\documentclass[11pt]{amsart}
\usepackage[leqno]{amsmath}
\usepackage{amssymb}
\usepackage{enumerate}
\usepackage{subfigure}
\usepackage{mathrsfs}
\usepackage{graphicx}
\usepackage{bm}
\usepackage[all]{xy}
\usepackage[usenames,dvipsnames]{xcolor}
\usepackage[pdftitle={Tightness of approximations to the chemical distance metric for simple conformal loop ensembles},
  pdfauthor={Jason Miller},
colorlinks=true,linkcolor=NavyBlue,urlcolor=RoyalBlue,citecolor=PineGreen,bookmarks=true,bookmarksopen=true,bookmarksopenlevel=2,unicode=true,linktocpage]{hyperref}

\usepackage{microtype}
\usepackage{centernot}
\usepackage{stmaryrd}
\usepackage{comment}

\numberwithin{equation}{section}
\numberwithin{figure}{section}
\newtheorem{theorem}{Theorem}[section]
\newtheorem{remark}[theorem]{Remark}
\newtheorem{lemma}[theorem]{Lemma}
\newtheorem{proposition}[theorem]{Proposition}

\newtheorem{definition}[theorem]{Definition}

\renewcommand{\C}{\mathbf{C}}
\newcommand{\D}{\mathbf{D}}
\newcommand{\E}{\mathbf{E}}
\newcommand{\F}{\mathbf{F}}
\newcommand{\h}{\mathbf{H}}
\newcommand{\N}{\mathbf{N}}
\newcommand{\Z}{\mathbf{Z}}
\newcommand{\p}{\mathbf{P}}

\newcommand{\R}{\mathbf{R}}

\newcommand{\Fh}{\mathfrak {h}}

\newcommand{\CA}{\mathcal {A}}
\newcommand{\CB}{\mathcal {B}}
\newcommand{\CC}{\mathcal {C}}
\newcommand{\CD}{\mathcal {D}}
\newcommand{\CE}{\mathcal {E}}
\newcommand{\CF}{\mathcal {F}}

\newcommand{\CL}{\mathcal {L}}

\newcommand{\CN}{\mathcal {N}}

\newcommand{\CQ}{\mathcal {Q}}

\newcommand{\CS}{\mathcal {S}}
\newcommand{\CT}{\mathcal {T}}

\newcommand{\CW}{\mathcal {W}}
\newcommand{\CX}{\mathcal {X}}
\newcommand{\CZ}{\mathcal {Z}}

\newcommand{\CG}{\mathcal {G}}
\newcommand{\CH}{\mathcal {H}}

\newcommand{\SLE}{{\rm SLE}}
\newcommand{\CLE}{{\rm CLE}}

\newcommand{\dist}{\mathrm{dist}}

\newcommand{\diam}{\mathrm{diam}}

\newcommand{\im}{\mathrm{Im}}
\newcommand{\re}{\mathrm{Re}}

\newcommand{\confrad}{{\rm CR}}

\newcommand{\one}{{\bf 1}}

\newcommand{\wt}{\widetilde}
\newcommand{\wh}{\widehat}
\newcommand{\ol}{\overline}
\newcommand{\ul}{\underline}

\newcommand{\giv}{\,|\,}

\newcommand{\BCLE}{\mathrm{BCLE}}
\newcommand{\ccwBCLE}{\BCLE^{\boldsymbol {\circlearrowleft}}}
\newcommand{\cwBCLE}{\BCLE^{\boldsymbol {\circlearrowright}}}

\newcommand{\strip} {\mathscr{S}}

\newcommand{\IG}{\mathrm{IG}}

\newcommand{\median}[1]{{\mathfrak m}_{#1}}
\newcommand{\medianHP}[1]{{\mathfrak m}_{#1}^{\mathrm{HP}}}
\newcommand{\quantHP}[2]{{{\mathfrak q}}_{#2}^{\mathrm{HP}}(#1)}

\newcommand{\exploreExp}{\alpha_{\mathrm{PERC}}}

\newcommand{\qmeasure}[1]{\mu_{#1}}
\newcommand{\qbmeasure}[1]{\nu_{#1}}
\newcommand{\qcarpet}[2]{\mu_{#1,#2}}
\newcommand{\lebneb}[1]{{\mathfrak N}_{#1}}

\newcommand{\Fd}{\mathfrak d}
\newcommand{\met}[3]{\Fd(#1,#2;#3)}
\newcommand{\metres}[4]{\Fd^{#1}(#2,#3;#4)}
\newcommand{\metplus}[3]{\Fd^+(#1,#2;#3)}
\newcommand{\metapprox}[4]{\Fd_{#1}(#2,#3;#4)}
\newcommand{\metapproxres}[5]{\Fd_{#1}^{#2}(#3,#4;#5)}

\newcommand{\dyad}{\mathsf{DyadDom}}

\newcommand{\distH}{d_{\mathrm H}}
\newcommand{\CK}{{\mathcal K}}
\newcommand{\funcmet}{{\mathbf d}}
\newcommand{\funcset}{\mathbf K}
\newcommand{\cp}{\mathrm{cap}}

\newcommand{\LBD}{\mathrm{LBD}}
\newcommand{\UBD}{\mathrm{UBD}}
\newcommand{\PP}{\mathrm{PP}}

\newcommand{\cwBoundary}[3]{[#1,#2]_{#3}^{\boldsymbol {\circlearrowright}}}
\newcommand{\cwBoundaryOpen}[3]{(#1,#2)_{#3}^{\boldsymbol {\circlearrowright}}}
\newcommand{\ccwBoundary}[3]{[#1,#2]_{#3}^{\boldsymbol {\circlearrowleft}}}
\newcommand{\ccwBoundaryOpen}[3]{(#1,#2)_{#3}^{\boldsymbol {\circlearrowleft}}}

\newcommand{\cpath}{\omega}

\newcommand{\qwedgeW}[2]{\mathsf{QWedge}_{\bm{\gamma}=#1}^{\mathbf{W}=#2}}
\newcommand{\qconeW}[2]{\mathsf{QCone}_{\bm{\gamma}=#1}^{\mathbf{W}=#2}}

\newcommand{\qdisk}[1]{\mathsf{QDisk}_{\bm{\gamma}=#1}}
\newcommand{\qdiskL}[2]{\mathsf{QDisk}_{\bm{\gamma}=#1}^{\bm{L}=#2}}
\newcommand{\qdiskCarpet}[2]{\mathsf{QDiskCLE}_{\bm{\gamma}=#1}^{\bm{L}=#2}}
\newcommand{\qdiskWeighted}[2]{\mathsf{QDiskW}_{\bm{\gamma}=#1}^{\bm{L}=#2}}

\newcommand{\net}{\mathrm{NET}}
\newcommand{\ball}{\mathrm{B}}
\newcommand{\KC}{\mathrm{KC}}
\newcommand{\HO}{\mathrm{HO}}

\begin{document}

\title[Tightness of the chemical distance metric for simple $\CLE$s]{Tightness of approximations to the chemical distance\\ metric for simple conformal loop ensembles}

\author{Jason Miller}

% arXiv abstract
%Suppose that $\Gamma$ is a conformal loop ensemble (CLE$_\kappa$) with simple loops ($\kappa \in (8/3,4)$) in a simply connected domain $D \subseteq {\mathbf C}$ whose boundary is itself a type of CLE$_\kappa$ loop.  Let $\Upsilon$ be the carpet of $\Gamma$, i.e., the set of points in $D$ not surrounded by a loop of $\Gamma$.  We prove that certain approximations to the chemical distance metric in $\Upsilon$ are tight.  More precisely, for each path $\omega \colon [0,1] \to \Upsilon$ and $\epsilon > 0$ we let ${\mathfrak N}_\epsilon(\omega)$ be the Lebesgue measure of the $\epsilon$-neighborhood of $\omega$.  For $z,w \in \Upsilon$ we let ${\mathfrak d}_\epsilon(z,w;\Gamma) = \inf_\omega {\mathfrak N}_\epsilon(\omega)$ where the infimum is over all paths $\omega \colon [0,1] \to \Upsilon$ with $\omega(0) = z$, $\omega(1) = w$ and let ${\mathfrak m}_\epsilon$ be the median of $\sup_{z,w \in \partial D} {\mathfrak d}_\epsilon(z,w;\Gamma)$.  We prove that $(z,w) \mapsto {\mathfrak m}_\epsilon^{-1} {\mathfrak d}_\epsilon(z,w;\Gamma)$ is tight and that any subsequential limit defines a geodesic metric on $\Upsilon$ which is H\"older continuous with respect to the Euclidean metric.  We conjecture that the subsequential limit is unique, conformally covariant, and describes the scaling limit of the chemical distance metric for discrete loop models which converge to CLE$_\kappa$ for $\kappa \in (8/3,4)$ such as the critical Ising model.

\begin{abstract}
Suppose that $\Gamma$ is a conformal loop ensemble ($\CLE_\kappa$) with simple loops ($\kappa \in (8/3,4)$) in a simply connected domain $D \subseteq \C$ whose boundary is itself a type of $\CLE_\kappa$ loop.  Let $\Upsilon$ be the carpet of $\Gamma$, i.e., the set of points in $D$ not surrounded by a loop of $\Gamma$.  We prove that certain approximations to the chemical distance metric in $\Upsilon$ are tight.  More precisely, for each path $\cpath \colon [0,1] \to \Upsilon$ and $\epsilon > 0$ we let $\lebneb{\epsilon}(\cpath)$ be the Lebesgue measure of the $\epsilon$-neighborhood of $\cpath$.  For $z,w \in \Upsilon$ we let $\metapprox{\epsilon}{z}{w}{\Gamma} = \inf_\cpath \lebneb{\epsilon}(\cpath)$ where the infimum is over all paths $\cpath \colon [0,1] \to \Upsilon$ with $\cpath(0) = z$, $\cpath(1) = w$ and let $\median{\epsilon}$ be the median of $\sup_{z,w \in \partial D} \metapprox{\epsilon}{z}{w}{\Gamma}$.  We prove that $(z,w) \mapsto \median{\epsilon}^{-1} \metapprox{\epsilon}{z}{w}{\Gamma}$ is tight and that any subsequential limit defines a geodesic metric on $\Upsilon$ which is H\"older continuous with respect to the Euclidean metric.  We conjecture that the subsequential limit is unique, conformally covariant, and describes the scaling limit of the chemical distance metric for discrete loop models which converge to $\CLE_\kappa$ for $\kappa \in (8/3,4)$ such as the critical Ising model.
\end{abstract}

\date{\today}
\maketitle

\setcounter{tocdepth}{1}

\parindent 0 pt
\setlength{\parskip}{0.20cm plus1mm minus1mm}

\tableofcontents

\section{Introduction}
\label{sec:intro}

\subsection{Overview}
\label{subsec:overview}

The conformal loop ensemble $\CLE_\kappa$ \cite{s2009cle,sw2012cle} is the canonical conformally invariant measure on loops in a simply connected domain $D \subseteq \C$.  It is the loop version of the Schramm-Loewner evolution ($\SLE_\kappa$) \cite{s2000sle}.  It consists of a countable collection of loops in $D$ each of which looks locally like an $\SLE_\kappa$ curve where $\kappa \in (8/3,8)$.  Just like for $\SLE_\kappa$, the loops are simple and do not intersect each other or the domain boundary for $\kappa \in (8/3,4]$ while for $\kappa \in (4,8)$ the loops are self-intersecting and intersect the domain boundary.  Its importance is that it describes (or is conjectured to describe) the scaling limit of all of the interfaces for a number of discrete models from statistical mechanics in two dimensions.  For example, on deterministic planar lattices, $\CLE_3$ was shown to be the scaling limit of the interfaces in the Ising model \cite{2014ising,bh2019ising}, $\CLE_{16/3}$ the interfaces in the FK-Ising model \cite{ks2019fkising,s2010ising}, $\CLE_6$ the interfaces in the percolation model \cite{s2001percolation,cn2006comm}, and $\CLE_8$ the Peano curve associated with the uniform spanning tree (UST) \cite{lsw2004lerwust}.

The purpose of this article is to initiate the program of constructing a conformally covariant metric associated with a $\CLE_\kappa$.  We will focus on the case that $\kappa \in (8/3,4)$ so that the $\CLE_\kappa$ loops are simple and do not intersect each other.  Suppose that~$D \subseteq \C$ is a simply connected domain and let~$\Gamma$ be a $\CLE_\kappa$ in $D$.  Let~$\Upsilon$ be the carpet of~$\Gamma$, i.e., the set of $z \in D$ which are not surrounded by any loop of~$\Gamma$.  For each path $\cpath \colon [0,1] \to \Upsilon$ and $\epsilon > 0$, we let $\lebneb{\epsilon}(\cpath)$ be the Lebesgue measure of the $\epsilon$-neighborhood of $\cpath$.  For $z,w \in \Upsilon$, we let $\metapprox{\epsilon}{z}{w}{\Gamma} = \inf_\cpath \lebneb{\epsilon}(\cpath)$ where the infimum is over all paths $\cpath \colon [0,1] \to \Upsilon$ with $\cpath(0) = z$, $\cpath(1) = w$.  Finally, we let $\median{\epsilon}$ be the median of the random variable $\sup_{z,w \in \partial D} \metapprox{\epsilon}{z}{w}{\Gamma}$.  Our aim is to show that $(z,w) \mapsto \median{\epsilon}^{-1} \metapprox{\epsilon}{z}{w}{\Gamma}$ is tight and that any subsequential limit defines a geodesic metric on $\Upsilon$ which is H\"older continuous with respect to the Euclidean metric.  

We note that for $z,w \in \Upsilon$ which are away from $\partial D$, it is natural to expect that the ``shortest'' path in~$\Upsilon$ from $z$ to $w$ is a fractal curve and that there exists $\alpha > 1$ so that the median of $\sup_{z,w \in \Upsilon} \metapprox{\epsilon}{z}{w}{\Gamma}$ is $\epsilon^{2-\alpha+o(1)}$ as $\epsilon \to 0$.  If $\partial D$ is smooth and $z,w \in \partial D$, then one could take the path which follows $\partial D$ from $z$ to $w$ hence in this case $\median{\epsilon} = \epsilon^{1+o(1)}$.  This means that if $\partial D$ is smooth then $\median{\epsilon}$ is the wrong normalization to get a finite metric on $\Upsilon$.  For this reason, we will state our main result for a domain with rough boundary.  In particular, we will take~$\partial D$ to be itself a $\CLE_\kappa$ loop.

\subsection{Main result}
\label{subsec:main_result}

The main result of this paper is the tightness of $\median{\epsilon}^{-1} \metapprox{\epsilon}{\cdot}{\cdotp}{\Gamma}$, which we note is a function defined on a random set.  Let us now specify with respect to which topology we will prove tightness.

We let $\funcset_n$ be the set of pairs consisting of a compact subset $K \subseteq \R^n$ and a function $f \colon K \to \R$.  We define a metric on $\funcset_n$ as follows.  Let~$\distH$ denote the Hausdorff distance on compact subsets of~$\R^n$.  For $(f,K), (g,L) \in \funcset_n$ we let
\[ \funcmet_{0,n}( (f,K),(g,L)) = \sup\{ |f(z)-g(w)| : z \in K, w \in L,\ |z-w| \leq \distH(K,L)\}.\]
We note that $\funcmet_{0,n}$ defines a pseudometric on $\funcset_n$ but not a metric.  Indeed, if $K$, $L$ are distinct compact subsets of $\R^n$ and $f,g$ are the zero function then $\funcmet_{0,n}((f,K),(g,L)) = 0$.  To define a metric, we thus set
\[ \funcmet_n((f,K),(g,L)) = \funcmet_{0,n}((f,K),(g,L)) + \distH(K,L).\]
Fix $R,S > 0$.  Suppose that $\F$ is a family in $\funcset_n$ so that:
\begin{enumerate}[(i)]
\item for every $\epsilon > 0$ there exists $\delta > 0$ so that for every $(f,K) \in \F$, $z,w \in K$, and $|z-w| < \delta$ we have that $|f(z)-f(w)| < \epsilon$ and
\item for every $(f,K) \in \F$ we have that $K \subseteq B(0,R)$ and $\sup_{z \in K} |f(z)| \leq S$.
\end{enumerate}
Then the closure of $\F$ is compact in $(\funcset_n,\funcmet_n)$.  Indeed, suppose that $((f_n,K_n))$ is a sequence in $\F$.  Then there exists a subsequence $((f_{n_j},K_{n_j}))$ so that $(K_{n_j})$ converges with respect to $\distH$.  The usual proof of the Arzela-Ascoli theorem then implies that, by taking a further subsequence if necessary, there exists $(f,K) \in \F$ so that $\funcmet_{0,n}((f_{n_j},K_{n_j}),(f,K)) \to 0$ as $j \to \infty$.

The main argument used in this paper works whenever $\partial D$ locally looks like an $\SLE_\kappa$ curve, though we will state our main result with a particular choice for the sake of concreteness (see Figure~\ref{fig:setup_sim} for a numerical simulation of the setup).

\begin{figure}
\begin{center}
\includegraphics[width=0.49\textwidth]{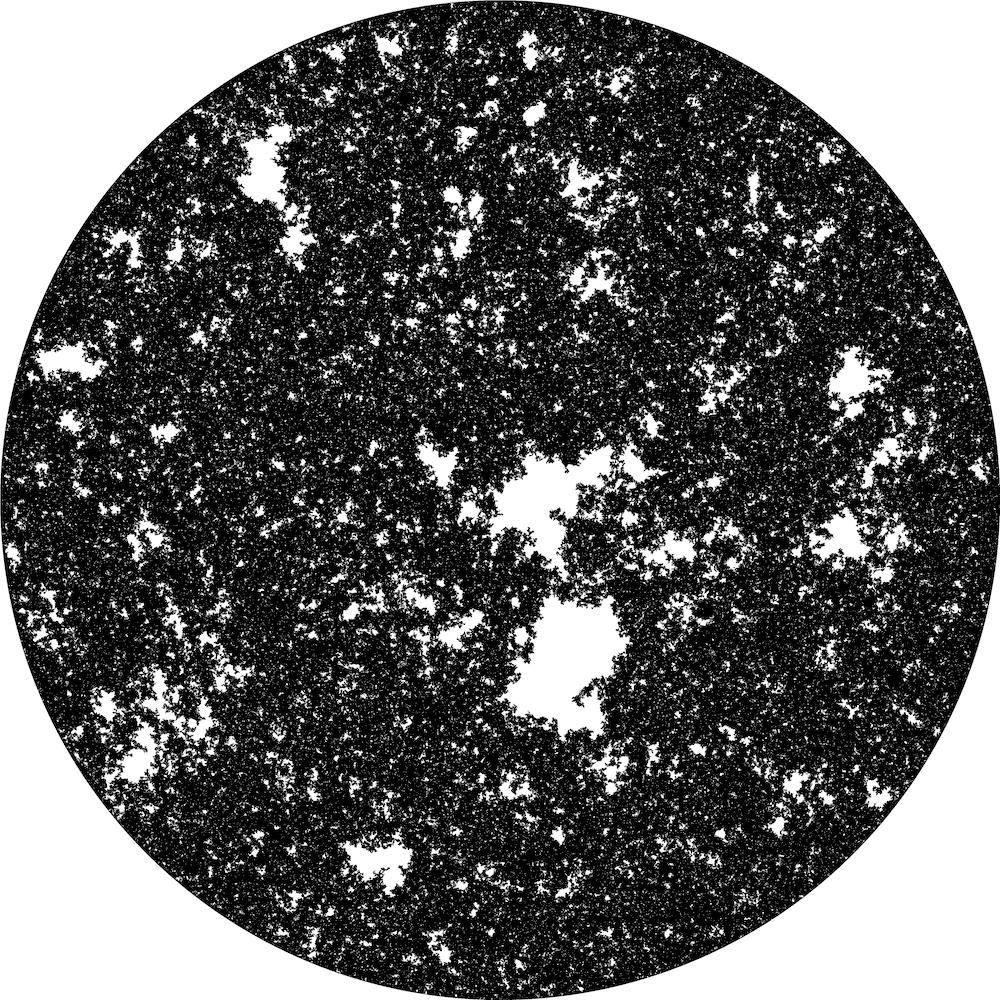}	
\hspace{0.005\textwidth}
\includegraphics[width=0.49\textwidth]{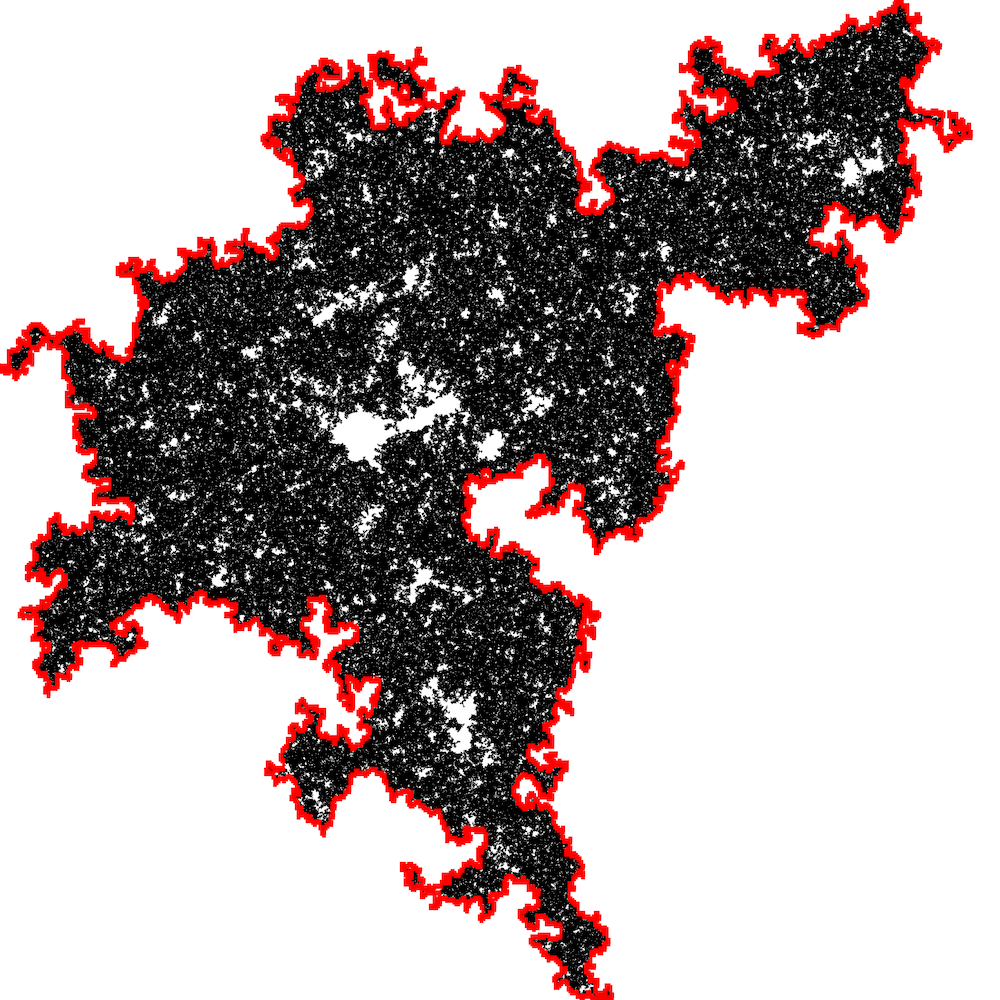}	

\vspace{0.01\textheight}

\includegraphics[width=0.7\textwidth]{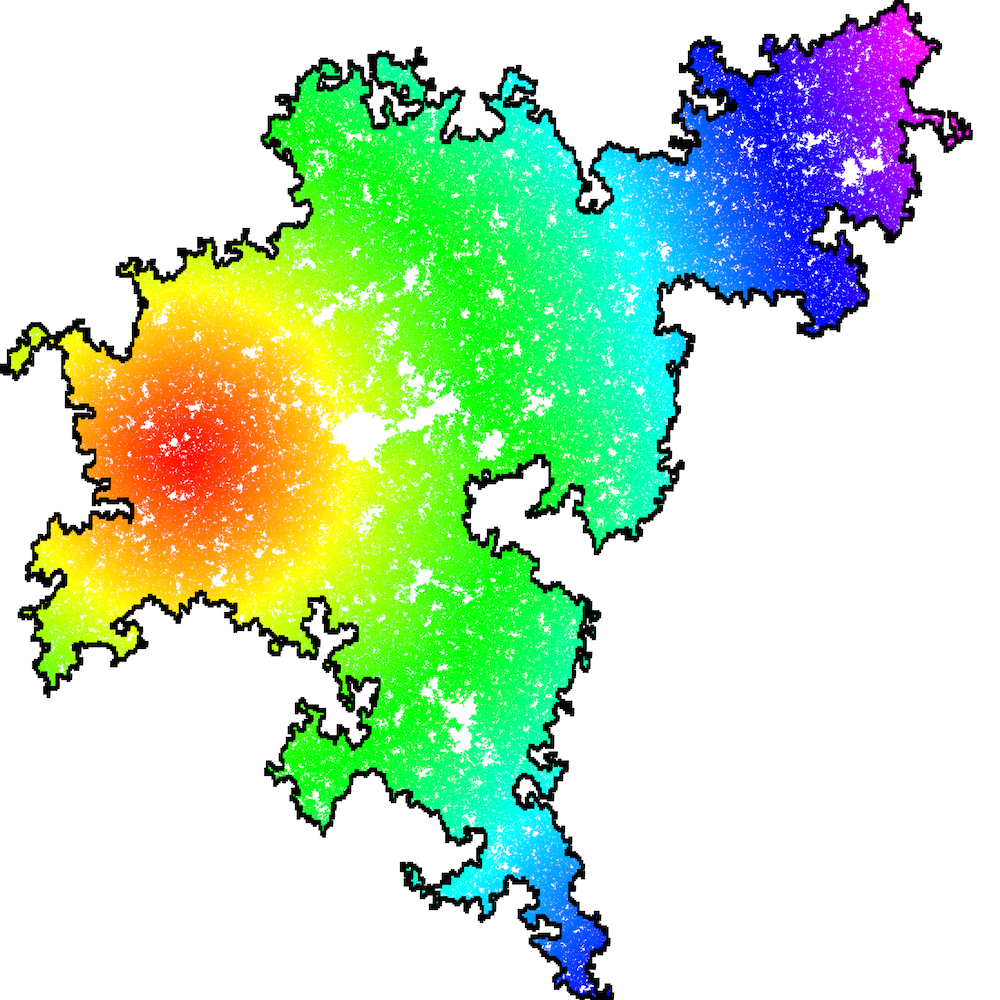}	
\end{center}
\caption{\label{fig:setup_sim} {\bf Top left:} A $\CLE_3$ on the unit disk.  {\bf Top right:} The loop which surrounds the origin from the top left (red) with an independent $\CLE_3$ in the domain surrounded by it.  {\bf Bottom:}  The top right with the $\CLE_3$ carpet colored according to the chemical distance from a uniformly chosen point in the carpet; the closest points are red and the farthest points are magenta.}
\end{figure}

\begin{theorem}
\label{thm:cle_loop}
Fix $\kappa \in (8/3,4)$.  Suppose that $\Gamma_\D$ is a $\CLE_\kappa$ in the unit disk $\D$ and let $D$ be the set of points surrounded by the loop $\CL \in \Gamma_\D$ which surrounds $0$.  Given $\CL$, let $\Gamma$ be a $\CLE_\kappa$ in $D$.  Let $\median{\epsilon}$ be the median of $\sup_{z,w \in \partial D} \metapprox{\epsilon}{z}{w}{\Gamma}$.  Then the law of the map $(z,w) \mapsto \median{\epsilon}^{-1} \metapprox{\epsilon}{z}{w}{\Gamma}$ is tight in the space $(\funcset_4,\funcmet_4)$.  Any subsequential limit is a.s.\ a geodesic metric on $\Upsilon$ which is H\"older continuous with respect to the Euclidean metric.
\end{theorem}

We conjecture that the limit in Theorem~\ref{thm:cle_loop} exists and is characterized by a certain list of axioms, as in the case of the Liouville quantum gravity metric \cite{mq2020geodesics,gm2021uniqueness}.  We plan to address this in future work.  We also conjecture that the limit describes the scaling limit of the chemical distance metric for discrete models which converge to $\CLE$, such as the critical Ising model \cite{bh2019ising} ($\kappa=3$).

\subsection{Outline and proof strategy}
\label{subsec:outline}

Fix $\kappa \in (8/3,4)$.  In proving Theorem~\ref{thm:cle_loop}, we will often work in an infinite volume setup which describes the local behavior of $\CL$ near a ``typical'' point in $\CL$.  That is, we will replace $\CL$ by a so-called two-sided whole-plane $\SLE_\kappa$ process $\eta$ and we let $\Gamma_+$ be a $\CLE_\kappa$ in the component of $\C \setminus \eta$ which is to the left of $\eta$.  Rather than normalize using the median and quantiles associated with the setup described in Theorem~\ref{thm:cle_loop}, we will instead use the median and quantiles associated with this infinite volume setup and then at the very end of the proof show that they are comparable.  In order to distinguish between these two situations, we will use the superscript HP to denote a quantity associated with the infinite volume (half-planar) setup and will not use a superscript when referring to a quantity associated with the setup described in the statement of Theorem~\ref{thm:cle_loop}.

The basic strategy to prove Theorem~\ref{thm:cle_loop} is the following.  For each $\epsilon > 0$ and $p \in (0,1)$ we let $\quantHP{p}{\epsilon}$ denote the $p$th quantile of $\sup_{z,w \in \eta([0,1])} \metapprox{\epsilon}{z}{w}{\Gamma_+}$ and let $\medianHP{\epsilon} = \quantHP{1/2}{\epsilon}$ be the median.  Note that for any $M_0 \geq 1$ there exists $M_1 \geq 1$ so that
\begin{equation}
\label{eqn:quant_comparison}
\frac{1}{M_1} \quantHP{p}{M_0 \epsilon} \leq \quantHP{p}{\epsilon} \leq M_1 \quantHP{p}{\epsilon/M_0}.	
\end{equation}
That is, replacing $\epsilon$ with a fixed multiple of $\epsilon$ changes the quantiles by at most a factor which does not depend on $\epsilon$.  (See also Lemma~\ref{lem:covering_lemma} for a more quantitative version of~\eqref{eqn:quant_comparison}.)

The proof is divided into two main steps:
\begin{itemize}
\item Establish the comparability of the quantiles $\quantHP{p}{\epsilon}$ uniformly in $\epsilon > 0$ and prove the tightness of $(\medianHP{\epsilon})^{-1} \metapprox{\epsilon}{\cdot}{\cdot}{\Gamma}$ in $(\funcset_4,\funcmet_4)$.  We emphasize here that the tightness result will be proved in the finite volume setting but using the normalization from the median defined in the infinite volume setting.
\item Show that every subsequential limit is a.s.\ positive definite hence a metric.
\end{itemize}
The geodesic property of a subsequential limit and the comparability of the infinite volume and finite volume quantiles will not be difficult to establish once we have accomplished both of the steps above.

Let $\kappa' = 16/\kappa \in (4,6)$.  In \cite{msw2017clepercolations}, it was shown how one can view an $\SLE_{\kappa'}(\kappa'-6)$ process $\eta'$ as a \emph{conformal percolation interface} (CPI) inside of the carpet $\Upsilon$ of $\Gamma$ which always keeps the loops of $\Gamma$ which it hits on its right side.  Moreover, the process obtained by following the loops of $\Gamma$ which are hit by $\eta'$ is an $\SLE_\kappa(\kappa-6)$ process $\eta$.  More generally, if we fix $\rho' \in [\kappa'-6,0]$ then there exists a value of $\beta \in [-1,1]$ so that we can couple an $\SLE_{\kappa'}(\rho';\kappa-6-\rho')$ process with $\Upsilon$ so that it keeps each loop of $\Gamma$ that it hits on its left (resp.\ right) side with probability $(1-\beta)/2$ (resp.\ $(1+\beta)/2$) and the process obtained by following the loops which are hit by $\eta'$ is an $\SLE_\kappa^\beta(\kappa-6)$ process $\eta$.  We will often abuse notation and refer to $\eta'$ together with the loops of $\Gamma$ that it hits (i.e., $\eta$) also as a CPI.

CPIs give us a way to ``explore'' the $\CLE_\kappa$ carpet in a Markovian way so that the conditional law of the loops of $\Gamma$ which are in the components of $D \setminus \eta([0,t])$ not surrounded by a loop of $\Gamma$ are independent $\CLE_\kappa$'s.  This means that we can explore further into each such component using a CPI which can also start from any point on the component boundary.  This type of exploration will be the basis for a spatial decomposition of the $\CLE_\kappa$ into ``chunks''.  Since the \emph{shape} of a domain strongly affects the chemical distance metric (e.g., as we explained earlier if the domain boundary is smooth then it is natural to expect that points along the boundary will be close to each other), knowing the shape of the successive chunks of the exploration will restrict the form of the chemical distance metric in the unexplored region.  We will thus want to consider a type of exploration which avoids revealing the exact shape of the explored region as much as possible.  For this reason, we will explore the $\CLE_\kappa$ using CPIs in a ``quantum'' manner using the relationship between $\SLE$ and Liouville quantum gravity (LQG) developed in \cite{s2016zipper,dms2014mating} as well as the relationship between $\CLE$ and LQG developed in \cite{msw2020simplecle}.  The reason for this is that exploring a $\CLE$ in a quantum manner is analogous to exploring a random planar map using a peeling process and in such an exploration the regions which are cut out depend only on the exploration through their boundary length.  In particular, conditioning on chunks which have been explored this way means that we do not have to condition on their actual shape -- only on how they are glued together.

We will call a chunk ``good'' if it is possible to traverse it with a path $\cpath$ with $\lebneb{\epsilon}(\cpath) \leq \quantHP{p}{\epsilon}$ and ``bad'' if it is not.  We will prove that by making the chunk size sufficiently small (but with $p \in (0,1)$ fixed), we can make the probability that a chunk is good as close to $1$ as we wish uniformly in $\epsilon \to 0$.  We will explain these estimates in Section~\ref{sec:chunk_estimates}.  We emphasize that this is the particular point in the proof where we make use of the normalization by $\medianHP{\epsilon}$ instead of making a choice such as by the median length of a crossing $\omega$ which minimizes $\lebneb{\epsilon}(\omega)$ among those which connect two disjoint boundary arcs.  The reason is that we will show that the probability that the diameter of the set of points $z \in \Upsilon_+$, $\Upsilon_+$ the carpet of $\Gamma_+$, with $\metapprox{\epsilon}{0}{z}{\Gamma_+} \leq \quantHP{p}{\epsilon}$ is at least $\delta$ tends to $1$ as $\delta \to 0$ provided $\epsilon > 0$ is sufficiently small relative to $\delta$.  The reason that this is the case is that $\eta([0,1])$ enjoys certain self-similarity properties.  In particular, if there was a positive chance that the aforementioned set was arbitrarily small with positive probability then that would have to be very likely to be true somewhere along $\eta([0,1])$ which in turn would contradict the definition of $\quantHP{p}{\epsilon}$.

We will show that the good vs bad chunks can be viewed as a type of strongly supercritical percolation Section~\ref{sec:percolation_exploration}, though in a more general framework.  In particular, we will show that it is likely that two points on the domain boundary are connected by a collection of such chunks.  The initial estimate that we will obtain here will not have sufficiently good concentration in order to deduce the tightness of $(\medianHP{\epsilon})^{-1} \metapprox{\epsilon}{\cdot}{\cdot}{\Gamma}$ because the exploration that we will consider has the possibility of hitting obstacles such as large $\CLE_\kappa$ loops.   In order to obtain stronger concentration, we will consider an adaptive version which circumvents these obstacles.

Using the definition of good vs bad chunks from Section~\ref{sec:chunk_estimates} and the adaptive percolation exploration from Section~\ref{sec:percolation_exploration}, we will show in Section~\ref{sec:boundary_tightness} that it is possible to string together the crossings between good chunks in order to deduce that with high probability we can connect two boundary points by a path $\cpath$ with $\lebneb{\epsilon}(\cpath)$ at most a constant times $\quantHP{p}{\epsilon}$.  The adaptive version of the exploration will give sufficiently good concentration so that we can establish the tightness of $(\medianHP{\epsilon})^{-1} \metapprox{\epsilon}{\cdot}{\cdot}{\Gamma}$ on $\partial D$ and deduce the comparability of the quantiles~$\quantHP{p}{\epsilon}$.

In Section~\ref{sec:interior_tightness} we then extend the tightness of $(\medianHP{\epsilon})^{-1} \metapprox{\epsilon}{\cdot}{\cdot}{\Gamma}$ from the boundary to the interior.  The first step is to define collections of points $P_j$ in $\Upsilon$ which consist of i.i.d.\ samples from the natural LQG measure on $\Upsilon$ constructed in \cite{msw2020simplecle}.  We will show in Appendix~\ref{app:carpet_measure} that there exists $\alpha_\net > 0$ so that if $|P_j| = 2^{\alpha_\net j}$ then with overwhelming probability $P_j$ forms a $2^{-j}$-net of $\Upsilon$ (with respect to the Euclidean metric).  We then aim to show that there exists $\beta > 0$ so that if $z, w \in P_j$ have Euclidean distance at most $2^{-j}$ then $(\medianHP{\epsilon})^{-1} \metapprox{\epsilon}{z}{w}{\Gamma}$ is very likely to be at most $2^{-\beta j}$.  Proving this statement will require a number of estimates.  For example, we will have to rule out the possibility that such points are separated from each other by a loop of $\Gamma$ with diameter significantly larger than $2^{-j}$.  Upon proving this, the remainder of the proof of tightness follows along the proof of the Kolmogorov-Centsov continuity criterion where the role of the dyadic decomposition in the latter is played by the sets $P_j$ in the former.

In the statement of Theorem~\ref{thm:cle_loop}, it is important to note that the normalization is by the median of $\sup_{z,w \in \eta([0,1])} \metapprox{\epsilon}{z}{w}{\Gamma_+}$ as opposed to the median length of a crossing which connects two disjoint boundary arcs.  In particular, due to this choice it is not obvious that a subsequential limit is positive definite (the other properties of being a metric will at this point be obvious due to how $\metapprox{\epsilon}{\cdot}{\cdot}{\Gamma}$ is defined).  Establishing the positive definiteness will be the focus of Section~\ref{sec:pos_def}.  The first observation is that a subsequential limit $\met{\cdot}{\cdot}{\Gamma}$ cannot be completely degenerate (i.e., vanish) because this would contradict the comparability of the quantiles $\quantHP{p}{\epsilon}$ uniformly in $\epsilon > 0$ which we have discussed earlier.  In particular, this implies that if we take two disjoint arcs $I, J$ on $\partial D$ then it is a positive probability event that the $\met{I}{J}{\Gamma} > 0$.  On this event, let $K = \{ z \in D : \met{z}{I}{\Gamma} = 0\}$ and let $U$ be the component of $D \setminus K^*$ with $J$ on its boundary.  Then $\partial U \setminus \partial D$ has the property that there cannot exist $z \in U$ and $w \in K$ with $\met{z}{w}{\Gamma} = 0$ (for otherwise $z \in K$).  Thus $D \cap \partial U$ can be thought of as a type of ``shield'' through which a path with ``$\met{\cdot}{\cdot}{\Gamma}$-length zero'' cannot pass.  The thrust of the argument in Section~\ref{sec:pos_def} is to use that such shields exist with positive probability and independence to show that there a.s.\ exists a family of such shields which separate~$I$ from~$J$ and therefore $\met{I}{J}{\Gamma} > 0$ a.s. Since $I$, $J$ were arbitrary boundary intervals, this will imply that $\met{\cdot}{\cdot}{\Gamma}$ is positive definite on $\partial D$.  Finally, using a CPI exploration, we will argue that if it is a positive probability event that exist $z,w \in \Upsilon \setminus \partial D$ distinct with $\met{z}{w}{\Gamma} = 0$ then it is also a positive probability event that there exists $z,w \in \partial D$ distinct with $\met{z}{w}{\Gamma} = 0$.  Altogether, this gives the positive definiteness of $\met{\cdot}{\cdot}{\Gamma}$ everywhere.  At the end of Section~\ref{sec:pos_def}, we will explain why the subsequential limit $\met{\cdot}{\cdot}{\Gamma}$ is a.s.\ geodesic and deduce the comparability of the quantiles defined in infinite volume and finite volume settings.

In Appendix~\ref{app:mod_of_cont}, we will give an upper bound for the modulus of continuity of the space-filling version of $\SLE$ from \cite{ms2017ig4} when drawn on top of an independent quantum disk and parameterized by quantum area.  In Appendix~\ref{app:carpet_measure} we will collect some estimates for the natural measure on the carpet of a $\CLE_\kappa$ which comes from LQG constructed in \cite{msw2020simplecle}.    Finally, in Appendix~\ref{app:levy_process}, we will collect some estimates for L\'evy processes.

\subsection{Relationship with other work}
\label{subsec:other_work}

Let us make some comments on the relationship between this work and some recent works focused on the construction of the LQG metric.  Recall that LQG refers to the random two-dimensional Riemannian manifold whose metric tensor is (formally) given by $e^{\gamma h(z)}(dx^2 + dy^2)$ where $h$ is an instance of (some form of) the Gaussian free field (GFF) on a planar domain $D$.  Since $h$ is a distribution and not a function, this expression does not make literal sense and requires interpretation.  The rigorous construction of the two-point distance function associated with LQG was first accomplished in the case that $\gamma=\sqrt{8/3}$ in \cite{ms2020qle1,ms2016qle2,ms2021qle3} and then subsequently (using a different method) in \cite{dddf2020tightness,gm2021uniqueness} for all $\gamma \in (0,2)$.  In the program of constructing a metric in the $\CLE_\kappa$ carpet for $\kappa \in (8/3,4)$, the present work plays the role of \cite{dddf2020tightness} in the program of constructing the LQG metric for $\gamma \in (0,2)$.

The argument that we will employ to establish tightness here will be rather different to that given in \cite{dddf2020tightness}, however.  Let us briefly highlight a few of the main differences.  First of all, the heart of \cite{dddf2020tightness} is to consider a certain mollification $h_\epsilon$ of the GFF $h$ and then to prove the (uniform in $\epsilon$) comparability of quantiles for lengths of side-to-side crossings in a box measured using the metric $e^{\xi h_\epsilon}$ where $\xi > 0$ is determined by $\gamma$.  Establishing this comparability in \cite{dddf2020tightness} is based on decomposing the larger box into smaller boxes and then using the Efron-Stein inequality to get a recursive relationship in the quantiles for different values of $\epsilon > 0$.  Once the comparability of side-to-side quantiles is accomplished, a standard sort of Kolmogorov-Centsov argument makes it possible to string together paths with length comparable to the median length crossing and deduce the tightness of the diameter and indeed the whole metric.  Moreover, the positive definiteness of the subsequential limit is immediate because if there were two points with zero distance in the limit then there would have to be some box across which the quantiles are not all comparable.

In this work, we will instead make the choice of looking at quantiles for the diameter of the whole boundary associated with our approximations.  We will then get comparability of the quantiles of the boundary diameter while at the same time proving that the subsequential limit is continuous with respect to the Euclidean topology.  The argument that we will employ does not use the Efron-Stein inequality or a recursion in order to compare the approximations for different values of $\epsilon$.  Moreover, stringing together short paths in \cite{dddf2020tightness} follows the usual dyadic subdivision as in the proof of the Kolmogorov-Centsov continuity criterion.  In the present work, stringing together short paths is rather delicate because of the presence of obstacles such as loops.  Finally, because we are normalizing our approximations by the median boundary diameter, it is not immediate that the subsequential limit is positive definite.

Next, we mention that the $\CLE_8$ metric was constructed in \cite{hs2018euclidean}, though the nature of the problem is rather different than for $\kappa \in (8/3,8)$.  Recall that $\CLE_8$ is the scaling limit of the UST \cite{lsw2004lerwust} so defining the metric in this case corresponds to properly assigning a length to the continuous analog of the UST branches.  These branches are $\SLE_2$-type curves \cite{ms2016ig1,ms2017ig4} and the appropriate notion of length to consider is the so-called natural parameterization of $\SLE$ \cite{ls2011natural,lr2015natural,ben2018gffnatural}.

Finally, we mention the work \cite{lgm2011largefaces} which studies random planar maps where the face sizes have law which is in the domain of attraction of an $\alpha$-stable distribution and it is shown that the resulting metric space properly renormalized is tight in the Gromov-Hausdorff topology.  The subsequentially limiting metric constructed in \cite{lgm2011largefaces} should correspond to a quantum version (i.e., where the lengths of paths in the approximations are defined using a quantity from LQG rather than a Euclidean quantity such as Lebesgue measure) of our main result, but we will not focus on that in the present paper.

\subsection{Notation}
 
\newcommand{\interior}[1]{\mathrm{int}(#1)}
\newcommand{\closure}[1]{\mathrm{cl}(#1)}

We will make use of the following notation in this article.  We will write $\D$ for the unit disk and $\h$ for the upper half-plane.  If $f, g$ are two functions, then we will write $f \lesssim g$ (resp.\ $f \gtrsim g$) if there exists a constant $c > 0$ so that $f \leq c g$ (resp.\ $f \geq c g$).  For a simply connected domain $D \subseteq \C$ and prime ends $a,b \in \partial D$, we let $\ccwBoundary{a}{b}{\partial D}$ (resp.\ $\cwBoundary{a}{b}{\partial D}$) denote the counterclockwise (resp.\ clockwise) arc of prime ends in $\partial D$ from $a$ to $b$.  For a set $A$, we let $\interior{A}$ (resp.\ $\closure{A}$) denote the interior (resp.\ closure) of $A$.

\subsection*{Acknowledgements}  This research was supported by the ERC starting grant 804166 (SPRS).  We thank Wendelin Werner for helpful discussions related to this work.  We also thank Naotaka Kajino for helpful discussions on related work.

\section{Preliminaries}
\label{sec:preliminaries}

The purpose of this section is to collect a number of preliminaries.  We will start by reviewing some facts about $\SLE$ in Section~\ref{subsec:sle}.  We will then review $\CLE$ in Section~\ref{subsec:cle}.  Next we will review some relevant results from imaginary geometry \cite{ms2016ig1,ms2017ig4} in Section~\ref{subsec:ig}.  Finally, we will review some basics of Liouville quantum gravity (LQG) and its relationship to $\SLE$ in Section~\ref{subsec:lqg}.  We assume that the reader has some familiarity with these objects and we will provide references for the various facts that we will use.

\subsection{Schramm-Loewner evolution}
\label{subsec:sle}

The Schramm-Loewner evolution ($\SLE_\kappa$) is a random fractal curve in a simply connected domain and was introduced by Schramm in \cite{s2000sle}.  The chordal version corresponds to a curve which connects two boundary points.  In the case that the domain is $\h$ and the curve goes from $0$ to $\infty$, it is defined by solving the chordal Loewner equation
\begin{equation}
\label{eqn:loewner}
\partial_t g_t(z) = \frac{2}{g_t(z) - W_t},\quad g_0(z) = z	
\end{equation}
where $W_t = \sqrt{\kappa} B_t$ and $B$ is a standard Brownian motion.  Let $\h_t$ be the domain of $g_t$ and $K_t = \h \setminus \h_t$.  In the case that $\kappa \neq 8$, it was shown by Rohde-Schramm \cite{rs2005basic} that there a.s.\ exists a curve $\eta$ so that $\h_t$ is equal to the unbounded component of $\h \setminus \eta([0,t])$.  The continuity in the case that $\kappa = 8$ follows from the convergence of the uniform spanning tree Peano curve to $\SLE_8$ established in \cite{lsw2004lerwust}.  If $D \subseteq \C$ is another simply connected domain (i.e., not $\h$) and $x,y \in \partial D$ are distinct, then an $\SLE_\kappa$ in $D$ from $x$ to $y$ is defined as $\varphi(\eta)$ where $\eta$ is an $\SLE_\kappa$ in $\h$ from $0$ to $\infty$ and $\varphi \colon \h \to \D$ is a conformal map which takes $0$ (resp.\ $\infty$) to $x$ (resp.\ $y$).

The so-called $\SLE_\kappa(\ul{\rho})$ processes are an important variant of $\SLE_\kappa$ in which one has extra marked points, say at $x_1,\ldots,x_n$, associated with weights $\ul{\rho} = (\rho_1,\ldots,\rho_n) \in \R^n$.  The $x_j$ for $1 \leq j \leq n$ are sometimes called force points and the associated weight $\rho_j$ determines how the curve interacts with $x_j$.  In particular, if $\rho_j$ is positive (resp.\ negative) then the $\SLE_\kappa(\ul{\rho})$ is pushed away from (resp.\ pulled towards) $x_j$.  These processes were first introduced in \cite{lsw2003confres} and are defined by solving~\eqref{eqn:loewner} where $W$ is taken to be the solution to the SDE
\begin{align*}
 dW_t = \sqrt{\kappa} dB_t + \sum_{i=1}^n \frac{\rho_i}{W_t - V_t^i} dt,\quad
 dV_t^i = \frac{2}{V_t^i - W_t} dt,\quad V_0^i = x_i\quad\text{for}\quad i=1,\ldots,n.
\end{align*}
Up until the first time that $W$ collides with one of the $V^i$, the law of an $\SLE_\kappa(\ul{\rho})$ is absolutely continuous with respect to the law of an $\SLE_\kappa$ hence corresponds to a continuous curve.  More generally, the standard $\SLE_\kappa(\ul{\rho})$ processes are defined up to the so-called continuation threshold \cite{ms2016ig1}: the first time $t$ that $\sum_{i=1}^n \rho_i \one_{W_t = V_t^i} \leq -2$.  The continuity of the $\SLE_\kappa(\ul{\rho})$ processes up to the continuation threshold was proved in \cite{ms2016ig1}.  One can also consider $\SLE_\kappa(\rho)$ processes with $-2-\kappa/2 < \rho < -2$, though in this case their behavior is rather than different.  In particular, an $\SLE_\kappa(\rho)$ process with $-2-\kappa/2 < \rho < -2$ is self-intersecting for $\kappa \in (0,4]$ while standard $\SLE_\kappa$ and $\SLE_\kappa(\rho)$ curves are simple.  The continuity in the case $-2-\kappa/2 < \rho < -2$ was proved in \cite{msw2017clepercolations,ms2019lightcone} in the presence of a single force point.  We will sometimes use the notation $\SLE_\kappa(\ul{\rho}_L;\ul{\rho}_R)$ to indicate an $\SLE_\kappa(\ul{\rho})$ process where $\ul{\rho}_L$ (resp.\ $\ul{\rho}_R$) indicates the weights of the of the force points which are to the left (resp.\ right) of $0$.  An $\SLE_\kappa(\ul{\rho})$ connecting two boundary points in a simply connected domain is defined as the conformal image of an $\SLE_\kappa(\ul{\rho})$ in $\h$ from $0$ to $\infty$.

Let us mention some of the special ranges of $\rho$ values.  Suppose that $\eta$ is an $\SLE_\kappa(\ul{\rho}_L;\ul{\rho}_R)$ with force points located at $(\ul{x}_L;\ul{x}_R)$.
\begin{itemize}
\item If $\sum_{i=1}^k \rho_{i,R} \geq \kappa/2-2$ then $\eta$ a.s.\ does not hit $[x_{i,R},x_{i+1,R})$ (and likewise with $L$ in place of~$R$).
\item If $\sum_{i=1}^k \rho_{i,R} \in (-2,\kappa/2-2)$ then $\eta$ can hit $[x_{i,R},x_{i+1,R})$ (and likewise with $L$ in place of $R$).
\item In the case $\kappa' > 4$ and $\sum_{i=1}^k \rho_{i,R} \in (-2,\kappa'/2-4]$ then with positive probability $\eta'$ fills an interval of the form $[x_{i,R},b)$ for $b \in (x_{i,R},x_{i+1,R})$ (and likewise with $L$ in place of $R$).
\end{itemize}

We can also consider whole-plane $\SLE_\kappa$ processes, which are defined by solving the equation
\[ \partial_t g_t(z) = g_t(z) \frac{W_t + g_t(z)}{W_t - g_t(z)}.\]
Here we take $W_t = e^{i \sqrt{\kappa} B_t}$ where $B_t$ is a two-sided (i.e., defined for $t \in \R$) Brownian motion.  For each time $t$, $g_t$ is the unique conformal transformation from the unbounded component of $\C \setminus \eta([0,t])$ to $\C \setminus \closure{\D}$ with positive derivative at $\infty$.  The whole-plane $\SLE_\kappa$ processes also correspond to a continuous curve \cite{rs2005basic,lsw2004lerwust} meaning there exists a continuous curve $\eta \colon \R \to \C$ so that for each $t \in \R$ the domain of $g_t$ is equal to the unbounded component of $\C \setminus \eta((-\infty,t])$.  There are also the whole-plane $\SLE_\kappa(\rho)$ processes where $\rho > -2$ which are constructed by replacing $W_t$ with the solution to the SDE (defined for all $t \in \R$)
\begin{align*}
dW_t &= \left( \frac{\rho}{2} \wt{\Psi}(O_t,W_t) -\frac{\kappa}{2} W_t \right) dt + i \sqrt{\kappa} W_t dB_t, \quad dO_t = \Psi(W_t,O_t) dt
\end{align*}
where
\[ \Psi(w,z) = -z\frac{z+w}{z-w} \quad\text{and}\quad \wt{\Psi}(z,w) = \frac{\Psi(z,w) + \Psi(1/\ol{z},w)}{2}.\]
The continuity of the whole-plane $\SLE_\kappa(\rho)$ processes with $\rho > -2$ was proved in \cite{ms2017ig4}.

\subsection{Conformal loop ensembles}
\label{subsec:cle}

We are now going to give a brief review of the conformal loop ensembles ($\CLE_\kappa$) \cite{s2009cle,sw2012cle}.  We will focus on their construction using the so-called boundary conformal loop ensembles ($\BCLE_\kappa$) which is developed in \cite{msw2017clepercolations} since this is the construction which will be most relevant for this work.  We will primarily focus on the case of the simple $\CLE_\kappa$ (i.e., $\kappa \in (8/3,4)$).

Fix $\kappa \in (2,4)$ and $\rho \in (-2,\kappa-4)$.  A $\cwBCLE_\kappa(\rho)$ (so that the loops have a clockwise orientation) is defined as follows.  Let $\eta$ be an $\SLE_\kappa(\rho ; \kappa-6-\rho)$ in $\h$ from $0$ to $\infty$ with force points at $0^-$, $0^+$.  Then the law of $\eta$ is \emph{target invariant}.  This means that if $\varphi$ is a conformal transformation $\h \to \h$ which fixes $0$ then $\varphi(\eta)$ and $\eta$, viewed modulo time parameterization, have the same law up until first disconnecting $\infty$ and $\varphi(\infty)$.  Let $(x_n)$ be a countable dense set in $\R$ and, for each $n \in \N$, let $\eta_n$ be an $\SLE_\kappa(\rho;\kappa-6-\rho)$ in $\h$ from $0$ to $x$ with force points at $0^-$, $0^+$.  The target invariance implies that we can couple the $\eta_n$ together so that any finite collection of them agree up until their target points are disconnected and then afterwards evolve independently.  The family of paths $(\eta_n)$ should be thought of as branches of a tree which is rooted at $0$.  As we will explain shortly, the $\cwBCLE_\kappa(\rho)$ in $\h$ is defined from the $(\eta_n)$ in a certain way; $\cwBCLE_\kappa(\rho)$ in domains other than $\h$ are defined by applying a conformal transformation.

The way that the loops are defined from the family $(\eta_n)$ is as follows.  Suppose that $\eta_n$ makes an excursion away from $\ccwBoundary{0}{x_n}{\partial \h}$ in the time-interval $[s,t]$.  Then $\eta_n|_{[s,t]}$ corresponds to part of a loop.  If $x_m \in (\eta_n(s),\eta_n(t))$ then there exists $t' > t$ so that $\eta_m$ makes an excursion away from $\ccwBoundary{0}{x_m}{\partial \h}$ in $[s,t']$ and $\eta_n|_{[s,t]} = \eta_m|_{[s,t]}$.  Thus $\eta_m|_{[s,t']}$ corresponds to a larger portion of the same loop.  The whole loop is obtained by by taking a sequence $(x_{m_j})$ in $\ccwBoundary{\eta_n(s)}{\eta_n(t)}{\partial \h}$ which decreases to $\eta_n(s)$ and considering the excursion of $\eta_{m_j}$ which contains $\eta_n|_{[s,t]}$.  It was shown in \cite{msw2017clepercolations} that the loops are simple continuous curves and that the law of the collection of loops does not depend on the root point (i.e., the origin).  By considering the intervals $\cwBoundary{0}{x_n}{\partial \h}$ in place of $\ccwBoundary{0}{x_n}{\partial \h}$ one can define the ``false'' loops of $\cwBCLE_\kappa(\rho)$ and note that these are surrounded counterclockwise.  $\ccwBCLE_\kappa(\rho)$ is defined in the same way except we take $\SLE_\kappa(\kappa-6-\rho; \rho)$ in place of $\SLE_\kappa(\rho ; \kappa-6-\rho)$ and the loops of a $\ccwBCLE_\kappa(\rho)$ correspond to the complementary components which are surrounded counterclockwise.

There is a natural exploration path of a $\cwBCLE_\kappa(\rho)$ process $\Gamma$ on a Jordan domain $D \subseteq \C$.  Namely, let~$\eta_0$ be a path which traverse $\partial D$ counterclockwise.  We then let~$\eta$ be the path which follows~$\eta_0$ except whenever it hits a loop of $\Gamma$ for the first time, it follows the loop in its entirety in the clockwise direction, and then continues following~$\eta_0$.  The continuity of~$\eta$ was proved in \cite{msw2017clepercolations} as a consequence of the continuity of space-filling $\SLE_{\kappa'}$ established in \cite{ms2017ig4}.  If we take $\eta$ and target it at a fixed boundary point (i.e., we parameterize $\eta$ by capacity as viewed from the given boundary point), then we recover the branch of the $\SLE_\kappa(\rho; \kappa-6-\rho)$ tree used to build the $\cwBCLE_\kappa(\rho)$.  We can similarly define an exploration path of a $\ccwBCLE_\kappa(\rho)$.

For $\kappa' \in (4,8)$ and $\rho \in (\kappa'/2-4,\kappa'/2-2)$ a $\cwBCLE_{\kappa'}(\rho')$ is defined in an analogous manner to the case $\kappa \in (2,4)$.  Moreover, one can associate with a $\cwBCLE_{\kappa'}(\rho')$ an exploration path which is defined in the same way as the case $\kappa \in (2,4)$.  When $\rho' = 0$ we will simply write $\cwBCLE_{\kappa'}$.  It turns out that the loops of a $\cwBCLE_{\kappa'}$ have the same law as the boundary touching loops of a $\CLE_{\kappa'}$.

Fix $\kappa \in (8/3,4)$ and let $\kappa'=16/\kappa \in (4,6)$.  The way that the iterated $\BCLE$ construction works is as follows.  Suppose that we first sample a $\cwBCLE_{\kappa'}$ process $\Gamma'$.  Then in each of the complementary components which are surrounded clockwise we sample a conditionally independent $\ccwBCLE_\kappa(-\kappa/2)$.  This procedure is then repeated in each of the components which are surrounded by a false (i.e., counterclockwise) loop of the original $\cwBCLE_{\kappa'}$ or by a false (i.e., clockwise) loop of one of the $\ccwBCLE_\kappa(-\kappa/2)$.  The resulting collection of $\ccwBCLE_\kappa(-\kappa/2)$ loops (from all of the iterations) gives a $\CLE_\kappa$ process $\Gamma$.

Let $\eta'$ be the exploration path associated with the $\Gamma'$.  Let $\eta$ be defined to be the path which follows $\eta'$ except whenever it hits a $\ccwBCLE_\kappa(-\kappa/2)$ loop for the first time then it follows that loop counterclockwise.  It is shown in \cite[Theorem~7.4]{msw2017clepercolations} that if we target $\eta$ at any fixed boundary point then we get an $\SLE_\kappa(\kappa-6)$ process and recall that we know that $\eta'$ targeted at any fixed boundary point is an $\SLE_{\kappa'}(\kappa'-6)$ process.  We call $\eta$ a \emph{conformal percolation interface} (CPI) in $\Gamma$ and $\eta'$ its trunk.  If we draw $\eta$ up to any stopping time in which it is not drawing a loop of $\Gamma$, then the conditional law of the loops of $\Gamma$ in the complementary components are conditionally independent $\CLE_\kappa$'s.  In order to emphasize that all of the loops that $\eta$ hits are to its right, we will sometimes write $\eta \sim \SLE_\kappa^1(\kappa-6)$.  By replacing the $\cwBCLE_{\kappa'}$ and $\ccwBCLE_\kappa(-\kappa/2)$ with $\ccwBCLE_{\kappa'}$ and $\cwBCLE_\kappa(-\kappa/2)$ we can similarly construct an $\SLE_\kappa^{-1}(\kappa-6)$ process meaning that the $\CLE_\kappa$ loops which it hits are always to its left.

What we have described above is a \emph{totally asymmetric} CPI because all of the loops which are hit by~$\eta$ lie to its right.  We can also consider the symmetric version which is constructed as follows.  We first sample a $\cwBCLE_{\kappa'}(\kappa'/2-3)$ process $\Gamma'$.  In each complementary component which is surrounded clockwise we sample a conditionally independent $\ccwBCLE_\kappa(\kappa/4-2)$ and in each component which is surrounded counterclockwise we sample a conditionally independent $\cwBCLE_\kappa(\kappa/4-2)$.  This procedure is then iterated in each of the false loops of the $\ccwBCLE_\kappa(\kappa/4-2)$ and $\cwBCLE_\kappa(\kappa/4-2)$.  The resulting ensemble of loops is again a $\CLE_\kappa$ (see \cite[Theorem~7.4, Corollary~7.5]{msw2017clepercolations}).  Let $\eta'$ be the exploration path associated with $\Gamma'$ and then let~$\eta$ be the path which is obtained by following~$\eta'$ and except it traverses each loop of a $\ccwBCLE_\kappa(\kappa/4-2)$ (resp.\ $\cwBCLE_\kappa(\kappa/4-2)$) that it hits counterclockwise (resp.\ clockwise).  Then $\eta$ targeted at any fixed boundary point has the law of an $\SLE_\kappa^0(\kappa-6)$ process.  Each time that~$\eta$ encounters a new loop of the $\CLE_\kappa$, it goes to its left or right based on the toss of a fair coin flip.

We will not need to consider even more general versions of these explorations, but we remark that one can get $\SLE_\kappa^\beta(\kappa-6)$ for $\beta \in [-1,1]$ using other combinations of $\rho$ values in the above construction properly (see \cite[Theorem~7.4]{msw2017clepercolations}).  The formula which relates $\beta$ and $\rho$ was computed in \cite{msw2020simplecle}.  (The corresponding formula with the roles of $\kappa$ and $\kappa'$ reversed was computed in \cite{msw2020nonsimplecle}.)

\subsection{Imaginary geometry}
\label{subsec:ig}

We are now going to collect some background from the theory of imaginary geometry \cite{ms2016ig1,ms2017ig4} which will be relevant for the present work.  Throughout, we take the convention that $\kappa \in (0,4)$ and $\kappa'=16/\kappa > 4$.  We also let 
\[ \chi = \frac{2}{\sqrt{\kappa}} - \frac{\sqrt{\kappa}}{2},\quad \lambda = \frac{\pi}{\sqrt{\kappa}},\quad \text{and}\quad \lambda' = \frac{\pi}{\sqrt{\kappa'}}.\]

Suppose that we have weights $(\ul{\rho}_L;\ul{\rho}_R)$ and force point locations $(\ul{x}_L;\ul{x}_R)$.  Let $h$ be a GFF on $\h$ with boundary conditions given by
\[ -\lambda\left(1+\sum_{i=1}^k \rho_{i,L} \right) \quad\text{in}\quad (x_{k+1,L},x_{k,L}]  \quad\text{and}\quad \lambda\left( 1+ \sum_{i=1}^k \rho_{i,R} \right) \quad\text{in}\quad [x_{k,R},x_{k+1,R}).\]
It is shown in \cite{ms2016ig1} that an $\SLE_\kappa(\ul{\rho}_L;\ul{\rho}_R)$ process with force points located at $(\ul{x}_L;\ul{x}_R)$ can be coupled with $h$ as a flow line of the formal vector field $e^{i h /\chi}$.  This means that if $(f_t)$ is the centered Loewner flow for $\eta$ (i.e., $f_t = g_t - W_t$ where $(g_t)$ is the Loewner flow) and $\tau$ is a stopping time for $\eta$ then $h \circ f_\tau^{-1} - \chi \arg( f_\tau^{-1})'$ is a GFF on $\h$ with boundary conditions given by
\[ -\lambda\left(1+\sum_{i=1}^k \rho_{i,L} \right) \quad\text{in}\quad (f_\tau(x_{k+1,L}),f_\tau(x_{k,L})]  \quad\text{and}\quad \lambda\left( 1+ \sum_{i=1}^k \rho_{i,R} \right) \quad\text{in}\quad [f_\tau(x_{k,R}),f_\tau(x_{k+1,R})).\]
Moreover, the boundary conditions on the image of the left (resp.\ right) side of $\eta|_{[0,\tau]}$ are given by $-\lambda$ (resp.\ $\lambda$).  Suppose that $D \subseteq \C$ is a simply connected domain, $\varphi \colon \h \to D$ is a conformal transformation, and $\wt{h} = h \circ \varphi^{-1} - \chi \arg (\varphi^{-1})'$.  Then we can similarly view $\varphi(\eta)$ as the flow line of $e^{i \wt{h}/\chi}$.  In the coupling of an $\SLE_\kappa(\ul{\rho}_L;\ul{\rho}_R)$ as a flow line of a GFF, we also have that the former is a.s.\ determined by the latter.  Finally, we can also define the flow line of a GFF $h$ with angle $\theta$ by considering the flow line of $h + \theta \chi$.

We now turn to the case of $\SLE_{\kappa'}$ processes.  Suppose that we have weights $(\ul{\rho}_L;\ul{\rho}_R)$ and force point locations $(\ul{x}_L; \ul{x}_R)$.  Let $h$ be a GFF on $\h$ with boundary conditions
\[ \lambda'\left(1+\sum_{i=1}^k \rho_{i,L}' \right) \quad\text{in}\quad (x_{k+1,L},x_{k,L}]  \quad\text{and}\quad -\lambda' \left( 1+ \sum_{i=1}^k \rho_{i,R}' \right) \quad\text{in}\quad [x_{k,R},x_{k+1,R}).\]
Then we can view an $\SLE_{\kappa'}(\ul{\rho}_L' ; \ul{\rho}_R')$ process $\eta'$ as a counterflow line of $h$ from $0$ to $\infty$.  This means that if $(f_t)$ is the centered Loewner flow for $\eta'$ and $\tau$ is a stopping time for $\eta'$ then we have that $h \circ f_\tau^{-1} - \chi \arg( f_\tau^{-1})'$ is a GFF on $\h$ with boundary conditions given by
\[ \lambda'\left(1+\sum_{i=1}^k \rho_{i,L}' \right) \quad\text{in}\quad (f_\tau(x_{k+1,L}),f_\tau(x_{k,L})]  \quad\text{and}\quad -\lambda' \left( 1+ \sum_{i=1}^k \rho_{i,R}' \right) \quad\text{in}\quad [f_\tau(x_{k,R}),f_\tau(x_{k+1,R})).\]
Moreover, the boundary conditions on the image of the left (resp.\ right) side of $\eta|_{[0,\tau]}$ are given by $\lambda'$ (resp.\ $-\lambda'$).  As in the case of GFF flow lines, counterflow lines are a.s.\ determined by the GFF and counterflow lines for GFFs defined on domains other than $\h$ are defined by applying the same change of coordinates formula as in the case of flow lines.

The discussion above implies that the iterated $\BCLE$ exploration is naturally coupled with the GFF.  We will explain this point in further detail in Appendix~\ref{app:mod_of_cont}.

We will also have occasion to consider GFF flow lines which start from an interior point, as developed in \cite{ms2017ig4}.  If we fix $\alpha > -\chi$ and we have a whole-plane GFF with values modulo $2\pi(\chi+\alpha)$, then its flow line starting from $0$ is a whole-plane $\SLE_\kappa(\rho)$ process with $\rho = 2-\kappa + 2\pi \alpha/\lambda$.  We can also consider flow lines of different angles starting from the origin and talk about their conditional laws.

Lastly, let us also mention the space-filling version of $\SLE$ which is considered in \cite{ms2017ig4}.  It is constructed by fixing a countable dense set of points, starting a flow line from each one, and then ordering the points according to how the flow lines merge.  It is shown in \cite{ms2017ig4} that there is a continuous path which visits these points in this order and this is space-filling $\SLE$.

\subsection{Liouville quantum gravity}
\label{subsec:lqg}

A Liouville quantum gravity (LQG) surface is the random surface whose Riemannian metric admits the formal expression:
\[ e^{\gamma h(z)} (dx^2 + dy^2)\]
where $h$ is an instance of some form of the GFF and $\gamma \in (0,2)$ is a parameter.  This expression does not make literal sense because $h$ is a distribution and not a function.  The volume form was constructed in \cite{ds2011kpz} (see also the references therein).  If $h$ is (some form of) the GFF with free boundary conditions, then one can also make sense of a boundary length measure by considering $e^{\gamma h(x)/2} dx$.  The metric (i.e., two-point distance function) for LQG was constructed first for $\gamma=\sqrt{8/3}$ in \cite{ms2020qle1,ms2016qle2,ms2021qle3} and subsequently for all $\gamma \in (0,2)$ in \cite{dddf2020tightness,gm2021uniqueness}.

If $\varphi \colon D \to \wt{D}$ is a conformal transformation and we let 
\begin{equation}
\label{eqn:q_surface_equiv}
\wt{h} = h \circ \varphi + Q\log|\varphi'| \quad\text{where}\quad Q = \frac{2}{\gamma} + \frac{\gamma}{2}
\end{equation}
then we have that $\qmeasure{\wt{h}}(\varphi(A)) = \qmeasure{h}(A)$ for all Borel sets $A \subseteq D$.  We say that two fields $h$, $\wt{h}$ are equivalent as quantum surfaces if they are related as in~\eqref{eqn:q_surface_equiv}.  A quantum surface is defined as an equivalence class with respect to this equivalence relation.  More generally, we can consider a quantum surface with marked points $x_1,\ldots,x_k \in D$.  Two quantum surfaces with marked points are considered to be equivalent if the fields are related as in~\eqref{eqn:q_surface_equiv} where the conformal map $\varphi$ takes the marked points for one surface to the marked points for the other.  We remark that it is not always obvious that two fields define equivalent quantum surfaces.

We will always assume that the parameters are matched by $\kappa = \gamma^2 \in (0,4)$ and $\kappa'=16/\kappa > 4$.  There is a way to define the quantum length of an $\SLE_\kappa$ on a $\gamma$-LQG surface which comes as a consequence of the main results of \cite{s2016zipper}.  The main results of \cite{dms2014mating} provide a way to define the quantum length (so-called quantum natural time) of an $\SLE_{\kappa'}$ process on a $\gamma$-LQG surface.

There are a number of important types of LQG surfaces: quantum wedges, cones, disks, and spheres.  We will not give a precise definition of these objects here because their exact definition will not matter for what follows, with the exception of the technical estimates which are proved in Appendices~\ref{app:carpet_measure} and~\ref{app:mod_of_cont}.  Briefly, a quantum wedge is an infinite volume surface which comes with two marked points (the ``origin'' and ``infinity'').  We will denote the law of a quantum wedge of weight $W > 0$ as $\qwedgeW{\gamma}{W}$.  Quantum wedges are homeomorphic to $\h$ when $W \geq \tfrac{\gamma^2}{2}$ and in this case we will use the notation $(\h,h,0,\infty)$ to indicate a quantum wedge which is parameterized by $\h$ with the origin point at $0$ and the infinity point at $\infty$.

Following \cite{msw2020simplecle}, we will refer to a sample from the law $\qwedgeW{\gamma}{2}$ as a \emph{quantum half-plane}.  The reason for this terminology is that the quantum half-plane is the LQG analog of the so-called Brownian half-plane \cite{gm2017halfplane}.  It possesses the following special property \cite{s2016zipper} which distinguishes it from other quantum wedges.  If $(\h,h,0,\infty)$ is a quantum half-plane, $r > 0$ is fixed, and $x \in [0,\infty)$ is such that $\qbmeasure{h}([0,x]) = r$, then $(\h,h,x,\infty)$ is a quantum half-plane.

 A quantum cone is an infinite volume surface which is homeomorphic to $\C$ and also has two marked points (the ``origin'' and ``infinity'').  We will denote the law of a quantum cone of weight $W > 0$ as $\qconeW{\gamma}{W}$.  We will use the notation $(\C,h,0,\infty)$ to indicate a quantum cone parameterized by $\C$ where the origin point is at $0$ and the infinity point is at $\infty$.  A quantum disk is a finite volume surface which is homeomorphic to $\D$ and its definition naturally equips it with two marked boundary points.  We will also denote the law of a quantum disk with boundary length $\ell$ using $\qdiskL{\gamma}{\ell}$ and use the notation $(D,h,x,y)$ to denote a quantum disk where the marked boundary points are at $x,y$.

The following theorem is proved in \cite{s2016zipper} in the case of $\qwedgeW{\gamma}{4}$ and an $\SLE_\kappa$ process and then generalized in \cite{dms2014mating} to the case of $\qwedgeW{\gamma}{W}$ for $W > 0$ and $\SLE_\kappa(\rho_1;\rho_2)$ for $\rho_1,\rho_2 > -2$.

\begin{theorem}
\label{thm:wedge_cutting}
Fix $W \geq \tfrac{\gamma^2}{2}$ and suppose that $W_1,W_2 > 0$ are such that $W = W_1 + W_2$.  Suppose that $\CW = (\h,h,0,\infty)$ has law $\qwedgeW{\gamma}{W}$.  Let $\eta$ be an independent $\SLE_\kappa(W_1 - 2; W_2 - 2)$ in $\h$ from $0$ to $\infty$.  Let $\CW_1$ (resp.\ $\CW_2$) be the quantum surface parameterized by the component of $\h \setminus \eta$ which are to the left (resp.\ right) of $\eta$.  Then $\CW_1$ (resp.\ $\CW_2$) has law $\qwedgeW{\gamma}{W_1}$ (resp.\ $\qwedgeW{\gamma}{W_2}$) and $\CW_1, \CW_2$ are independent.
\end{theorem}

This is proved in \cite{dms2014mating}.

\begin{theorem}
\label{thm:cone_cutting}
Fix $W > 0$ and suppose that $\CC = (\C,h,0,\infty)$ has law $\qconeW{\gamma}{W}$.  Let $\eta$ be an independent whole-plane $\SLE_\kappa(W-2)$ process in $\C$ from $0$ to $\infty$.  Then the quantum surface parameterized by $\C \setminus \eta$ is a quantum wedge of weight $W$.
\end{theorem}

\begin{remark}
\label{rem:cutting_gluing_weight_4_cone}
As a consequence of Theorems~\ref{thm:wedge_cutting}, \ref{thm:cone_cutting} we have the following fact.  Suppose that $\CC = (\C,h,0,\infty)$ has law $\qconeW{\gamma}{4}$.  Let $\eta_-$ be an independent whole-plane $\SLE_\kappa(2)$ in $\C$ from $0$ to $\infty$ and, given $\eta_-$, let $\eta_+$ be a chordal $\SLE_\kappa$ in $\C \setminus \eta_-$ from $0$ to $\infty$.  Let $H_+, H_-$ be the two components of $\C \setminus (\eta_- \cup \eta_+)$ and let $\CH_\pm = (H_\pm,h,0,\infty)$ be the quantum surfaces parameterized by~$H_\pm$.  Then~$\CH_\pm$ are independent quantum half-planes.  We note that if one considers the concatenation of the time-reversal of~$\eta_-$ with~$\eta_+$ then one obtains a curve in the Riemann sphere from~$\infty$ to~$\infty$ which in the literature is sometimes referred to as a two-sided whole-plane $\SLE_\kappa$ from~$\infty$ to~$\infty$ through $0$.  One can also construct the law $\eta_\pm$ using the whole-plane GFF using the machinery from \cite{ms2017ig4}.  Namely, let~$h^w$ be a whole-plane GFF with values modulo $2\pi(\chi + \sqrt{\kappa}/2) = 4\pi/\sqrt{\kappa}$ and let $h^\IG = h^w - \tfrac{\sqrt{\kappa}}{2} \arg(\cdot)$, also with values defined modulo $4\pi/\sqrt{\kappa}$.  Then the joint law of the pair $(\eta_-,\eta_+)$ is equal to that of the flow lines of $h^\IG$ from $0$ to $\infty$ with angles $2\pi/(2-\kappa/2)$, $0$, respectively.

Using the invariance of the law of the quantum-plane under the operation of shifting its origin point a given amount of quantum length to the left or right, one also has the following.  Suppose that $\eta_+$ is parameterized according to quantum length and $r > 0$.  Then $(H_+,h|_{H_+},\eta_+(r),\infty)$, $(H_-,h|_{H_-},\eta_+(r),\infty)$ are independent quantum half-planes.  Therefore the quantum surface $(\C,h,\eta_+(r),\infty)$ also has law $\qconeW{\gamma}{4}$.  Moreover, up to a translation and rescaling the law of the paths $\eta_+|_{[r,\infty)}$ and the concatenation of the time-reversal of $\eta_+|_{[0,r]}$ with $\eta_-$ is the same as the law of $(\eta_-,\eta_+)$.
\end{remark}

The following is proved in \cite{msw2020simplecle}.

\begin{theorem}
\label{thm:cpi_wedge_explore}
Suppose that $\CH = (\h,h,0,\infty)$ is a quantum half-plane.  Fix $\beta \in [-1,1]$.  Let $\eta$ be an independent $\SLE_\kappa^\beta(\kappa-6)$ in $\h$ from $0$ to $\infty$ which is subsequently parameterized by the quantum natural time of its trunk.  Then for each time $t$, the quantum surface $\CH_t$ parameterized by the unbounded component of $\h \setminus \eta([0,t])$ is a quantum half-plane which is independent of the quantum surface $\CK_t$ parameterized by its complement in $\h$ and decorated by $\eta|_{[0,t]}$.  This holds more generally if we replace $t$ by a stopping time of the filtration $\CF_t = \sigma(\CK_s : s \leq t)$.   Finally, let $L_t$ denote the quantum length of the part of $\h \cap \partial \CK_t$ which is to the left of $\eta(t)$ minus the quantum length of the part of $\partial \h \cap \partial \CK_t$ which is to the left of $0$.  Define $R_t$ analogously with right in place of left.  Then the pair $(L_t,R_t)$ evolve as independent $4/\kappa$-stable L\'evy processes.
\end{theorem}

In the statement of Theorem~\ref{thm:cpi_wedge_explore}, the positivity parameter (which determines the relative intensity of upward vs downward jumps) and relative overall intensity of jumps for $L_t$ and $R_t$ depend on $\beta$ and were determined in \cite{msw2020simplecle}.  We will have some further discussion of this in Appendix~\ref{app:levy_process}.

\section{Crossing estimates}
\label{sec:chunk_estimates}

The purpose of this section is to start to set up the percolation argument that will be used in order to establish Theorem~\ref{thm:cle_loop}.  In Section~\ref{subsec:setup}, we will describe the setup that we will use throughout this section as well as give the main two statements.  The first main statement (Lemma~\ref{lem:one_chunk_quantile}) gives that it is very likely that a quantum surface $\CN$ cut out of a quantum half-plane by an independent $\SLE_\kappa^0(\kappa-6)$ process run until the first time it hits the boundary after time $\delta^{4/\kappa}$ contains paths $\cpath$ in the $\CLE_\kappa$ carpet with $\lebneb{\epsilon}(\cpath) \leq \quantHP{p}{\epsilon}$ (as defined in Section~\ref{subsec:outline}; we will also recall the definition of $\quantHP{p}{\epsilon}$ just below) which connect the bottom right corner of $\CN$ to any sufficiently large interval on the top of $\CN$ which intersects the $\CLE_\kappa$ carpet and otherwise stays away from $\partial \CN$, provided $\delta > 0$ is sufficiently small.  The second main statement (Lemma~\ref{lem:two_chunks_quantile}) is of a similar flavor except with two chunks of quantum surface cut out by successive $\SLE_\kappa^0(\kappa-6)$ processes.  In order to read the rest of this article, one only needs these two main statements, so the rest of the section can be skipped on a first reading.  In Section~\ref{subsec:cpi_close_to_path}, we will show that a CPI in a $\CLE_\kappa$ carpet $\Upsilon$ has positive conditional probability given~$\Upsilon$ of being close to any fixed path in~$\Upsilon$.  The focus of Section~\ref{subsec:crossing_between_disks} is to prove an intermediate result for the likelihood of such a crossing $\omega$ in $\Upsilon$ as above but in a slightly different setting.  In Section~\ref{subsec:chunk_proofs} we will use this intermediate result in order to complete the proofs of the two main statements.

\subsection{Setup and main statements}
\label{subsec:setup}

Throughout this section, we assume that we have the following setup.  We assume that $\CC = (\C,h,0,\infty)$ has law $\qconeW{\gamma}{4}$ with the circle average embedding.  We let~$h^w$ be a whole-plane GFF with values modulo $2\pi(\chi + \sqrt{\kappa}/2) = 4\pi/\sqrt{\kappa}$ and set $h^\IG = h^w - \tfrac{\sqrt{\kappa}}{2} \arg(\cdot)$, also with values modulo $4\pi/\sqrt{\kappa}$.  We assume that $h^\IG$ is independent of $\CC$.  We let $\eta_-$, $\eta_+$ be the flow lines of $h^\IG$ from $0$ to $\infty$ with angles $2\pi/(2-\kappa/2)$, $0$, respectively.  Then $\eta_\pm$ are each whole-plane $\SLE_\kappa(2)$ processes in $\C$ from $0$ to $\infty$ \cite[Theorem~1.4]{ms2017ig4}.  Moreover, the conditional law of $\eta_-$ (resp.\ $\eta_+$) given $\eta_+$ (resp.\ $\eta_-$) is that of a chordal $\SLE_\kappa$ process in $\C \setminus \eta_+$ (resp.\ $\C \setminus \eta_-$) from $0$ to $\infty$.  That is, the concatenation of the time-reversal of $\eta_-$ together with $\eta_+$ is a two-sided whole-plane $\SLE_\kappa$ process from $\infty$ to $\infty$ through $0$ and the same is true with the roles of $\eta_\pm$ swapped.  We assume that $\eta_\pm$ are parameterized according to quantum length according to $h$.  Let $\CH_+$ (resp.\ $\CH_-$) be the quantum surface parameterized by the component of $\C \setminus (\eta_- \cup \eta_+)$ part of whose boundary consists of the left side of $\eta_+$.  By Theorems~\ref{thm:wedge_cutting}, \ref{thm:cone_cutting} we have that $\CH_\pm$ are independent quantum half-planes, which we assume to be marked by $0$ and $\infty$.  Let $\Gamma_+$ be a $\CLE_\kappa$ in $\CH_+$ and let~$\Upsilon_+$ be the carpet of~$\Gamma_+$.  As in Section~\ref{subsec:outline}, for each $\epsilon > 0$ and $p \in (0,1)$ we let $\quantHP{p}{\epsilon}$ denote the $p$th quantile of $\sup_{z,w \in \eta_+([0,1])} \metapprox{\epsilon}{z}{w}{\Gamma_+}$ and let $\medianHP{\epsilon} = \quantHP{1/2}{\epsilon}$

\begin{figure}[ht!]
\includegraphics[scale=1]{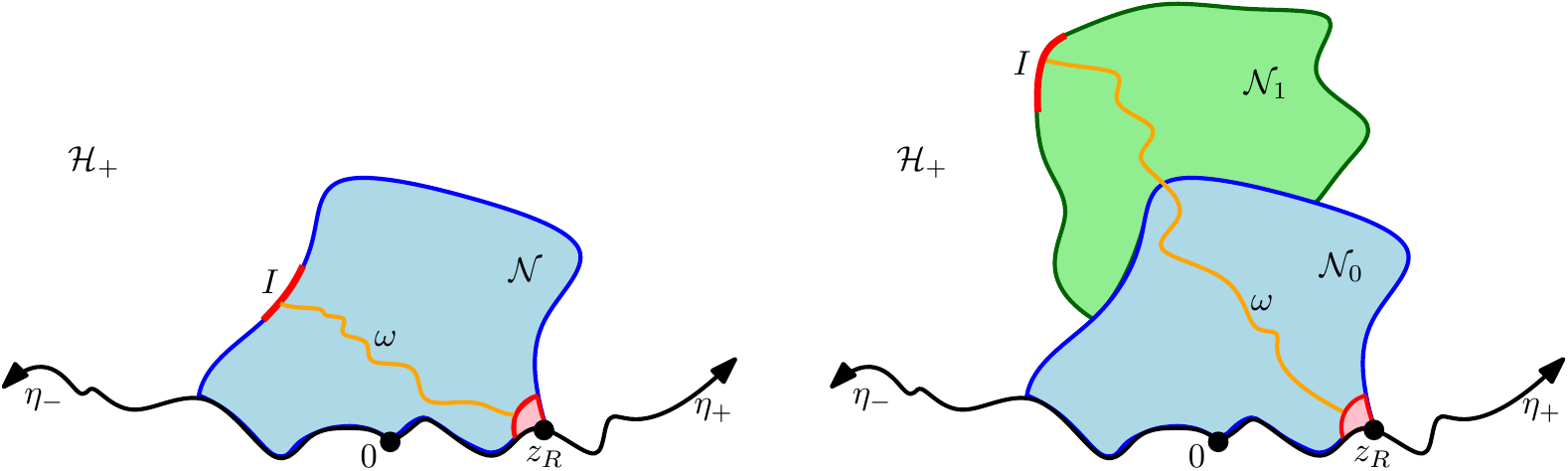}
\caption{\label{fig:crossssing_chunks} Illustration of the setup and statement of Lemma~\ref{lem:one_chunk_quantile} (left) and Lemma~\ref{lem:two_chunks_quantile} (right).  Not shown are the loops of $\Gamma$, which are in particular dense in $\CH_+ \cap \partial \CN$, $\CH_+ \cap \partial \CN_0$, and $\CH_+ \cap \partial \CN_1$.}
\end{figure}

\begin{lemma}
\label{lem:one_chunk_quantile}
For each $p_0,\zeta_0,p \in (0,1)$ there exists $\xi_0, \delta_0 > 0$ so that for all $\delta \in (0,\delta_0)$ there exists $\epsilon_0 > 0$ so that for all $\epsilon \in (0,\epsilon_0)$ the following is true.  Suppose that~$\eta$ is an $\SLE_\kappa^0(\kappa-6)$ in $\CH_+$ from~$0$ to~$\infty$ which is coupled with $\Gamma_+$ as a CPI and parameterized by the quantum natural time of its trunk.  Let $\sigma = \inf\{t \geq \delta^{4/\kappa} : \eta(t) \in \partial \CH_+\}$ and let~$\CN$ be the quantum surface disconnected from~$\infty$ by $\eta|_{[0,\sigma]}$.  Let $d_0 = \diam(\eta([0,\sigma]))$ and let $z_R$ be the rightmost point on the bottom of $\CN$.  Let $G$ be the event that for each interval $I$ on the top of $\partial \CN$ with quantum length at least $\zeta_0 \delta$ and which intersects $\Upsilon_+$ there is a path $\cpath$ in $\Upsilon_+ \cap \CN$ which
\begin{enumerate}[(i)]
\item connects $\partial B(z_R, \zeta_0 d_0)$ to $\wt{I}$ where $\wt{I}$ is the interval on the top of $\partial \CN$ with the same center as $I$ but with twice the quantum length,
\item $\dist(\cpath,\partial \CN \setminus \wt{I}) \geq \xi_0 d_0$, and
\item satisfies $\lebneb{\epsilon}(\cpath) \leq \quantHP{p}{\epsilon}$.
\end{enumerate}
Then $\p[G] \geq p_0$.
\end{lemma}

\begin{lemma}
\label{lem:two_chunks_quantile}
For each $p_0,\zeta_0,p \in (0,1)$ there exists $\xi_0, \delta_0 > 0$ so that for all $\delta \in (0,\delta_0)$ there exists $\epsilon_0 > 0$ so that for all $\epsilon \in (0,\epsilon_0)$ the following is true.  Suppose that $\eta_0$ is an $\SLE_\kappa^0(\kappa-6)$ in $\CH_+$ from~$0$ to~$\infty$ which is coupled with~$\Gamma_+$ as a CPI and parameterized by the quantum natural time of its trunk.  Let $\sigma_0 = \inf\{t \geq \delta^{4/\kappa} : \eta_0(t) \in \partial \CH_+\}$.  Let $\CN_0$ be the quantum surface disconnected from $\infty$ by $\eta_0|_{[0,\sigma_0]}$.  Let $\eta_1$ be an $\SLE_\kappa^0(\kappa-6)$ process in $\CH_{+,1} = \CH_+ \setminus \CN_0$ starting from a uniformly random point chosen from the set of points on the top $\CH_+ \cap \partial \CN_0$ of $\partial \CN_0$ whose clockwise boundary length distance to~$\partial \CH_+$ is an integer multiple of~$\zeta_0 \delta$ and are within boundary length distance $\zeta_0 \delta$ of a point in $\Upsilon_+ \cap \partial \CN_0$ and coupled with the loops of $\Gamma_+$ contained in $\CH_{+,1}$ as a CPI.  We assume that $\eta_1$ is parameterized according to the quantum natural time of its trunk.  Let $\sigma_1 = \inf\{ t \geq \delta^{4/\kappa} : \eta_1(t) \in \partial \CH_{+,1}\}$.  Let $d_0 = \diam(\eta_0([0,\sigma_0]))$ and let $z_R$ be the rightmost point on the bottom of $\CN_0$.  Let $\CN$ be the quantum surface parameterized by $\interior{\closure{\CN_0 \cup \CN_1}}$.  Let $G$ be the event that for each interval $I$ on the top of $\partial \CN_1$ with quantum length at least $\zeta_0 \delta$ and which intersects $\Upsilon_+$ there is a path $\cpath$ in $\Upsilon_+ \cap \CN$ which
\begin{enumerate}[(i)]
\item connects $\partial B(z_R, \zeta_0 d_0)$ to $\wt{I}$ where $\wt{I}$ is the interval on the top of $\partial \CN$ with the same center as $I$ but with twice the quantum length,
\item $\dist(\cpath, \partial \CN \setminus \wt{I}) \geq \xi_0 d_0$, and
\item satisfies $\lebneb{\epsilon}(\cpath) \leq \quantHP{p}{\epsilon}$.
\end{enumerate}
Then $\p[G] \geq p_0$.
\end{lemma}

Let us now explain the main steps used to prove Lemmas~\ref{lem:one_chunk_quantile} and~\ref{lem:two_chunks_quantile}.
\begin{enumerate}
\item[Step 1.] Show that a CPI in $\Gamma_+$ has a positive chance given $\Upsilon_+$ of being close to any fixed path in $\Upsilon_+$ (Proposition~\ref{prop:cpi_path_close}).  This result will be used several other times in this article.  Its proof does not need to be read in order to understand the proofs of Lemmas~\ref{lem:one_chunk_quantile} and~\ref{lem:two_chunks_quantile} or the rest of the article.
\item[Step 2.] We will construct a domain in $\CH_+$ using a pair of CPIs of $\Gamma_+$ whose boundary has positive distance from $\partial \CH_+$.  We will argue that if we make this domain sufficiently small then it is likely to contain a crossing $\omega$ in $\Upsilon_+$ satisfying $\lebneb{\epsilon}(\omega) \leq \quantHP{p}{\epsilon}$.  As we mentioned in Section~\ref{subsec:outline}, the key point in the proof where we define the quantiles using $\sup_{z,w \in \eta_+([0,1])} \metapprox{\epsilon}{z}{w}{\Gamma_+}$ (rather than using the length of a crossing in $\Upsilon_+$ connecting two disjoint boundary arcs) will be contained in this step and it is used to show that the probability that the (Euclidean) diameter of the set of points $z \in \Upsilon_+$ with $\metapprox{\epsilon}{0}{z}{\Gamma_+} \leq \quantHP{p}{\epsilon}$ is at least $\delta > 0$ tends to $1$ as $\delta \to 0$ provided $\epsilon > 0$ is sufficiently small relative to $\delta$.
\item[Step 3.] Complete the proofs of Lemmas~\ref{lem:one_chunk_quantile} and~\ref{lem:two_chunks_quantile}.  In the case of Lemma~\ref{lem:one_chunk_quantile}, we will use Proposition~\ref{prop:cpi_path_close} to argue that given $\CN$ there is a positive chance that some additional CPIs construct a domain as in Step 2 which contains a crossing $\omega$ in $\Upsilon_+$ satisfying  $\lebneb{\epsilon}(\omega) \leq \quantHP{p}{\epsilon}$ where the shape of the auxiliary domain is such that the crossing is forced to connect $\partial B(z_R, \zeta_0 d_0)$ to $\wt{I}$ while staying at distance at least $\xi_0 d_0$ from $\partial \CN \setminus \wt{I}$.  The proof of Lemma~\ref{lem:two_chunks_quantile} is the same except we want the crossing to traverse $\interior{\closure{\CN_0 \cup \CN_1}}$ instead of a single $\SLE_\kappa^0(\kappa-6)$ chunk so we will not explain the argument in this case in detail.
\end{enumerate}

\subsection{A CPI has a positive chance of being close to any path in the carpet}
\label{subsec:cpi_close_to_path}

In \cite{msw2020nonsimplenotdetermined} it was proved that a CPI inside of a $\CLE_\kappa$ carpet~$\Upsilon$ is a.s.\ not determined by~$\Upsilon$.  That is, the conditional law of a CPI in $\Upsilon$ given $\Upsilon$ is a.s.\ non-atomic.  In what follows, we will upgrade this statement to show that if~$\omega$ is any continuous path in~$\Upsilon$ connecting distinct boundary points then a CPI connecting the same boundary points has positive conditional probability given~$\Upsilon$ of staying in the $\epsilon$-neighborhood of $\omega$.  We will phrase the following statement for a deterministic path $\omega$ but by considering the finite collection of piecewise polygonal paths which connect the vertices in $(\epsilon \Z^2) \cap \D$ the same statement a.s.\ holds for all $\omega$ simultaneously.

\begin{proposition}
\label{prop:cpi_path_close}
Let $\omega \colon [0,1] \to \D$ be a continuous curve with $\omega(0) = -i$ and $\omega(1) = i$.  Fix $\epsilon > 0$, let $\omega^\epsilon$ be the $\epsilon$-neighborhood of $\omega$, and let $K = \D \setminus \omega^\epsilon$.  Suppose that $\Gamma$ is a $\CLE_\kappa$ on $\D$, $\Upsilon$ its carpet, and let $K^*$ be the closure of the union of $K$ and the loops of $\Gamma$ which intersect $K$.  Let $\eta$ be an $\SLE_\kappa^1(\kappa-6)$ in $\D$ from $-i$ to $i$ coupled with $\Gamma$ as a CPI and let $\eta'$ be its trunk.  Let $E$ be the event that $K^*$ does not disconnect $-i$ from $i$ in $\omega^\epsilon$.  Then 
\[ \p[ \eta' \subseteq \omega^\epsilon \giv \Upsilon] > 0 \quad\text{a.s.\ on}\quad E.\]
\end{proposition}

In order to complete the proof of Proposition~\ref{prop:cpi_path_close} we need to recall a few things from \cite{msw2020nonsimplenotdetermined}.  Suppose that $\Gamma'$ is a $\CLE_{\kappa'}$ on $\D$.  Let $\eta'$ be the branch of the exploration tree of $\Gamma'$ from $-i$ to $i$ and let $\ol{\eta}'$ be its time-reversal.  Suppose that $t > 0$.  On the event that $\eta'|_{[0,t]}$ has not yet reached $i$,  we have that~$\eta'$ is a.s.\ exploring some loop $\CL' \in \Gamma'$ at the time $t$.  Let $t_0$ be such that $[t_0,t]$ is the interval of time in which $\eta'|_{[0,t]}$ is exploring $\CL'$.  Suppose that $\ol{t} > 0$ and we are on the event that $\ol{\eta}'|_{[0,\ol{t}]}$ has not hit $\eta'([0,t])$.  Then $\ol{\eta}'$ is a.s.\ exploring some loop $\ol{\CL}' \in \Gamma'$ at the time $\ol{t}$.  Let $\ol{t}_0$ be such that $[\ol{t}_0,\ol{t}]$ is the interval of time in which $\ol{\eta}'|_{[0,\ol{t}]}$ is exploring $\ol{\CL}'$.  Let $D_{t,\ol{t}}$ be the component of $\D \setminus (\eta'([0,t]) \cup \ol{\eta}'([0,\ol{t}]))$ with $\eta'(t)$, $\ol{\eta}'(\ol{t})$ on its boundary.  Then it is shown in \cite[Lemma~3.1]{msw2020nonsimplenotdetermined} that the conditional law of the part of $\Gamma'$ contained in $D_{t,\ol{t}}$ is a conformally invariant function of $D_{t,\ol{t}}$ and the marked points $\eta'(t)$, $\eta'(t_0)$, $\ol{\eta}'(\ol{t})$, $\ol{\eta}'(\ol{t}_0)$.  Moreover, let $\eta_0'$ be equal to $\eta'$ starting from $\eta'(t)$, $\eta_1'$ the time-reversal of $\eta_0'$ starting from $\eta'(t_0)$, $\ol{\eta}_0'$ be equal to $\ol{\eta}'$ starting from $\ol{\eta}'(\ol{t})$, and $\ol{\eta}_1'$ the time-reversal of $\ol{\eta}_0'$ starting from $\ol{\eta}'(\ol{t}_0)$.  Then $\eta_0'$ can terminate at either $\eta'(t_0)$, in which case $\CL' \neq \ol{\CL}'$, or at $\ol{\eta}'(\ol{t})$, in which case $\CL' = \ol{\CL}'$.  (The probability of the two ways of the paths hooking up was subsequently computed explicitly in \cite{mw2018connection}.)  Then it is shown in \cite[Section~4.1]{msw2020nonsimplenotdetermined} that conditionally on the ranges of $\eta_0'$, $\eta_1'$, $\ol{\eta}_0'$, $\ol{\eta}_1'$ and the event $\CL' \cap \ol{\CL}' \neq \emptyset$, either possibility has positive probability.

We recall that there is a natural exploration path $\Lambda$ associated with the iterated $\BCLE$ construction of a $\CLE_\kappa$ \cite{msw2017clepercolations} which we will now describe.  For simplicity, we will assume that the $\CLE_\kappa$ is in $\D$.  We first let $\Lambda_0$ be given by the path which follows $\partial \D$ counterclockwise starting from $-i$.  Let $\Lambda_1$ be given by $\Lambda_0$ except whenever $\Lambda_0$ hits a $\cwBCLE_{\kappa'}(0)$ loop then it follows it clockwise in its entirety before continuing along $\Lambda_0$.  We then let $\Lambda_2$ be given by $\Lambda_1$ except whenever it hits a $\ccwBCLE_\kappa(-\kappa/2)$ loop it follows it counterclockwise in its entirety before continuing along $\Lambda_1$.  We then continue iterating this construction in the false loops of the $\cwBCLE_{\kappa'}(0)$ and of the $\ccwBCLE_\kappa(-\kappa/2)$'s to obtain a path which follows the entire iterated $\BCLE$ construction of the $\CLE_\kappa$ instance.  It is explained in \cite{msw2017clepercolations} that this a.s.\ defines a continuous path, which we call $\Lambda$.  We note that if we take $\Lambda$ and target it at any point in $\partial \D$ then we obtain a CPI in $\Gamma$ starting from~$-i$ (and which has the law of an $\SLE_\kappa^1(\kappa-6)$ process).

Finally, let us recall that a trunk $\eta'$ of an $\SLE_\kappa^1(\kappa-6)$ process in $\h$ from $0$ to $\infty$ is an $\SLE_{\kappa'}(\kappa'-6)$ process where the force point is at $0^+$.  Moreover, by \cite[Theorem~1.4]{ms2016ig1} (see also \cite[Figure~2.5]{mw2017intersections}) the law of the left (resp.\ right) boundary of the trunk is that of an $\SLE_\kappa(\kappa-4;2-\kappa)$ (resp.\ $\SLE_\kappa(\kappa/2-2;-\kappa/2)$) process.  In particular, the right boundary a.s.\ does not hit $(-\infty,0)$ since $\kappa/2-2$ is the critical value at or above which an $\SLE_\kappa(\rho)$ process does not hit the boundary.  On the other hand, since $\kappa \in (8/3,4)$ we have that $2-\kappa < \kappa/2-2$ so that the left boundary a.s.\ does hit $(0,\infty)$.  Finally, $\eta'$ hits the points in its left (resp.\ right) boundary in order so can be decomposed into a collection of it excursions from the left (resp.\ right) boundary.

\begin{proof}[Proof of Proposition~\ref{prop:cpi_path_close}]
We consider the following Markovian resampling operation which preserves the law of $\Lambda$.  Let $z \in \D$ and $U \in [0,1]$ be independently sampled from Lebesgue measure and independently of $\Lambda$.  For each $t \geq 0$, we let $D_t$ be the component of $\D \setminus \Lambda([0,t])$ which contains $z$.  We let~$\tau_0$ be the first time~$t$ that the conformal radius $\confrad(z,D_t)$ of~$D_t$ as seen from~$z$ is at most~$U$ or $\partial D_t$ is a loop of $\Gamma$.  In the latter case we set $\tau = \tau_0$.  In the former case, if $\Lambda$ is not drawing a $\CLE_\kappa$ loop at time $\tau_0$ then we also set $\tau = \tau_0$.  If $\Lambda$ is drawing a $\CLE_\kappa$ loop at time $\tau_0$, we let $\tau$ be the time after $\tau_0$ that $\Lambda$ completes the $\CLE_\kappa$ loop.  Suppose that we are working on the event that $\partial D_\tau$ is not a loop of $\Gamma$.  Let $\varphi \colon D_\tau \to \D$ be the unique conformal map with $\varphi(z) = 0$ and $\varphi(\Lambda(\tau)) = -i$.  Let $\sigma > \tau$ be such that $[\tau,\sigma]$ is the interval of time in which $\Lambda$ is in $D_\tau$.  By the definition of $\Lambda$ and the iterated $\BCLE$ construction, modulo time parameterization we have that $\varphi(\Lambda|_{[\tau,\sigma]})$ has the same law as $\Lambda$.  Let $\ol{x} \in \partial \D$ be sampled from Lebesgue measure and $t, \ol{t} \sim \exp(1)$ independently of each other and everything else.  Let $\eta_0'$, $\eta_1'$, $\ol{\eta}_0'$, $\ol{\eta}_1'$ be the $4$-tuple of paths as described above defined in terms of the branch of the exploration tree of the $\CLE_{\kappa'}$ in $\D$ from $-i$ to $\ol{x}$ described by $\varphi(\Lambda|_{[\tau,\sigma]})$ and the times $t$, $\ol{t}$.  We then resample $\eta_0'$, $\eta_1'$, $\ol{\eta}_0'$, $\ol{\eta}_1'$ from their conditional law given their range.  As we have explained above, by \cite[Section~4.1]{msw2020nonsimplenotdetermined} we have that on the event that the $\CLE_{\kappa'}$ loops described by $\eta_0'$, $\eta_1'$, $\ol{\eta}_0'$, $\ol{\eta}_1'$ intersect it is a positive conditional probability event that after resampling $\eta_0'$, $\ol{\eta}_0'$ describe part of the same $\CLE_{\kappa'}$ loop or a different $\CLE_{\kappa'}$ loop.

Suppose that $N$ is sampled from the geometric distribution with parameter $1/2$ independently of everything else.  In the remainder of the proof, we will show that if we run this Markov resampling step $N$ times then on the event $E$ from the statement it is a positive probability event given $\Upsilon$ that a CPI in $\Upsilon$ from $-i$ to $i$ reaches $i$ before leaving $\omega^\epsilon$.  Upon showing this, the proof of the lemma will be complete.

We assume that we are working on $E$.  Let $V$ be the component of $\D \setminus K^*$ with $-i,i$ on its boundary, $\Gamma_V$ the loops of $\Gamma$ contained in~$V$, and let $I_- = \ccwBoundary{a_-}{b_-}{\partial \D}$ (resp.\ $I_+ = \cwBoundary{a_+}{b_+}{\partial \D}$) be the arc of $\partial \D \cap \partial V$ which contains $-i$ (resp.\ $i$).  Fix $x_-$ (resp.\ $x_+$) in the interior of $I_-$ (resp.\ $I_+$).  Let $\Gamma_V'$ be a $\CLE_{\kappa'}$ in $V$ coupled with $\Gamma_V$ as in the $\BCLE$ construction of $\CLE$.  Let $\eta'$ be the branch of the exploration tree of $\Gamma_V'$ from $x_-$ to $x_+$.  Then $\eta'$ is the trunk of an $\SLE_\kappa^1(\kappa-6)$ process from $x_-$ to $x_+$ in $V$ coupled with $\Gamma_V$ as a CPI.  As explained above, the right boundary of $\eta'$ a.s.\ does not hit $\cwBoundaryOpen{x_-}{x_+}{\partial V}$.  Let $x_-' \in \ccwBoundaryOpen{a_-}{x_-}{\partial \D}$ and $x_+' \in \cwBoundaryOpen{a_+}{x_+}{\partial \D}$.  Let $n$ be the number of excursions that $\eta'$ makes from its right boundary which intersect $\cwBoundary{x_-'}{x_+'}{\partial V}$ and note that $n$ is a.s.\ finite by the continuity of $\eta'$.

More generally, let $\eta_j$, $\eta_j'$ be the paths after performing the resampling step $j$ times and let $n_j$ be the number of excursions that $\eta_j'$ makes from its right boundary which intersect $\cwBoundary{x_-'}{x_+'}{\partial V}$.  Let us assume that $n \geq 1$ and let $\eta_{0,1}',\ldots,\eta_{0,n}'$ be the $n$ excursions that $\eta' = \eta_0'$ makes from its right boundary which hit $\cwBoundary{x_-'}{x_+'}{\partial V}$ ordered chronologically.  We assume that $\eta_{0,1}'$ starts from where $\eta'$ first starts drawing it and let $\ol{\eta}_{0,1}'$ be its time-reversal.  Let $\tau_{0,1}$ (resp.\ $\ol{\tau}_{0,1}$) be the first time that $\eta_{0,1}'$ (resp.\ $\ol{\eta}_{0,1}'$) hits $\cwBoundary{x_-'}{x_+'}{\partial V}$ and let $W$ be a component of $V \setminus (\eta_{0,1}'([0,\tau_{0,1}]) \cup \ol{\eta}_{0,1}'([0,\ol{\tau}_{0,1}]))$ whose boundary is contained in and has non-empty intersection with both $\eta_{0,1}'([0,\tau_{0,1}])$ and in $\ol{\eta}_{0,1}'([0,\ol{\tau}_{0,1}])$.  We note that such a component a.s.\ exists.

Given $\eta$ and $\Upsilon$, it is a positive probability event that $z \in W$ and $U$ is such that $\Lambda|_{[0,\tau]}$ is in $\partial W \setminus \eta_{0,1}'$ and has not drawn all of $\partial W \cap \eta_{0,1}'$.  Given this, there is a positive chance that $\ol{x}$, $t$, $\ol{t}$ are such that $\varphi^{-1}(\eta_0' \cup \eta_1' \cup \ol{\eta}_0' \cup \ol{\eta}_1')$ and the loops of $\Gamma$ it intersects contains $\partial W$.  On this event, as we have explained above, the Markovian resampling step has positive chance of making it so that $n_1 = n-1$.  By repeating this $n-1$ further times, we see that there is a positive chance that $\eta_n'$ does not hit $\cwBoundary{x_-'}{x_+'}{\partial V}$.

The above implies that on $E$ it is a positive conditional probability event given $\Upsilon$ that $\eta'$ does not hit $\cwBoundary{x_-'}{x_+'}{\partial V}$.  We note that by definition $\eta'$ is given by targeting at $x_+$ the path which visits the loops of $\Gamma_V'$ ordered according to when they are hit by $\ccwBoundary{x_-}{x_+}{\partial V}$.  Thus what we have shown is that on $E$ it is a positive conditional probability event given $\Upsilon$ that the loops of $\Gamma_V'$ which hit $\ccwBoundary{x_-}{x_+}{\partial V}$ do not hit $\cwBoundary{x_-'}{x_+'}{\partial V}$.  Applying this principle a second time implies that the following is true.  Suppose that $x_-'' \in \ccwBoundaryOpen{x_-}{b_-}{\partial \D}$ and $x_+'' \in \cwBoundaryOpen{x_+}{b_+}{\partial \D}$.  Then on $E$ it is a positive conditional probability event given $\Upsilon$ that the loops of $\Gamma_V'$ which hit either $\cwBoundary{x_-'}{x_+'}{\partial V}$ or $\ccwBoundary{x_-''}{x_+''}{\partial V}$ do not disconnect $x_-$ from $x_+$ in $V$.  Let $W$ be the complementary component in~$V$ of these loops which has $x_-,x_+$ on its boundary.  We note that the loops $\Gamma_W'$ of $\Gamma_V'$ which are contained in $\closure{W}$ have the law of a $\CLE_{\kappa'}$ in $W$.

On the above event, we note that there a.s.\ exists a loop $\CL_-'$ (resp.\ $\CL_+'$) of $\Gamma_W'$ which disconnects $x_-$ (resp.\ $x_+$) from $x_+$ (resp.\ $x_-$) and does not hit either $\cwBoundary{x_-'}{x_+'}{\partial V}$ or $\ccwBoundary{x_-''}{x_+''}{\partial V}$.  Since the graph of $\CLE_{\kappa'}$ loops is a.s.\ connected, there a.s.\ exists $n$ loops $\CL_1',\ldots,\CL_n'$ in $\Gamma_W'$ so that $\CL_1' \cap \CL_-' \neq \emptyset$, $\CL_n' \cap \CL_+' \neq \emptyset$, and $\CL_j' \cap \CL_{j-1}' \neq \emptyset$ for each $2 \leq j \leq n$.  Thus applying the resampling procedure one more time, it is a positive probability event that $\CL_1',\ldots,\CL_n'$ get connected into a single~$\CLE_{\kappa'}$ loop in~$\closure{W}$ which disconnects~$x_-$ from~$x_+$ and vice-versa.  On this event, the branch of the exploration tree of the resulting $\CLE_{\kappa'}$ does not hit either $\cwBoundary{x_-'}{x_+'}{\partial V}$ or $\ccwBoundary{x_-''}{x_+''}{\partial V}$ and is a CPI of~$\Gamma_V$.  Altogether, this completes the proof.
\end{proof}

\subsection{Crossing a disk bounded by CPIs away from the boundary}
\label{subsec:crossing_between_disks}

\begin{figure}[ht!]
\begin{center}
\includegraphics[scale=1]{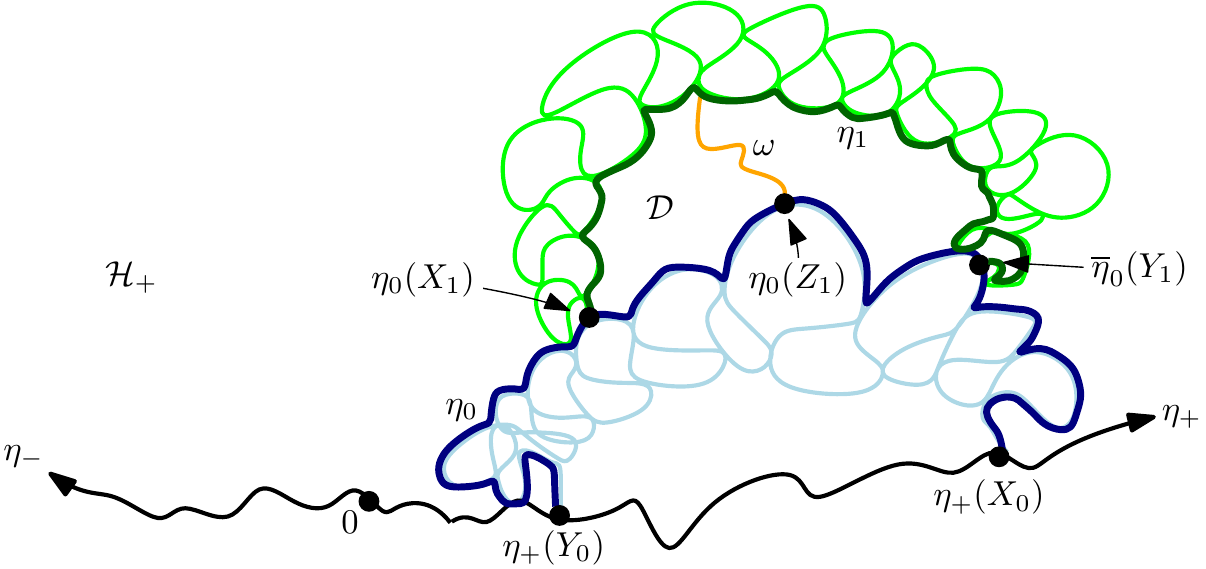}	
\end{center}
\caption{Illustration of the statement of Lemma~\ref{lem:disk_good_connection_intermediate}.  The light blue curve is the trunk of $\wt{\eta}_0$ and the green curve is the trunk of $\wt{\eta}_1$.  The left boundary ($\eta_0$ and its time-reversal $\ol{\eta}_0$) of $\wt{\eta}_0$ (as viewed from $\eta_+(Y_0)$) is shown in dark blue and the left boundary ($\eta_1$) of $\wt{\eta}_1$ (as viewed from $\ol{\eta}_0(Y_1)$) is shown in dark green.  The $\CLE_\kappa$ loops visited by $\wt{\eta}_0$, $\wt{\eta}_1$ are not shown as we have taken them to be $\SLE_\kappa^{-1}(\kappa-6)$ processes (i.e., $\CLE_\kappa$ loops are only on the left side of the trunk) so that $\partial \CD$ is only composed of parts of the boundaries of the trunks of $\wt{\eta}_0$, $\wt{\eta}_1$ and not of $\CLE_\kappa$ loops.}	
\end{figure}

\begin{lemma}
\label{lem:disk_good_connection_intermediate}
For each $\zeta_0, p \in (0,1)$ there exists $\delta_0 > 0$ so that for all $\delta \in (0,\delta_0)$ there exists $\epsilon_0 > 0$ so that for all $\epsilon \in (0,\epsilon_0)$ the following is true.  Let $X_0$ be an exponential random variable with mean~$\delta$ and let $Y_0$ be uniform in $[0,X_0]$, independently of everything else.  Let $\wt{\eta}_0$ be an $\SLE_\kappa^{-1}(\kappa-6)$ in~$\CH_+$ from $\eta_+(X_0)$ to $\eta_+(Y_0)$ coupled with $\Gamma_+$ as a CPI, let $\eta_0$ be its left boundary, assume that $\eta_0$ is parameterized according to quantum length, and let $L$ be the total quantum length of $\eta_0$.  Given $L$, let $X_1,Y_1$ be i.i.d.\ uniform in $[0,L]$ independently of everything else.  Let $\ol{\eta}_0$ be the time-reversal of~$\eta_0$.  On the event that $\eta_0([0,X_1]) \cap \ol{\eta}_0([0,Y_1]) = \emptyset$ and the part of $\eta_0$ from $\eta_0(X_1)$ to $\ol{\eta}_0(Y_1)$ (i.e., $\eta_0|_{[X_1,L-Y_1]}$) does not hit $\partial \CH_+$, let $\wt{\eta}_1$ be an $\SLE_\kappa^{-1}(\kappa-6)$ process in the unbounded component of $\CH_+ \setminus \eta_0$ from $\eta_0(X_1)$ to $\ol{\eta}_0(Y_1)$, coupled as a CPI of the loops of $\Gamma_+$ in this component.  Let $\eta_1$ be the left boundary of $\wt{\eta}_1$.  Let $Z_1$ be uniform in $[X_1, L-Y_1]$ and let $\CD$ be the component of $\CH_+ \setminus (\eta_0 \cup \eta_1)$ with $\eta_0(Z_1)$ on its boundary.  Let $E$ be the event that there does not exist a path $\cpath$ in $\CD \cap \Upsilon_+$ from $\eta_0(Z_1)$ to $\eta_1$ with $\lebneb{\epsilon}(\cpath) \leq \quantHP{p}{\epsilon}$.  Then $\p[E] \leq \zeta_0$.
\end{lemma}

The strategy to prove Lemma~\ref{lem:disk_good_connection_intermediate} is the following.  First, we will prove a technical result (Lemma~\ref{lem:middle_part_abs_cont}) which gives that the conditional law of the intermediate part of an $\SLE_\kappa(\rho_1;\rho_2)$ process is absolutely continuous with respect to ordinary $\SLE_\kappa$ on the event it does not intersect the domain boundary.  This is purely a statement about $\SLE_\kappa(\rho_1;\rho_2)$ processes whose proof can be skipped on a first reading of the proof of Lemma~\ref{lem:disk_good_connection_intermediate}.  Next, we will prove in Lemma~\ref{lem:chunk_exit} that if we consider an $\SLE_\kappa^{-1}(\kappa-6)$ process $\eta$ on $\CH_+$ from $\eta_-(\delta)$ to $\infty$ then provided $\delta > 0$ is sufficiently small it is very likely that there is a path $\cpath$ in $\CH_+$ from $0$ to the top of the quantum surface disconnected by $\eta$ from $\infty$ with $0$ on its boundary satisfying $\lebneb{\epsilon}(\cpath) \leq \quantHP{p}{\epsilon}$.  We will then complete the proof of Lemma~\ref{lem:disk_good_connection_intermediate} by using that the law of the part of $\partial \CD$ on $\eta_0$ near $\eta_0(Z_1)$ is comparable to the law of~$\partial \CH_+$ near~$0$.

\begin{lemma}
\label{lem:middle_part_abs_cont}
Suppose that $D \subseteq \C$ is a simply connected domain and $x,y \in \partial D$ are distinct.  Fix $\rho_1,\rho_2 > -2$.  Let $\eta$ be an $\SLE_\kappa(\rho_1; \rho_2)$ process in $D$ from $x$ to $y$.  Let $\tau$ (resp.\ $\ol{\tau}$) be a stopping time (resp.\ reverse stopping time) for $\eta$.  The conditional law of $\eta|_{[\tau,\ol{\tau}]}$ given $\eta|_{[0,\tau]}$ and $\eta|_{[\ol{\tau},\infty)}$ on the event that $\eta|_{[\tau,\ol{\tau}]}$ does not intersect $\partial D$ is absolutely continuous with respect to that of an $\SLE_\kappa$ from $\eta(\tau)$ to $\eta(\ol{\tau})$ in the component of $D \setminus (\eta([0,\tau]) \cup \eta([\ol{\tau},\infty))$ with $\eta(\tau)$ on its boundary.
\end{lemma}
\begin{proof}
It suffices to prove the result in the case that $D = \h$, $x = 0$, and $y = \infty$.  We can view $\eta$ as a flow line of a GFF $h$ on $\h$ with boundary conditions given by $-\lambda(1+\rho_1)$ on $\R_-$ and $\lambda(1+\rho_2)$ on~$\R_+$.  Let $\eta_L'$ be the counterflow line of $h+\pi \chi/2$ from $\infty$ to $0$.  Then $\eta$ is equal to the right boundary of $\eta_L'$.  Let $\eta_R'$ be the counterflow line of $h-\pi \chi/2$ from $\infty$ to $0$.  Then $\eta$ is also equal to the left boundary of $\eta_R'$.  In other words, $\eta$ is equal to the common boundary of $\eta_L'$ and $\eta_R'$.

Let $\tau_L$ (resp.\ $\tau_R$) be the first time that $\eta_L'$ (resp.\ $\eta_R'$) hits $\eta(\ol{\tau})$.  Then the common boundary of $\eta_L'([0,\tau_L])$ and $\eta_R'([0,\tau_R])$ is equal to $\eta([\ol{\tau},\infty))$.  It is shown in the proof of \cite[Lemma~5.13]{ms2016ig2} that $A_{\ol{\tau}} = \eta_L'([0,\tau_L]) \cup \eta_R'([0,\tau_R])$ is a local set for $h$.  Let $\CF_{\tau,\ol{\tau}}$ be the $\sigma$-algebra generated by $\eta|_{[0,\tau]}$, $\eta_L'|_{[0,\tau_L]}$, $\eta_R'|_{[0,\tau_R]}$.  Then we moreover have that the conditional law of $\eta|_{[\tau,\ol{\tau}]}$ given $\CF_{\tau,\ol{\tau}}$ is that of an $\SLE_\kappa(\rho_1, \tfrac{\kappa}{2}-2-\rho_1; \rho_2, \tfrac{\kappa}{2}-2-\rho_2)$ process in the component $D_{\tau,\ol{\tau}}$ of $\h \setminus (\eta([0,\tau]) \cup A_{\ol{\tau}})$ with $\eta(\tau)$ on its boundary from $\eta(\tau)$ to $\eta(\ol{\tau})$.  The force points are located at the leftmost and rightmost intersections of $\eta|_{[0,\tau]}$ with $\partial \h$ and the intersections of $\eta_L'|_{[0,\tau_L]}$ and $\eta_R'|_{[0,\tau_R]}$ on~$\R_+$ and~$\R_-$, respectively, which are closest to~$0$.

Let $D_L$ (resp.\ $D_R$) be the component of $\h \setminus \eta$ which is to the left (resp.\ right) of $\eta$ with $\eta(\ol{\tau})$ on its boundary.  Let $\eta_L$ (resp.\ $\eta_R$) denote the left (resp.\ right) boundary of $\eta_L'([0,\tau_L])$ (resp.\ $\eta_R'([0,\tau_R])$).  Then $\eta_L$ (resp.\ $\eta_R$) is the flow line of $h$ restricted to $D_L$ (resp.\ $D_R$) starting from $\eta(\ol{\tau})$ with angle $\pi$ (resp.\ $-\pi$).  In particular, the conditional law of $\eta_L$ given $\eta$ is that of an $\SLE_\kappa(\kappa/2-2,\rho_1;-\kappa/2)$ in $D_L$ starting from $\eta(\ol{\tau})$ and the conditional law of $\eta_R$ given $\eta$ is that of an $\SLE_\kappa(-\kappa/2;\kappa/2-2,\rho_2)$ in $D_R$ starting from $\eta(\ol{\tau})$.  Let $\CG_{\tau,\ol{\tau}}$ be the $\sigma$-algebra generated by $\eta|_{[0,\tau]}$, $\ol{\eta}|_{[\ol{\tau},\infty)}$, $\eta_L$, and $\eta_R$.  It is moreover shown in \cite{ms2016ig2} that the conditional law of $\eta$ given $\CG_{\tau,\ol{\tau}}$ is the same as the conditional law of~$\eta$ given $\CF_{\tau,\ol{\tau}}$.

Fix $\epsilon > 0$.  Let $\ol{\tau}_\epsilon$ be the first time $t$ that $\eta$ gets within distance $\epsilon$ of $A_{\ol{\tau}}$.  Then it follows that the law of $\eta|_{[\tau,\ol{\tau}_\epsilon]}$ on the event that it does not hit $\partial D$ is absolutely continuous with respect to the law of an $\SLE_\kappa$ in $D_{\tau,\ol{\tau}}$ from $\eta(\tau)$ to $\eta(\ol{\tau})$ stopped upon getting within distance $\epsilon$ of $A_{\ol{\tau}}$.  Since $\eta$ hits $A_{\ol{\tau}}$ exactly at $\eta(\ol{\tau})$, it follows that if we let $\tau_\epsilon$ be the first time $t$ that $\eta$ gets within distance $\epsilon$ of $\eta(\ol{\tau})$ then the law of $\eta|_{[\tau,\tau_\epsilon]}$ on the event that it does not hit $\partial D$ is absolutely continuous with respect to the law of an $\SLE_\kappa$ in $D_{\tau,\ol{\tau}}$ from $\eta(\tau)$ to $\eta(\ol{\tau})$ stopped upon getting within distance $\epsilon$ of $\eta(\ol{\tau})$.  Let $U_{\tau,\ol{\tau}}$ be the component of $\h \setminus (\eta([0,\tau]) \cup \eta([\ol{\tau},\infty))$ with $\eta(\tau)$ on its boundary.  Then it follows from \cite[Proposition~5.3]{lsw2003confres} that the law of $\eta|_{[\tau,\tau_\epsilon]}$ on the event that it does not hit $\partial D$ is absolutely continuous with respect to the law of an $\SLE_\kappa$ in $U_{\tau,\ol{\tau}}$ from $\eta(\tau)$ to $\eta(\ol{\tau})$ stopped at the first time it gets within distance $\epsilon$ of $\eta(\ol{\tau})$. 

By reversing the roles of $\tau$ and $\ol{\tau}$ and using the reversibility of $\SLE_\kappa(\rho_1;\rho_2)$ established in \cite{ms2016ig2}, this implies that the following is true.  Let $\ol{\eta}$ be the time-reversal of $\eta$ and let $\ol{\sigma},\sigma$, respectively, be the times for $\ol{\eta}$ which correspond to $\ol{\tau},\tau$, respectively, for $\eta$.  Let $\ol{\sigma}_\epsilon$ be the first time that $\ol{\eta}$ gets within distance $\epsilon$ of $\eta(\tau)$.  Then the law of $\ol{\eta}|_{[\ol{\sigma},\ol{\sigma}_\epsilon]}$ on the event that it does not hit $\partial D$ is absolutely continuous with respect to the law of an $\SLE_\kappa$ in $U_{\tau,\ol{\tau}}$ from $\eta(\ol{\tau})$ to $\eta(\tau)$ stopped at the first time it gets within distance $\epsilon$ of $\eta(\tau)$.

For $\epsilon, \delta > 0$, we let $E_{\epsilon,\delta}$ be the event that $\ol{\eta}|_{[\ol{\sigma}_\epsilon,\sigma]}$ does not leave $B(\eta(\tau),\delta)$ and $\eta_L,\eta_R$ do not hit $B(\eta(\tau),2\delta)$.  Then for every $\zeta > 0$ there exists $\delta_0 > 0$ so that for all $\delta \in (0,\delta_0)$ there exists $\epsilon_0 > 0$ so that for all $\epsilon \in (0,\epsilon_0)$ we have that the probability of $E_{\epsilon,\delta}$ is at least $1-\zeta$.  Let $(\wt{\eta},\wt{\eta}_L,\wt{\eta}_R)$ be a triple where
\begin{enumerate}[(i)]
\item $\wt{\eta}$ is an $\SLE_\kappa$ in $U_{\tau,\ol{\tau}}$ from $\eta(\ol{\tau})$ to $\eta(\tau)$ and
\item the conditional law of $(\wt{\eta}_L,\wt{\eta}_R)$ given $\wt{\eta}$ is the same as the conditional law of $(\eta_L,\eta_R)$ given $\eta$.
\end{enumerate}
On $E_{\epsilon,\delta}$ and the event that $\ol{\eta}|_{[\ol{\sigma},\ol{\sigma}_\epsilon]}$ does not hit $\partial D$, we have that the joint law of the triple consisting of $\ol{\eta}|_{[\ol{\sigma},\ol{\sigma}_\epsilon]}$, $\eta_L$, and $\eta_R$ is absolutely continuous with respect to the triple consisting of $\wt{\eta}$ stopped upon getting within distance $\epsilon$ of $\eta(\tau)$, $\wt{\eta}_L$, and $\wt{\eta}_R$.

Let $\sigma_\epsilon^L$ (resp.\ $\sigma_\epsilon^R$) be the first time that $\eta_L$ (resp.\ $\eta_R$) gets within distance $\epsilon$ of $\partial \h$.  Then the conditional law of $\eta_L|_{[0,\sigma_\epsilon^L]}$ given $\eta$ is absolutely continuous with respect to the law of an $\SLE_\kappa(\kappa/2-2;-\kappa/2)$ stopped at the corresponding time.  Likewise, the conditional law of $\eta_R|_{[0,\sigma_\epsilon^R]}$ given $\eta$ is absolutely continuous with respect to the law of an $\SLE_\kappa(-\kappa/2;\kappa/2-2)$ stopped at the corresponding time.   Let $(\wh{\eta}, \wh{\eta}_L,\wh{\eta}_R)$ be a triple where
\begin{enumerate}[(i)]
\item $\wh{\eta}$ is an $\SLE_\kappa$ in $U_{\tau,\ol{\tau}}$ from $\eta(\ol{\tau})$ to $\eta(\tau)$,
\item the conditional law of $\wh{\eta}_L$ given $\wh{\eta}$ is that of an $\SLE_\kappa(-\kappa/2;\kappa/2-2)$ starting at $\ol{\eta}(\ol{\tau})$, and
\item the conditional law of $\wh{\eta}_R$ given $\wh{\eta}$ is that of an $\SLE_\kappa(\kappa/2-2;-\kappa/2)$ starting at $\ol{\eta}(\ol{\tau})$.
\end{enumerate}
On the event that $\ol{\eta}|_{[\ol{\sigma},\ol{\sigma}_\epsilon]}$ does not hit $\partial \h$, we have that the law of the triple consisting of $\ol{\eta}|_{[\ol{\sigma},\ol{\sigma}_\epsilon]}$, $\eta_L|_{[0,\sigma_\epsilon^L]}$, and $\eta_R|_{[0,\sigma_\epsilon^R]}$ is absolutely continuous with respect to the law of the triple consisting of $\wh{\eta}$ stopped at the first time it gets within distance $\epsilon$ of $\eta(\tau)$ and $\wh{\eta}_L, \wh{\eta}_R$ both stopped at the first time they get within distance $\epsilon$ of $\partial \h$.  In particular, on the event that $\ol{\eta}|_{[\ol{\sigma},\ol{\sigma}_\epsilon]}$ does not hit $\partial \h$, the law of the pair consisting of $\eta_L|_{[0,\sigma_\epsilon^L]}$ and $\eta_R|_{[0,\sigma_\epsilon^R]}$ is absolutely continuous with respect to the law of the pair $(\wh{\eta}_L,\wh{\eta}_R)$ both stopped at the corresponding time.

Let $F_{\epsilon,\delta}$ be the event that $\eta$ does not get within distance $\delta$ of $\eta_L([\sigma_\epsilon^L,\infty))$ or $\eta_R([\sigma_\epsilon^R,\infty))$.  Then for every $\zeta > 0$ there exists $\delta_0 > 0$ so that for all $\delta \in (0,\delta_0)$ there exists $\epsilon_0 > 0$ so that for all $\epsilon \in (0,\epsilon_0)$ we have that the probability of $F_{\epsilon,\delta}$ is at least $1-\zeta$.  On $F_{\epsilon,\delta}$, it follows that the conditional law of $\eta|_{[\tau,\ol{\tau}]}$ given $\eta_L$, $\eta_R$ and on the event that $\eta|_{[\tau,\ol{\tau}]}$ does not hit $\partial \h$ is absolutely continuous with respect to the conditional law of $\wh{\eta}$ given $\wh{\eta}_L$, $\wh{\eta}_R$.  Altogether, this implies that the conditional law of $\eta|_{[\tau,\ol{\tau}]}$ given $\eta_L$, $\eta_R$ and on the event that $\eta|_{[\tau,\ol{\tau}]}$ does not hit $\partial \h$ is absolutely continuous with respect to the conditional law of $\wh{\eta}$ given $\wh{\eta}_L$, $\wh{\eta}_R$.  Therefore the law of $\eta|_{[\tau,\ol{\tau}]}$ on the event that it does not hit $\partial \h$ is absolutely continuous with respect to the law of $\wh{\eta}$.
\end{proof}

\begin{figure}[ht!]
\begin{center}
\includegraphics[scale=1]{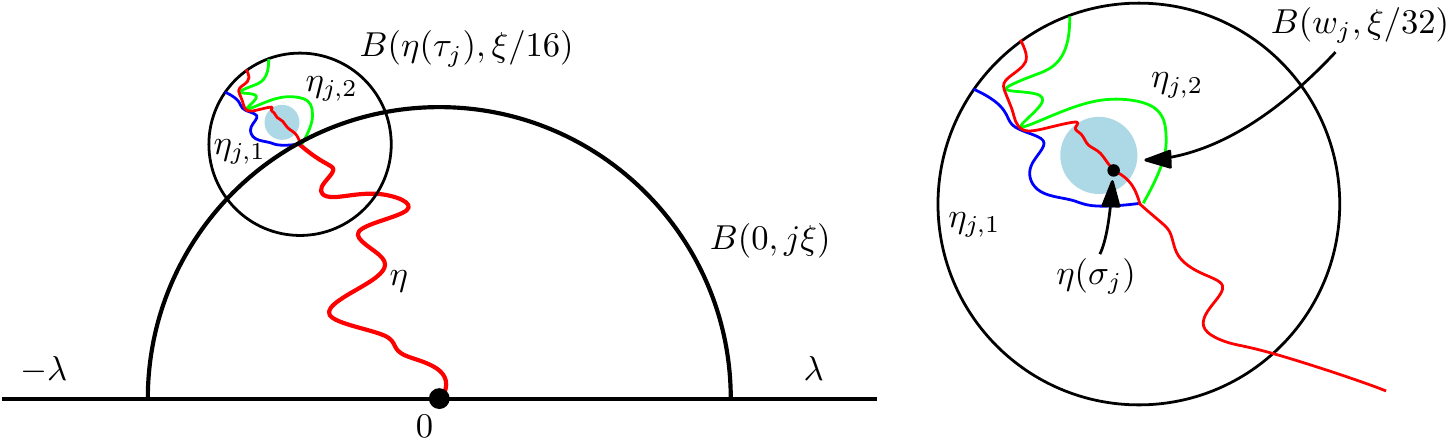}
\end{center}
\caption{\label{fig:chunk_exit_proof}  Illustration of the proof of Lemma~\ref{lem:chunk_exit}; the light blue ball is $B(w_j,\xi/32)$.}
\end{figure}

\begin{lemma}
\label{lem:chunk_exit}
Suppose that $\eta$ is an $\SLE_\kappa^{-1}(\kappa-6)$ in $\CH_+$ from $\eta_-(\delta)$ to $\infty$ which is coupled with $\Gamma_+$ as a CPI and parameterized by the quantum natural time of its trunk.  For each $p, \zeta_0 \in (0,1)$ there exists $\delta_0, \epsilon_0 > 0$ so that for all $\delta \in (0,\delta_0)$ and $\epsilon \in (0,\epsilon_0)$ the following is true.  Let $\CD$ be the component of $\CH_+ \setminus \eta$ with $\eta_+(0) = 0$ on its boundary.  Let $X_\epsilon = \inf_\cpath \lebneb{\epsilon}(\cpath)$ where the infimum is over all paths $\cpath$ in $\CH_+$ contained in $\Upsilon_+ \cap \CD$ and which connect~$0$ to $\eta \cap \partial \CD$.  Then $\p[X_\epsilon \geq \quantHP{p}{\epsilon}] \leq \zeta_0$.
\end{lemma}
\begin{proof}
\noindent{\it Step 1. Setup.} Fix $p \in (0,1)$.  It suffices to show that for every $\zeta_0 > 0$ there exists $\xi, \epsilon_0 > 0$ so that for all $\epsilon \in (0,\epsilon_0)$ the following is true.  Let $Y_\epsilon$ be the Euclidean diameter of the set of points in~$\CH_+$ which are connected to $0$ by paths $\cpath$ in $\Upsilon_+$ with $\lebneb{\epsilon}(\cpath) \leq \quantHP{p}{\epsilon}$.  Then $\p[ Y_\epsilon \leq \xi] \leq \zeta_0$.  Suppose that this is not true.  Then there exists $\zeta_0 > 0$ and a sequence $(\epsilon_n)$ of positive numbers with $\epsilon_n \to 0$ as $n \to \infty$ so that $\p[Y_{\epsilon_n} \leq 1/n] \geq \zeta_0$ for every $n \in \N$.  We are going to assume that this is true throughout the rest of the proof in order to show that
\[ \p\!\left[ \sup_{z,w \in \eta_+([0,1])} \metapprox{\epsilon_n}{z}{w}{\Gamma_+} \geq \quantHP{p}{\epsilon_n} \right] \to 1 \quad\text{as}\quad n \to \infty.\]
This, in turn, will contradict the definition of $\quantHP{p}{\epsilon}$.

\noindent{\it Step 2. First observation.}  Fix $\rho_L, \rho_R > -2$ and suppose that $\wt{\eta}$ is an $\SLE_\kappa(\rho_L ; \rho_R)$ in $\D$ from $-i$ to $i$.  Let $\wt{\Gamma}$ be a collection of loops in the components of $\D \setminus \wt{\eta}$ which are to the left of $\wt{\eta}$ which in each such component is a conditionally independent $\CLE_\kappa$.  Let $\wt{\Upsilon}$ be the closure of the union of the loops in $\wt{\Gamma}$.  Fix $z \in \D$ and $r > 0$ so that $B(z,2r) \subseteq \D$.  Then it follows from our assumption above that there exists a sequence $(\epsilon_n)$ of positive numbers with $\epsilon_n \to 0$ as $n \to \infty$ so that the following is true.  Let $\wt{\sigma} = \inf\{t \geq 0 : \wt{\eta}(t) \in B(z,r)\}$.  On $\wt{\sigma} < \infty$, let $\wt{Y}_{\epsilon_n}$ be the Euclidean diameter of the set of points to the left of $\wt{\eta}$ which are accessible by paths $\cpath$ in $\wt{\Upsilon}$ starting from $\wt{\eta}(\wt{\sigma})$ with $\lebneb{\epsilon_n}(\cpath) \leq \quantHP{p}{\epsilon_n}$.  Then $\p[ \wt{Y}_{\epsilon_n} \leq 1/n,\ \wt{\sigma} < \infty]$ is bounded from below by a positive constant which depends only on $\rho_L$, $\rho_R$, $z$, $r$, and $\zeta_0$ from Step 1.  Indeed, this follows from absolute continuity and our assumption in Step 1.

\noindent{\it Step 3. Approximate $\CLE_\kappa$-metric ball has small Euclidean diameter in many places on the interface.}  We are now going to show that the event that the ``$\metapprox{\epsilon_n}{\cdot}{\cdot}{\Gamma_+}$-metric ball'' at a point along $\eta_+$ has small diameter is likely to happen at many points.  In this step, we are going to focus on proving a version of this statement in the setting of~$\h$; in the next step we will subsequently transfer the result to the setting of $\eta_+$.  To this end, we suppose that~$h$ is a GFF on~$\h$ with boundary conditions given by $-\lambda$ on~$\R_-$ and~$\lambda$ on~$\R_+$.  Let~$\eta$ be the flow line of $h$ from $0$ to $\infty$.  Then $\eta$ is an $\SLE_\kappa$ curve in $\h$ from $0$ to $\infty$.  Fix $\xi > 0$ and $\theta \in (0,\pi)$.  For each $j$, we let $\tau_j = \inf\{t \geq 0 : \eta(t) \notin B(0, \xi j)\}$.  We then let $\eta_{j,1}$ (resp.\ $\eta_{j,2}$) be the flow line of $h$ starting from $\eta_j(\tau_j)$ with angle $2\lambda'/\chi - \pi$ (resp.\ $-\theta$).  Then $\eta_{j,1}$ is equal to the right boundary of the counterflow line $\eta_{j,1}'$ of $h + 2\lambda' - \pi \chi/2$ from $\infty$ to $\eta(\tau_j)$.  We note that the angle difference between $\eta_{j,1}$ and $\eta_{j,2}$ is equal to $2\lambda'/\chi - \pi + \theta$ and that flow lines intersect whenever the angle difference is between $0$ and $2\lambda'/\chi$.  That is, with any choice of $\theta \in (0,\pi)$ we have that $\eta_{j,1}$, $\eta_{j,2}$ a.s.\ intersect each other.

Given $\eta$, we let $\Gamma_L$ be a $\CLE_\kappa$ in the component of $\h \setminus \eta$ which is to the left of $\eta$.  We assume that $\Gamma_L$ is generated from $h$ using the coupling of $\SLE_\kappa(\kappa-6)$ with the GFF so that $\eta_{j,1}'$ is a CPI in $\Gamma_L$ for each $j$.  Then the conditional law of $\Gamma_L$ in any component $U$ of $\h \setminus (\eta \cup \eta_{1,j})$ which is to the left of $\eta$ and to the right of $\eta_{1,j}$ is that of a $\CLE_\kappa$ in $U$.

For each $j \in \N$ and $k \in \{1,2\}$, we let $\tau_{j,k}$ be the first time that $\eta_{j,k}$ leaves $B(\eta_j(\tau_j),\xi/16)$.  We let $\CF_j$ be the $\sigma$-algebra generated by $\eta|_{[0,\tau_j]}$ and $\eta_{i,k}|_{[0,\tau_{i,k}]}$ for $1 \leq i \leq j$ and $k \in \{1,2\}$.  We also let $E_j$ be the event that:
\begin{enumerate}[(i)]
\item\label{it:ej1} There exists $w \in \h$ so that $\eta_{j,1}|_{[0,\tau_{j,1}]}$, $\eta_{j,2}|_{[0,\tau_{j,2}]}$ do not intersect $B(w,\xi/32)$ and disconnect $B(w,\xi/32)$ from $\infty$ and the harmonic measure of $\eta_{j,k}([0,\tau_{j,k}])$ in $\h \setminus \cup_{k=1}^2 \eta_{j,k}([0,\tau_{j,k}])$ as seen from $w$ is at least $1/4$ for $k \in \{1,2\}$.
\item\label{it:ej2} Let $V_j$ be the component of $\h \setminus \cup_{k=1}^2 \eta_{j,k}([0,\tau_{j,k}])$ disconnected from $\infty$ as in~\eqref{it:ej1} which is first visited by $\eta_{j,1}$.  Let $w_j$ be the point $w$ as in~\eqref{it:ej1} with the smallest real part, breaking ties by taking the point with smallest imaginary part.  Let $\Gamma_{j,L}$ be the loops of $\Gamma_L$ which are contained in $V_j$.  Let $\sigma_j$ be the first time that $\eta$ visits $B(w_j,\xi/64)$.  Then $\sigma_j < \infty$ and $\metapprox{\epsilon_n}{\eta(\sigma_j)}{\partial B(w_j,\xi/32)}{\Gamma_{j,L}} \geq \quantHP{p}{\epsilon_n}$.
\end{enumerate}

We claim that there exists $p_0 > 0$ which does not depend on $j$ so that
\begin{align}
\label{eqn:ejlbd}
\p[ E_j \giv \CF_{j-1}] \geq p_0 \quad\text{for each}\quad j \in \N.
\end{align}
To start to prove~\eqref{eqn:ejlbd}, for each $j$ we let $z_j$ be the point on $\h \cap \partial B(0, \xi j)$ with $\arg(z_j) = ((\im(\eta(\tau_{j-1})) \wedge (\pi-\xi)) \vee \xi$.  Let $A_j$ be the union of $\eta([0,\tau_j])$ and $\eta_{i,k}([0,\tau_{i,k}])$ for $1 \leq i \leq j$ and $k \in \{1,2\}$.  Let $\varphi_j$ be the unique conformal map from the unbounded component of $\h \setminus A_{j-1}$ to $\h$ which fixes $\infty$ and takes $z_j$ to $i$.  Then there exists a constant $M  > 0$ so that $\varphi(\eta(\tau_{j-1}))$ is contained in $[-M,M]$ (as the harmonic measure in $\h \setminus \eta([0,\tau_{j-1}])$ as seen from $z_j$ of the left side of $\eta$ and $\R_-$ is bounded from below and the same is also true for the and right side of $\eta$ and $\R_+$). \cite[Lemmas~2.3--2.5]{mw2017intersections} thus imply that there exists $p_0 > 0$ which does not depend on $j$ such that the conditional probability given $\CF_{j-1}$ of part~\eqref{it:ej1} of $E_j$ is at least $p_0$.

We are now going to show that, by possibly decreasing the value of $p_0 > 0$, we have that~\eqref{eqn:ejlbd} holds.  Let $w_j$, $V_j$ be as in~\eqref{it:ej2}.  Let $\psi_j$ be the unique conformal transformation $V_j \to \D$ which takes the first (resp.\ last) point on $\partial V$ visited by $\eta$ to $-i$ (resp.\ $i$) and is such that $\im(\psi_j(w_j)) = 0$.  Then there exists $\rho_L, \rho_R > -2$ so that $\wt{\eta}_j = \psi_j(\eta)$ is an $\SLE_\kappa(\rho_L; \rho_R)$ in $\D$ from $-i$ to $i$. Since the harmonic measure of the clockwise (resp.\ counterclockwise) arc of $\partial \D$ from $-i$ to $i$ as seen from $\wt{w}_j = \psi_j(w_j)$ is at least $1/4$, it follows that there exists $r_0 > 0$ so that $B(\wt{w}_j,2r_0) \subseteq \D$.  Let $\wt{\sigma}_j$ be the first time that $\wt{\eta}_j$ visits $B(\wt{w}_j,r_0)$.  Let $\wt{Y}_{\epsilon_n}$ be the Euclidean diameter of the set of points to the left of $\wt{\eta}_j$ which are accessible by paths in $\cpath$ in $\psi_j(\Gamma_{j,L})$ starting from $\wt{\eta}_j(\wt{\sigma}_j)$ with $\lebneb{\epsilon_n}(\cpath) \leq \quantHP{p}{\epsilon_n}$.  Step 2 implies that there exists $q_0 > 0$ so that $\p[ \wt{Y}_{\epsilon_n} \leq 1/n,\ \wt{\sigma}_j < \infty] \geq q_0$.  We note that the restriction of $(\psi_j^{-1})'$ to $\psi_j(B(w_j,\xi/32))$ is bounded from below by a constant times $\xi$.  Therefore~\eqref{eqn:ejlbd} follows by conformally mapping back.

\noindent{\it Step 4.  Contradiction to the definition of $\quantHP{p}{\epsilon}$.}  Suppose that we have the setup described in Step~3.  Fix $N \in \N$ and let $F_N = \cup_{j=1}^N E_j$ be the event that $E_j$ occurs for some $1 \leq j \leq N$.  It follows from~\eqref{eqn:ejlbd} that $\p[F_N] \to 1$ as $N \to \infty$.  Let $K_\xi = [-\xi^{1/2}, \xi^{1/2}] \times [\xi/128,\xi^{1/2}]$ and let $G_N$ be the event that $E_j$ occurs and $\eta(\tau_j) \in K_\xi$.  Then it further follows that for each $p_0 \in (0,1)$ we can choose $\xi > 0$ sufficiently small so that with $N = \xi^{-1/2}$ we have that $\p[G_N] \geq p_0$.  Let $\varphi$ be the unique conformal transformation from $\h$ to $\C \setminus \eta_-$ which fixes $0$ and $\infty$ and takes $-1$ to the prime end corresponding to the point $\eta_-(1)$ which is on $\partial \CH_+$.  Then there exists a constant $c \geq 1$ depending only on $\eta_-$ so that the restriction of $|\varphi'|$ to $K_\xi$ is between $c^{-1} \xi^{-1}$ and $c \xi^{-1/2}$.  We assume that $\eta_+ = \varphi(\eta)$, modulo parameterization.  It therefore follows that on $G_N$ we have that the ``$\metapprox{\epsilon_n}{\cdot}{\cdot}{\Gamma_+}$-metric ball'' centered at $0$ of radius $\quantHP{p}{\epsilon_n}$ does not contain all of the image under $\varphi$ of $\eta$ up until time $\tau_N$.  As the probability that $\varphi(\eta([0,\tau_N]))$ is contained in $\eta_+([0,1])$ (recall that $\eta_+$ is parameterized by quantum length) tends to $1$ as $\xi \to 0$, the desired contradiction to the definition of $\quantHP{p}{\epsilon_n}$ follows.
\end{proof}

\begin{proof}[Proof of Lemma~\ref{lem:disk_good_connection_intermediate}]
In order to start to prove the lemma, we will first explain why a simpler version of the statement holds.  Let $\eta$ be an $\SLE_\kappa^{-1}(\kappa-6)$ process from $0$ to $\infty$ in $\CH_+$ coupled as a CPI of $\Gamma_+$.  Let $X$ be an exponential random variable with mean $1$ which is independent of everything else and let~$\CD_0$ be the component of $\CH_+ \setminus \eta$ with $\eta_+(X \delta)$ on its boundary.  It follows from the argument used to prove Lemma~\ref{lem:chunk_exit} that for each $\zeta_0 \in (0,1)$ there exists $\delta_0 \in (0,1)$ so that for each $\delta \in (0,\delta_0)$ there exists $\epsilon_0 \in (0,1)$ so that for every $\epsilon \in (0,\epsilon_0)$ the probability that there is a path $\cpath$ in $\Upsilon_+$ from $\eta_+(X\delta)$ to the part of $\partial \CD_0$ which is contained in $\eta$ with $\lebneb{\epsilon}(\cpath) \leq \quantHP{p}{\epsilon}$ is at least $1-\zeta_0$.

Let now describe the law of the ensemble of paths which make up $\partial \CD_0$.  We know that the conditional law of $\eta_+$ given~$\eta_-$ is that of an $\SLE_\kappa$ curve in $\C \setminus \eta_-$ from~$0$ to $\infty$.  Let $\eta_R$ be the right boundary of~$\eta$.  Note that~$\eta_R$ is also the right boundary of the trunk of~$\eta$.  Recall that the trunk is an $\SLE_{\kappa'}(\kappa'-6)$ process in~$\CH_+$ from~$0$ to~$\infty$ with its force point located at~$0^-$.   We thus know from \cite[Theorem~1.5]{ms2016ig1} that the law of~$\eta_R$ is equal to that of an $\SLE_\kappa(\kappa-4;2-\kappa)$ process from~$\infty$ to~$0$ in~$\CH_+$.  Also, note that~$\CD_0$ is the component of $\CH_+ \setminus \eta_R$ with $\eta_+(X\delta)$ on its boundary.

We are now going to make a comparison between the law of $\CD$ as defined in the lemma statement to the law of $\CD_0$ as described just above.  Let $L$ be the quantum length of $\eta_0$.  Let $G$ be the event that $\eta_0([0,X_1]) \cap \ol{\eta}_0([0,Y_1]) = \emptyset$ and the part of $\eta_0$ from $\eta_0(X_1)$ to $\ol{\eta}_0(Y_1)$, i.e., $\eta_0|_{[X_1,L-Y_1]}$ does not hit $\partial \CH_+$.  On $G$, Lemma~\ref{lem:middle_part_abs_cont} implies that the conditional law of $\eta_0|_{[X_1,L-Y_1]}$ is absolutely continuous with respect to the law of an $\SLE_\kappa$ in the unbounded component $U$ of $\CH_+ \setminus (\eta_0([0,X_1]) \cup \ol{\eta}_0([0,Y_1]))$ from $\eta_0(X_1)$ to $\ol{\eta}_0(Y_1)$.  We also have that the law of $\eta_1$ given $\eta_0$ is that of an $\SLE_\kappa(\kappa-4;2-\kappa)$ from $\ol{\eta}_0(Y_1)$ to $\eta_0(X_1)$.

The previous two paragraphs imply that the following is true.  Let $\varphi$ be the unique conformal transformation from $\C \setminus \eta_-$ to $U$ which takes $0$ to $\eta_0(X_1)$ and $\infty$ to $\ol{\eta}_0(Y_1)$.  Then it follows from the above that the law of the pair $(\eta_0, \eta_1)$ on $G$ is absolutely continuous with respect to the law of the pair $(\varphi(\eta_+), \varphi(\eta_R))$.  Moreover, since $\eta_+(X \delta)$ has a positive chance of being in the boundary of any fixed component of $\CH_+ \setminus \eta$ which intersects $\eta_+$, it follows that the law of $\CD$ is absolutely continuous with respect to the law of $\varphi(\CD_0)$.  The result then follows since we know that $|\varphi'|$ is a.s.\ bounded on a neighborhood of the closure of $\CD_0$ (recall~\eqref{eqn:quant_comparison}).
\end{proof}

\subsection{Proofs of main crossing statements}
\label{subsec:chunk_proofs}

\begin{figure}[ht!]
\begin{center}
\includegraphics[scale=1]{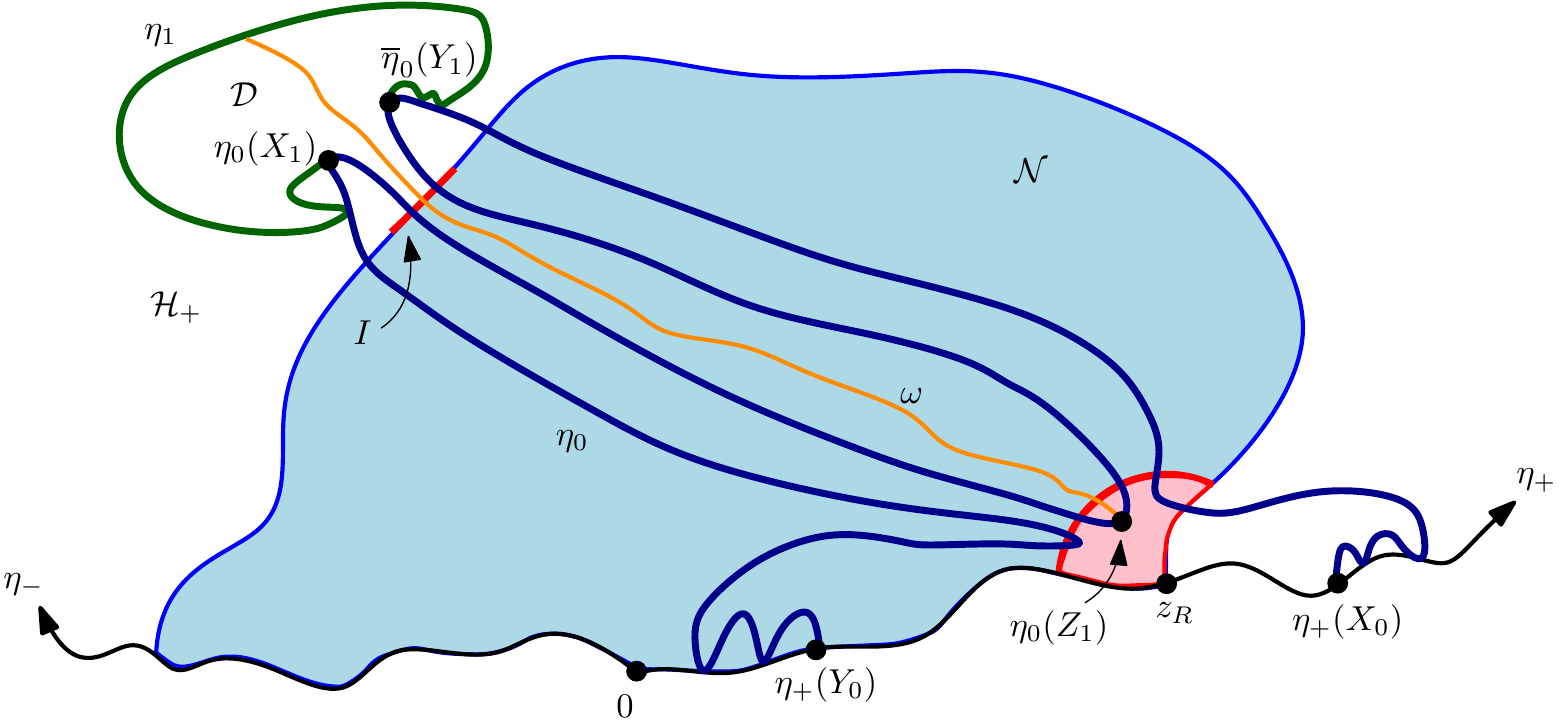}	
\end{center}
\caption{\label{fig:chunk_quantile_proof} Illustration of the proof of Lemma~\ref{lem:one_chunk_quantile}.  Shown are just the left boundaries of $\wt{\eta}_0$ ($\eta_0$, dark blue) as viewed from $\eta_+(Y_0)$ and $\wt{\eta}_1$ ($\eta_1$, dark green) as viewed from $\eta_0(X_1)$, using the notation from Lemma~\ref{lem:disk_good_connection_intermediate}.  Proposition~\ref{prop:cpi_path_close} implies that on the event that $I \cap \Upsilon \neq \emptyset$, it is a positive conditional probability event given $\CN$ that a path in $\CD$ from $\eta_0(Z_1)$ to $\eta_1$ in $\Upsilon$ passes through $I$ and $\eta_1(Z_1) \in B(z_R,\zeta_0 d_0)$, as shown.}
\end{figure}

\begin{proof}[Proof of Lemma~\ref{lem:one_chunk_quantile}]
See Figure~\ref{fig:chunk_quantile_proof} for an illustration of the argument.  Let $\eta$ be an $\SLE_\kappa^0(\kappa-6)$ process in $\CH_+$ from $0$ to $\infty$ coupled with $\Gamma_+$ as a CPI.  We assume that $\eta$ is parameterized according to the quantum natural time of its trunk.  Let $\tau_\delta = \inf\{t \geq \delta^{4/\kappa} : \eta(t) \in \partial \CH_+\}$ and let $\CN$ be the quantum surface parameterized by the region disconnected from $\infty$ by $\eta([0,\tau_\delta])$.   Let also $d_0 = \diam(\CN)$.

Let $\wt{\eta}_0$, $\wt{\eta}_1$, $\eta_0$, $\eta_1$, $X_0$, $Y_0$, $X_1$, $Y_1$, $Z_1$ be as in the statement of Lemma~\ref{lem:disk_good_connection_intermediate}, taken to be conditionally independent of everything else given $\Gamma_+$.  On $\eta_0([0,X_1]) \cap \ol{\eta}_0([0,Y_1]) = \emptyset$ let also $\CD$ be the component of $\CH_+ \setminus (\eta_0 \cup \eta_1)$ between $\eta_0$ and $\eta_1$ with $\eta_0(Z_1)$ on its boundary as in Lemma~\ref{lem:disk_good_connection_intermediate}.  Let $Q$ be the event that $\eta_0([0,X_1]) \cap \ol{\eta}_0([0,Y_1]) \neq \emptyset$, the part of $\eta_0$ from $\eta_0(X_1)$ to $\ol{\eta}_0(Y_1)$ does not hit $\partial \CH_+$, and there does not exist a path $\cpath$ in $\Upsilon_+ \cap \CD$ which connects $\eta_0(Z_1)$ to $\eta_1$ with $\lebneb{\epsilon}(\cpath) \leq \quantHP{p}{\epsilon}$.  Fix $\zeta_1 > 0$ and assume that we have chosen $\delta > 0$ sufficiently small so that Lemma~\ref{lem:disk_good_connection_intermediate} implies that $\p[Q] \leq \zeta_1$.

Let $Z$ be chosen uniformly from the boundary measure on the top of $\CN$ and let $I$ (resp.\ $\wt{I}$) be the interval of quantum length $\zeta_0 \delta$ (resp.\ $3\zeta_0 \delta /2$) on the top of $\CN$ centered at $Z$.  Let $E = \{I \cap \Upsilon_+ \neq \emptyset\}$.  Let $F$ be the event that $\eta_0([0,X_1]) \cap \ol{\eta}_0([0,Y_1]) = \emptyset$, the part of $\eta_0$ from $\eta_0(X_1)$ to $\ol{\eta}_0(Y_1)$ does not hit $\partial \CH_+$, and every path in $\Upsilon_+ \cap \CD$ from $\eta_0(Z_1)$ to $\eta_1$ passes through $I$, $\eta_0(Z_1) \in B(z_R,\zeta_0 d_0)$, and $\partial \CD$ has distance at least $\xi_0 d_0$ from $\partial \CN \setminus \wt{I}$.  Let $\CF$ be the $\sigma$-algebra generated by $\Gamma_+$, $\CN$, and $Z$.  Note that $E$ is $\CF$-measurable.  It follows from Proposition~\ref{prop:cpi_path_close} that
\[ \p\big[ \p[ F \giv \CF] > 0 \giv E \big] \to 1 \quad\text{as}\quad \xi_0 \to 0.\]
Fix $p_0 > 0$ and let $G = \{ \p[ F \giv \CF] \one_E \geq p_0 \one_E \}$.  Fix $\zeta_2 > 0$.  By making $p_0, \xi_0 > 0$ sufficiently small we have that $\p[G^c \cap E] \leq \zeta_2$.  We note that by scale invariance the choice of~$p_0$ does not depend on~$\delta$.  Therefore we can take $\zeta_1 = p_0 \zeta_2$ and assume that $\delta > 0$ is sufficiently small so that $\p[Q] \leq \zeta_1$.

Let $H$ be the event that $E$ occurs and every path $\cpath$ in $\Upsilon_+ \cap \CN$ connecting $B(z_R, \zeta_0 d_0)$ to $\wt{I}$ with distance at least $\xi_0 d_0$ from $\partial \CN \setminus \wt{I}$ satisfies $\lebneb{\epsilon}(\cpath) \geq \quantHP{p}{\epsilon}$.  Then we have that
\begin{align*}
   \p[ H ]
&= \E[ \one_E \one_H]
 \leq \frac{1}{p_0} \E[ \one_E \one_H \p[ F \giv \CF]] + \p[G^c \cap E]\\
&= \frac{1}{p_0} \p[E,\ H,\ F] + \p[G^c \cap E] \quad\text{($E,H \in \CF$)}\\
&\leq \frac{1}{p_0} \p[Q] + \p[G^c]
 \leq \frac{\zeta_1}{p_0} + \zeta_2 \leq 2 \zeta_2.
\end{align*}

Consider the event that there exists $J$ on the top of $\partial \CN$ with quantum length $\zeta_0$, intersects $\Upsilon_+$, and such that if $\wt{J}$ is the interval on the top of $\partial \CN$ with the same center as $J$ but twice the quantum length every path $\cpath$ from $B(z_R,\zeta_0 d_0)$ to $\wt{J}$ in $\Upsilon_+ \cap \CN$ and with distance at least $\xi_0 d_0$ to $\partial \CN \setminus \wt{J}$ satisfies $\lebneb{\epsilon}(\cpath) \geq \quantHP{p}{\epsilon}$.  On this event, there is a positive chance that $\wt{I} \subseteq \wt{J}$.  Therefore the probability of this event is at most a constant times $\zeta_2$, from which the result follows.
\end{proof}

\begin{proof}[Proof of Lemma~\ref{lem:two_chunks_quantile}]
This follows from the same argument used to prove Lemma~\ref{lem:one_chunk_quantile}.
\end{proof}

\section{Percolation exploration}
\label{sec:percolation_exploration}

The purpose of this section is to describe a type of supercritical percolation exploration inside of a $\CLE_\kappa$ carpet drawn on top of a quantum half-plane.  In particular, we will decompose the $\CLE_\kappa$ carpet using ``chunks'' of $\SLE_\kappa^0(\kappa-6)$ processes coupled as CPIs.  Our ultimate aim is to show that two points on the boundary of the quantum half-plane are likely to be connected by a collection of such chunks which are in a certain sense ``good''.  The first step in this (Section~\ref{subsec:exploration_def}) is to define a version of the exploration where we start with chunks which each consist of order $2^{-K}$ units of quantum natural time where $K \in \N$ large, with the chunk size growing larger until we see chunks with of order $2^{-J}$ units of quantum natural time where $J \in \N$ with $J \leq K$ is fixed, and then getting smaller with chunks with quantum natural time tending to $0$.  One can think of the chunks in this path as being analogous to a path between two boundary points on $\partial \h$ which passes through at most a fixed number of squares of each scale in a Whitney cube decomposition of $\h$.  We will then explain in Section~\ref{subsec:perc_estimates} that it is likely that this exploration ``succeeds'' in the sense that it is likely that we can construct a path of such chunks all of which satisfy an event which allows us to say that the chunk is ``good''.  Finally, in Section~\ref{subsec:limiting_exploration} we will explain how one can take a (subsequential) limit as $K \to \infty$ to construct an exploration which starts with arbitrarily small chunks and connects two boundary points.

\subsection{Definition of the exploration}
\label{subsec:exploration_def}

Suppose that $\CH = (\h,h,0,\infty)$ is a quantum half-plane.  Suppose that we have fixed $\delta_0,a_0 \in (0,1)$ and $\exploreExp \in (1-\kappa/4,1)$.  Let $x_{-1} < 0 < x_1$ be such that $\qbmeasure{h}([x_{-1},0]) = \qbmeasure{h}([0,x_1]) = 1$.  We assume that $\Gamma$ is a $\CLE_\kappa$ on $\h$.  We will describe below an adaptive exploration of $\Gamma$ and $\CH$.  We will assume that the various $\SLE_\kappa^0(\kappa-6)$ curves in what follows are coupled with $\Gamma$ as CPIs.  Suppose that $\eta$ is an $\SLE_\kappa^0(\kappa-6)$ in $\h$ from $0$ to $\infty$ which is parameterized by the quantum natural time of its trunk, $\CN_t$ is the quantum surface parameterized by the domain disconnected by $\eta|_{[0,t]}$ from $\infty$, and $\sigma$ is a stopping time for the filtration generated by $\CN_t$.  Then the \emph{top} (resp.\ \emph{bottom}) of $\CN_t$ is $\partial \CN_t \cap \h$ (resp.\ $\partial \CN_t \cap \partial \h$).

\begin{figure}[ht!]
\begin{center}
\includegraphics[scale=1]{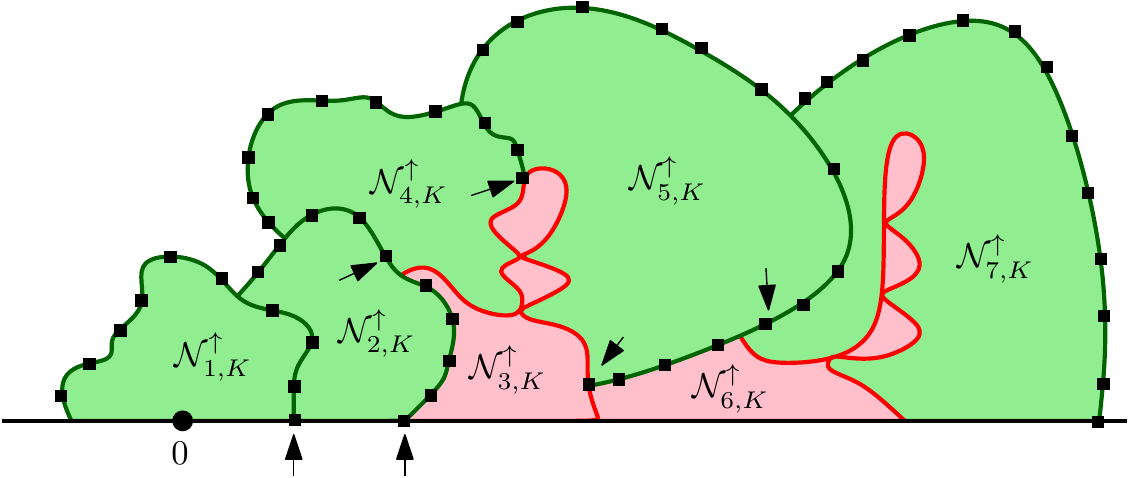}	
\end{center}
\caption{\label{fig:perc_illustration} Illustration of the first few chunks in the percolation exploration of a $\CLE_\kappa$ decorated quantum half-plane.  The chunks which are ``good'' are indicated in green and the chunks which are ``bad'' are indicated in red.  Part of our definition of a chunk being good is that the quantum surface it disconnects from $\infty$ is simply connected, which is why the green chunks are homeomorphic to $\D$ while the bad chunks are not (there are also other ways in which a chunk can be bad).  The arrows indicate the starting points of the $\SLE_\kappa^0(\kappa-6)$'s which cut out the chunks and the squares indicate the points on the tops of the good chunks whose clockwise boundary length distance to the wedge boundary in which they are drawn is an integer multiple of $a_0 2^{-(\kappa/4)K}$; we always take chunks to start at such points.}
\end{figure}

\subsubsection{Increasing chunk sizes}

Fix $J, K \in \N$ with $J \leq K$.  We consider the following exploration of~$\CH$ by $\SLE_\kappa^0(\kappa-6)$ curves.  Let $\eta_{1,K}^\uparrow$ be an independent $\SLE_\kappa^0(\kappa-6)$ curve in $\CH$ from $0$ to $\infty$ and let $\sigma_{1,K}^\uparrow = \inf\{t \geq 2^{-K} \delta_0 : \eta_{1,K}^\uparrow(t) \in \partial \CH\} \wedge 2^{-K}$.  Let $\CN_{1,K}^\uparrow$ be the quantum surface disconnected from $\infty$ by $\eta_{1,K}^\uparrow([0,\sigma_{1,K}^\uparrow])$ and let $\CT_{1,K}^\uparrow$ be the top of $\CN_{1,K}^\uparrow$.   We also let $\CF_{1,K}^\uparrow$ be the $\sigma$-algebra generated by the quantum surface $\CN_{1,K}^\uparrow$ decorated by the path $\eta_{1,K}^\uparrow|_{[0,\sigma_{1,K}^\uparrow]}$ and let $E_{1,K}^\uparrow$ be an $\CF_{1,K}^\uparrow$-measurable event.  We then inductively define $\sigma$-algebras~$\CF_{j,K}^\uparrow$, quantum half-planes~$\CH_{j,K}^\uparrow$, $\SLE_\kappa^0(\kappa-6)$ curves~$\eta_{j,K}^\uparrow$, events~$E_{j,K}^\uparrow$, stopping times $\sigma_{j,K}^\uparrow$, and quantum surfaces $\CN_{j,K}^\uparrow$ with tops $\CT_{j,K}^\uparrow$ as follows.  We let $\CH_{j+1,K}^\uparrow$ be the quantum surface parameterized by the unbounded component of $\CH_{j,K}^\uparrow \setminus \eta_{j,K}^\uparrow([0,\sigma_{j,K}^\uparrow])$ and marked by the points $\eta_{j,K}^\uparrow(\sigma_{j,K}^\uparrow)$ and $\infty$.  If $E_{j,K}^\uparrow$ occurs, we let $\eta_{j+1,K}^\uparrow$ be an $\SLE_\kappa^0(\kappa-6)$ process in $\CH_{j+1,K}^\uparrow$ starting from the rightmost intersection of $\eta_{j,K}^\uparrow([0,\sigma_{j,K}^\uparrow])$ with $\partial \CH_{j,K}^\uparrow$.  For $1 \leq i \leq j$, let $\CX_{i,K}^\uparrow$ be the set of points in $\CT_{i,K}^\uparrow$ whose clockwise boundary length distance to $\partial \CH_{i,K}^\uparrow$ is an integer multiple of $a_0 2^{-(\kappa/4) K}$.  If $E_{j,K}^\uparrow$ does not occur, we let $\eta_{j+1,K}^\uparrow$ be an $\SLE_\kappa^0(\kappa-6)$ process in $\CH_{j+1,K}^\uparrow$ starting from the rightmost point $z_{j,K}^\uparrow$ on $\partial \CH_{j+1,K}^\uparrow$ which is:
\begin{itemize}
\item to the left of the leftmost point of $\CT_{j,K}^\uparrow$,
\item in $\CX_{i,K}^\uparrow$ with $i < j$ such that the part of $\CT_{i,K}^\uparrow \cap \partial \CH_{j+1,K}^\uparrow$ to the left of $z_{j,K}^\uparrow$ has quantum length at least $a_0 2^{-(\kappa/4) K}$ and with boundary length distance at most $a_0 2^{-(\kappa/4)K}$ from $\Upsilon_+$.
\end{itemize}
In the case that there is no such point, we take the starting point of $\eta_{j+1,K}^\uparrow$ to be $\eta_{j,K}^\uparrow(\sigma_{j,K}^\uparrow)$.  We let $\sigma_{j+1,K}^\uparrow = \inf\{t \geq 2^{-K} \delta_0 : \eta_{j+1,K}^\uparrow(t) \in \partial \CH_{j+1,K}^\uparrow \} \wedge 2^{-K}$, let $\CN_{j+1,K}^\uparrow$ be the quantum surface in $\CH_{j+1,K}^\uparrow$ disconnected by $\eta_{j+1,K}^\uparrow([0,\sigma_{j+1,K}^\uparrow])$ from $\infty$, $\CF_{j+1,K}^\uparrow$ be the $\sigma$-algebra generated by the quantum surface parameterized by $\interior{\closure{\cup_{i=1}^{j+1} \CN_{i,K}^\uparrow}}$ and decorated by the paths $\eta_{i,K}^\uparrow|_{[0,\sigma_{i,K}^\uparrow]}$ for $1 \leq i \leq j+1$.  Finally, we let $E_{j+1,K}^\uparrow$ be an $\CF_{j+1,K}^\uparrow$-measurable event.  Let $n_K^\uparrow$ be the first $j$ so that either
\begin{itemize}
\item $x_1$ is disconnected from $\infty$, or
\item $E_{j,K}^\uparrow$ occurs and the boundary length distance along $\partial \CH_{j+1,K}^\uparrow$ from $\eta_{j+1,K}^\uparrow(0)$ to $x_1$ is at most $1-2^{-\exploreExp(\kappa/4) K}$.
\end{itemize}

Suppose that we have defined the exploration procedure for some $J+1 \leq M \leq K$ and we are on the event that $x_1$ has not been disconnected from $\infty$.  We then define it for $M-1$ as follows.  Let $\eta_{1,M-1}^\uparrow$ be an independent $\SLE_\kappa^0(\kappa-6)$ curve in $\CH_{1,M-1}^\uparrow = \CH_{n_M^\uparrow+1,M}^\uparrow$ from the rightmost intersection of $\eta_{n_M,M}^\uparrow([0,\sigma_{n_M,M}^\uparrow])$ with $\partial \CH_{n_M,M}^\uparrow$ to $\infty$.  Let $\sigma_{1,M-1}^\uparrow = \inf\{t \geq 2^{-(M-1)} \delta_0 : \eta_{1,M-1}^\uparrow(t) \in \partial \CH_{1,M-1}^\uparrow \} \wedge 2^{-(M-1)}$.  We then inductively define  quantum half-planes $\CH_{j,M-1}^\uparrow$, $\SLE_\kappa^0(\kappa-6)$ curves $\eta_{j,M-1}^\uparrow$, events $E_{j,M-1}^\uparrow$, stopping times $\sigma_{j,M-1}^\uparrow$, and quantum surfaces $\CN_{j,M-1}^\uparrow$ with tops $\CT_{j,M-1}^\uparrow$ as follows.  We let $\CH_{j+1,M-1}^\uparrow$ be the quantum surface parameterized by the unbounded component of $\CH_{j,M-1}^\uparrow \setminus \eta_{j,M-1}^\uparrow([0,\sigma_{j,M-1}^\uparrow])$ and marked by the points $\eta_{j,M-1}^\uparrow(\sigma_{j,M-1}^\uparrow)$ and $\infty$.  If $E_{j,M-1}^\uparrow$ occurs, we let $\eta_{j+1,M-1}^\uparrow$ be an $\SLE_\kappa^0(\kappa-6)$ process in $\CH_{j+1,M-1}^\uparrow$ starting from the rightmost intersection of $\eta_{j,M-1}^\uparrow([0,\sigma_{j,M-1}^\uparrow])$ with $\partial \CH_{j,M-1}^\uparrow$.  For $1 \leq i \leq j$, let $\CX_{i,M-1}^\uparrow$ be the set of points in $\CT_{i,M-1}^\uparrow$ whose clockwise boundary length distance to $\partial \CH_{i,M-1}^\uparrow$ is an integer multiple of $a_0 2^{-(\kappa/4) (M-1)}$.  If $E_{j,M-1}^\uparrow$ does not occur, we let $\eta_{j+1,M-1}^\uparrow$ be an $\SLE_\kappa^0(\kappa-6)$ process in $\CH_{j+1,M-1}^\uparrow$ starting from the rightmost point $z_{j,M-1}^\uparrow$ on $\CH_{j+1,M-1}^\uparrow$ which is:
\begin{itemize}
\item to the left of the leftmost point of $\CT_{j,M-1}^\uparrow$ and
\item in $\CX_{i,N}^\uparrow$ with $M-1 \leq N \leq K$ such that the part of $\CT_{i,N}^\uparrow \cap \partial \CH_{j+1,N}^\uparrow$ to the left of $z_{j,M-1}^\uparrow$ has quantum length at least $a_0 2^{-(\kappa/4) N}$ and with boundary length distance at most $a_0 2^{-(\kappa/4)N}$ from $\Upsilon_+$.
\end{itemize}
In the case that there is no such point, we take the starting point of $\eta_{j+1,M-1}^\uparrow$ to be $\eta_{j,M-1}^\uparrow(\sigma_{j,M-1}^\uparrow)$.  We let $\sigma_{j+1,M-1}^\uparrow = \inf\{t \geq 2^{-(M-1)} \delta_0 : \eta_{j+1,M-1}^\uparrow(t) \in \partial \CH_{j+1,M-1}^\uparrow \} \wedge 2^{-(M-1)}$, let $\CN_{j+1,M-1}^\uparrow$ be the quantum surface in $\CH_{j+1,M-1}^\uparrow$ disconnected by $\eta_{j+1,M-1}^\uparrow([0,\sigma_{j+1,M-1}^\uparrow])$ from $\infty$, $\CF_{j+1,M-1}^\uparrow$ the $\sigma$-algebra generated by $\CF_{j,M}^\uparrow$ and the quantum surface parameterized by $\interior{\closure{\cup_{i=1}^{j+1} \CN_{i,M-1}^\uparrow}}$ and decorated by the paths $\eta_{i,M-1}^\uparrow|_{[0,\sigma_{i,M-1}^\uparrow]}$ for $1 \leq i \leq j+1$.  Finally, we let $E_{j+1,M-1}^\uparrow$ be an $\CF_{j+1,M-1}^\uparrow$-measurable event.  Let $n_{M-1}^\uparrow$ be the first $j$ so that either
\begin{itemize}
\item $x_1$ is disconnected from $\infty$, or
\item $E_{j,M-1}^\uparrow$ occurs and the boundary length distance in $\CH_{j+1,M-1}^\uparrow$ from $\eta_{j+1,M-1}^\uparrow(0)$ to $x_1$ is at most $1-2^{-\exploreExp (\kappa/4) (M-1)}$.
\end{itemize}

\subsubsection{Decreasing chunk sizes}

We will now define the continuation of the exploration after scale $J$ as above.  The definition will be similar except the chunk sizes will be decreasing (rather than increasing) as the exploration gets closer to $x_1$.

Let $\eta_{1,J}^\downarrow$ be an independent $\SLE_\kappa^0(\kappa-6)$ curve in $\CH_{1,J}^\downarrow = \CH_{n_J^\uparrow+1,J}^\uparrow$ from the rightmost intersection of $\eta_{n_J,J}^\uparrow([0,\sigma_{n_J,J}^\uparrow])$ with $\partial \CH_{n_J,J}^\uparrow$ to $\infty$.  Let $\sigma_{1,J}^\downarrow = \inf\{t \geq 2^{-J} \delta_0 : \eta_{1,J}^\downarrow(t) \in \partial \CH_{1,J}^\downarrow \} \wedge 2^{-J}$.  Let $\CN_{1,J}^\uparrow$ be the quantum surface disconnected from $\infty$ by $\eta_{1,J}^\uparrow([0,\sigma_{1,J}^\uparrow])$ and let $\CT_{1,J}^\uparrow$ be the top of $\CN_{1,J}^\uparrow$.  We also let $\CF_{1,J}^\downarrow$ be the $\sigma$-algebra generated by $\CF_{n_J^\uparrow,J}^\uparrow$ and the quantum surface $\CN_{1,J}^\downarrow$ decorated by the path $\eta_{1,J}^\downarrow|_{[0,\sigma_{1,J}^\downarrow]}$ and let $E_{1,J}^\downarrow$ be an $\CF_{1,J}^\downarrow$-measurable event.  We then inductively define $\sigma$-algebras $\CF_{j,J}^\downarrow$, quantum half-planes $\CH_{j,J}^\downarrow$, $\SLE_\kappa^0(\kappa-6)$ curves $\eta_{j,J}^\downarrow$, events $E_{j,J}^\downarrow$, stopping times $\sigma_{j,J}^\downarrow$, and quantum surfaces $\CN_{j,J}^\downarrow$ with tops $\CT_{j,J}^\downarrow$ as follows.  We let $\CH_{j+1,J}^\downarrow$ be the quantum surface parameterized by the unbounded component of $\CH_{j,J}^\downarrow \setminus \eta_{j,J}^\downarrow([0,\sigma_{j,J}^\downarrow])$ and marked by the points $\eta_{j,J}^\downarrow(\sigma_{j,J}^\downarrow)$ and $\infty$.  If $E_{j,J}^\downarrow$ occurs, we let $\eta_{j+1,J}^\downarrow$ be an $\SLE_\kappa^0(\kappa-6)$ process in $\CH_{j+1,J}^\downarrow$ starting from the rightmost intersection of $\eta_{j,J}^\downarrow([0,\sigma_{j,J}^\downarrow])$ with $\partial \CH_{j,J}^\downarrow$.  For $1 \leq i \leq j$, let $\CX_{i,J}^\downarrow$ be the set of points in $\CT_{i,J}^\downarrow$ whose clockwise boundary length distance to $\partial \CH_{i,J}^\downarrow$ is an integer multiple of $a_0 2^{-(\kappa/4) J}$.  If $E_{j,J}^\downarrow$ does not occur, we let $\eta_{j+1,J}^\downarrow$ be an $\SLE_\kappa^0(\kappa-6)$ process in $\CH_{j+1,J}^\downarrow$ starting from the rightmost point $z_{j,J}^\downarrow$ on $\partial \CH_{j+1,J}^\downarrow$ which is:
\begin{itemize}
\item to the left of the leftmost point of $\CT_{j,J}^\downarrow$ and
\item in $\CX_{i,N}^\bullet$ for $\bullet \in \{\uparrow, \downarrow\}$ where $K \leq N \leq J$ ($\bullet = \uparrow$) or $N = J$ ($\bullet = \downarrow$) such that the part of $\CT_{i,N}^\bullet \cap \partial \CH_{j+1,J}^\downarrow$ to the left of $z_{j,J}^\downarrow$ has quantum length at least $a_0 2^{-(\kappa/4) N}$ and with boundary length distance at most $a_0 2^{-(\kappa/4)N}$ from $\Upsilon_+$.
\end{itemize}
In the case that there is no such point, we take the starting point of $\eta_{j+1,J}^\downarrow$ to be $\eta_{j,J}^\downarrow(\sigma_{j,J}^\downarrow)$.  We let $\sigma_{j+1,J}^\downarrow = \inf\{t \geq 2^{-J} \delta_0 : \eta_{j+1,J}^\downarrow(t) \in \partial \CH_{j+1,J}^\downarrow \} \wedge 2^{-J}$, let $\CN_{j+1,J}^\downarrow$ be the quantum surface in $\CH_{j+1,J}^\downarrow$ disconnected by $\eta_{j+1,J}^\downarrow([0,\sigma_{j+1,J}^\downarrow])$ from $\infty$, $\CF_{j+1,J}^\downarrow$ the $\sigma$-algebra generated by $\CF_{1,J}^\downarrow$ and the quantum surface parameterized by $\interior{\closure{\cup_{i=1}^{j+1} \CN_{i,J}^\downarrow}}$ and decorated by the paths $\eta_{i,J}^\downarrow|_{[0,\sigma_{i,J}^\downarrow]}$ for $1 \leq i \leq j+1$.  Finally, we let $E_{j+1,J}^\downarrow$ be an $\CF_{j+1,J}^\downarrow$-measurable event.  Let $n_J^\downarrow$ be the first $j$ so that either
\begin{itemize}
\item $x_1$ is disconnected from $\infty$, or
\item $E_{j,J}^\downarrow$ occurs and the boundary length distance in $\CH_{j+1,J}^\downarrow$ from $\eta_{j+1,J}^\downarrow(0)$ to $x_1$ is at most $2^{-\exploreExp (\kappa/4) J}$.
\end{itemize}

Suppose that we have defined the downward exploration procedure for some $M \geq J$.  We then define it for $M+1$ as follows.  Let $\eta_{1,M+1}^\downarrow$ be an independent $\SLE_\kappa^0(\kappa-6)$ curve in $\CH_{1,M+1}^\downarrow = \CH_{n_M^\downarrow+1,M}^\downarrow$ from the rightmost intersection of $\eta_{j,M}^\downarrow([0,\sigma_{j,M}^\downarrow])$ with $\partial \CH_{j,M}^\uparrow$ to $\infty$.  Let $\sigma_{1,M+1}^\downarrow = \inf\{t \geq 2^{-(M+1)} \delta_0 : \eta_{1,M+1}^\downarrow(t) \in \partial \CH_{1,M+1}^\downarrow \} \wedge 2^{-(M+1)}$.  We then inductively define  quantum half-planes $\CH_{j,M+1}^\downarrow$, $\SLE_\kappa^0(\kappa-6)$ curves $\eta_{j,M+1}^\downarrow$, events $E_{j,M+1}^\downarrow$, stopping times $\sigma_{j,M+1}^\downarrow$, and quantum surfaces $\CN_{j,M+1}^\downarrow$ with tops $\CT_{j,M+1}^\downarrow$ as follows.  We let $\CH_{j+1,M+1}^\downarrow$ be the quantum surface parameterized by the unbounded component of $\CH_{j,M+1}^\downarrow \setminus \eta_{j,M+1}^\downarrow([0,\sigma_{j,M+1}^\downarrow])$ and marked by the points $\eta_{j,M+1}^\uparrow(\sigma_{j,M+1}^\downarrow)$ and $\infty$.  If $E_{j,M+1}^\downarrow$ occurs, we let $\eta_{j+1,M+1}^\downarrow$ be an $\SLE_\kappa^0(\kappa-6)$ process in $\CH_{j+1,M+1}^\downarrow$ starting from the rightmost intersection of $\eta_{j,M+1}^\downarrow([0,\sigma_{j,M+1}^\downarrow])$ with $\partial \CH_{j,M+1}^\downarrow$.  For $1 \leq i \leq j$, let $\CX_{i,M+1}^\downarrow$ be the set of points in $\CT_{i,M+1}^\downarrow$ whose clockwise boundary length distance to $\partial \CH_{i,M+1}^\downarrow$ is an integer multiple of $a_0 2^{-(\kappa/4) (M+1)}$.  If $E_{j,M+1}^\downarrow$ does not occur, we let $\eta_{j+1,M+1}^\downarrow$ be an $\SLE_\kappa^0(\kappa-6)$ process in $\CH_{j+1,M+1}^\downarrow$ starting from the rightmost point $z_{j,M+1}^\downarrow$ on $\partial \CH_{j+1,M+1}^\downarrow$ which is:
\begin{itemize}
\item to the left of the leftmost point of $\CT_{j,M+1}^\downarrow$ and
\item in $\CX_{i,N}^\bullet$ for $\bullet \in \{\uparrow, \downarrow\}$ where $K \leq N \leq J$ ($\bullet = \uparrow$) or $M+1 \leq N \leq J$ ($\bullet = \downarrow$) such that the part of $\CT_{i,N}^\bullet \cap \partial \CH_{j+1,M+1}^\downarrow$ to the left of $z_{j,M+1}^\downarrow$ has quantum length at least $a_0 2^{-(\kappa/4) N}$ and with boundary length distance at most $a_0 2^{-(\kappa/4)N}$ from $\Upsilon_+$.
\end{itemize}
In the case that there is no such point, we take the starting point of $\eta_{j+1,M+1}^\downarrow$ to be $\eta_{j,M}^\downarrow(\sigma_{j,M}^\downarrow)$.  We let $\sigma_{j+1,M+1}^\downarrow = \inf\{t \geq 2^{-(M+1)} \delta_0 : \eta_{j+1,M+1}^\downarrow(t) \in \partial \CH_{j+1,M+1}^\downarrow \} \wedge 2^{-(M+1)}$, let $\CN_{j+1,M+1}^\downarrow$ be the quantum surface in $\CH_{j+1,M+1}^\downarrow$ disconnected by $\eta_{j+1,M+1}^\downarrow([0,\sigma_{j+1,M+1}^\downarrow])$ from $\infty$, $\CF_{j+1,M+1}^\downarrow$ the $\sigma$-algebra generated by $\CF_{j+1,M}^\downarrow$ and the quantum surface parameterized by $\interior{\closure{\cup_{i=1}^{j+1} \CN_{i,M+1}^\downarrow}}$ and decorated by the paths $\eta_{i,M+1}^\downarrow|_{[0,\sigma_{i,M+1}^\downarrow]}$ for $1 \leq i \leq j+1$.  Finally, we let $E_{j+1,M+1}^\downarrow$ be an $\CF_{j+1,M+1}^\downarrow$-measurable event. Let $n_{M+1}^\downarrow$ be the first $j$ so that either
\begin{itemize}
\item $x_1$ is disconnected from $\infty$, or
\item $E_{j,M-1}^\downarrow$ occurs and the boundary length distance in $\CH_{j+1,M+1}^\downarrow$ from $\eta_{j+1,M+1}^\downarrow(0)$ to $x_1$ is at most $2^{-\exploreExp (\kappa/4) (M+1)}$.
\end{itemize}

\subsubsection{Definition of events}

\begin{figure}[ht!]
\begin{center}
\includegraphics[scale=1]{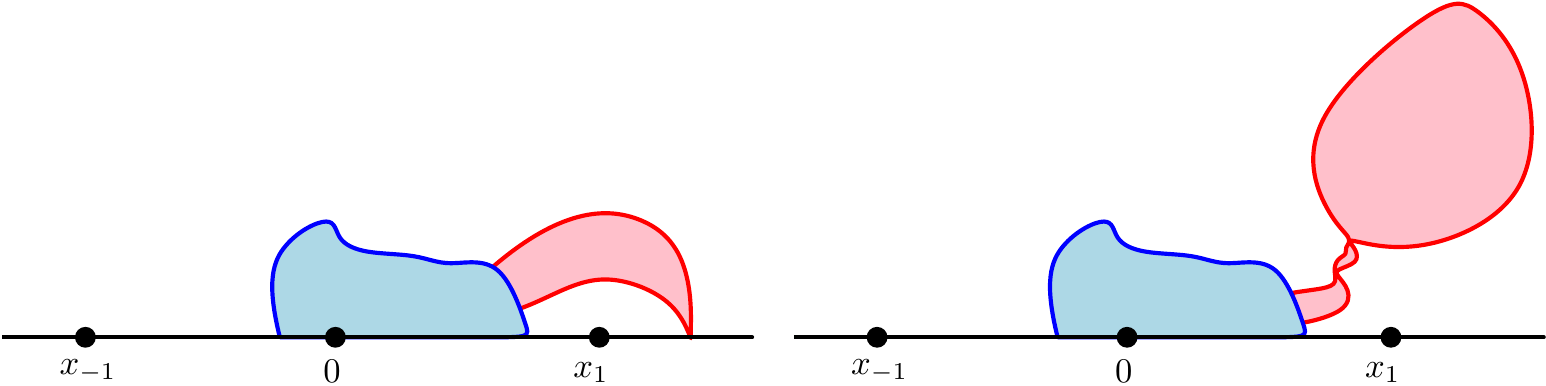}	
\end{center}
\caption{\label{fig:chunk_failures} Illustration of two of the ways in which the exploration can fail.  The blue region illustrates the quantum surface parameterized by the exploration before discovering the chunk which led to the failure.  On the left, the exploration failed due to a large downward jump (which also disconnected $x_1$ from $\infty$.  On the right, the exploration failed due to a large upward jump (corresponding to the discovery of a large $\CLE_\kappa$ loop).}
\end{figure}

Fix $c_F > 1$; we will choose its value later.  We let $F_K^\uparrow$ be the event that 
\begin{itemize}
\item $x_{-1}$ or $x_1$ is disconnected from $\infty$ by one of $\eta_{j,K}^\uparrow|_{[0,\sigma_{j,K}^\uparrow]}$ for $1 \leq j \leq n_K^\uparrow$,
\item there exists $1 \leq j \leq n_K^\uparrow$ so that $\eta_{j,K}^\uparrow|_{[0,\sigma_{j,K}^\uparrow]}$ makes a jump (downward or upward) of size at least $2^{-\exploreExp (\kappa/4) K}$, or
\item $n_K^\uparrow \geq c_F 2^{(1-\exploreExp)(\kappa/4) K}$.
\end{itemize}

For each $J \leq M \leq K-1$, we let $F_M^\uparrow$ be the event that
\begin{itemize}
\item $x_{-1}$ or $x_1$ is disconnected from $\infty$ by one of $\eta_{j,M}^\uparrow|_{[0,\sigma_{j,M}^\uparrow]}$ for $1 \leq j \leq n_M^\uparrow$,
\item there exists $1 \leq j \leq n_M^\uparrow$ so that $\eta_{j,M}^\uparrow|_{[0,\sigma_{j,M}^\uparrow]}$ makes a jump (downward or upward) of size at least $2^{-\exploreExp (\kappa/4) M}$,
\item $n_M^\uparrow \geq c_F 2^{(1-\exploreExp) (\kappa/4) M}$, or
\item the curves $\eta_{j,M}^\uparrow|_{[0,\sigma_{j,M}^\uparrow]}$ for $1 \leq j \leq n_M^\uparrow$ disconnect all of the tops $\CT_{j,M+1}^\uparrow$ for $1 \leq j \leq n_{M+1}^\uparrow$ from $\infty$.
\end{itemize}

For each $J \leq M$, we let $F_M^\downarrow$ be the event that
\begin{itemize}
\item $x_{-1}$ or $x_1$ is disconnected from $\infty$ by one of $\eta_{j,M}^\downarrow|_{[0,\sigma_{j,M}^\downarrow]}$ for $1 \leq j \leq n_M^\downarrow$,
\item there exists $1 \leq j \leq n_M^\downarrow$ so that $\eta_{j,M}^\downarrow|_{[0,\sigma_{j,M}^\downarrow]}$ makes a jump (downward or upward) of size at least $2^{-\exploreExp (\kappa/4) M}$,
\item $n_M^\downarrow \geq c_F 2^{(1-\exploreExp)(\kappa/4) M}$ ($M \geq J+1$) or $n_M^\downarrow \geq c_F 2^{(\kappa/4) M}$ ($M=J$), or
\item the curves $\eta_{j,M}^\downarrow|_{[0,\sigma_{j,M}^\downarrow]}$ for $1 \leq j \leq n_M^\downarrow$ disconnect all of the tops 
 	\begin{itemize}
 	\item (if $M \geq J+1$) $\CT_{j,M-1}^\downarrow$ for $1 \leq j \leq n_{M-1}^\downarrow$
 	\item (if $M = J$) $\CT_{j,J}^\uparrow$ for $1 \leq j \leq n_J^\uparrow$
 	\end{itemize}
 	from $\infty$.
\end{itemize}

We say that the exploration \emph{fails} if one of $F_M^\uparrow$ for $J \leq M \leq K$ or $F_M^\downarrow$ for $J \leq M$ occurs.

\subsection{Estimates}
\label{subsec:perc_estimates}

\begin{lemma}
\label{lem:point_to_point_exploration}
Fix $C_0 > 0$.  There exist constants $c,\delta_0, a_0 \in (0,1)$ and $c_F > 0$ so that with $p_0 = 1-a_0$ the following is true for every $J \in \N$.  Assume that
\begin{enumerate}[(i)]
\item $\p[E_{j,M}^\uparrow \giv \CF_{j-1,M}^\uparrow] \geq p_0$ for all $j \geq 1$ and $J \leq M \leq K$,
\item $\p[E_{j,M}^\downarrow \giv \CF_{j-1,M}^\downarrow] \geq p_0$ for all $j \geq 1$ and $J \leq M$ and
\item $E_{j,M}^\uparrow$ (resp.\ $E_{j,M}^\downarrow$) implies that the top boundary length of $\CN_{j,M}^\uparrow$ (resp.\ $\CN_{j,M}^\downarrow$) is at most $C_0 2^{-(\kappa/4) M}$ for all $j \geq 1$ and $J \leq M \leq K$ (resp.\ $J \leq M$).
\end{enumerate}
Then the probability that the exploration as defined in Section~\ref{subsec:exploration_def} fails is $O(2^{-c J})$.
\end{lemma}
\begin{proof}

We are going to show that the following is true.  There exists a constant $c > 0$ so that for each $J \leq M \leq K$ we have $\p[F_M^\uparrow] = O(2^{-c M})$ and for $M \geq J$ we have $\p[F_M^\downarrow] = O(2^{-c M})$.  Then the result follows by applying a union bound.  We will explain carefully how to obtain the upper bound for $\p[F_K^\uparrow]$ and then subsequently explain the necessary modifications to bound $\p[F_M^\uparrow]$ for $J \leq M \leq K$ and $\p[F_M^\downarrow]$ for $J \leq M$.

\noindent{\it Step 1.  Definition of boundary length processes.}  For each $j$ and $M$, we let $T_{j,M}^\uparrow$ denote the quantum length of the top $\CT_{j,M}^\uparrow$ of $\CN_{j,M}^\uparrow$.  We also let $B_{L,j,M}^\uparrow$ (resp.\ $B_{R,j,M}^\uparrow$) denote the quantum length of the part of the bottom of $\CN_{j,M}^\uparrow$ which is to the left (resp.\ right) of $\eta_{j,M}^\uparrow(0)$.  Finally, we define $L_{j,M}^\uparrow$ (resp.\ $R_{j,M}^\uparrow$) to be equal to the quantum length of the part of $\partial \CH_{j,M}^\uparrow \setminus \partial \CH$ to the left (resp.\ right) of $\eta_{j,M}^\uparrow(0)$ minus the quantum length of the part of $\partial \CH \setminus \partial \CH_{j,M}^\uparrow$ to the left (resp.\ right) of $0$.  We define $L_{0,M}^\uparrow = L_{n_{M-1},M-1}^\uparrow$ and $R_{0,M}^\uparrow = R_{n_{M-1},M-1}^\uparrow$.  On $(E_{j,M}^\uparrow)^c$, we let $S_{j,M}^\uparrow$ be the quantum length of the interval on $\partial \CH_{j+1,M}^\uparrow$ from $\eta_{j+1,M}^\uparrow(0)$ to the leftmost point on the top of $\CN_{j,M}^\uparrow$.  For each $j \geq 0$, we then have that
\begin{equation}
\label{eqn:l_change}
L_{j+1,M}^\uparrow - L_{j,M}^\uparrow = (T_{j,M}^\uparrow - B_{L,j,M}^\uparrow) \one_{E_{j,M}^\uparrow} - (B_{L,j,M}^\uparrow + S_{j,M}^\uparrow) \one_{(E_{j,M}^\uparrow)^c}.	
\end{equation}
We similarly have that
\begin{equation}
\label{eqn:r_change}
R_{j+1,M}^\uparrow - R_{j,M}^\uparrow = -B_{R,j,M}^\uparrow \one_{E_{j,M}^\uparrow} + (T_{j,M}^\uparrow - B_{R,j,M}^\uparrow + S_{j,M}^\uparrow) \one_{(E_{j,M}^\uparrow)^c}.
\end{equation}

\noindent{\it Step 2. Dominating the boundary length processes.}  We are now going to define a sequence $\wt{L}_{j,M}^\uparrow$ which we can use to dominate $L_{j,M}^\uparrow$ from above and a sequence $\wt{R}_{j,M}^\uparrow$ which we can use to dominate $R_{j,M}^\uparrow$ from below.

We set $\wt{L}_{0,K}^\uparrow = 0$.  We define $\wt{L}_{j,K}^\uparrow$ using the recurrence relation for $j \geq 0$
\begin{equation}
\label{eqn:l_k_dom_change}
\wt{L}_{j+1,K}^\uparrow - \wt{L}_{j,K}^\uparrow = (T_{j,K}^\uparrow - B_{L,j,K}^\uparrow - a_0 2^{-(\kappa/4)K}) \one_{E_{j,K}^\uparrow} - B_{L,j,K}^\uparrow \one_{(E_{j,K}^\uparrow)^c}.	
\end{equation}
We then inductively define $\wt{L}_{j,M}^\uparrow$ for $M \leq K-1$ by setting $\wt{L}_{0,M}^\uparrow = \wt{L}_{n_{M-1}^\uparrow,M-1}^\uparrow$ and we define $\wt{L}_{j,M}^\uparrow$ using the recurrence relation for $j \geq 0$
\begin{equation}
\label{eqn:l_m_dom_change}
\wt{L}_{j+1,M}^\uparrow - \wt{L}_{j,M}^\uparrow = (T_{j,M}^\uparrow - B_{L,j,M}^\uparrow - a_0 2^{-(\kappa/4)M}) \one_{E_{j,M}^\uparrow} - B_{L,j,M}^\uparrow \one_{(E_{j,M}^\uparrow)^c}.	
\end{equation}
Let $N_{j,M}^\uparrow = \sum_{i=1}^j \one_{(E_{i,M}^\uparrow)^c}$ and let $\wt{\Delta}_{j,M}^\uparrow$ be the sum of $N_{j,M}^\uparrow (a_0 2^{-(\kappa/4) M})$ and the $N_{j,M}^\uparrow$ largest upward jumps made by either the left or the right boundary length processes up to the time that $\CN_{j,M}^\uparrow$ has been explored.  One can see that $\wt{L}_{j,M}^\uparrow - \wt{\Delta}_{j,M}^\uparrow \leq L_{j,M}^\uparrow$ for all $j,M$ because in the definition of the exploration, we have a slide to the left of size at most $a_0 2^{-(\kappa/4)M}$ on the top of each good chunk while in the definition of $\wt{L}_{j,M}^\uparrow$ we always have such a shift.  Moreover, in the definition of $L_{j,M}^\uparrow$ we slide to the left over $N_{j,M}^\uparrow$ loops of $\Gamma$ previous discovered by the exploration as well as at most $N_{j,M}^\uparrow (a_0 2^{-(\kappa/4) M})$ units of extra length while $\wt{\Delta}_{j,M}^\uparrow$ gives the sum of the lengths of the longest $N_{j,M}^\uparrow$ loops discovered by the exploration together with this extra length.

We also set $\wt{R}_{0,K}^\uparrow = 0$.  We define $\wt{R}_{j,K}^\uparrow$ using the recurrence relation for $j \geq 0$
\begin{equation}
\label{eqn:r_k_dom_change}
\wt{R}_{j+1,K}^\uparrow - \wt{R}_{j,K}^\uparrow = (a_0 2^{-(\kappa/4)K} - B_{R,j,K}^\uparrow) \one_{E_{j,K}^\uparrow} + (T_{j,K}^\uparrow - B_{R,j,K}^\uparrow) \one_{(E_{j,K}^\uparrow)^c}.
\end{equation}
We then inductively define $\wt{R}_{j,M}^\uparrow$ for $M \leq K-1$ by setting $\wt{R}_{0,M}^\uparrow = \wt{R}_{n_{M-1}^\uparrow,M-1}^\uparrow$ and we define $\wt{R}_{j,M}^\uparrow$ using the recurrence relation for $j \geq 0$
\begin{equation}
\label{eqn:r_m_dom_change}
\wt{R}_{j+1,M}^\uparrow - \wt{R}_{j,M}^\uparrow = (a_0 2^{-(\kappa/4)M} - B_{R,j,M}^\uparrow) \one_{E_{j,M}^\uparrow} + (T_{j,M}^\uparrow - B_{R,j,M}^\uparrow) \one_{(E_{j,M}^\uparrow)^c}.
\end{equation}
One can similarly see that $\wt{R}_{j,M}^\uparrow + \wt{\Delta}_{j,M}^\uparrow \geq R_{j,M}^\uparrow$ for all $j,M$.

\noindent{\it Step 3.  Analysis of dominating boundary length processes.}  

Lemma~\ref{lem:bottom_length_moment_bound} implies that $\E[B_{R,j,K}^\uparrow \giv \CF_{j-1,K}^\uparrow] \gtrsim (\log \delta_0^{-1}) \delta_0^{\kappa/4} 2^{-(\kappa/4)K}$.  We also have that 
\begin{align*}
	\E[T_{j,K}^\uparrow \one_{(E_{j,K}^\uparrow)^c} \giv \CF_{j-1,K}^\uparrow]
&\leq \E[T_{j,K}^\uparrow \one_{\{\sigma_{j,K}^\uparrow = 2^{-K}\}}  \giv \CF_{j-1,K}^\uparrow] + \E[T_{j,K}^\uparrow \one_{(E_{j,K}^\uparrow)^c} \one_{\{\sigma_{j,K}^\uparrow < 2^{-K}\}}  \giv \CF_{j-1,K}^\uparrow].
\end{align*}
Lemma~\ref{lem:top_length_moment_bound} implies that the first term on the right hand side is at most a constant times $\delta_0^{\kappa/4} 2^{-(\kappa/4) K}$.  For the second term on the right hand side, fix $p \in (1,4/\kappa)$ and let $q > 1$ be such that $p^{-1} + q^{-1} = 1$.  Then
\begin{align*}
\E[T_{j,K}^\uparrow \one_{(E_{j,K}^\uparrow)^c} \one_{\{\sigma_{j,K}^\uparrow < 2^{-K}\}}  \giv \CF_{j-1,K}^\uparrow]
&\leq \E[ (T_{jK}^\uparrow)^p \one_{\{\sigma_{j,K}^\uparrow < 2^{-K}\}} \giv \CF_{j-1,K}^\uparrow]^{1/p} \p[ (E_{j,K}^\uparrow)^c\giv \CF_{j-1,K}^\uparrow]^{1/q}.
\end{align*}
The first factor on the right hand side is at most a constant times $2^{-(\kappa/4) K}$ by \cite[Chapter VIII, Proposition~4]{bertoin1996levy}.  The second term can be made arbitrarily small by making $a_0 > 0$ sufficiently small.  In particular, we choose $a_0 > 0$ sufficiently small so that the second term is at most $\delta_0^{\kappa/4}$.

Therefore by choosing $\delta_0 > 0$ sufficiently small and then $a_0 > 0$ sufficiently small given our choice of~$\delta_0$ we see for a constant $d_0 > 0$ that 
\begin{align}
\label{eqn:r_k_dom_change_cond_mean}
\E[ \wt{R}_{j,K}^\uparrow - \wt{R}_{j-1,K}^\uparrow \giv \CF_{j-1,K}^\uparrow] \leq -d_0 2^{-(\kappa/4)K}.
\end{align}
As the increments $\wt{R}_{j,K}^\uparrow - \wt{R}_{j-1,K}^\uparrow - \E[ \wt{R}_{j,K}^\uparrow - \wt{R}_{j-1,K}^\uparrow \giv \CF_{j-1,K}^\uparrow]$ form a martingale difference sequence with a bounded $p$th moment for every $p \in (1,4/\kappa)$, it follows from~\eqref{eqn:r_k_dom_change_cond_mean} and \cite{chatterji1969mgdiff} that for every $p \in (1,4/\kappa)$ there exist constants $c_0,c_1,c_2 > 0$ so that
\begin{equation}
\label{eqn:right_move}
\p[ \wt{R}_{N,K}^\uparrow \notin [-c_0 2^{-(\kappa/4)K} N, -c_1 2^{-(\kappa/4)K} N]] \leq c_2 N^{1-p}
\end{equation}
and by Doob's maximal inequality the same holds $\min_{1 \leq n \leq N} \wt{R}_{n,K}^\uparrow$ in place of $\wt{R}_{N,K}^\uparrow$.  A similar analysis implies that there exist constants $c_3,c_4,c_5 > 0$ so that
\begin{equation}
\label{eqn:left_move}
\p[ \wt{L}_{N,K}^\uparrow \notin [c_3 2^{-(\kappa/4)K} N, c_4 2^{-(\kappa/4)K} N] ] \leq c_5 N^{1-p}	
\end{equation}
and by Doob's maximal inequality the same holds $\max_{1 \leq n \leq N} \wt{L}_{n,K}^\uparrow$ in place of $\wt{L}_{n,K}^\uparrow$.

Let us now collect a few observations from the above.  Let $\wt{n}_K^\uparrow$ be the first $j$ so that $\wt{R}_{j,K}^\uparrow$ is at most $-2^{-\exploreExp(\kappa/4)K}$.
\begin{enumerate}[(i)]
\item\label{it:tilde_n_bound} By~\eqref{eqn:right_move}, we have that
\[ \p[ c_0 2^{(1-\exploreExp)(\kappa/4) K} \leq \wt{n}_K^\uparrow \leq c_1 2^{(1-\exploreExp)(\kappa/4) K} ] = 1- O(2^{-cK}).\]
\item\label{it:tilde_l_bound} By~\eqref{eqn:right_move} and~\eqref{eqn:left_move}, we have that
\[ \p[ \wt{L}_{\wt{n}_K^\uparrow,K}^\uparrow \geq c_3 2^{-\exploreExp(\kappa/4) K}] = 1-O(2^{-cK}).\]
\item\label{it:tilde_jump_bound} On $\wt{n}_K^\uparrow \leq c_1 2^{(1-\exploreExp)(\kappa/4) K}$, the amount of quantum natural time elapsed by the first $\wt{n}_K^\uparrow$ chunks is at most $c_1 2^{((1-\exploreExp)(\kappa/4)-1) K}$.  Therefore (from the explicit form of the L\'evy measure) on $\wt{n}_K^\uparrow \leq c_1 2^{(1-\exploreExp)(\kappa/4) K}$, the number of jumps that either the left or right boundary length processes make of size at least $2^{-\exploreExp(\kappa/4) K}$ is stochastically dominated by a Poisson random variable with mean a constant times
\[ (2^{-\exploreExp(\kappa/4) K})^{-4/\kappa} \times 2^{((1-\exploreExp)(\kappa/4)-1) K} = 2^{(1-\exploreExp)(\kappa/4-1) K}.\]
As $\kappa \in (8/3,4)$ we note that $(1-\exploreExp)(\kappa/4-1) < 0$.  Altogether, we conclude that the probability that either the left or right boundary length processes make a jump of size at least $2^{-\exploreExp(\kappa/4) K}$ is $O(2^{-cK})$.
\end{enumerate}

\noindent{\it Step 4.  Analysis of boundary length processes.}  Since we have that $\wt{L}_{j,M}^\uparrow - \wt{\Delta}_{j,M}^\uparrow \leq L_{j,M}^\uparrow$ and $\wt{R}_{j,M}^\uparrow + \wt{\Delta}_{j,M}^\uparrow \geq R_{j,M}^\uparrow$, it follows from the above and Lemma~\ref{lem:stablejumpsum} that (possibly changing the values of $c,c_0,c_1$) we have
\begin{enumerate}[(i)]
\item\label{it:n_bound} 
\[ \p[ c_0 2^{(1-\exploreExp)(\kappa/4) K} \leq n_K^\uparrow \leq c_1 2^{(1-\exploreExp)(\kappa/4) K} ] = 1- O(2^{-cK}).\]
\item\label{it:l_bound} We have that
\[ \p[ L_{n_K^\uparrow,K}^\uparrow \geq c_3 2^{-\exploreExp(\kappa/4) K}] = 1-O(2^{-cK}).\]
\item\label{it:jump_bound} The probability that either the left or right boundary length processes make a jump of size at least $2^{-\exploreExp(\kappa/4) K}$ is $O(2^{-cK})$.
\end{enumerate}

\noindent{\it Step 5. Completion of the proof.}  On the event that the above hold, we consider the exploration starting at the next scale.  As in the case $M=K$, we consider the dominating boundary length processes.  We note that when we start the exploration for $M=K-1$, on the event described in~\eqref{it:tilde_l_bound}, the amount of boundary length to the left of the starting point which is contained in chunks for which $E_{j,K}^\uparrow$ occurs is at least a constant times $2^{-\exploreExp(\kappa/4)K}$.  On the event described in~\eqref{it:tilde_n_bound}, the number of such chunks is at most $c_1 2^{(1-\exploreExp)(\kappa/4)K}$.  Since the top boundary length of each such chunk is at most $C_0 2^{-(\kappa/4)K}$, if $A$ is the number of chunks with top boundary length at most $a_0 2^{-(\kappa/4)K}$ and $B$ is the number of the remaining chunks, we have that
\[ c_3 2^{-\exploreExp(\kappa/4)K} \leq a_0 2^{-(\kappa/4)K} \times A + C_0 2^{-(\kappa/4)K} \times B.\]
By decreasing the value of $a_0 > 0$ if necessary, the only way that this inequality can hold is if $B$ is at least a constant times $2^{(1-\exploreExp)(\kappa/4) K}$.

The same analysis as for $M=K$ thus implies for $M=K-1$ that
\begin{enumerate}[(i)]
\item By~\eqref{eqn:right_move}, we have that
\[ \p[ c_0 2^{(1-\exploreExp)(\kappa/4) M} \leq n_M^\uparrow \leq c_1 2^{(1-\exploreExp)(\kappa/4) M}] =  1- O(2^{-cK}).\]
\item By~\eqref{eqn:right_move} and~\eqref{eqn:left_move}, we have that
\[ \p[ L_{n_M^\uparrow,M}^\uparrow \geq c_3 2^{(1-\exploreExp)(\kappa/4) M}] = 1- O(2^{-cM}).\]
\item On $n_M^\uparrow \leq c_1 2^{(1-\exploreExp)(\kappa/4) M}$, the amount of quantum natural time elapsed by the chunks is at most $c_1 2^{((1-\exploreExp)(\kappa/4)-1) M}$.  Therefore (from the explicit form of the L\'evy measure) the probability that either the left or right boundary length processes make a jump of size at least a constant times $2^{-\exploreExp(\kappa/4) M}$ is $O(2^{-cM})$.
\end{enumerate}
We can therefore iterate this for the rest of the steps in the upward exploration and likewise for the downward exploration to obtain the result.
\end{proof}

\subsection{Limiting exploration}
\label{subsec:limiting_exploration}

The exploration described in Section~\ref{subsec:exploration_def} and the result of Lemma~\ref{lem:point_to_point_exploration} from Section~\ref{subsec:perc_estimates} shows that with probability $1-O(e^{-c J})$ there exists a path of ``good chunks'' in a percolation exploration which starts with a chunk corresponding to an $\SLE_\kappa^0(\kappa-6)$ run for time of order $2^{-K}$ which forms a neighborhood of $0$ and terminates precisely at $x_1$, the point on $\R_+$ so that $\qbmeasure{h}([0,x_1]) = 1$.  The purpose of this section is to show that one can construct such an exploration which starts precisely at $0$ and ends precisely at $x_1$ by taking a limit as $K \to \infty$ (Lemma~\ref{lem:point_to_point_half_plane_exploration}).  We will then extend the construction to the setting of a quantum disk in Lemma~\ref{lem:point_to_point_disk_exploration}.

\begin{lemma}
\label{lem:point_to_point_half_plane_exploration}
Suppose that $\CH = (\h,h,0,\infty)$ is a quantum half-plane.  Let $x_1 \in \R_+$ be such that $\qbmeasure{h}([0,x_1]) = 1$.  Fix $J \in \N$ and suppose that we have families of events $E_{i,M}^\uparrow$, $E_{i,M}^\downarrow$ for path-decorated quantum surfaces for each $i \in \N$ and $M \geq J$ so that if we fix $K \geq J$ and perform the exploration as in Section~\ref{subsec:exploration_def} with $J \leq M \leq K$ then the conditions of Lemma~\ref{lem:point_to_point_exploration} are satisfied.

There exist families of quantum surfaces $(\CE_{M,J}^\uparrow)_{M \geq J}$ and $(\CE_{M,J}^\downarrow)_{M \geq J}$ coupled with $\CH$ so that the following are true.	
\begin{enumerate}[(i)]
\item $\CE_{M,J}^\uparrow$ is decreasing in $M$ and $\CE_{M,J}^\downarrow$ is increasing in $M$.
\item\label{it:unexplored} For each $M \geq J$, the conditional law of $\CH \setminus \CE_{M,J}^\uparrow$ given $\CE_{M,J}^\uparrow$ is a quantum half-plane.  The same holds with $\CE_{M,J}^\downarrow$ in place of $\CE_{M,J}^\uparrow$.
\item\label{it:cond_law} Fix $K \geq J$.  The conditional law of the surfaces $\CE_{M,J}^\uparrow \setminus \CE_{K,J}^\uparrow$ for $J \leq M \leq K$ and $\CE_{M,J}^\downarrow \setminus \CE_{K,J}^\uparrow$ for $J \leq M$ given $\CE_{K,J}^\uparrow$ is described by the exploration procedure from Section~\ref{subsec:exploration_def} where we abuse notation and write $E_{i,M}^\uparrow$ for the event that $E_{i,M}^\uparrow$ holds for the path-decorated quantum surface up until when the $i$th $\SLE_\kappa^0(\kappa-6)$ chunk is explored in $\CE_{M,J}^\uparrow$.  (In particular, each $\CE_{M,J}^\uparrow \setminus \CE_{M+1,J}^\uparrow$ and $\CE_{M,J}^\downarrow \setminus \CE_{M-1,J}^\downarrow$ can be further decomposed into a collection of $\SLE_\kappa^0(\kappa-6)$ chunks.)
\end{enumerate}
Moreover, there exists a constant $c > 0$ (which does not depend on $J$) so that the probability that any of the following occur is $O(2^{-c J})$:
\begin{itemize}
\item There exists $M \geq J$ so that the number of $\SLE_\kappa^0(\kappa-6)$ chunks which make up $\CE_{M,J}^\uparrow \setminus \CE_{M+1,J}^\uparrow$ for $J \leq M$ or $\CE_{M,J}^\downarrow \setminus \CE_{M-1,J}^\downarrow$ for $J \leq M+2$ is at least $c_F 2^{(1-\exploreExp)(\kappa/4) M}$.  The number of $\SLE_\kappa^0(\kappa-6)$ chunks which make up $\CE_{J,J}^\downarrow \setminus \CE_{J,J}^\uparrow$ is at least $c_F 2^{(\kappa/4) J}$.
\item There exists $M \geq J$ so that the boundary length process for an $\SLE_\kappa^0(\kappa-6)$ chunk which is part of either $\CE_{M,J}^\uparrow \setminus \CE_{M+1,J}^\uparrow$ or $\CE_{M,J}^\downarrow \setminus \CE_{M-1,J}^\downarrow$ makes a jump of size at least $2^{-\exploreExp (\kappa/4)M}$.
\item Either $0$ or $x_1$ is disconnected from $\infty$ by $\CE_{M,J}^\uparrow \setminus \CE_{M+1,J}^\uparrow$ or $\CE_{M,J}^\downarrow \setminus \CE_{M-1,J}^\downarrow$.
\item The difference in the boundary length of $\partial \CE_{M,J}^\uparrow \setminus \partial \h$ and $\partial \CE_{M,J}^\uparrow \cap \partial \h$ is at least $2^{-c M}$.  The difference in the boundary length of $\partial \CE_{M,J}^\downarrow \setminus \partial \h$ and $\partial \CE_{M,J}^\downarrow \cap \partial \h$ is at least $2^{-c J}$.
\end{itemize}
Moreover, the probability that any of the above three items occur for some fixed value of $M \geq J$ is $O(2^{-c M})$ (for the same constant $c$).
\end{lemma}

Before we proceed to the proof of Lemma~\ref{lem:point_to_point_half_plane_exploration}, we first need to collect the following lemma which will be used in order to show that the law of the diameter of the approximations to the $(\CE_{M,J}^\uparrow)_{M \geq J}$ and $(\CE_{M,J}^\downarrow)_{M \geq J}$ is tight.  The proof of this lemma can be skipped on a first reading.

\begin{figure}[ht!]
\begin{center}
\includegraphics[scale=1]{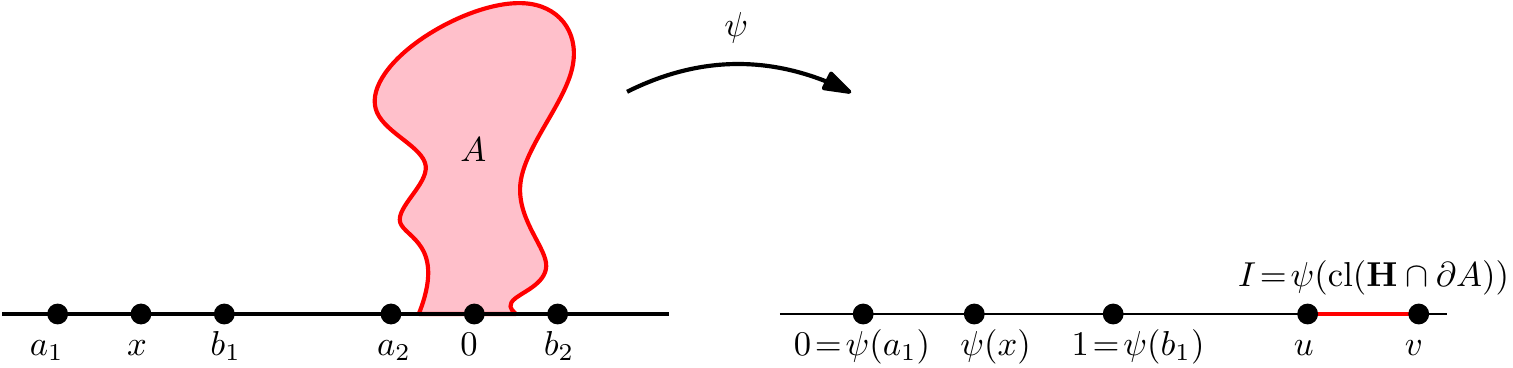}	
\end{center}
\caption{\label{fig:sle_hitting} Illustration of the setup and proof of Lemma~\ref{lem:sle_hitting_diam_bound}.}
\end{figure}

\begin{lemma}
\label{lem:sle_hitting_diam_bound}
Suppose that $-\infty < a_1 < b_1 < a_2 < 0 < b_2 < \infty$.  For every $p \in (0,1)$ there exists $d_0 \in (0,\infty)$ so that the following is true.  Let $A$ be a compact $\h$-hull (i.e., $A \subseteq \h$, $A = \closure{A} \cap \h$, $\closure{A}$ is compact, and $\h \setminus A$ is simply connected) so that $\closure{A} \cap \partial \h \subseteq [a_2,b_2]$.  Let $\eta$ be an $\SLE_\kappa^0(\kappa-6)$ process in $\h \setminus A$ from $x \in [a_1,b_1]$ to $\infty$.  If the probability that $\eta$ hits $A$ is at most $p$ then $\diam(A) \leq d_0$.
\end{lemma}
\begin{proof}
See Figure~\ref{fig:sle_hitting} for an illustration of the setup. Let $\psi$ be the unique conformal transformation from $\h \setminus A$ to $\h$ which fixes $\infty$, takes $a_1$ to $0$, and $b_1$ to $1$. Then $\psi(\eta)$ is an $\SLE_\kappa^0(\kappa-6)$ process in $\h$ from $\psi(x) \in [0,1]$ to $\infty$. Let $I = \psi(\closure{\h \cap \partial A}) = [u,v]$.  By assumption, the probability that $\psi(\eta)$ hits $I$ is at most $p$.

We claim that there exists $M \in (0,\infty)$ depending only on $a_1$, $b_1$, $a_2$, $b_2$ so that $u \leq M$.  Indeed, suppose that $\eta'$ is an $\SLE_6$ in $\h$ from $b_1$ to $\infty$.  Then there exists $q \in (0,1)$ so that the probability that $\eta'$ hits $[b_2,\infty)$ before hitting $(-\infty,a_1]$ is at least $q$.  Therefore the probability that $\eta'$ hits $A$ or $[b_2,\infty)$ before hitting $(-\infty,a_1]$ is also at least $q$.  Now assume that $\eta'$ is an $\SLE_6$ in $\h \setminus A$ instead of $\h$.  By the locality property for $\SLE_6$, the same statement holds.  Therefore the probability that $\psi(\eta')$, an $\SLE_6$ in $\h$ from $1$ to $\infty$, hits $\psi(\h \cap \partial A) \cup \psi([b_2,\infty)) \subseteq [u,\infty)$ before hitting $(-\infty,0]$ is at least $q$.  This proves the claim. 

Let $N \geq M$ be such that the probability that $\psi(\eta)$ hits $[M,N]$ is at least $p$.  We note that such an $N$ exists as the trunk of $\eta$ is an $\SLE_{\kappa'}(\kappa'/2-3;\kappa'/2-3)$ process in $\h$ from $x$ to $\infty$ and the probability that such a process hits any infinite boundary interval is $1$.  Note that $N$ depends only on $a_1$, $a_2$, $b_1$, $b_2$, and $p$.  Since the probability that $\psi(\eta)$ hits $I$ is at most $p$ and $u \leq M$ it must be that $v \leq N$.

To complete the proof, it suffices to show that for each $R \in (0,\infty)$ there exists $d \in (0,\infty)$ so that if $\diam(A) \geq d$ then $v \geq R$.  To see that this is the case, let $\eta'$ be as above.  Suppose that we have $k \in \N$ so that $A$ crosses the annulus $A_k = B(0,2^k) \setminus B(0,2^{k-1})$.  Let $\tau_k = \inf\{t \geq 0 : \eta'(t) \in \partial B(0,3 \cdot 2^{k-2})\}$ (i.e., $\tau_k$ is the first time that $\eta'$ hits the circle in the middle of $A_k$) and let $\sigma_k = \inf\{t \geq \tau_k : \eta(t) \notin A_k\}$.  If $\eta'|_{[\tau_k,\sigma_k]}$ hits $[-2^k,-2^{k-1}]$ and then hits $[2^{k-1},2^k]$, then $\eta'|_{[\tau_k,\sigma_k]}$ hits $A$.  As there exists $p_0 > 0$ so that the conditional probability of this event given $\eta'|_{[0,\sigma_{k-1}]}$ is at least $p_0$, it follows that if $\diam(A) \geq d$ then the probability that $\eta'$ hits $A$ is at least $1-O(d^{-a})$ for a constant $a > 0$.  Arguing as above, this implies that $v \to \infty$ as $\diam(A) \to \infty$.
\end{proof}

\begin{proof}[Proof of Lemma~\ref{lem:point_to_point_half_plane_exploration}]
We are going to construct the family of quantum surfaces described in the lemma statement using the exploration that we have defined in Section~\ref{subsec:exploration_def} and by taking an appropriate subsequential limit.  The first step is to control the change in the boundary length process (Step 1), then we will obtain a diameter bound for the approximations (Step 2), then deduce the existence of the subsequential limit (Step 3), and then finally describe the form of the conditional law given the exploration (Step 4).

Fix $K \geq J$.  Suppose that we perform the exploration as in Section~\ref{subsec:exploration_def} with $J \leq M \leq K$.  Let $\CE_{M,J,K}^\uparrow$ (resp.\ $\CE_{M,J,K}^\downarrow$) be the quantum surface which consists of all of the $\SLE_\kappa^0(\kappa-6)$ chunks for the exploration up to when the increasing (resp.\ decreasing) part with chunks with quantum natural time in $[\delta_0 2^{-M}, 2^{-M}]$ has been completed.

\noindent{\it Step 1. Change in boundary length.}
Let $F$ be the event that the exploration fails so that $F^c$ is the event that it does not fail.  For each $J \leq M \leq K$, let $T_M^\uparrow$ denote the quantum length of $\partial \CE_{M,J,K}^\uparrow \setminus \partial \CH$.  For each $J \leq M$ we also let $T_M^\downarrow$ denote the quantum length of $\partial \CE_{M,J,K}^\downarrow \setminus \partial \CH$.  We are now going to show that for each $\epsilon > 0$ there exists a constant $C \in (0,\infty)$ so that
\[ \p\!\left[ \left( \left\{\max_{J \leq M\leq K} T_M^\uparrow \geq C \right\} \cup \left\{\sup_{J \leq M} T_M^\downarrow \geq C \right\} \right) \cap F^c \right] \leq \epsilon.\]

Consider the exploration when we have parameterized it by the quantum natural time of the individual $\SLE_\kappa^0(\kappa-6)$ chunks.  Let $X_t$ be the overall boundary length process.  That is, $X_t$ is equal to the quantum length of the top boundary (part in $\h$) of the exploration minus the quantum length of the bottom boundary (part in $\partial \h$) of the exploration when $t$ units of quantum natural time has elapsed.  Note that the quantum length of the top boundary at time $t$ is equal to $X_t$ plus the quantum length of the bottom at time $t$.  On the event that the exploration does not fail, i.e., $F^c$ occurs, we have that the bottom is contained in $[x_{-1},x_1]$.  As $\qbmeasure{h}([x_{-1},x_1]) = 2$, the quantum length of the bottom boundary is at most $2$ on $F^c$.  Therefore to give an upper bound to the quantum length of the top boundary on this event it suffices to give an upper bound to $X_t$.

Recall that $X_t$ evolves as a $4/\kappa$-stable L\'evy process.  On $F^c$, we have that $n_M^\uparrow \leq c_F 2^{(1-\exploreExp)(\kappa/4)M}$ for each $J \leq M \leq K$, $n_M^\downarrow \leq c_F 2^{(1-\exploreExp)(\kappa/4)M}$ for each $M \geq J+1$, and $n_J^\downarrow \leq c_F 2^{(\kappa/4)J}$.  At the $M$th step of the upward or downward exploration (i.e., when exploring chunks with quantum natural time in $[\delta_0 2^{-M},2^{-M}]$), on $F^c$ the amount of quantum natural time involved is thus at most $c_F 2^{((1-\exploreExp)(\kappa/4)-1)M}$ for $J \leq M \leq K$ (upward exploration) or $M \geq J+1$ (downward exploration) or at most $c_F 2^{(\kappa/4-1)J}$ (first part of the downward exploration).  Note that $(1-\exploreExp)(\kappa/4)-1 < 0$ as $\exploreExp \in (0,1)$ and $\kappa \in (8/3,4)$ and also that $\kappa/4-1 < 0$.  By summing over $M \geq J$, we thus see that the total amount of quantum natural time involved in the exploration on $F^c$ is at most
\begin{equation}
\label{eqn:t_bound}
T = c_F 2^{(\kappa/4-1)J} + 2 c_F \sum_{M=J}^\infty 2^{((1-\exploreExp)(\kappa/4)-1)M}.
\end{equation}
We emphasize here that $T$ is be taken to be uniform in $K$.  By \cite[Chapter~VIII, Proposition~4]{bertoin1996levy}, we have that $\sup_{0 \leq t \leq T} X_t$ is a.s.\ finite.  Altogether, what we have shown implies that the law of the top boundary length of the exploration is tight as $K \to \infty$ on $F^c$.

\noindent{\it Step 2.  Diameter bound.}  Fix $M \geq J$.  We are going to show that the diameter of $\CE_{M,J,K}^\uparrow$ is tight on $F^c$ as $K \to \infty$.  Let $x_{-2} < 0$ be so that $\qbmeasure{h}([x_{-2},0]) = 2$.  Let $\eta$ be an $\SLE_\kappa^0(\kappa-6)$ process in $\h$ from $x_{-2}$ to $\infty$ coupled as a CPI in $\Gamma$ conditionally independently of the rest of the exploration.  Then we know that the conditional law of $\eta$ given $\CE_{M,J,K}^\uparrow$ is that of an $\SLE_\kappa^0(\kappa-6)$ process in $\h \setminus \CE_{M,J,K}^\uparrow$ from $x_{-2}$ to $\infty$.  Note that we can write $\partial \h  \cap \partial \CE_{I,J,K}^\uparrow$ as $[a,b]$.  Let $u$ (resp.\ $v$) be the quantum length of $[x_{-2},a]$ (resp.\ $[x_{-2},a]$ and $\h \cap \partial \CE_{I,J,K}^\uparrow$, i.e., the top of $\CE_{I,J,K}^\uparrow$).  As the quantum surface parameterized by $\h \setminus \CE_{I,J,K}^\uparrow$ is a quantum half-plane, it follows that the conditional probability given $\CE_{M,J,K}^\uparrow$ that $\eta$ hits $\CE_{M,J,K}^\uparrow$ is equal to the probability that the running infimum of a $4/\kappa$-stable L\'evy process with both upward and downward jumps starting from $0$ hits $[2-v,2-u]$.  As explained above, the quantum length $v-u$ of the top of $\CE_{I,J,K}^\uparrow$ is tight on $F^c$ as $K \to \infty$.  Therefore the probability that $\eta$ hits $\CE_{M,J,K}^\uparrow$ is bounded away from $1$ on $F^c$ as $K \to \infty$.  By Lemma~\ref{lem:sle_hitting_diam_bound}, this implies that the law of the diameter of $\CE_{M,J,K}^\uparrow$ is tight on $F^c$ as $K \to \infty$.

\noindent{\it Step 3. Subsequential limit and law of unexplored region.}  Let $g_{M,J,K}$ be the unique conformal map from $\h \setminus \CE_{M,J,K}^\uparrow$ to $\h$ with $g_{I,J,K}(z) - z \to 0$ as $z \to \infty$.  Then it follows that the law of $g_{I,J,K}$ (with the Caratheodory topology) on $F^c$ is tight as $K \to \infty$ since the diameter of $\CE_{M,J,K}^\uparrow$ on $F^c$ is tight as $K \to \infty$.  Indeed, this follows from \cite{law2005conformally}.  Suppose that we pass to a sequence $(K_n)$ so that the joint law of any finite collection of the $g_{I,J,K_n}$ together with $h$ converges weakly as $n \to \infty$.  Let $g_{I,J}$ denote the limit as $n \to \infty$.  We abuse notation and write $h$ for the instance of $h$ which is coupled with $g_{I,J}$ as well as with $g_{I,J,K_n}$ for each $n$.  As the quantum surface described by $h \circ g_{I,J,K_n}^{-1} + Q \log| (g_{I,J,K_n}^{-1})'|$ is a quantum half-plane for each $n$, it follows that $h \circ g_{I,J}^{-1} + Q \log |(g_{I,J}^{-1})'|$ is a quantum half-plane as well.  This proves that~\eqref{it:unexplored} holds.

\noindent{\it Step 4.  Form of the conditional law.}  By the construction of the subsequential limit, it follows that~\eqref{it:cond_law} holds.  That the final points hold is clear from the construction.
\end{proof}

We now deduce the analog of Lemma~\ref{lem:point_to_point_half_plane_exploration} in the case of the quantum disk.  We note that the exploration that we have defined in Section~\ref{subsec:exploration_def} makes sense on a doubly marked quantum disk, although the estimates we have previously established for its basic properties are for the exploration in the case of the quantum half-plane.  The purpose of the following lemma is to deduce these properties when we work instead on the disk.

\begin{lemma}
\label{lem:point_to_point_disk_exploration}
Fix $\ell \geq 2$ and suppose that $\CD = (\D,h) \sim \qdiskL{\gamma}{\ell}$.  Suppose that $x \in \partial \CD$ is sampled from $\qbmeasure{h}$ and $y \in \partial \CD$ is such that $\qbmeasure{h}(\ccwBoundary{x}{y}{\partial \CD}) = 1$.  Let $z \in \ccwBoundary{x}{y}{\partial \CD}$ be such that $\qbmeasure{h}(\cwBoundary{x}{z}{\partial \CD}) = \qbmeasure{h}(\cwBoundary{z}{y}{\partial \CD}) = \qbmeasure{h}(\cwBoundary{x}{y}{\partial \CD})/2 = (\ell-1)/2$.  For each $J \in \N$, there is a family of quantum surfaces $\CE_{I,J}^\uparrow$ and $\CE_{I,J}^\downarrow$ so that all of the properties of Lemma~\ref{lem:point_to_point_half_plane_exploration} hold but with the quantum half-plane replaced by $\CD$ and the marked points at $0$, $x_1$, $\infty$ replaced by $x$, $y$, and $z$, respectively.
\end{lemma}

In order to deduce Lemma~\ref{lem:point_to_point_disk_exploration} from Lemma~\ref{lem:point_to_point_half_plane_exploration}, we first record the following Radon-Nikodym derivative (\cite[Proposition~5.1]{msw2020simplecle}) which serves to compare $\SLE_\kappa^0(\kappa-6)$ explorations on quantum disks and on quantum half-planes.

\begin{lemma}
\label{lem:rn_derivative}
Fix $\ell > 0$ and suppose that $\CD = (\D,h) \sim \qdiskL{\gamma}{\ell}$.  Suppose that $x \in \partial \CD$ is sampled from $\qbmeasure{h}$, let $L_0,R_0 > 0$ be such that $\ell = L_0 + R_0$, and let $y \in \partial \CD$ be such that $\qbmeasure{h}(\cwBoundary{x}{y}{\partial \CD}) = L_0$ and $\qbmeasure{h}(\ccwBoundary{x}{y}{\partial \CD}) = R_0$.  Suppose that $\eta$ is an independent $\SLE_\kappa^0(\kappa-6)$ in $\CD$ from $x$ to $y$ which is parameterized by the quantum natural time of its trunk.  For each $t \geq 0$, we let $\CD_t$ denote the component of $\CD \setminus \eta([0,t])$ with $y$ on its boundary and let $L_t$ (resp.\ $R_t$) denote the quantum length of $\cwBoundary{\eta(t)}{y}{\partial \CD_t}$ (resp.\ $\ccwBoundary{\eta(t)}{y}{\partial \CD_t}$).  Let $E_t$ be the event that $\inf_{0 \leq s \leq t} L_s > 0$ and $\inf_{0 \leq s \leq t} R_s > 0$.  On $E_t$, the Radon-Nikodym derivative between the law of $(L,R)|_{[0,t]}$ and a pair of independent $4/\kappa$-stable L\'evy processes starting from $(L_0,R_0)$ on $[0,t]$ is given by $(\Delta_0/\Delta_t)^{4/\kappa+1}$ where $\Delta_t = L_t + R_t$.
\end{lemma}

\begin{proof}[Proof of Lemma~\ref{lem:point_to_point_disk_exploration}]
Suppose that $\CD$, $x$, $y$, $z$ are as in the statement of the lemma.  Then we can consider the same approximation scheme using the exploration from Section~\ref{subsec:exploration_def} as in the proof of Lemma~\ref{lem:point_to_point_half_plane_exploration}.  Indeed, even though we defined the exploration in Section~\ref{subsec:exploration_def} in the case of the half-plane, it still makes sense to consider it in the case of the quantum disk.  Let $\Delta_0 = \ell$ be the boundary length of $\partial \CD$.  For each $J \leq M$, we let $\Delta_M^\downarrow$ be the boundary length of the complementary component with $z$ on its boundary after we have completed the part of the exploration where the $\SLE_\kappa^0(\kappa-6)$ chunk sizes are in $[\delta_0 2^{-M},2^{-M}]$.  Note that the event that the exploration has not failed up until this point, the analog of the event $E_t$ from Lemma~\ref{lem:rn_derivative} up to this time holds.  Therefore on this event, Lemma~\ref{lem:rn_derivative} implies that the Radon-Nikodym derivative between the half-planar and disk explorations is of the form $(\Delta_0/\Delta_M^\downarrow)^{4/\kappa+1}$.  Due to the definition of the exploration failing, on the event that has not failed we have that $\Delta_M^\downarrow$ is bounded from below.  This completes the proof of the result.
\end{proof}

\section{Tightness on the boundary}
\label{sec:boundary_tightness}

The purpose of this section is to prove the tightness of $(\medianHP{\epsilon})^{-1} \metapprox{\epsilon}{\cdot}{\cdot}{\Gamma}$ restricted to the domain boundary.  We are going to prove the result in a more general setting than in the statement of Theorem~\ref{thm:cle_loop} in the sense that we will not restrict to the case that the domain boundary is given by a $\CLE_\kappa$ loop.  Instead, we will make an assumption on the behavior of the quantum area measure when we consider a sample from the law $\qdiskL{\gamma}{1}$ parameterized by the domain.  (We note that this is effectively an assumption about the roughness of $\partial D$.)  In what follows, when we write ``quantum length metric on $\partial D$'' for $D \subseteq \C$ simply connected we mean the metric on (the prime ends of) $\partial D$ where the distance between $x,y$ is given by the minimum of the quantum lengths of $\cwBoundary{x}{y}{\partial D}$ and $\ccwBoundary{x}{y}{\partial D}$.

\begin{proposition}
\label{prop:boundary_distance_tail_bounds}
Let $\alpha_\LBD > 2$ and fix $\beta \in (0,2/\alpha_\LBD)$.  Suppose that $\CD = (D,h,x,y) \sim \qdiskL{\gamma}{1}$.  Fix $c_0, \epsilon_0 > 0$ and let $E$ be the event that for every $z \in D$ and $\epsilon \in (0,\epsilon_0)$ so that $B(z,\epsilon) \subseteq D$ we have that $\qmeasure{h}(B(z,\epsilon)) \geq c_0 \epsilon^{\alpha_\LBD}$.  Let $\Gamma$ be an independent $\CLE_\kappa$ in $D$ and let $\Upsilon$ be its carpet.  Let $X_{\epsilon,\beta}$ be the $\beta$-H\"older norm of $(z,w) \mapsto (\medianHP{\epsilon})^{-1} \metapprox{\epsilon}{z}{w}{\Gamma}$ with respect to the quantum length metric on $\partial D$.  For every $p > 0$ we have that
\begin{equation}
\label{eqn:holder_norm_tight}
\p[E,\ X_{\epsilon,\beta} \geq A] = O(A^{-p}) \quad\text{as}\quad A \to \infty
\end{equation}
where the implicit constants depend on $p,\epsilon_0,c_0,\alpha_\LBD,\beta$.

Moreover, let $Y_{\beta}$ be the $\beta$-H\"older norm of the Euclidean metric with respect to the quantum length metric on $\partial D$.  Then we have that
\begin{equation}
\label{eqn:quantum_holder_norm_tight}
\p[E,\ Y_\beta \geq A] = O(A^{-p}) \quad\text{as}\quad A \to \infty
\end{equation}
where the implicit constants depend on $p,\epsilon_0,c_0,\alpha_\LBD,\beta$.
\end{proposition}

In the setting of Proposition~\ref{prop:boundary_distance_tail_bounds}, truncating on the event that the quantum area measure is bounded from below serves to give a lower bound on the size of the $\SLE_\kappa^0(\kappa-6)$ chunks which make up the exploration described in Lemma~\ref{lem:point_to_point_disk_exploration}.  The reason that this is important is that our definition of the good event for a chunk which is given just below will involve the minimal ``$\lebneb{\epsilon}$-length'' of various crossings when one has parameterized the chunk by $\D$.  Embedding the chunk into $D$ involves applying a conformal transformation which can lead to an effective change in the size of the neighborhood used to approximate the ``length'' of a path.  As we will see in the proof of Proposition~\ref{prop:boundary_distance_tail_bounds}, the lower bound on the quantum area measure will serve to give an upper bound on the size of an embedded chunk hence an upper bound on the ``length'' of an embedded path.

Let us now explain the main steps used to prove Proposition~\ref{prop:boundary_distance_tail_bounds}.
\begin{enumerate}
\item[Step 1.] We will first give the events for the good chunks in Definition~\ref{def:good_chunk} and then show in Lemma~\ref{lem:good_chunk_occurs} that we can adjust the parameters in the definition to make the probability that they occur sufficiently close to $1$ so that the results of Section~\ref{sec:percolation_exploration} apply.
\item[Step 2.] The definition of a good chunk in Definition~\ref{def:good_chunk} will involve quantities of the form $\lebneb{\epsilon}(\omega)$ for paths $\omega$ which arise after we have parameterized the chunk by $\D$.  We then need to control the Lebesgue measure of the $\epsilon$-neighborhood of these paths after embedding them and this will be the usage of Lemma~\ref{lem:covering_lemma}.
\item[Step 3.] We complete the proof of Proposition~\ref{prop:boundary_distance_tail_bounds} using Lemma~\ref{lem:point_to_point_disk_exploration}.  The idea is that we will start with two boundary points $x$, $y$ and perform the exploration as in Lemma~\ref{lem:point_to_point_disk_exploration}.  If the exploration succeeds, then our previous estimates and the definition of a good chunk will allow us to construct a path $\omega$ in $\Upsilon$ with control on $\lebneb{\epsilon}(\omega)$.  If the exploration fails, then we will repeat the exploration as in Lemma~\ref{lem:point_to_point_disk_exploration} in the remaining domains.  The exploration will be repeated either once or twice depending on the type of failure.  We will refer to these new explorations as the ``children'' of the failed exploration so that we get a tree structure on such explorations.  We will show that by properly adjusting the parameters we can dominate the number of times the exploration has to be repeated by the size of a subcritical Galton-Watson tree, which will complete the proof.
\end{enumerate}

Let $\CH = (\h,h,0,\infty)$ be a quantum half-plane.  Suppose that $\eta$ is an $\SLE_\kappa^0(\kappa-6)$ curve on~$\h$ from~$0$ to~$\infty$ which is taken to be independent of $h$ and then parameterized by the quantum natural time of its trunk.  For the quantum surface $\CN_t$ disconnected from~$\infty$ by $\eta|_{[0,t]}$, we let $z_L(\CN_t)$ (resp.\ $z_R(\CN_t)$) denote the leftmost (resp.\ rightmost) point on the bottom of $\CN_t$.  We also let $z_C(\CN_t) = \eta(0)$.

\begin{figure}[ht!]
\begin{center}
\includegraphics[scale=1]{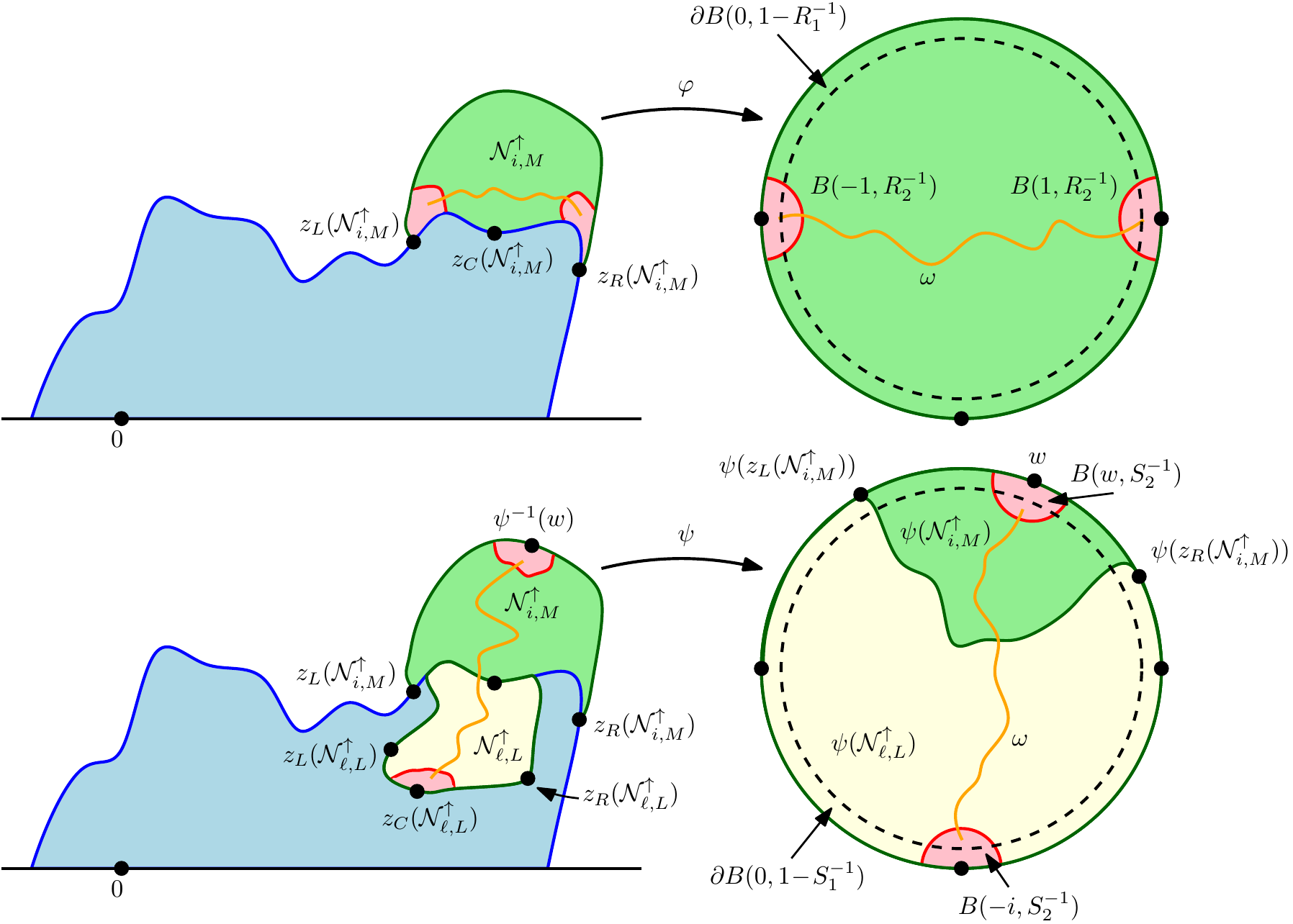}	
\end{center}
\caption{\label{fig:good_boundary_def1} {\bf Top:} Illustration of part~\eqref{it:path_across_chunk} of Definition~\ref{def:good_chunk}.  Shown in blue is the exploration up until $\CN_{i,M}^\uparrow$ is discovered and the images of $z_L(\CN_{i,M}^\uparrow)$, $z_C(\CN_{i,M}^\uparrow)$, and $z_R(\CN_{i,M}^\uparrow)$ under $\varphi$ are respectively given by $-1$, $-i$, and $1$, respectively. {\bf Bottom:} Illustration of part~\eqref{it:path_across_two_chunks} of Definition~\ref{def:good_chunk}.  The chunk $\CN_{\ell,L}^\uparrow$ is shown in light yellow even though it is ``good'' in order to differentiate it from $\CN_{i,M}^\uparrow$.  The images of $z_L(\CN_{\ell,L}^\uparrow)$, $z_C(\CN_{\ell,L}^\uparrow)$, and $z_R(\CN_{\ell,L}^\uparrow)$ under $\psi$ are $-1$, $-i$, and $1$, respectively.}
\end{figure}

\begin{definition}
\label{def:good_chunk}
We suppose that we have the setup described in Section~\ref{subsec:exploration_def}.  Fix $R_1, R_2, S_1, S_2 \geq 1$.  For each $i,M$, we define $E_{i,M}^\uparrow$ to be the event for $\CN_{i,M}^\uparrow$ and $\eta_{i,M}^\uparrow$ that $\eta_{i,M}^\uparrow(\sigma_{i,M}^\uparrow) \in \partial \h$ and the following hold.
\begin{enumerate}[(i)]
\item\label{it:chunk_lengths} The quantum lengths of the top, left side of the bottom, and the right side of the bottom of~$\CN_{i,M}^\uparrow$ are in $[R_1^{-1} 2^{-(\kappa/4)M}, R_1 2^{-(\kappa/4)M}]$.
\item\label{it:path_across_chunk} Let $\varphi \colon \CN_{i,M}^\uparrow \to \D$ be the unique conformal transformation which takes $z_L(\CN_{i,M}^\uparrow)$ to $-1$, $z_C(\CN_{i,M}^\uparrow)$ to $-i$, and $z_R(\CN_{i,M}^\uparrow)$ to $1$.  There is a path $\cpath$ connecting $B(-1,R_2^{-1})$ to $B(1,R_2^{-1})$ in $\varphi(\Upsilon \cap \CN_{i,M}^\uparrow)$ which has distance at least $R_1^{-1}$ from $\partial \D$ and satisfies $\lebneb{\epsilon}(\cpath) \leq R_1 \medianHP{\epsilon}$.
\item\label{it:chunk_area_ubd} $\qmeasure{h}(\CN_{i,M}^\uparrow)\leq  R_1 2^{-M}$.
\item\label{it:path_across_two_chunks} Let $\CN_{\ell,L}^\uparrow$ be the chunk whose top contains $\eta_{i,M}^\uparrow(0)$ (so that the left side of the bottom of $\CN_{i,M}^\uparrow$ is glued to the top of $\CN_{\ell,L}^\uparrow$) if it exists, otherwise we set $\CN_{\ell,L}^\uparrow = \emptyset$.  If $\CN_{\ell,L}^\uparrow \neq \emptyset$, let $\psi$ be the unique conformal map which takes $\CN = \interior{\closure{\CN_{\ell,L}^\uparrow \cup \CN_{i,M}^\uparrow}}$ to $\D$ so that $z_L(\CN_{\ell,L}^\uparrow)$ is taken to $-1$, $z_C(\CN_{\ell,L}^\uparrow)$ is taken to $-i$, and $z_R(\CN_{\ell,L}^\uparrow)$ is taken to $1$.  The harmonic measure of each of $\cwBoundary{-1}{\psi(z_L(\CN_{i,M}^\uparrow))}{\partial \D}$, $\cwBoundary{\psi(z_L(\CN_{i,M}^\uparrow))}{\psi(z_R(\CN_{i,M}^\uparrow))}{\partial \D}$, and $\cwBoundary{\psi(z_R(\CN_{i,M}^\uparrow))}{1}{\partial \D}$ as seen from $0$ is at least $R_2^{-1}$.  Let $w$ be the midpoint of $\cwBoundary{\psi(z_L(\CN_{i,M}^\uparrow))}{\psi(z_R(\CN_{i,M}^\uparrow))}{\partial \D}$.  Then there moreover is a path $\cpath$ in $\psi(\Upsilon \cap \CN)$ which connects $B(-i,S_2^{-1})$ to $B(w,S_2^{-1})$ which has distance at least $S_1^{-1}$ from $\partial \D$ and satisfies $\lebneb{\epsilon}(\cpath) \leq S_1 \medianHP{\epsilon}$.
\item\label{it:path_across_future_chunk} If $\wh{\CN}$ is another $\SLE_\kappa^0(\kappa-6)$ chunk which starts from a point on the top of $\CN_{i,M}^\uparrow$ with clockwise boundary length distance from $\partial \CH_{i,M}^\uparrow$ given by an integer multiple of $a_0 2^{-(\kappa/4)M}$ and with boundary length distance from a point in $\Upsilon$ at most $a_0 2^{-(\kappa/4)M}$ and is run until the first time it hits the domain boundary after $\delta_0 2^{-J}$, $J \in \{M-1,M,M+1\}$, then the conditional probability given $\CF_{i,M}^\uparrow$ that the previous item holds for the pair $(\CN_{i,M}^\uparrow,\wt{\CN})$ in place of $(\CN_{\ell,L}^\uparrow,\CN_{i,M}^\uparrow)$ is at least $1-R_2^{-1}$.
\end{enumerate}
We define $E_{i,M}^\downarrow$ analogously.
\end{definition}

\begin{lemma}
\label{lem:good_chunk_occurs}
Suppose that we have the setup described in Section~\ref{subsec:exploration_def} and the events $E_{i,M}^\uparrow$, $E_{i,M}^\downarrow$ are as in Definition~\ref{def:good_chunk}.  For every $p_0 \in (0,1)$ there exists $R_1 \geq R_2 \geq 1$ and $S_2 \geq S_1 \geq 1$ so that
\begin{enumerate}[(i)]
\item On the event that the exploration has not failed before exploring $\CN_{i,M}^\uparrow$ (resp.\ $\CN_{i,M}^\downarrow$) we have that $\p[E_{i,M}^\uparrow \giv \CF_{i-1,M}^\uparrow] \geq p_0$ (resp.\ $\p[E_{i,M}^\downarrow \giv \CF_{i-1,M}^\downarrow] \geq p_0$) for all $i,M \in \N$.
\item Suppose that $\CN_{\ell,L}^\uparrow$ is as in part~\eqref{it:path_across_two_chunks} of Definition~\ref{def:good_chunk}, $\cpath_{i,M}^\uparrow$ (resp.\ $\cpath_{\ell,L}^\uparrow$) is any path for $\CN_{i,M}^\uparrow$ (resp.\ $\CN_{\ell,L}^\uparrow$) satisfying part~\eqref{it:path_across_chunk} of Definition~\ref{def:good_chunk}.  Then any path as in part~\eqref{it:path_across_two_chunks} of Definition~\ref{def:good_chunk} crosses both $\cpath_{i,M}^\uparrow$ and $\cpath_{\ell,L}^\uparrow$.
\end{enumerate}
\end{lemma}
\begin{proof}
It is clear from the scaling properties of $4/\kappa$-stable L\'evy process that~\eqref{it:chunk_lengths} holds with probability tending to $1$ as $R_1 \to \infty$.  Lemma~\ref{lem:one_chunk_quantile} implies that with any choice of $R_2 \geq 1$ we have that~\eqref{it:path_across_chunk} holds with probability tending to $1$ as $R_1 \to \infty$.  We have that $\E[ \qmeasure{h}(\CN_{i,M}^\uparrow) \giv \CF_{i-1,M}^\uparrow] = \E[ \qmeasure{h}(\CN_{i,M}^\uparrow)] = O(2^{-M})$ as $\E[ \qmeasure{h}(\CN_{i,M}^\uparrow)]$ is equal to a constant times the sum of the squares of the jumps made by the associated L\'evy process (which are run for time at most $2^{-M}$).  Indeed, the range of the $\SLE_\kappa^0(\kappa-6)$ itself has zero quantum area so all of the quantum area is in the quantum disks that it disconnects from $\infty$ and the expected quantum area associated with a sample from the law $\qdiskL{\gamma}{\ell}$ is equal to a constant times $\ell^2$.  Therefore Markov's inequality implies that~\eqref{it:chunk_area_ubd} holds with probability tending to $1$ as $R_1 \to \infty$.  That~\eqref{it:path_across_two_chunks} and~\eqref{it:path_across_future_chunk} hold with probability tending to $1$ as $R_1, S_1 \to \infty$ for any fixed choice of $R_2,S_2$ follows from Lemma~\ref{lem:two_chunks_quantile}.  Finally, for any values $R_1 \geq R_2 \geq 1$ it is clear that the second assertion of the lemma holds by making $S_2 \geq 1$ sufficiently large. 
\end{proof}

Before we proceed to the proof of Proposition~\ref{prop:boundary_distance_tail_bounds}, we need to collect the following elementary lemma.  As in the case of paths, for $K \subseteq \C$ and $\epsilon > 0$ we let $\lebneb{\epsilon}(K)$ denote the Lebesgue measure of the $\epsilon$-neighborhood of $K$.

\begin{lemma}
\label{lem:covering_lemma}
There exists a constant $c_0 > 0$ so that the following is true for all $\epsilon > 0$ and $\delta \in (0,1/2)$.  Suppose that $K \subseteq \C$ is connected and compact.  Then we have that
\[ \lebneb{\delta \epsilon}(K) \geq c_0 \delta \lebneb{\epsilon}(K).\] 
\end{lemma}
\begin{proof}
Fix $\epsilon > 0$ and $\delta \in (0,1/2)$.  Let $(z_k)$ be a countable, dense subset of $K$.  By the compactness of $K$, there exists $z_{j_1},\ldots,z_{j_n}$ so that $K \subseteq \cup_{i=1}^n B(z_{j_i},\epsilon)$.  By the Vitali covering lemma, there exists $i_1,\ldots,i_\ell \subseteq \{ j_1,\ldots,j_n\}$ so that the balls $B(z_{i_j},\epsilon)$ for $1 \leq j \leq \ell$ are pairwise disjoint and $K \subseteq \cup_{j=1}^\ell B(z_{i_j}, 3\epsilon)$.  Note that
\[ \pi \epsilon^2 \ell \leq \lebneb{\epsilon}(K) \leq 9 \pi \epsilon^2 \ell.\]
Since $K$ is connected, for each $1 \leq j \leq \ell$, we can find at least $\lfloor \delta^{-1} / 2 \rfloor$ disjoint balls with centers in $K \cap B(z_{i_j},\epsilon /2)$ and radius $\epsilon \delta$.  It therefore follows that
\[ \lebneb{\epsilon \delta}(K) \geq \pi (\epsilon \delta)^2 \times \lfloor \delta^{-1} / 2 \rfloor \times \ell \geq \frac{\delta^2 \lfloor \delta^{-1} /2 \rfloor}{9} \lebneb{\epsilon}(K).\]
This proves the result as $\delta^2 \lfloor \delta^{-1} / 2 \rfloor \geq \delta/2$ for all $\delta \in (0,1/2)$.
\end{proof}

We now proceed to give the proof of Proposition~\ref{prop:boundary_distance_tail_bounds}.

\begin{proof}[Proof of Proposition~\ref{prop:boundary_distance_tail_bounds}]
We will first explain the proof of~\eqref{eqn:holder_norm_tight} and then explain at the end how the same argument gives~\eqref{eqn:quantum_holder_norm_tight}.  We assume that we have chosen $\delta_0,p_0 \in (0,1)$ so that the assertion of Lemma~\ref{lem:point_to_point_exploration} holds.  With these parameters, we then assume that we have chosen $R_2 \geq R_1 \geq 1$ and $S_2 \geq S_1 \geq 1$ so that the assertion of Lemma~\ref{lem:good_chunk_occurs} holds.  In what follows, we may therefore view $R_1$, $R_2$, $S_1$, $S_2$ as constants.  We perform the exploration from Lemma~\ref{lem:point_to_point_disk_exploration} using the events defined in Definition~\ref{def:good_chunk} in a recursive manner as follows.

Fix $\delta,\zeta > 0$.  Suppose that $x_0,x_1 \in \partial \CD$ are so that $\qbmeasure{h}(\ccwBoundary{x}{x_0}{\partial \CD}) = \zeta$ and $\qbmeasure{h}(\ccwBoundary{x_0}{x_1}{\partial \CD}) = \delta$.  Let~$\CD_0$ be given by taking $\CD$ and replacing $h$ with $h_0 = h + \tfrac{2}{\gamma} \log \delta^{-1}$.  Let $x_0^0,x_1^0$ be the corresponding points for~$\CD_0$.  Then $\qbmeasure{h_0}(\ccwBoundary{x_0^0}{x_1^0}{\partial \CD_0}) = 1$.  We then perform the exploration as in Lemma~\ref{lem:point_to_point_disk_exploration} in $\CD_0$ from $x_0^0$ to $x_1^0$ until it fails.  We recall that the exploration fails if it makes either a large upward jump (i.e., of size at least $2^{- \exploreExp (\kappa/4) M}$ when exploring chunks with quantum natural time in $[\delta_0 2^{-M}, 2^{-M}]$), a large downward jump (i.e., of size at least $2^{- \exploreExp (\kappa/4) M}$ when exploring chunks with quantum natural time in $[\delta_0 2^{-M}, 2^{-M}]$), requires too many chunks at a particular scale (i.e., at least $c_F 2^{(1-\exploreExp)(\kappa/4) M}$ when exploring chunks of size in $[\delta_0 2^{-M}, 2^{-M}]$ for $M > J$ or at least $c_F 2^{(\kappa/4)J}$ when exploring chunks of size in $[\delta_0 2^{-J},2^{-J}]$), or the deviations in the boundary length process are too large (i.e., the absolute value of the quantum length of the top minus the bottom exceeds $c_F 2^{(1-4/\kappa) M}$ when exploring chunks of size in $[\delta_0 2^{-M}, 2^{-M}]$).  We stop the exploration immediately if either of these possibilities happens, and then proceed as described below.

\begin{figure}[ht!]
\begin{center}
\includegraphics[scale=1]{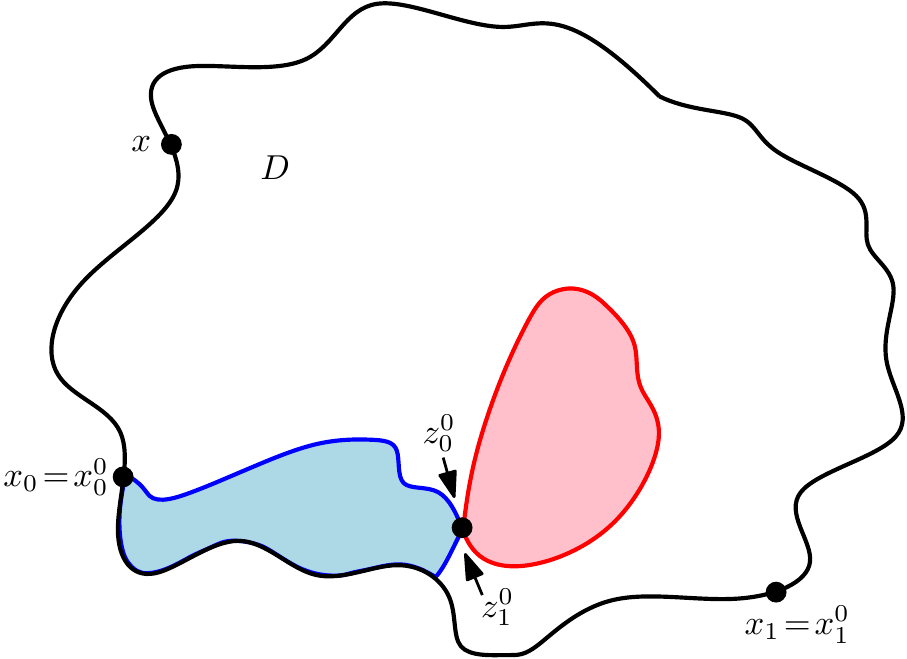}
\end{center}
\caption{\label{fig:upward_jump_failure} Illustration of the exploration in $D$ from $x_0 = x_0^0$ to $x_1 = x_1^0$ (blue) up until it fails due to making a large upward jump (i.e., a large $\CLE_\kappa$ loop is discovered) shown in red.  The prime ends $z_0^0$ and $z_1^0$ correspond to where the loop is rooted on the exploration.  When this happens, two new explorations are started in the remaining domain going from $x_0^0$ to $z_0^0$ and from $z_1^0$ to $x_1^0$.}
\end{figure}

\begin{figure}[ht!]
\begin{center}
\includegraphics[scale=1]{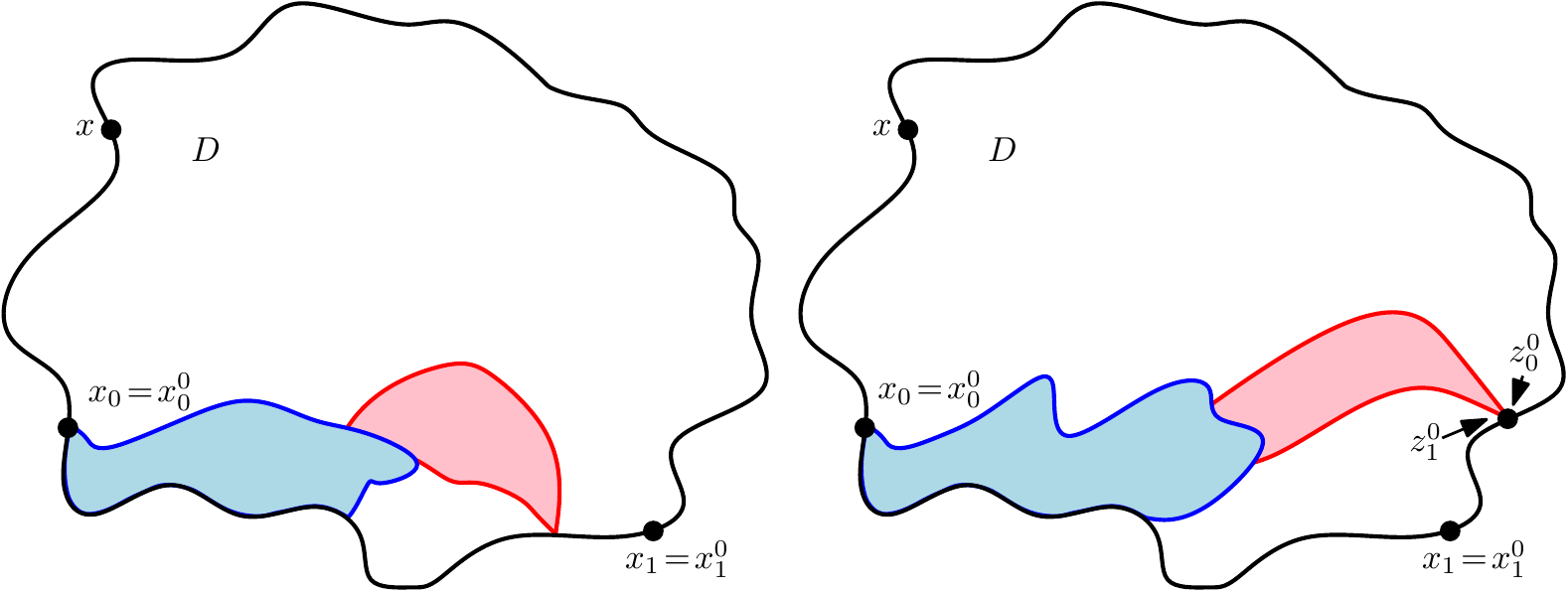}
\end{center}
\caption{\label{fig:downward_jump_failure} Illustration of the exploration in $D$ from $x_0 = x_0^0$ to $x_1 = x_1^0$ (blue) up until it fails due to making a large downward jump with the corresponding chunk shown in red.  {\bf Left:} In the case that $x_0^0$ and $x_1^0$ are not disconnected, a new exploration is started in the remaining domain with $x_0^0$, $x_1^0$ on its boundary from $x_0^0$ to $x_1^0$.  {\bf Right:} In the case that $x_0^0$ and $x_1^0$ are disconnected, we let $z_0^0$ and $z_1^0$ be the two prime ends in the two complementary domains corresponding to where the downward jump occurred.  Two new explorations are started, one from $x_0^0$ to $z_0^0$ and the other from $z_1^0$ to $x_1^0$.}
\end{figure}

\noindent{\it Case 1: large upward jump.} See Figure~\ref{fig:upward_jump_failure} for an illustration.  If the exploration fails by making a large upward jump, we let $\wt{\CD}_{0}$ be the quantum disk in the complement of the exploration with $x_0^0$, $x_1^0$ on its boundary and let $\wt{h}_0$ be the field which describes $\wt{\CD}_{0}$.  We let $z_0^0$, $z_1^0$ be the two points on $\partial \wt{\CD}_{00}$ which correspond to where the loop corresponding to the upward jump is rooted.  We then let $\CD_{00}$ be the surface described by $h_{00} = \wt{h}_0 + \tfrac{2}{\gamma} \log L_{00}^{-1}$ where $L_{00} = \qbmeasure{\wt{h}_0}(\ccwBoundary{x_0^0}{z_0^0}{\partial \wt{\CD}_0})$ and we write $x_0^{00}$, $x_1^{00}$ for the corresponding points.  We also let $\CD_{01}$ be the surface described by $h_{01} = \wt{h}_0 + \tfrac{2}{\gamma} \log L_{01}^{-1}$ where $L_{01} = \qbmeasure{\wt{h}_0}(\ccwBoundary{z_1^0}{x_1^0}{\partial \wt{\CD}_0})$ and we write $x_0^{01}$, $x_1^{01}$ for the corresponding points.  We refer to $(\CD_{0i},x_0^{0i},x_1^{0i})$ for $i=0,1$ as the \emph{children} of $(\CD_0,x_0^0,x_1^0)$. We then perform the exploration inside of $\CD_{0i}$ from $x_0^{0i}$ to $x_1^{0i}$ for $i=0,1$ starting with $i=0$ and then with $i=1$ with the following exception.  If the exploration with $i=0$ fails with a large downward jump which disconnects $x_0^{01} = z_1^0$ from $x_1^{01} = x_1^0$ then we do not perform the exploration from $x_0^{01} = z_1^0$ to $x_1^{01} = x_1^0$ and instead at the next stage perform the exploration from $x_0^{00}$ to $x_1^{01}$ (analogously to how the exploration proceeds in the case of a large downward jump which does not disconnect the initial and terminal points as described just below).

\noindent{\it Case 2: large downward jump.} See Figure~\ref{fig:downward_jump_failure} for an illustration.  If the exploration fails by making a large downward jump, there are two possibilities.  Either $x_0$ and $x_1$ are disconnected or they are not.  If $x_0$, $x_1$ are not disconnected, then we let $\wt{\CD}_0$ be the quantum disk in the complement of the exploration with $x_0^0$, $x_1^0$ on its boundary and let $\wt{h}_0$ be the field which describes $\wt{\CD}_{0}$.  We let $\CD_{00}$ be the quantum surface described by $h_{00} = \wt{h}_0 + \tfrac{2}{\gamma} \log L_{00}^{-1}$ where $L_{00} = \qbmeasure{\wt{h}_0}(\ccwBoundary{x_0^0}{x_1^0}{\partial \wt{\CD}_0})$ and we write $x_0^{00}$, $x_1^{00}$ for the corresponding points.  We refer to $(\CD_{00}, x_{0}^{00}, x_1^{00})$ as the \emph{child} of $(\CD_{0},x_0^0,x_1^0)$ and then continue the exploration in $\CD_{00}$ from $x_0^{00}$ to $x_1^{00}$.

Now suppose that $x_0$, $x_1$ are disconnected.  Let $\wt{\CD}_0$ (resp.\ $\wt{\CD}_1$) be the quantum disk with $x_0$ (resp.\ $x_1$) on its boundary.  Let $\wt{h}_0$ (resp.\ $\wt{h}_1$) be the field which describes $\wt{\CD}_0$ (resp.\ $\wt{\CD}_1$).  Let $z_0$ (resp.\ $z_1$) be the point on $\partial \wt{\CD}_0$ (resp.\ $\partial \wt{\CD}_1$) which corresponds to where the downward jump happened.  We then let $\CD_{00}$ be the quantum surface described by $h_{00} = \wt{h}_0 + \tfrac{2}{\gamma} \log L_{00}^{-1}$ where $L_{00} = \qbmeasure{\wt{h}_0}(\ccwBoundary{x_0}{z_0}{\partial \wt{\CD}_0})$ and call $x_0^{00}$, $x_1^{00}$ the corresponding points.  We also let $\CD_{01}$ be the quantum surface described by $h_{01} = \wt{h}_1 + \tfrac{2}{\gamma} \log L_{01}^{-1}$ where $L_{01} = \qbmeasure{\wt{h}_1}(\ccwBoundary{z_1}{x_1}{\partial \wt{\CD}_1})$ and write $x_0^{01}$, $x_1^{01}$ for the corresponding points.   We refer to $(\CD_{0i},x_0^{0i},x_1^{0i})$ for $i=0,1$ as the children of $(\CD_0,x_0^0,x_1^0)$. We then perform the exploration inside of $\CD_{0i}$ from $x_0^{0i}$ to $x_1^{0i}$ for $i=0,1$.

\noindent{\it Case 3: too many chunks or top boundary length too long.} Finally, we consider the situation in which the exploration fails by either having $n_M^\uparrow$, $n_M^\downarrow$ too large for $M \geq J$ or the quantum length of the top of the exploration exceeds $C 2^{(1-4/\kappa) J}$ for a constant $C > 0$ (we will adjust the value of $C$ later in the proof; note that $1-4/\kappa < 0$ since $\kappa \in (8/3,4)$).  We note that the quantum length of the top can only exceed $C 2^{(1-4/\kappa) J}$ by at most $2^{- \exploreExp (\kappa/4) J}$ for otherwise the exploration would have failed due to a large upward jump.  We let $\wt{\CD}_0$ be the quantum disk in the complement of the exploration with $x_0^0$, $x_1^0$ on its boundary.  Let $\wt{h}_0$ be the field which describes $\wt{\CD}_0$.  We let $\CD_{00}$ be the quantum surface described by $h_{00} = \wt{h}_0 + \tfrac{2}{\gamma} \log L_{00}^{-1}$ where $L_{00} = \qbmeasure{\wt{h}_0}(\ccwBoundary{x_0^0}{x_1^0}{\partial \wt{\CD}_0})$ and we write $x_0^{00}$, $x_1^{00}$ for the corresponding points.  We refer to $(\CD_{00}, x_{0}^{00}, x_1^{00})$ as the child of $(\CD_{0},x_0^0,x_1^0)$ and then continue the exploration in $\CD_{00}$ from $x_0^{00}$ to $x_1^{00}$.

\noindent{\it Completion of the proof.} The above defines a tree structure of explorations $(\CD_o, h_o, x_0^o, x_1^o)$ together with scaling factors $L_o$ for $o \in \{0,1\}^n$ and $n \in \N$ where each node can either have zero children (exploration does not fail), one child (large downward jump which does not separate the endpoints, too many chunks at a fixed stage, or the top boundary length is too large), or two children (large downward jump which separates the endpoints or a large upward jump).  Let $\CT$ be the tree associated with the exploration and for $o \in \CT$ we let $o'$ be the parent of $o$.  Note that $x_0^o, x_1^o$ correspond to points $\wt{x}_0^o$, $\wt{x}_1^o$ on the boundary of the exploration from $x_0^{o'}$ to $x_1^{o'}$ in $\CD_{o'}$.  Moreover, $L_o$ is the quantum length of the counterclockwise arc from $\wt{x}_0^{o'}$ to $\wt{x}_1^{o'}$ along the boundary of the exploration from $x_0^o$ to $x_0^{o'}$.  In particular, $L_0 = \delta$.  On the event $E$ that $\qmeasure{h}(B(z,\epsilon)) \geq c_0 \epsilon^{\alpha_\LBD}$ for all $\epsilon \in (0,\epsilon_0)$ and $z \in D$ so that $B(z,\epsilon) \subseteq D$ we have that $h_0$ satisfies $\qmeasure{h_0}(B(z,\epsilon)) \geq c_0 \delta^{-2} \epsilon^{\alpha_\LBD}$ for all $\epsilon \in (0,\epsilon_0)$ and $z \in D$ so that $B(z,\epsilon) \subseteq D$.  The reason for this is that adding $\tfrac{2}{\gamma} \log \delta^{-1}$ to the field has the effect of multiplying quantum areas by the factor $\delta^{-2}$.

Lemma~\ref{lem:point_to_point_disk_exploration} implies that by choosing $J \in \N$ and $C > 0$ large enough we can ensure that the mean number of children of each node in $\CT$ is strictly smaller than $1$ so that the corresponding Galton-Watson tree is subcritical.  Since the offspring distribution is supported on $\{0,1,2\}$, the size of the Galton-Watson tree also has an exponential tail.

We are now going to use the definition of the $E_{i,M}^\uparrow$, $E_{i,M}^\downarrow$ from Definition~\ref{def:good_chunk} to construct a path $\cpath$ from $x_0$ to $x_1$ contained in $\Upsilon$ and bound the tail of $\lebneb{\epsilon}(\cpath)$.

Let us first consider the exploration from $x_0^0$ to $x_0^1$ in $\CD_0$ in the case that it does not fail.  We say that two chunks $\CN_{i,M}^\uparrow$, $\CN_{\ell,L}^\uparrow$ for which $E_{i,M}^\uparrow$ and $E_{\ell,L}^\uparrow$ occur and $L \geq M$ (also $\ell \geq i$ if $L=M$) are adjacent if $\eta_{\ell,L}^\uparrow(0)$ is on the top of $\CN_{i,M}^\uparrow$ (and the same with $\downarrow$ in place of $\uparrow$).  Let $\CN_n$ for $n \in \Z$ be the bi-infinite sequence of adjacent chunks.  Then parts~\eqref{it:path_across_chunk} and~\eqref{it:path_across_two_chunks} of Definition~\ref{def:good_chunk} give us a way to construct a path connecting $x_0^0$ to $x_1^0$ as follows.  For each $n \in \Z$, we let $\cpath_n$ be the path in~$\CN_n$ corresponding to part~\eqref{it:path_across_chunk} of Definition~\ref{def:good_chunk} and we let~$\cpath_{n-1,n}$ be the path in $\interior{\closure{\CN_{n-1} \cup \CN_n}}$ corresponding to part~\eqref{it:path_across_two_chunks} of Definition~\ref{def:good_chunk}.  Then successively concatenating the paths $\cpath_{n-1,n}$, $\cpath_n$ leads to a path $\cpath$ in $\Upsilon$ connecting $x_0^0$ to $x_1^0$ and we have that
\[ \lebneb{\epsilon}(\cpath) \leq \sum_{n \in \Z} (\lebneb{\epsilon}(\cpath_n) + \lebneb{\epsilon}(\cpath_{n-1,n})).\]
Let $\varphi_n \colon \CN_n \to \D$ be the map as defined in part~\eqref{it:path_across_chunk} of Definition~\ref{def:good_chunk}.  Then we know that $\lebneb{\epsilon}(\varphi_n(\cpath_n)) \leq R_1 \medianHP{\epsilon}$ and $\varphi_n(\cpath_n)$ has distance at least $R_1^{-1}$ from $\partial \D$.  Let 
\[ \ol{d}_n = \sup_{z \in B(0,1-R_1^{-1})} |(\varphi_n^{-1})'(z)| \quad\text{and}\quad \ul{d}_n = \inf_{z \in B(0,1-R_1^{-1})} |(\varphi_n^{-1})'(z)|.\]
For each $\zeta > 0$, let $A_{n,\zeta}$ (resp.\ $B_{n,\zeta}$) be the $\zeta$-neighborhood of $\varphi_n(\cpath_n)$ (resp.\ $\cpath_n$).  Then as $B_{n,\ul{d}_n \epsilon} \subseteq \varphi_n(A_{n,\epsilon})$ we have that
\begin{align}
\label{eqn:n_delta_eps_ubd}
\lebneb{\ul{d}_n \epsilon}(\cpath_n) \leq \int_{A_{n,\epsilon}} |(\varphi_n^{-1})'(z)|^2 dz \leq R_1 \ol{d}_n^2 \medianHP{\epsilon}.
\end{align}
By Lemma~\ref{lem:covering_lemma} we have for a constant $c_0 > 0$ that
\begin{equation}
\label{eqn:n_delta_eps_lbd}
\lebneb{\ul{d}_n \epsilon}(\cpath_n) \geq c_0 \ul{d}_n \lebneb{\epsilon}(\cpath_n).
\end{equation}
Distortion estimates for conformal maps imply that $\ol{d}_n / \ul{d}_n$ is bounded from above by a constant which depends only on $R_1$.  Consequently, by combining~\eqref{eqn:n_delta_eps_ubd} with~\eqref{eqn:n_delta_eps_lbd} we see for a constant $c_1 > 0$ that
\begin{equation}
\label{eqn:n_eps_ubd}
\lebneb{\epsilon}(\cpath_n) \leq c_1 \ol{d}_n \medianHP{\epsilon}.	
\end{equation}
If $\CN_n = \CN_{i,M}^\uparrow$ or $\CN_n = \CN_{i,M}^\downarrow$, then we have that $\qmeasure{h_0}(\CN_n) \leq R_1 2^{-M}$.  Consequently, $\qmeasure{h}(\CN_n) \leq R_1 \delta^2 2^{-M}$.

It therefore follows that on $E$ we have for a constant $c_2 > 0$ that
\begin{equation}
\label{eqn:delta_n_ubd}
\ul{d}_n \leq c_2 \delta^{2/\alpha_\LBD} 2^{-M / \alpha_\LBD}.
\end{equation}
Combining~\eqref{eqn:n_eps_ubd} with~\eqref{eqn:delta_n_ubd} implies for a constant $c_3 > 0$ that
\begin{equation}
\label{eqn:gamma_n_bound}
\lebneb{\epsilon}(\cpath_n) \leq c_3 \delta^{2/\alpha_\LBD} 2^{- M / \alpha_\LBD} \medianHP{\epsilon}.
\end{equation}
By a similar argument, we have (possibly increasing $c_3 > 0$) that
\begin{equation}
\label{eqn:gamma_n_n_bound}
\lebneb{\epsilon}(\cpath_{n-1,n}) \leq c_3 \delta^{2/\alpha_\LBD} 2^{- M / \alpha_\LBD} \medianHP{\epsilon}.
\end{equation}
Recall from the definition of the exploration succeeding that if it does succeed then the number of chunks with quantum natural time in $[\delta_0 2^{-M}, 2^{-M}]$ with $M \geq J+1$ is at most a constant times $2^{(1-\exploreExp)(\kappa/4) M}$ and with $M = J$ is at most a constant times  $2^{(\kappa/4) J}$.  We assume that we have chosen $\exploreExp \in (0,1)$ sufficiently close to $1$ so that $\alpha = (1-\exploreExp)(\kappa/4) - 1 / \alpha_\LBD < 0$.  If we sum \eqref{eqn:gamma_n_bound}, \eqref{eqn:gamma_n_n_bound} over $n$, we thus see that those terms coming from chunks with natural time in $[\delta_0 2^{-M}, 2^{-M}]$ for $M \geq J+1$ contribute a constant times $2^{-\alpha J} \delta^{2/\alpha_\LBD} \medianHP{\epsilon}$.  With $\alpha = (\kappa/4)-\alpha_\LBD^{-1}$, those terms with $M =J$ contribute a constant times $2^{\alpha J} \delta^{2/\alpha_\LBD} \medianHP{\epsilon}$.  Altogether, we have for a constant $c_4 > 0$ that
\[ \lebneb{\epsilon}(\cpath) \leq c_4 2^{\alpha J} \delta^{2/\alpha_\LBD} \medianHP{\epsilon}.\]

Let us now describe how we recursively construct the path $\cpath$ in the case that the exploration fails at some stage.  Consider the tree $\CT$ associated with iterating the exploration as defined above.  We note that $\CT$ has a natural planar structure which comes from the ordering of the exploration.  We let $\wt{\CT}$ consist of the nodes in $\CT$ which either have zero or two children and we view $\wt{\CT}$ as a planar tree with the tree structure and ordering coming from that of $\CT$.  Let $o_1,\ldots,o_n$ be the leaves of $\wt{\CT}$ (given according to their contour order).  For each $1 \leq j \leq n$, we let $\cpath_j$ be the path which is defined in the same manner as $\cpath$ just as above except using the exploration associated with the node $o_j$ (which we emphasize does not fail as $o_j$ is a leaf).  We let $\cpath$ be the concatenation of the~$\cpath_j$.  Then
\[ \lebneb{\epsilon}(\cpath_j) \leq c_4 2^{\alpha J} \left( \prod_{o} L_{o}^{2/\alpha_\LBD} \right) \medianHP{\epsilon}\]
where the product is over those $o \in \CT$ which are ancestors of $o_j$.  The reason for this is that if we add $\tfrac{2}{\gamma} \log L_o^{-1}$ to the field to multiply boundary lengths by $L_o^{-1}$ it has the effect of multiplying areas by $L_o^{-2}$.  Since $L_o \leq C$, $L_0 = \delta$, and the number of ancestors of $o_j$ is trivially at most $|\CT|$, the above is in turn bounded from above by
\[ c_4 2^{\alpha J} C^{2|\CT|/\alpha_\LBD} \delta^{2/\alpha_\LBD} \medianHP{\epsilon}.\]
Therefore
\[ \lebneb{\epsilon}(\cpath) \leq c_4  2^{\alpha J} |\CT| C^{2 |\CT|/\alpha_\LBD} \delta^{2/\alpha_\LBD} \medianHP{\epsilon}.\]
We note that making $J$ large has the effect of improving the tail of $|\CT|$.  Therefore for every value of $p > 0$ we can choose $J$ sufficiently large so that for a constant $c_5 > 0$ we have
\[ \E[ ( (\medianHP{\epsilon})^{-1}\lebneb{\epsilon}(\cpath))^p \one_E] \leq c_5 \delta^{2 p/\alpha_\LBD}.\]
Fix $\beta \in (0,2/\alpha_\LBD)$.  As $p > 0$ was arbitrary, the Kolmogorov-Centsov theorem thus gives that the law of the $\beta$-H\"older norm of $(\medianHP{\epsilon})^{-1} \metapprox{\epsilon}{\cdot}{\cdot}{\Gamma}$ (with respect to the quantum length metric on $\partial D$) satisfies~\eqref{eqn:holder_norm_tight}.

We note that~\eqref{eqn:quantum_holder_norm_tight} follows from the proof of~\eqref{eqn:holder_norm_tight} except instead of using the paths $\omega_n$, $\omega_{n-1,n}$ we instead use the images of straight lines under $\varphi_n^{-1}$.  The same argument used to prove~\eqref{eqn:gamma_n_bound}, \eqref{eqn:gamma_n_n_bound} implies that the diameter of such a path is at most a constant times $\delta^{2/\alpha_\LBD} 2^{-M/\alpha_\LBD}$.  Therefore we can string together these paths in place of the $\omega_n$, $\omega_{n-1,n}$ to complete the proof of~\eqref{eqn:quantum_holder_norm_tight}.
\end{proof}

\section{Tightness in the interior}
\label{sec:interior_tightness}

The aim of this section is to extend the tightness result from Section~\ref{sec:boundary_tightness} to points in the $\CLE_\kappa$ carpet which are contained in the interior of the domain.  As in Section~\ref{sec:boundary_tightness}, we will state and prove the result in a more general setting than described in the statement of Theorem~\ref{thm:cle_loop}.

\newcommand{\approxball}[3]{{\mathfrak B}_{#1}(#2,#3)}

\begin{proposition}
\label{prop:interior_tightness}
For each $\alpha_\LBD > 2 > \alpha_\UBD > 0$ there exists $\alpha_\KC > 0$ so that the following is true.  Suppose that $\CD = (D,h,x,y) \sim \qdiskL{\gamma}{1}$.  Fix $\epsilon_0 > 0$ and let $E$ be the event that for every $z \in D$ and $\epsilon \in (0,\epsilon_0)$ so that $B(z,\epsilon) \subseteq D$ we have that $\epsilon^{\alpha_\LBD} \leq \qmeasure{h}(B(z,\epsilon)) \leq \epsilon^{\alpha_\UBD}$.  Let $\Gamma$ be an independent $\CLE_\kappa$ in $D$ and let $\Upsilon$ be its carpet.  Let $X_{\epsilon}$ be the $\alpha_\KC$-H\"older norm of $(z,w) \mapsto (\medianHP{\epsilon})^{-1} \metapprox{\epsilon}{z}{w}{\Gamma}$ with respect to the Euclidean metric on $D$.  For every $p > 0$ we have that
\[ \p[E,\ X_{\epsilon} \geq A] = O(A^{-p}) \quad\text{as}\quad A \to \infty\]
where the implicit constants depend only on $p,\epsilon_0,\alpha_\LBD,\alpha_\UBD$.
\end{proposition}

The proof of Proposition~\ref{prop:interior_tightness} will involve several steps which we will now describe.
\begin{enumerate}
\item[Step 1.] We fix $\alpha_\net > 0$ and for each $j \in \N$ we let $N_j = 2^{\alpha_\net j}$ then pick $(z_n)$ i.i.d.\ from~$\qcarpet{h}{\Upsilon}$.  By Lemma~\ref{lem:quantum_measure_points_dense}, the points $z_1,\ldots,z_{N_j}$ are very likely to form a $2^{-j}$-net of $\Upsilon$.  We then aim to show in the remaining steps that if we have $1 \leq k,\ell \leq N_j$ such that $|z_\ell - z_k| \leq 2^{-j}$ then $(\medianHP{\epsilon})^{-1} \metapprox{\epsilon}{z_\ell}{z_k}{\Gamma}$ is very likely to be at most $2^{-\alpha_\KC j}$ for some constant $\alpha_\KC > 0$.
\item[Step 2.] We will argue in Lemma~\ref{lem:disk_in_disk} that there is a constant $\alpha_\PP > 0$ so that there does not exist $z \in \Upsilon$ and $\CL \in \Gamma$ so that $\CL \setminus B(z,\delta)$ has more than one component of diameter at least $\delta^{1/\alpha_\PP}$.  (We will also need a version where we rule out there being more than one component of $\CL \setminus B(z,\delta)$ with quantum length at least $\delta^{1/\alpha_\PP}$.)
\item[Step 3.]  We then use the percolation exploration from Section~\ref{sec:percolation_exploration} to show that for every $z \in \Upsilon$ with $\dist(z, \partial D) < \delta$ there exists a path $\eta$ in $\Upsilon$ so that the component $U$ of $D \setminus \eta$ containing~$z$ has diameter at most a power of $\delta$ and every pair of points $u,v \in \eta \cap \partial U$ satisfy that $(\medianHP{\epsilon})^{-1} \metapprox{\epsilon}{u}{v}{\Gamma}$ is also at most a power of $\delta$ (Lemma~\ref{lem:bounds_holder}).  We now aim to combine this with Steps 1 and 2 to bound $(\medianHP{\epsilon})^{-1} \metapprox{\epsilon}{z_\ell}{z_k}{\Gamma}$ as described at the end of Step 1.
\item[Step 4.] For each $z \in \Upsilon$ and $r > 0$ we let $\approxball{\epsilon}{z}{r}$ be the set of $w \in \Upsilon$ such that $(\medianHP{\epsilon})^{-1} \metapprox{\epsilon}{z}{w}{\Gamma} \leq r$.  We next establish a \emph{lower bound} on the Euclidean diameter of $\approxball{\epsilon}{z_\ell}{r}$ for each $1 \leq \ell \leq N_j$ (Lemma~\ref{lem:quantum_typical_point_diamter_lbd}).  The bound will be so that it implies for a constant $\alpha_\ball > 0$ that if $|z_\ell - z_k| \leq 2^{-j}$ then it very likely that $\approxball{\epsilon}{z_\ell}{2^{-\alpha_\ball j}}$, $\approxball{\epsilon}{z_k}{2^{-\alpha_\ball j}}$ have diameter much larger than $2^{-j}$ but it will not be clear at this point that $\approxball{\epsilon}{z_\ell}{2^{-\alpha_\ball j}}$, $\approxball{\epsilon}{z_k}{2^{-\alpha_\ball j}}$ intersect.
\item[Step 5.] We now fix a value of $1 \leq \ell \leq N_j$ and consider an $\SLE_\kappa^0(\kappa-6)$ exploration from a point on $\partial D$ targeted at $z_\ell$ and coupled with $\Gamma$ as a CPI.  We assume we are working on the event that points within distance $2^{-j}$ of each other are not separated from each other by a loop of $\Gamma$ of size at least $2^{-j/\alpha_\PP}$.  It is a consequence of the results in Appendix~\ref{app:mod_of_cont} that if we parameterize $\eta$ according to the amount of quantum area that its trunk $\eta'$ has disconnected from $z_\ell$ then $\eta'$ is H\"older continuous with some exponent $\alpha_\HO$.  For each $m \in \N$ we let $t_m = m 2^{-j/(\alpha_\HO \alpha_\PP)}$.  By Step 3, we have for each $m \in \N$ that the condition of Step 3 holds for the component of $D \setminus \eta([0,t_m])$ which contains $z_\ell$ with very high probability.  Let $\tau$ be the first time $t$ that $\eta'$ gets within distance $2^{-j/\alpha_\PP}$ of $z_\ell$ and let $m_0$ be such that $\tau \in (t_{m_0},t_{m_0+1}]$.  We fix another value of $1 \leq k \leq N_j$.  On the event that $|z_\ell - z_k| \leq 2^{-j}$ we have that $\eta|_{[0,\tau]}$ hence also $\eta|_{[0,t_{m_0}]}$ does not disconnect $z_\ell$, $z_k$ (as this would only be possible if there were a loop which violated Step 2).  Step 3 implies that we can find paths $\omega_\ell, \omega_k$ in the component of $D \setminus \eta([0,t_{m_0}])$ which contains $z_\ell$ which satisfy the properties mentioned above for $z_\ell$, $z_k$.  The sets $\approxball{\epsilon}{z_k}{2^{-\alpha_\ball j}}$, $\approxball{\epsilon}{z_\ell}{2^{-\alpha_\ball j}}$ from Step 4 then have to intersect $\omega_\ell$, $\omega_k$ which gives an overall bound on $\metapprox{\epsilon}{z_\ell}{\eta([0,t_{m_0}])}{\Upsilon}$ and $\metapprox{\epsilon}{z_k}{\eta([0,t_{m_0}])}{\Upsilon}$.  Proposition~\ref{prop:boundary_distance_tail_bounds} then gives an upper bound on $\metapprox{\epsilon}{u_\ell}{u_k}{\Upsilon}$ where $u_\ell$, $u_k$ are where $\omega_\ell$, $\omega_k$, respectively, hit $\eta([0,t_{m_0}])$.  Altogether, this gives an upper bound to $\metapprox{\epsilon}{z_\ell}{z_k}{\Upsilon}$ which is a power of $2^{-j}$.  All of these estimates will be combined carefully in Lemma~\ref{lem:distance_two_quantum_typical}.
\item[Step 6.] We complete the proof by applying an argument analogous to that used to prove the Kolmogorov-Centsov continuity criterion. 
\end{enumerate}

\subsection{$\CLE$ loop pinch point bound}
\label{subsec:cle_loop_pinch_point}

Suppose that $D \subseteq \C$ is a bounded simply connected domain, $(D,h,x,y) \sim \qdiskL{\gamma}{1}$, and $\Gamma$ is an independent $\CLE_\kappa$ in $D$.  Fix $\epsilon, \alpha_\PP > 0$.  We say that $\CL \in \Gamma$ has an $(\epsilon,\alpha_\PP)$-pinch point if one of the following hold.
\begin{enumerate}[(i)]
\item \label{it:pp1} The quantum length of $\CL$ exceeds $2\epsilon$ and there exists $z \in \CL$ so that if $z_1$ (resp.\ $z_2$) is the point in $\CL$ so that $\qbmeasure{h}(\ccwBoundary{z}{z_1}{\CL})= \epsilon$ (resp.\ $\qbmeasure{h}(\cwBoundary{z}{z_2}{\CL}) = \epsilon$) then $\ccwBoundary{z_1}{z_2}{\CL} \cap B(z,\epsilon^{\alpha_\PP}) \neq \emptyset$ or
\item \label{it:pp2} The diameter of $\CL$ exceeds $2\epsilon$ and there exists $z \in \CL$ such that if $z_1$ (resp.\ $z_2$) is the first counterclockwise (resp.\ clockwise) point in $\CL$ starting from $z$ so that $\diam(\ccwBoundary{z}{z_1}{\CL}) = \epsilon$ (resp.\ $\diam(\cwBoundary{z}{z_2}{\CL}) = \epsilon)$ then $\ccwBoundary{z_1}{z_2}{\CL} \cap B(z,\epsilon^{\alpha_\PP}) \neq \emptyset$.
\end{enumerate}

\begin{lemma}
\label{lem:disk_in_disk}
Suppose that $D \subseteq \C$ is a bounded simply connected domain and $(D,h,x,y) \sim \qdiskL{\gamma}{1}$.  Suppose that $\alpha_\UBD, \epsilon_0 > 0$.  There exists $\alpha_\PP, \beta > 0$ depending only on $\alpha_\UBD$ so that the following is true.  Let $E_1$ be the event that for every $\epsilon \in (0,\epsilon_0)$ and $z \in D$ so that $B(z,\epsilon) \subseteq D$ we have that $\qmeasure{h}(B(z,\epsilon)) \leq \epsilon^{\alpha_\UBD}$.  Let $E_2$ be the event that there is $\CL \in \Gamma$ with an $(\epsilon,\alpha_\PP)$-pinch point.  Then $\p[E_1 \cap E_2] = O(\epsilon^\beta)$ for all $\epsilon \in (0,\epsilon_0)$.
\end{lemma}

In what follows, for $\ell > 0$ we let $\qdiskWeighted{\gamma}{\ell}$ be the law on quantum surfaces $(D,h)$ whose Radon-Nikodym derivative with respect to $\qdiskL{\gamma}{\ell}$ is given by a normalizing constant times $\qmeasure{h}(D)$.  As $\E[ \qmeasure{h}(D)] < \infty$ under $\qdiskL{\gamma}{\ell}$, we have that $\qdiskWeighted{\gamma}{\ell}$ indeed makes sense as a probability measure.  We note that it suffices to prove Lemma~\ref{lem:disk_in_disk} under $\qdiskWeighted{\gamma}{1}$ in place of $\qdiskL{\gamma}{1}$ since $\qmeasure{h}(D)$ has finite negative moments of all orders \cite[Theorem~1.2]{ag2019disk}.  Given that $(D,h)$  has law $\qdiskWeighted{\gamma}{\ell}$, it is natural to consider $(D,h)$ marked by a point $z$ chosen from $\qmeasure{h}$ and we will also write $(D,h,z) \sim \qdiskWeighted{\gamma}{\ell}$ in this case.

\begin{lemma}
\label{lem:quantum_disk_length_bound}
There exists $\alpha, \beta, c_0 > 0$ so that the following is true.  Suppose that $(\D,h,0) \sim \qdiskWeighted{\gamma}{1}$.  Then
\[ \p[ \exists \theta \in [0,2\pi) : \qbmeasure{h}(\ccwBoundary{e^{i \theta}}{e^{i (\theta+\epsilon)}}{\partial \D}) \geq \epsilon^\beta] \leq c_0 \epsilon^\alpha \quad\text{for all}\quad \epsilon > 0.\]
\end{lemma}
\begin{proof}
In the case that $(\D,h,-i,i) \sim \qdiskL{\gamma}{1}$, this follows from the same argument used to prove Lemma~\ref{lem:quantum_disk_multifractal} except here we use the moment upper bound for the boundary length measure for a free boundary GFF.  We will thus omit the details in this case.  The result in the case of $\qdiskWeighted{\gamma}{1}$ follows since \cite[Theorem~1.2]{ag2019disk} implies that $\qmeasure{h}(D)$ has a finite $p$th moment for each $p < 4/\gamma^2$ \cite[Theorem~1.2]{ag2019disk} so we can get the result from the case of $\qdiskL{\gamma}{1}$ and apply H\"older's inequality.
\end{proof}

\begin{proof}[Proof of Lemma~\ref{lem:disk_in_disk}]
Let $\alpha$, $\beta$ be as in Lemma~\ref{lem:quantum_disk_length_bound} and let $\epsilon, \delta > 0$.  We assume that $(D,h,z) \sim \qdiskWeighted{\gamma}{1}$.  As remarked after the lemma statement, it suffices to give the proof in this setting.  Note that $z$ is a.s.\ surrounded by a loop $\CL_z$ of $\Gamma$.  Let $D_z$ be the component of $\C \setminus \CL_z$ which contains $z$.  Let $\varphi_z \colon D_z \to \D$ be the unique conformal map with $\varphi_z(z) = 0$, $\varphi_z'(z) > 0$, and let $h_z = h \circ \varphi_z^{-1} + Q \log|(\varphi_z^{-1})'|$.  Let $F$ be the event that $\CL_z$ has quantum length at least $\epsilon$ and there exists $\theta \in [0,2\pi)$ so that $\qbmeasure{h_z}(\ccwBoundary{e^{i \theta}}{e^{i(\theta+\delta)}}{\partial \D}) \geq \delta^\beta$.  Letting $\ell > 0$ be the quantum length of $\CL_z$, we have that the conditional law of $(\D,h_z,0)$ given $\ell$ is equal to $\qdiskWeighted{\gamma}{\ell}$.  Consequently, Lemma~\ref{lem:quantum_disk_length_bound} implies that $\p[F] = O((\delta/\epsilon)^\alpha)$.

Fix $\xi > 1$ and $\sigma > 0$.  Let $G$ be the event that
\begin{enumerate}[(i)]
\item $E_1$ occurs
\item $\qmeasure{h}(D) \leq \epsilon^{-\sigma}$ and
\item there is $\CL \in \Gamma$ with quantum length at least $2\epsilon$ which satisfies~\eqref{it:pp1} in the definition of an $(\epsilon,\alpha_\PP)$-pinch point and the domain surrounded by $\CL$ has quantum area at least $\epsilon^{2\xi}$.
\end{enumerate}
Let us now make a few observations.  First, it follows from Lemma~\ref{lem:number_of_loops} that for each $a > 0$ there exists $b > 0$ so that the probability that the number of loops of $\Gamma$ in $D$ with quantum length at least $\epsilon$ is $O(\epsilon^{-4/\kappa-1/2-a})$ is $1-O(\epsilon^b)$.  By \cite[Theorem~1.2]{ag2019disk}, we have that the total quantum area associated with a sample from $\qdiskWeighted{\gamma}{\ell}$ for $\ell \geq \epsilon$ is at least $\epsilon^{2\xi}$ off an event whose probability decays to $0$ as $\epsilon \to 0$ faster than any power of $\epsilon$.  Combining these two facts implies that off an event whose probability decays to $0$ faster than any power of $\epsilon$ as $\epsilon \to 0$, the quantum area of the region surrounded by each loop of $\Gamma$ with quantum length at least $\epsilon$ is at least $\epsilon^{2\xi}$.  Moreover, by \cite[Theorem~1.2]{ag2019disk} we have that $\p[ \qmeasure{h}(D) \geq \epsilon^{-\sigma}] = O(\epsilon^{(4/\gamma^2-1)\sigma})$.  It thus follows that $\p[E_1 \cap E_2] \leq \p[G] + O(\epsilon^{(4/\gamma^2-1)\sigma})$.

We will argue that there exists a constant $c_0 > 0$ so that
\begin{align}
\label{eqn:p_f_given_g_lbd}
\p[F \giv G] \geq c_0 \epsilon^{2\xi+\sigma}
\end{align}
which, in turn, implies that
\begin{align}
\label{eqn:p_g_bound}
\p[G] \leq \frac{\p[F]}{\p[F \giv G]} = \frac{O( (\delta/\epsilon)^{\alpha})}{c_0 \epsilon^{2\xi+\sigma}} = O( \delta^\alpha \epsilon^{-\alpha-2\xi-\sigma}).
\end{align}
Combining, this will prove the result provided we take $\delta$ to be a sufficiently large power of $\epsilon$.

Suppose that we are working on $G$, that $\CL$ is a loop of $\Gamma$ with quantum length at least $2 \epsilon$ and suppose that $\CL$ satisfies~\eqref{it:pp1} in the definition of an $(\epsilon,\alpha_\PP)$-pinch point.  We assume further that the domain $U$ surrounded by $\CL$ satisfies $\qmeasure{h}(U) \geq \epsilon^{2\xi}$ and that among all such loops $\CL$ is leftmost (smallest real part) and among all such leftmost loops is bottommost (smallest imaginary part).  We let $w$ be the $(\epsilon,\alpha_\PP)$-pinch point of $\CL$ which is leftmost and among all leftmost $(\epsilon,\alpha_\PP)$-pinch points of $\CL$ is bottommost.  We note that if $d \in (0,\epsilon_0)$ and $\diam(U) \leq d$ then on $E_1$ we have that $\qmeasure{h}(U) \leq d^{\alpha_\UBD}$.  That is, we have that $\epsilon^{2\xi} \leq d^{\alpha_\UBD}$.  Therefore $\diam(U) \geq \min(\epsilon^{2\xi/\alpha_\UBD},\epsilon_0)$.  Since $\qmeasure{h}(D) \leq \epsilon^{-\sigma}$ and $\qmeasure{h}(U) \geq \epsilon^{2\xi}$, it follows that $\p[ z \in U \giv G] \geq \epsilon^{2\xi+\sigma}$.  Moreover, $\p[ z \in B(w, \delta) \giv G,\ z \in U] =O(\delta^{\alpha_\UBD} \epsilon^{-2\xi})$.  Suppose that we are working on $G$ and $z \in U \setminus B(w,\delta)$.  Let $\varphi \colon U \to \D$ be the unique conformal transformation with $\varphi(z) = 0$ and $\varphi'(z) > 0$.  Then there exists a component $L$ of $\CL \setminus B(z,\epsilon^{\alpha_\PP})$ which has quantum length at least $\epsilon$ and is not part of the boundary of the component of $U \setminus B(w,\delta)$ which contains~$z$.  The Beurling estimate implies that $\varphi(L)$ is mapped to an interval of $\partial \D$ of (Euclidean) length $O((\epsilon^{\alpha_\PP}/\delta)^{1/2})$.  Altogether, this implies that 
\begin{align*}
\p[F \giv G] 
&\geq \p[ z \in U \setminus B(w,\delta) \giv G] \geq \p[ z \in U \giv G] (1-\p[ z \in B(w, \delta) \giv G,\ z \in U])\\
&\geq \epsilon^{2\xi + \sigma}(1-O(\delta^{\alpha_\UBD} \epsilon^{-2\xi})).
\end{align*}
We assume that $\delta$ is a sufficiently large power of $\epsilon$ so that $O(\delta^{\alpha_\UBD} \epsilon^{-2\xi}) = o(1)$ as $\epsilon \to 0$.  This implies~\eqref{eqn:p_f_given_g_lbd}.  Taking $\delta$ to be an even larger power of $\epsilon$ so that the power of $\epsilon$ in the right side of~\eqref{eqn:p_g_bound} is positive and then $\alpha_\PP > 0$ large enough so that $(\epsilon^{\alpha_\PP}/\delta)^{1/2} = O(\delta^\beta)$ completes the proof in the case of $(\epsilon,\alpha_\PP)$ pinch points satisfying~\eqref{it:pp1}.  The proof is analogous for $(\epsilon,\alpha_\PP)$ pinch points satisfying~\eqref{it:pp2}.
\end{proof}

\subsection{Proof of tightness in the interior}
\label{subsec:proof_of_tightness}

The first step to completing the proof of Proposition~\ref{prop:interior_tightness} is to give the definition of a good chunk for the adaptive $\SLE_\kappa^0(\kappa-6)$ exploration of a quantum half-plane $\CH = (\h,h,0,\infty)$ (see Figure~\ref{fig:good_interior_def_illustration} for an illustration).  The definition that we will give below will be different than in Section~\ref{sec:boundary_tightness} because the percolation exploration will serve a different purpose.  Suppose that $\eta$ is an $\SLE_\kappa^0(\kappa-6)$ from $0$ to $\infty$ sampled independently of $h$ and then subsequently parameterized according to the quantum natural time of its trunk.  As in the beginning of Section~\ref{sec:boundary_tightness}, for the quantum surface $\CN_t$ disconnected from $\infty$ by $\eta|_{[0,t]}$, we let $z_L(\CN_t)$ (resp.\ $z_R(\CN_t)$) denote the leftmost (resp.\ rightmost) point on the bottom of $\CN_t$.  We also let $z_C(\CN_t) = \eta(0)$.

\begin{figure}[ht!]
\begin{center}
\includegraphics[scale=1]{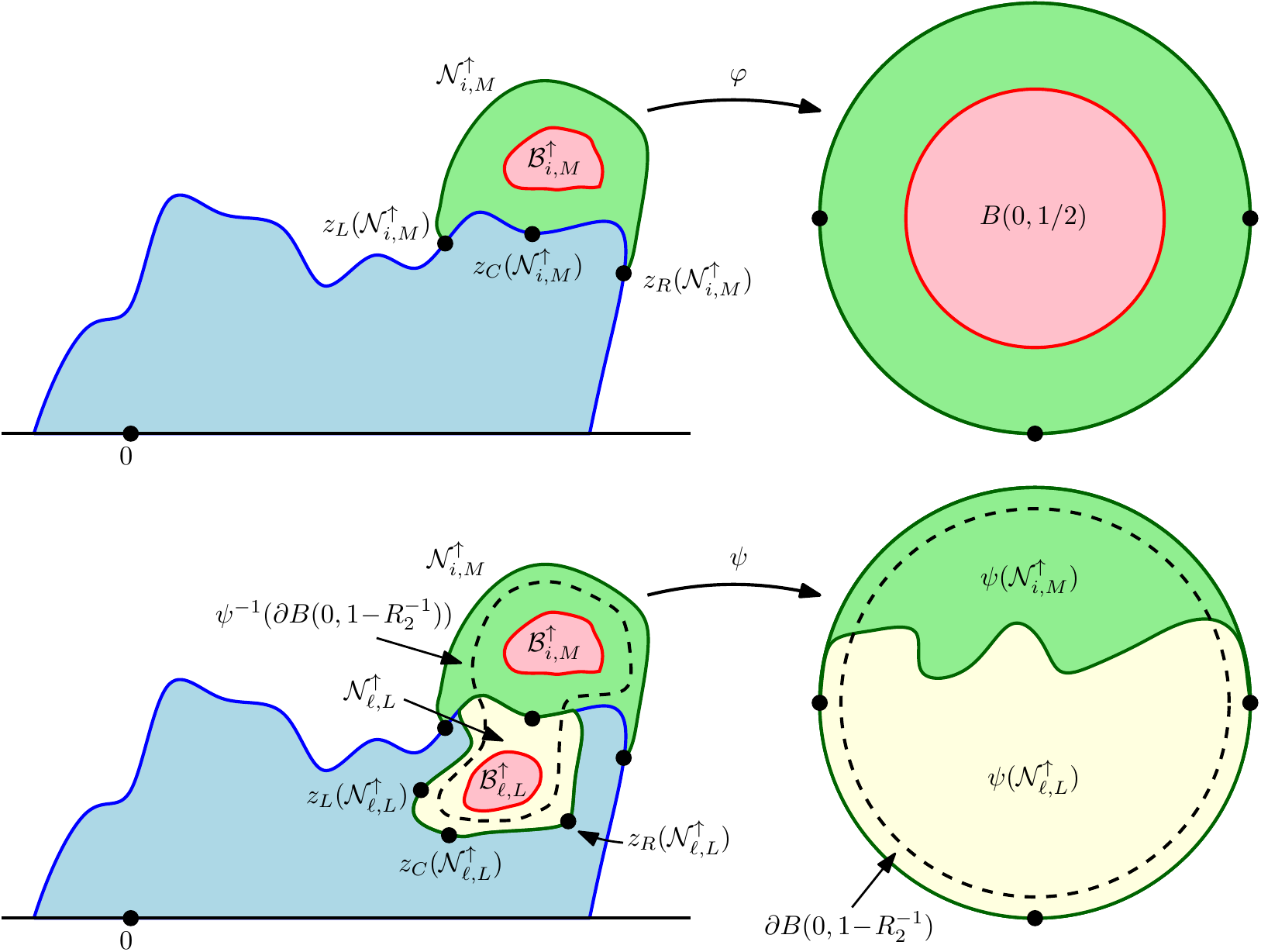}	
\end{center}
\caption{\label{fig:good_interior_def_illustration}  
{\bf Top:} Illustration of part~\eqref{it:int_chunk_mass} of Definition~\ref{def:interior_good_chunk}.  Shown in blue is the exploration up until $\CN_{i,M}^\uparrow$ is discovered and the images of $z_L(\CN_{i,M}^\uparrow)$, $z_C(\CN_{i,M}^\uparrow)$, and $z_R(\CN_{i,M}^\uparrow)$ under $\varphi$ are respectively given by $-1$, $-i$, and $1$. {\bf Bottom:} Illustration of part~\eqref{it:int_chunk_neighbormass} of Definition~\ref{def:interior_good_chunk}.  The chunk $\CN_{\ell,L}^\uparrow$ is shown in light yellow even though it is ``good'' in order to differentiate it from $\CN_{i,M}^\uparrow$.  The images of $z_L(\CN_{\ell,L}^\uparrow)$, $z_C(\CN_{\ell,L}^\uparrow)$, and $z_R(\CN_{\ell,L}^\uparrow)$ under $\psi$ are $-1$, $-i$, and $1$, respectively.}
\end{figure}

\begin{definition}
\label{def:interior_good_chunk}
We suppose that we have the setup described in Section~\ref{subsec:exploration_def}.  Fix $R_1, R_2, R_3 \geq 1$.  For each $i,M$, we define $E_{i,M}^\uparrow$ to be the event for $\CN_{i,M}^\uparrow$ and $\eta_{i,M}^\uparrow$ that $\eta_{i,M}^\uparrow(\sigma_{i,M}^\uparrow) \in \partial \h$ and the following hold.
\begin{enumerate}[(i)]
\item\label{it:int_chunk_qlength} The quantum lengths of the top, the left side of the bottom, and the right side of the bottom of $\CN_{i,M}^\uparrow$ are in $[R_1^{-1} 2^{-(\kappa/4) M}, R_1 2^{-(\kappa/4) M}]$.
\item\label{it:int_chunk_mass} Let $\varphi \colon \CN_{i,M}^\uparrow \to \D$ be the unique conformal transformation with $\varphi(z_L(\CN_{i,M}^\uparrow)) = -1$, $\varphi(z_C(\CN_{i,M}^\uparrow)) = -i$, and $\varphi(z_R(\CN_{i,M}^\uparrow)) = 1$.  Let $\CB_{i,M}^\uparrow = \varphi^{-1}(B(0,1/2))$.  Then we have that
 \[ R_2^{-1} 2^{-M} \leq \qmeasure{h}(\CB_{i,M}^\uparrow) \leq \qmeasure{h}(\CN_{i,M}^\uparrow) \leq R_2 2^{-M}.\]
\item\label{it:int_chunk_neighbormass} Let $\CN_{\ell,L}^\uparrow$ be the chunk whose top contains $\eta_{i,M}^\uparrow(0)$ (so that the left side of the bottom of $\CN_{i,M}^\uparrow$ is glued to the top of $\CN_{\ell,L}^\uparrow$) if it exists, otherwise we set $\CN_{\ell,L}^\uparrow = \emptyset$.  If $\CN_{\ell,L}^\uparrow \neq \emptyset$, let $\wt{\CN}$ be the quantum surface parameterized by $\interior{\closure{\CN_{\ell,L}^\uparrow \cup \CN_{i,M}^\uparrow}}$ and let $\psi \colon \wt{\CN} \to \D$ be the unique conformal map so that $\psi(z_L(\CN_{\ell,L}^\uparrow)) = -1$, $\psi(z_C(\CN_{\ell,L}^\uparrow)) = -i$, and $\psi(z_R(\CN_{\ell,L}^\uparrow))=1$.  Then
\begin{enumerate}[(I)]
 \item $R_2^{-1} 2^{-M} \leq \qmeasure{h}(\psi^{-1}(B(0,1-R_2^{-1}))) \leq \qmeasure{h}(\wt{\CN}) \leq 4 R_2 2^{-M}$ and
  \item $\psi^{-1}(B(0,1-R_2^{-1}))$ contains both $\CB_{i,M}^\uparrow$ and $\CB_{\ell,L}^\uparrow$.
\end{enumerate}
\item\label{it:int_cond_prob} If $\wh{\CN}$ is another $\SLE_\kappa^0(\kappa-6)$ chunk defined using the time-interval $[\delta_0 2^{-j}, 2^{-j}]$ for $j \in \{M-1,M,M+1\}$ which starts from a point on the top of $\CN_{i,M}^\uparrow$ with clockwise boundary length distance from $\partial \CH_{i,M}^\uparrow$ given by an integer multiple of $a_0 2^{-(\kappa/4) M}$ chosen in an $\CF_{i,M}^\uparrow$-measurable manner and with boundary length distance from a point in $\Upsilon$ at most $a_0 2^{-(\kappa/4) M}$, then the conditional probability given $\CF_{i,M}^\uparrow$ that the previous item holds for the pair $(\CN_{i,M}^\uparrow, \wh{N})$ in place of $(\CN_{\ell,L}^\uparrow, \CN_{i,M}^\uparrow)$ is at least $1-R_3^{-1}$.
\end{enumerate}
We define $E_{i,M}^\downarrow$ analogously to $E_{i,M}^\uparrow$.
\end{definition}

\begin{lemma}
\label{lem:good_chunk_occurs_interior}
Suppose that we have the setup described in Section~\ref{subsec:exploration_def} and the events $E_{i,M}^\uparrow$, $E_{i,M}^\downarrow$ are as in Definition~\ref{def:interior_good_chunk}.  For every $p_0 \in (0,1)$ there exists $R_1, R_2, R_3 \geq 1$, $\delta_0 > 0$ so that on the event that the exploration has not failed before exploring $\CN_{i,M}^\uparrow$ (resp.\ $\CN_{i,M}^\downarrow$), we have that $\p[E_{i,M}^\uparrow \giv \CF_{i-1,M}^\uparrow] \geq p_0$ (resp.\ $\p[E_{i,M}^\downarrow \giv \CF_{i-1,M}^\downarrow] \geq p_0$) for all $i,M \in \N$.
\end{lemma}
\begin{proof}
Lemmas~\ref{lem:bottom_length_moment_bound} and~\ref{lem:top_length_moment_bound} imply that we can take $\delta_0 \in (0,1)$ sufficiently small so that the probability that $\eta_{i,M}^\uparrow(\sigma_{i,M}^\uparrow) \in \partial \h$ is at least $1-q_0/3$ where $q_0 = 1-p_0$.  The scaling properties of $4/\kappa$-stable L\'evy processes imply that~\eqref{it:int_chunk_qlength} holds with probability tending to $1$ as $R_1 \to \infty$.  Likewise, it is also clear that~\eqref{it:int_chunk_mass} holds with probability tending to $1$ as $R_2 \to \infty$.  By the definition of $E_{\ell,L}^\uparrow$, we have that~\eqref{it:int_chunk_neighbormass} holds with probability at least $1-R_3^{-1}$.  We choose $R_3 \geq 1$ sufficiently large so that $1-R_3^{-1} \geq 1-q_0/3$.  To see that~\eqref{it:int_cond_prob} holds with probability tending to $1$ as $R_2 \to \infty$ (when $R_3$ is fixed), we note that we can construct the gluing of $\CN_{i,M}^\uparrow$ and $\wh{\CN}$ as follows.  Recall that $\CH_{i,M}^\uparrow$ is the quantum half-plane which corresponds to the unexplored region when generating $\eta_{i,M}^\uparrow$.  We suppose that we are working on the event that $\eta_{i,M}^\uparrow(\sigma_{i,M}^\uparrow) \in \partial \h$ and properties~\eqref{it:int_chunk_qlength} and~\eqref{it:int_chunk_mass} of Definition~\ref{def:interior_good_chunk} hold.  Then there are at most $R_1 a_0^{-1}$ points on the top of $\CN_{i,M}^\uparrow$ with clockwise boundary length distance to $\partial \CH_{i,M}^\uparrow$ given by an integer multiple of $a_0 2^{-(\kappa/4) M}$ and with boundary length distance from a point in $\Upsilon$ at most $a_0 2^{-(\kappa/4) M}$.  For a given such point chosen in an $\CF_{i,M}^\uparrow$-measurable manner, we run another $\SLE_\kappa^0(\kappa-6)$ process $\wh{\eta}$ in $\CH_{i+1,M}^\uparrow = \CH_{i,M}^\uparrow \setminus \CN_{i,M}^\uparrow$.  Suppose that we have fixed $j \in \{M-1,M,M+1\}$.  On the event that $\wh{\sigma} = \inf\{t \geq \delta_0 2^{-j} : \wh{\eta}(t) \in \partial \CH_{i+1,M}^\uparrow\} \leq 2^{-j}$ and property~\eqref{it:int_chunk_qlength} of Definition~\ref{def:interior_good_chunk} holds for the surface $\wh{\CN}$ disconnected from $\infty$ by $\wh{\eta}|_{[0,\wh{\sigma}]}$, the surface $\wt{\CN}$ parameterized by $\interior{\closure{\CN_{i,M}^\uparrow \cup \wh{\CN}}}$ is homeomorphic to $\D$.  Therefore it is clear that~\eqref{it:int_cond_prob} holds with probability tending to $1$ as $R_2 \to \infty$ (with $R_1,R_3 \geq 1$ fixed) as well.
\end{proof}

\begin{lemma}
\label{lem:disk_neighborhood}
Suppose that $D$ is a bounded, simply connected domain, $(D,h) \sim \qdiskL{\gamma}{1}$, and $\Gamma$ is a $\CLE_\kappa$ on $D$ which is independent of $h$.  Fix $0 < \alpha_\UBD < 2 < \alpha_\LBD$ and $\alpha_\PP,\epsilon_0 > 0$.  Let $E$ be the event that for every $z \in D$ and $\epsilon \in (0,\epsilon_0)$ we have that $\epsilon^{\alpha_\LBD} \leq \qmeasure{h}(B(z,\epsilon)) \leq \epsilon^{\alpha_\UBD}$ and no loop of $\Gamma$ has an $(\epsilon,\alpha_\PP)$-pinch point for $\epsilon \in (0,\epsilon_0)$.  For each $\beta > 0$ there exists $\xi > 0$ so that the following is true.  Fix $\delta \in (0,\epsilon)$ and let $F$ be the event that for every $z \in D$ with $\dist(z, \partial D) \leq \epsilon$ there exists a curve $\eta$ in $\Upsilon$ so that
\begin{enumerate}[(i)]
\item\label{it:distance_property} $\metapprox{\delta}{u}{v}{\Gamma}/\medianHP{\delta} \leq \epsilon^\xi$ for all $u,v \in \eta$ and
\item\label{it:diameter_property} the component $U$ of $D \setminus \eta$ which contains $z$ satisfies $\diam(U) \leq \epsilon^\xi$ and $\partial U \cap \partial D \neq \emptyset$.
\end{enumerate}
Uniformly in $\delta \in (0,\epsilon)$ we have that $\p[E \cap F^c] = O(\epsilon^\beta)$.
\end{lemma}
\begin{proof}

\noindent{\it Step 0: Setup.} Suppose that $x \in \partial D$ is picked from $\qbmeasure{h}$.  For each $k \in \N$, we let $N_k = 2^k$ and let $I_1,\ldots,I_{N_k}$ be arcs of $\partial D$ of equal quantum length $2^{-k}$ with disjoint interior whose union is $\partial D$.  We assume that the $I_j$ are given in counterclockwise order and the left endpoint of $I_1$ is $x$.  Fix $1 \leq j \leq N_k$ and let $I = I_j$.  Fix $\zeta > 0$.  We will adjust the value of $\zeta$ in the proof.  The role of $\epsilon > 0$ in the statement of the lemma will be played by $2^{- \zeta k}$. Let $U = U_j$ be the $2^{-\zeta k}$-neighborhood of $I$ with respect to the internal (Euclidean) metric in $D$.  Suppose that $\delta \in (0,2^{-\zeta k})$.  We are going to show that there exists $\xi > 0$ (which does not depend on $k$) so that there exists a path $\eta$ in $\Upsilon$ which satisfies~\eqref{it:distance_property} and such that $\diam(U) \leq \epsilon^\xi$ off an event of probability $O(2^{-(1+\beta)k})$.  The result will then follow by performing a union bound over $1 \leq j \leq N_k$ and then over $k \in \N$.

Let $x_0$ be the left endpoint of $I_{j-1}$ and let $x_1$ be the right endpoint of $I_{j+1}$, where we take the convention that $I_0 = I_{N_k}$ and $I_{N_k+1} = I_1$.  Then $\qbmeasure{h}(\ccwBoundary{x_0}{x_1}{\partial D}) = 3 \cdot 2^{-k}$ and $I \subseteq \ccwBoundary{x_0}{x_1}{\partial D}$.  We let $h_0 = h + \tfrac{2}{\gamma} \log (2^k/3)$ so that with $x_0^0 = x_0$ and $x_1^0 = x_1$ we have that $\qbmeasure{h_0}(\ccwBoundary{x_0^0}{x_1^0}{\partial D}) = 1$.  We let $\CD_0 = (D,h_0)$ be the resulting (rescaled) quantum disk.  We then perform the exploration as in Lemma~\ref{lem:point_to_point_disk_exploration} from $x_0^0$ to $x_1^0$ where the definition of a good chunk is as in Definition~\ref{def:interior_good_chunk} and where we take $J = \xi_0 k$ and $\xi_0 > 0$ is a fixed small constant whose value we will adjust later in the proof.

\noindent{\it Step 1.  Exploration succeeds.}  We first suppose that we are on the event that the exploration does not fail and we will explain why there exists a path $\eta$ as claimed.  See Figure~\ref{fig:interior_neighborhood} for an illustration of the setup.  Afterwards, we will explain how to iterate the exploration to construct a path $\eta$ in the case that the exploration fails.  Let $\CN_1,\ldots,\CN_n$ be the (good) chunks in the exploration which have (the maximal) quantum natural time in $[\delta_0 2^{-\xi_0 k},2^{-\xi_0 k}]$ (as measured using $h_0$).  Note that if we rescale the boundary length back by the factor $3/2^k$ (to transform $h_0$ back to $h$) then quantum natural time gets scaled by the factor $(3/2^k)^{4/\kappa}$.  Therefore the (maximal) quantum natural time of a good chunk under the exploration as measured by $h$ is equal to a constant times $2^{-\xi_1 k}$ where $\xi_1 = \xi_0 + 4/\kappa$.  For each $1 \leq \ell \leq n-1$, let $\psi_\ell$ be the unique conformal map from $\interior{\closure{\CN_\ell \cup \CN_{\ell+1}}}$ to $\D$ which takes $z_L(\CN_\ell)$, $z_C(\CN_\ell)$, and $z_R(\CN_\ell)$ to $-1$, $-i$, and $1$, respectively.  We then let $\CA_\ell = \psi_\ell^{-1}(B(0,1-R_2^{-1}))$.

\begin{figure}[ht!]
\begin{center}
\includegraphics[scale=1]{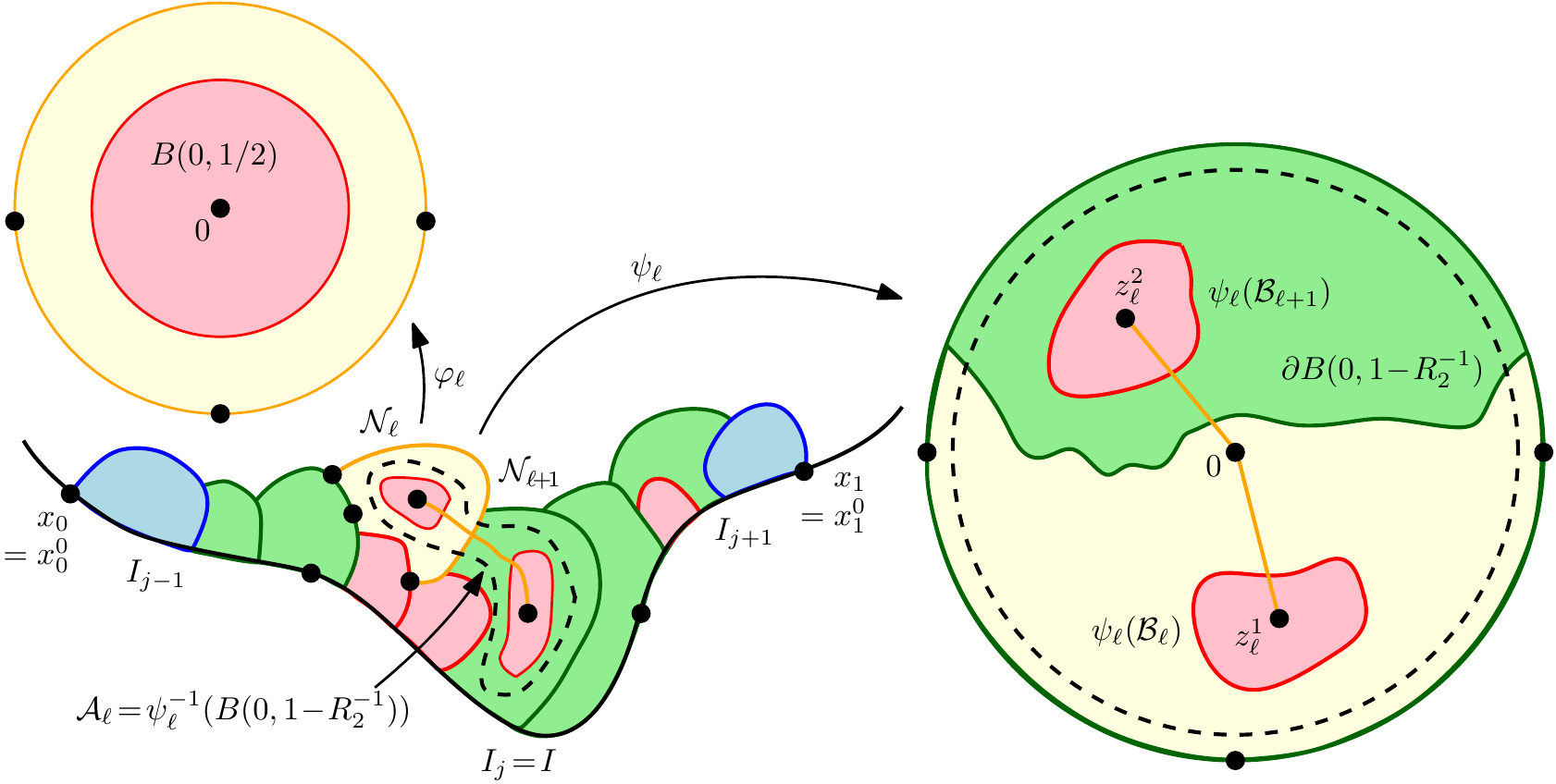}	
\end{center}
\caption{\label{fig:interior_neighborhood} Illustration of the definition of $\omega_\ell$ in the case that the exploration succeeds.  Shown is part of $\partial D$ which contains $I_{j-1}$, $I_j = I$, and $I_{j+1}$.  The blue regions are the initial and terminal parts of the exploration before and after the chunks with the maximal size are discovered.  Among the chunks with maximal size, those in green are good (with the exception of $\CN_\ell$ to differentiate it from $\CN_{\ell+1}$) and those in red are bad.  Recall that $\psi_\ell$ is the unique conformal map from $\interior{\closure{\CN_\ell \cup \CN_{\ell+1}}}$ to $\D$ which takes $z_L(\CN_\ell)$, $z_C(\CN_\ell)$, $z_R(\CN_\ell)$ respectively to $-1$, $-i$, and $1$ and $\varphi_\ell$ is the unique conformal map from $\CN_\ell$ to $\D$ which takes $z_L(\CN_\ell)$, $z_C(\CN_\ell)$, $z_R(\CN_\ell)$ respectively to $-1$, $-i$, and $1$.  Shown also in red are the regions $\CB_\ell = \varphi_\ell^{-1}(B(0,1/2))$ and $\CB_{\ell+1} = \varphi_{\ell+1}^{-1}(B(0,1/2))$.  The dashed region which contains $\CB_\ell$, $\CB_{\ell+1}$ is $\CA_\ell = \psi_\ell^{-1}(B(0,1-R_2^{-1}))$.  The path $\omega_\ell$ (orange) is the image under $\psi_\ell^{-1}$ of the concatenation of the line segment from $z_\ell^1 = \psi_\ell(\varphi_\ell^{-1}(0))$ to $0$ and from $0$ to $z_\ell^2 = \psi_\ell(\varphi_{\ell+1}^{-1}(0))$.}
\end{figure}

We are now going to show that we can adjust the value of $\zeta > 0$ so that $\CT = \CN_1 \cup \CN_n \cup (\cup_{\ell=1}^{n-1} \CA_\ell)$ contains a path~$\cpath$ which connects $I_{j-1}$ to $I_{j+1}$ and has (Euclidean) distance at least $2^{-\zeta k}$ from $I_j$.  This, in turn, will imply that there is a component of $D \setminus \cpath$ which contains $U$.  To see this, for each $1 \leq \ell \leq n-1$ we let $\varphi_\ell \colon \CN_\ell \to \D$ be the unique conformal map which takes $z_L(\CN_\ell)$, $z_C(\CN_\ell)$, and $z_R(\CN_\ell)$ to $-1$, $-i$, and $1$, respectively.  Let $\CB_\ell = \varphi_\ell^{-1}(B(0,1/2))$.  Then we know from property~\eqref{it:int_chunk_neighbormass} of Definition~\ref{def:interior_good_chunk} that $\CA_\ell$ contains $\CB_\ell$ and $\CB_{\ell+1}$.  By property~\eqref{it:int_chunk_mass} of Definition~\ref{def:interior_good_chunk},
\begin{equation}
\label{eqn:b_area_lbd}
\qmeasure{h}(\CB_\ell) = \left(\frac{3}{2^{k}}\right)^2 \qmeasure{h_0}(\CB_\ell) \geq R_2^{-1} 2^{-(\xi_0+2) k} = R_2^{-1} 2^{-\xi_2 k} \quad\text{where}\quad \xi_2 = \xi_0+2.
\end{equation}
From the definition of $E$, we have that
\[ \diam(\CB_\ell)^{\alpha_\UBD} \geq \qmeasure{h}(\CB_\ell) \geq R_2^{-1} 2^{-\xi_2 k}.\]
That is, on $E$ we have that
\begin{equation}
\label{eqn:b_diam_lbd}
\diam(\CB_\ell) \geq R_2^{-1/\alpha_\UBD} 2^{-\xi_2 k /\alpha_\UBD}.	
\end{equation}
Assume that $1 \leq \ell \leq n-1$.  We now define a path $\cpath_\ell$ as follows. We let $z_\ell^1 = \psi_\ell(\varphi_{\ell}^{-1}(0))$ and $z_\ell^2 = \psi_\ell(\varphi_{\ell+1}^{-1}(0))$. We then let $\cpath_\ell$ be the image under $\psi_\ell^{-1}$ of the concatenation of $[z_\ell^1,0]$ with $[0,z_\ell^2]$.  As $B(0,1-R_2^{-1})$ contains $\psi_\ell(\CB_\ell)$ and $\CB_\ell$ has diameter at least $R_2^{-1/\alpha_\UBD} 2^{- \xi_2 k /\alpha_\UBD}$, standard distortion estimates for conformal maps imply that
\[ |(\psi_\ell^{-1})'(z)| \gtrsim 2^{-\xi_2 k/\alpha_\UBD} \quad\text{for all}\quad z \in B(0,1-R_2^{-1})\]
where the implicit constant in $\gtrsim$ depends only on $R_2$. Therefore there exists a ($R_2$-dependent) constant $c_0 > 0$ so that $\CT$ contains the (Euclidean) $c_0 2^{-\xi_2 k/\alpha_\UBD}$ neighborhood of $\cpath_\ell$.  We take $\zeta > \xi_2/\alpha_\UBD$.

In order to finish proving the claim, it is left to explain why:
\begin{itemize}
\item $\CN_1$ (resp.\ $\CN_n$) intersects~$I_{j-1}$ (resp.\ $I_{j+1}$) and
\item We can define a path $\cpath_0$ in $\CN_1 \cup \CA_1$ which connects~$I_{j-1}$ to the starting point of~$\cpath_1$ and a path~$\cpath_n$ in $\CA_{n-1} \cup \CN_n$ which connects the ending point of~$\cpath_{n-1}$ to~$I_{j+1}$ so that both~$\cpath_0$ and~$\cpath_n$ have distance at least $2^{-\zeta k}$ from $I_j$.
\end{itemize}
The first point is in fact a consequence of the definition of the exploration succeeding.  This leaves us to address the second point.

Let us first explain how to construct the path $\cpath_0$.  Consider the set $\CB_1 = \varphi_1^{-1}(B(0,1/2))$.  As explained above, we know that $\diam(\CB_1) \geq R_2^{-1/\alpha_\UBD} 2^{-\xi_2 k/\alpha_\UBD}$.  Let $L_1$ (resp.\ $L_2$) be the line segment in $\D$ from $-1/4$ (resp.\ $1/4$) to the point $z_1 = -i$ (resp.\ the point $z_2$ on the bottom of $\partial \D$ with real part $1/4$).  We claim that there exists a constant $c > 0$ so that $\dist(\varphi_1^{-1}(L_1),\varphi_1^{-1}(L_2)) \geq c 2^{- \xi_2 k/\alpha_\UBD}$.  To see this, let $B$ be a Brownian motion starting from $0$.  Then $B$ has a positive chance of exiting $\partial \D$ in $\ccwBoundary{z_1}{z_2}{\partial \D}$ before hitting $L_1$, $L_2$.  Therefore the same is also true for the Brownian motion $\varphi_1^{-1}(B)$. The claim therefore follows from the Beurling estimate.  We can therefore take $\cpath_0$ to be the image under $\varphi_1^{-1}$ of the line segment from the point on the bottom of $\partial \D$ with real part $1/8$ to $0$ and we see that the distance of $\cpath_0$ to $I_j$ is at least $c 2^{-\xi_2 k/\alpha_\UBD}$ (possibly decreasing $c > 0$).  We define $\cpath_n$ in a similar manner.

We now define the path $\eta$.  Let $x \in \partial D \setminus (I_{j-1} \cup I_j \cup I_{j+1})$ be any fixed point.  On the event that the exploration succeeds, its outer boundary as viewed from $x$ does not depend on the choice of $x$.  We take $\eta$ to be the curve which corresponds to the outer boundary of the exploration as viewed from $x$ and assume that $\eta$ is parameterized by quantum length.  We note that in this case the quantum length of $\eta$ is at most a constant times $2^{-k}$.  On $E$, \eqref{eqn:quantum_holder_norm_tight} from Proposition~\ref{prop:boundary_distance_tail_bounds} thus gives an upper bound on $\diam(I_{j-1} \cup I_j \cup I_{j+1})$ and $\diam(\eta)$ which is a power of $2^{-k}$ off an event which occurs with probability $O(2^{-(1+\beta)k})$, which completes the proof in the case that the exploration does not fail.

\noindent{\it Step 2.  Exploration fails.} We now consider the possibilities which can occur if the exploration described above fails.  Measuring quantities using $h_0$, we recall that the exploration fails if it makes either a large upward jump (i.e., of size at least $2^{- \exploreExp (\kappa/4) M}$ when exploring chunks with quantum natural time in $[\delta_0 2^{-M}, 2^{-M}]$), a large downward jump (i.e., of size at least $2^{- \exploreExp (\kappa/4) M}$ when exploring chunks with quantum natural time in $[\delta_0 2^{-M}, 2^{-M}]$), requires too many chunks at a particular scale (i.e., at least $c_F 2^{(1-\exploreExp)(\kappa/4) M}$ when exploring chunks of size in $[\delta_0 2^{-M}, 2^{-M}]$ for $M > J$ or at least $c_F 2^{(\kappa/4)J}$ when exploring chunks of size in $[\delta_0 2^{-J},2^{-J}]$), or the deviations in the boundary length process are too large (i.e., the absolute value of the quantum length of the top minus the bottom exceeds $c_F 2^{(1-4/\kappa) M}$ when exploring chunks of size in $[\delta_0 2^{-M}, 2^{-M}]$).  We stop the exploration immediately if either of these possibilities happens, and then proceed as described below.

\begin{figure}[ht!]
\begin{center}
\includegraphics[scale=1]{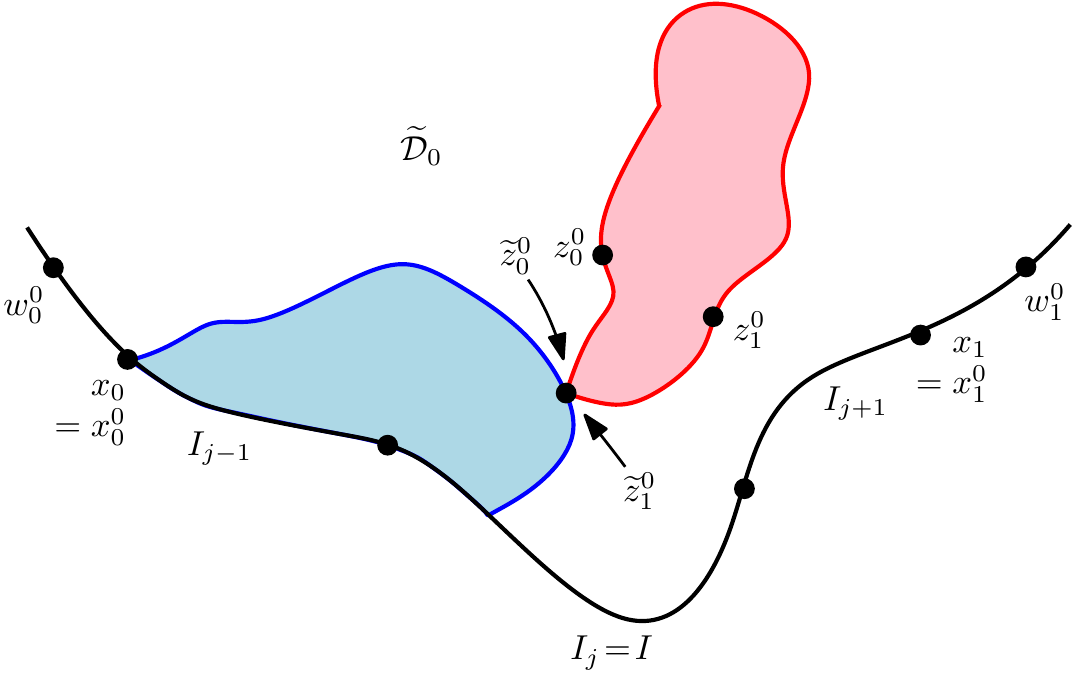}	
\end{center}
\caption{\label{fig:interior_neighborhood_upward_jump} Illustration of the possibility that the exploration fails by making an upward jump whose size measured using $h_0$ is at least $3$.  Shown in blue is the exploration up until it makes the large upward jump and the corresponding $\CLE_\kappa$ loop is shown in red.  The two prime ends which correspond to where the loop is rooted on the previous exploration are $\wt{z}_0^0$, $\wt{z}_1^0$.  The point $z_0^0$ (resp.\ $z_0^1$) is the one on the loop so that the clockwise (resp.\ counterclockwise) boundary length from $\wt{z}_0^0$ to $z_0^0$ (resp.\ $\wt{z}_1^0$ to $z_1^0$) measured using $h_0$ is equal to $1$.  Since we are working on the event that $\Gamma$ does not have any $(\epsilon,\alpha_\PP)$-pinch points for $\epsilon \in (0,\epsilon_0)$, by choosing $\zeta > 0$ sufficiently large (depending only on $\alpha_\PP$) we have that the clockwise arc of the loop from $z_0^0$ to $z_1^0$ does not hit the ball of radius $2^{-\zeta k}$ centered at $\wt{z}_0^0$.  The point $w_0^0$ (resp.\ $w_1^0$) is the one which is $1$ unit of boundary length distance in the clockwise (resp.\ counterclockwise) direction from $x_0^0$ (resp.\ $x_1^0$) measured using $h_0$.  The exploration is then continued from $w_0^0$ to $z_0^0$ and from $z_1^0$ to $w_1^0$.}
\end{figure}

\noindent{\it Case 1: large upward jump.}  We let $\wt{\CD}_0$ be the quantum disk in the complement of the exploration with $x_0^0$ and $x_1^0$ on its boundary.  We let $\wt{z}_0^0$, $\wt{z}_1^0$ be the two prime ends on $\partial \wt{\CD}_0$ which correspond to where the discovered $\CLE_\kappa$ loop $\CL$ is rooted to the boundary.  There are two possibilities.  Either the size of the upward jump (measured using $h_0$) is smaller or larger than $3$.

We first suppose that the upward jump (measured using $h_0$) has size at most $3$.  We let $z_0^0$ (resp.\ $z_1^0$) be the point on $\partial \wt{\CD}_0$ so that $\qbmeasure{h_0}(\ccwBoundary{z_0^0}{x_0^0}{\partial \wt{\CD}_0}) = 1$ (resp.\ $\qbmeasure{h_0}(\ccwBoundary{x_1^0}{z_1^0}{\partial \wt{\CD}_0}) = 1$).  We let $\CD_{00}$ be given by $\wt{\CD}_0$ after adding $\tfrac{2}{\gamma} \log L_{00}^{-1}$ to the field where $L_{00} = \qbmeasure{h_0}(\ccwBoundary{z_0^0}{z_1^0}{\partial \wt{\CD}_0})$ and we write $x_0^{00}$, $x_1^{00}$ for the points which correspond to $z_0^0$, $z_1^0$.  We refer to $(\CD_{00},x_0^{00},x_1^{00})$ as the \emph{child} of $(\CD_0,x_0^0,x_1^0)$.  We then perform the exploration inside of $\CD_{00}$ from $x_0^{00}$ to $x_1^{00}$ with the definition of the good chunk as in Definition~\ref{def:interior_good_chunk}.  We note in this case that 
\begin{equation}
\label{eqn:upward_jump_small_dev_bound}
3 - c_F 2^{(1-4/\kappa) J} \leq L_{00} \leq 6 + c_F 2^{(1-4/\kappa) J}.
\end{equation}

We now suppose that the upward jump (measured using $h_0$) has size at least $3$.  See Figure~\ref{fig:interior_neighborhood_upward_jump} for an illustration in this case.  We let $z_0^0$ (resp.\ $z_1^0$) be the point on $\CL$ so that $\qbmeasure{h_0}(\cwBoundary{\wt{z}_0^0}{z_0^0}{\CL}) = 1$ (resp.\ $\qbmeasure{h_0}(\cwBoundary{z_1^0}{\wt{z}_1^0}{\CL}) = 1$).  This means that
\[ \qbmeasure{h}(\cwBoundary{\wt{z}_0^0}{z_0^0}{\CL}) \geq 2^{-k} \quad\text{and}\quad  \qbmeasure{h}(\cwBoundary{z_1^0}{\wt{z}_1^0}{\CL}) \geq 2^{-k}.\]
It thus follows from the definition of $E$ that if we assume $\zeta > 1/ \alpha_\PP$ then $\ccwBoundary{z_0^0}{z_1^0}{\partial \wt{\CD}_0}$ does not intersect $B(\wt{z}_0^0,2^{-\zeta k})$.
  
We also let $w_0^0$ (resp.\ $w_1^0$) be the point on $\partial \wt{\CD}_0$ so that $\qbmeasure{h_0}(\ccwBoundary{w_0^0}{x_0^0}{\partial \wt{\CD}_0}) = 1$ (resp.\ $\qbmeasure{h_0}(\ccwBoundary{x_1^0}{w_1^0}{\partial \wt{\CD}_0}) = 1$.  We let $\CD_{00}$ be given by $\wt{\CD}_0$ after adding $\tfrac{2}{\gamma} \log L_{00}^{-1}$ to the field where $L_{00} = \qbmeasure{h_0}(\ccwBoundary{w_0^0}{z_0^0}{\partial \wt{\CD}_0})$ and we write $x_0^{00}$, $x_1^{00}$ for the points corresponding to $w_0^0$, $z_0^0$.  We also let $\CD_{01}$ be given by $\wt{\CD}_0$ after adding $\tfrac{2}{\gamma} \log L_{01}^{-1}$ where $L_{01} = \qbmeasure{h_0}(\ccwBoundary{z_1^0}{w_1^0}{\partial \wt{\CD}_0})$ and we write $x_0^{01}$, $x_1^{01}$ for the points which correspond to $z_1^0$, $w_1^0$.  We refer to $(\CD_{0i},x_0^{0i},x_1^{0i})$ for $i=0,1$ as the \emph{children} of $(\CD_0,x_0^0,x_1^0)$.  We then perform the exploration inside of $\CD_{0i}$ from $x_0^{0i}$ to $x_1^{0i}$ for $i=0,1$ with the definition of the good chunk as in Definition~\ref{def:interior_good_chunk}.  We note in this case that
\begin{equation}
\label{eqn:upward_jump_large_dev_bound}
2 \leq L_{0i} \leq 3 + c_F 2^{(1-4/\kappa) J}\quad\text{for}\quad i=0,1.
\end{equation}  

We will now explain how to construct the path $\cpath$ in the case that both of the explorations succeed.  As mentioned just above,  we have that the distance between $\cwBoundary{z_0^0}{z_0^1}{\CL}$ and $\wt{z}_0^0$ is at least $2^{-\zeta k}$.  Let $\CN_1,\ldots,\CN_\ell$ be the good chunks from the exploration in $\CD_{00}$ from $x_0^{00}$ to $x_1^{00}$ which have the maximal chunk size.  Arguing as in the proof that the exploration succeeds to begin with, we can construct a path inside of $\closure{\cup_{j=1}^\ell \CN_j}$ which has distance at least $2^{-\zeta k}$ (possibly increasing $\zeta > 0$) from $I_j$ and connects $I_{j-1}$ to $x_1^{00} = z_0^0$.  We then take the concatenation of this path with $\ccwBoundary{z_0^0}{\wt{z}_0^0}{\CL}$.  Finally, we concatenate this path with $\ccwBoundary{\wt{z}_1^0}{z_1^0}{\CL}$ and the analogous path constructed from the exploration in $\CD_{01}$ from $x_0^{01} = z_1^0$ to $x_1^{01} = x_1^0$.

If one of the two explorations fail, then we iterate the procedure depending on the type of failure in the same manner that we continue the iteration procedure in the case that the initial exploration has failed with the following exception.  If the exploration in $\CD_{00}$ from $x_0^{00}$ to $x_1^{00}$ has a large downward jump whose terminal point disconnects $x_0^{01}$ and $x_1^{01}$, then we do not perform the exploration in $\CD_{01}$ from $x_0^{01}$ to $x_1^{01}$ and only perform an exploration from $x_{0}^{00}$ to $x_1^{01}$ in the component which contains these two points (as in the first type of large downward jump considered just below).

\begin{figure}[ht!]
\begin{center}
\includegraphics[scale=1]{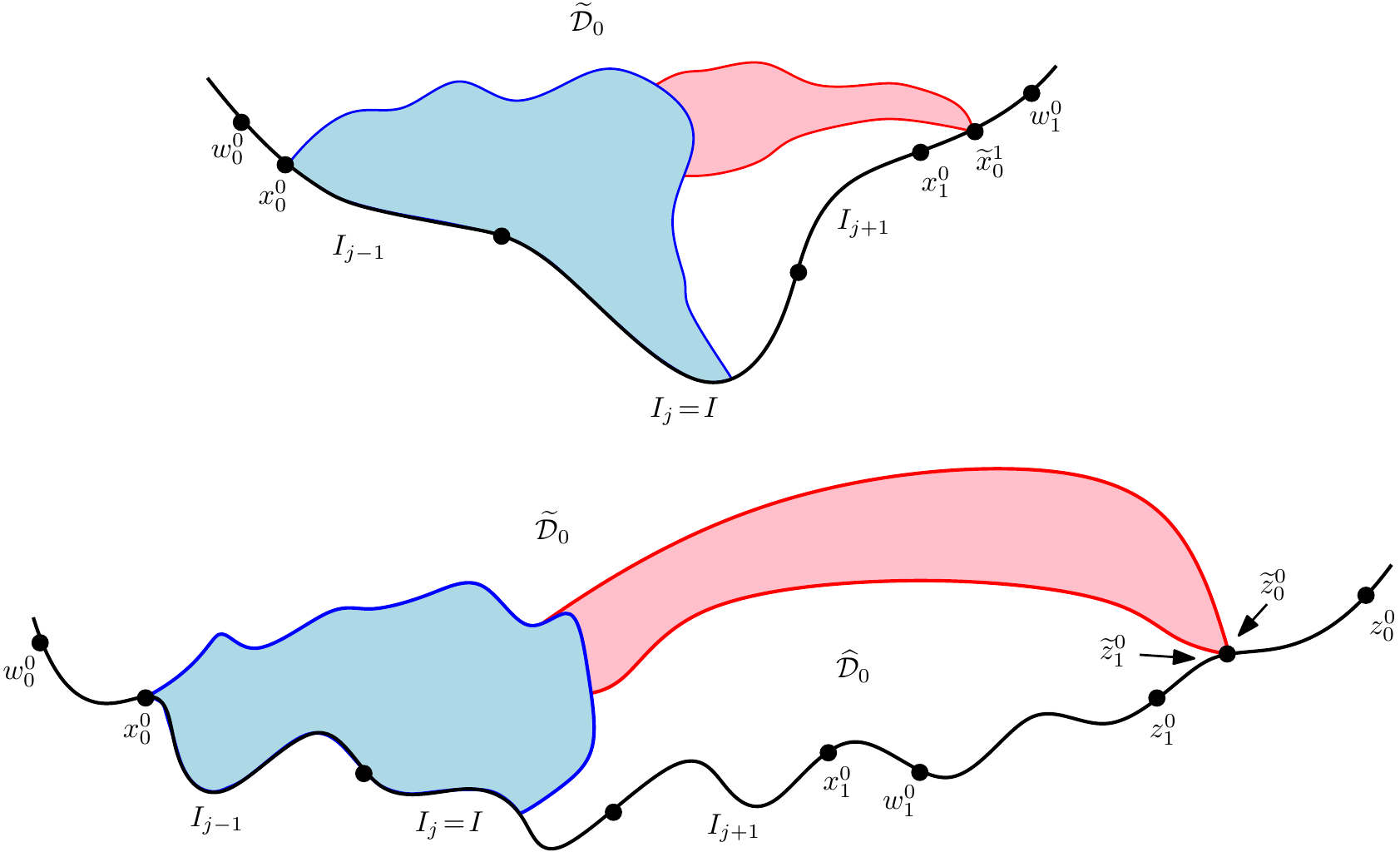}	
\end{center}
\caption{\label{fig:interior_neighborhood_downward_jump} Illustration of the possibility that the exploration fails by making a large downward jump.  {\bf Top:} The terminal point of the downward jump is within $2$ units of boundary length distance (measured using $h_0$) from $I_{j-1} \cup I_j \cup I_{j+1}$.  Shown is the case that the terminal point is not in $I_{j-1} \cup I_j \cup I_{j+1}$ but it is also possible for it to be in $I_{j-1} \cup I_j \cup I_{j+1}$.  {\bf Bottom:} The terminal point of the downward jump is more than $2$ units of boundary length distance (measured using $h_0$) from $I_{j-1} \cup I_j \cup I_{j+1}$.}
\end{figure}

\noindent{\it Case 2: large downward jump.}  There are two possibilities in the case that the exploration makes a large downward jump: either the terminal point of the downward jump is within boundary length distance $2$ (measured using $h_0$) of $I_{j-1} \cup I_j \cup I_{j+1}$ or it is not.  See Figure~\ref{fig:interior_neighborhood_downward_jump} for an illustration of the setup.

Suppose that the terminal point of the downward jump is within boundary length distance $2$ (measured using $h_0$) of $I_{j-1} \cup I_j \cup I_{j+1}$.  Let $w_0^0$ (resp.\ $w_1^0$) be such that $\qbmeasure{h_0}(\cwBoundary{w_0^0}{x_0^0}{\partial \CD}) = 3$ (resp.\ $\qbmeasure{h_0}(\ccwBoundary{x_1^0}{w_1^0}{\partial \CD}) = 3$) and let $\wt{\CD}_0$ be the quantum disk in the complement of the exploration whose boundary contains $\cwBoundary{w_0^0}{w_1^0}{\partial \CD}$.  We then let $\CD_{00}$ be given by $\wt{\CD}_{0}$ after adding $\tfrac{2}{\gamma} \log L_{00}^{-1}$ to the field where $L_{00} = \qbmeasure{h_0}(\ccwBoundary{z_0^0}{z_1^0}{\partial \wt{\CD}_0})$ and we write $x_0^{00}$, $x_1^{00}$ for the points on $\partial \CD_{00}$ which correspond to $w_0^0$, $w_1^0$.  We refer to $(\CD_{00},x_0^{00},x_1^{00})$ as the child of $(\CD_0,x_0^0,x_1^0)$.  We then perform the exploration inside of $\CD_{00}$ from $x_0^{00}$ to $x_1^{00}$.  In the case that this exploration succeeds, we take the path to be defined in an analogous manner as to when the exploration did not have a large downward jump.  We note in this case that
\begin{equation}
\label{eqn:downward_jump_close_dev_bound}
4 \leq L_{00} \leq 7 + c_F 2^{(1-4/\kappa) J}.
\end{equation}

Now suppose that the terminal point of the downward jump has boundary length distance from $I_{j-1} \cup I_j \cup I_{j+1}$ at least $2$ (measured using $h_0$).  We let $\wt{\CD}_0$ be the quantum disk in the complement of the exploration with $x_0^0$ on its boundary.  We let $\wt{z}_0^0$ be the point on $\partial \wt{\CD}_0$ which is the terminal point of the large downward jump and let $z_0^0 \in \partial \wt{\CD}_0$ be such that $\qbmeasure{h_0}(\ccwBoundary{\wt{z}_0^0}{z_0^0}{\partial \wt{\CD}_0}) = 1/2$.  We also let $w_0^0 \in \partial \wt{\CD}_0$ be such that $\qbmeasure{h_0}(\ccwBoundary{w_0^0}{x_0^0}{\partial \wt{\CD}_0}) = 1/2$.  We let $\CD_{00}$ be given by $\wt{\CD}_{00}$ after adding $\tfrac{2}{\gamma} \log L_{00}^{-1}$ to the field where $L_{00} = \qbmeasure{h_0}(\ccwBoundary{w_0^0}{z_0^0}{\partial \wt{\CD}_0})$ and we write $x_0^{00}$, $x_1^{00}$ for the points corresponding to $w_0^0$, $z_0^0$.  We let $\wh{\CD}_0$ be the quantum disk in the complement of the exploration with $x_1^0$ on its boundary, let $\wt{z}_1^0 \in \partial \wh{\CD}_0$ be the terminal point of the large downward jump, and let $z_1^0 \in \partial \wh{\CD}_0$ be such that $\qbmeasure{h_0}(\ccwBoundary{z_1^0}{\wt{z}_1^0}{\partial \wh{\CD}_0}) = 1/2$.  We also let $w_1^0 \in \partial \wh{\CD}_0$ be such that $\qbmeasure{h_0}(\ccwBoundary{x_1^0}{w_1^0}{\partial \wh{\CD}_0}) = 1/2$ and let $\CD_{01}$ be given by $\wh{\CD}_0$ after adding $\tfrac{2}{\gamma} \log L_{01}^{-1}$ where $L_{01} = \qbmeasure{h_0}(\ccwBoundary{z_1^0}{w_1^0}{\partial \wh{\CD}_0})$ and we write $x_0^{01}$, $x_1^{01}$ for the points corresponding to $z_1^0$, $w_1^0$, respectively.  We refer to $(\CD_{0i},x_0^{0i},x_1^{0i})$ for $i=0,1$ as the children of $(\CD_0,x_0^0,x_1^0)$.  We then perform the exploration inside of $\CD_{0i}$ from $x_0^{0i}$ to $x_1^{0i}$ for $i=0,1$.  We note in this case that
\begin{equation}
\label{eqn:downward_jump_far_dev_bound}
1 \leq L_{0i} \leq 2 + c_F 2^{(1-4/\kappa) J} \quad\text{for}\quad i=0,1.
\end{equation} 

In the case that both explorations succeed, we take the path which is given by concatenating the path associated with $(\CD_{00},x_0^{00},x_1^{00})$ followed by the clockwise arc of $\partial D$ from $x_1^{00}$ to $x_0^{01}$ and then the path associated with $(\CD_{01},x_0^{01},x_1^{01})$.  Arguing as in the same way as the exploration succeeds, if both of these explorations succeed then this path will have the desired properties.

If either exploration fails, we continue the exploration as described above in the new quantum disk.

\noindent{\it Case 3: too many chunks or top boundary length too long.}  Finally, we consider the situation in which the exploration fails by either having $n_M^\uparrow$, $n_M^\downarrow$ too large for $M \geq J$ or the absolute value of the difference of the top of the exploration minus the bottom (as measured using $h_0$) exceeds $c_F 2^{(1-4/\kappa) J}$.  We note that the quantum length of the top can only exceed that of the bottom by at most $c_F 2^{(1-4/\kappa) J} + 3$ for otherwise the exploration would have failed due to a large upward jump.  We let $\wt{\CD}_0$ be the quantum disk in the complement of the exploration with $x_0^0$, $x_1^0$ on its boundary.  Let $\wt{h}_0$ be the field which describes $\wt{\CD}_0$ and let $w_0^0$ (resp.\ $w_1^0$) be the point in $\partial \wt{\CD}_0$ so that $\qbmeasure{h_0}(\ccwBoundary{w_0^0}{x_0^0}{\partial \wt{\CD}_0}) = 1$ (resp.\ $\qbmeasure{h_0}(\ccwBoundary{x_1^0}{w_0^0}{\partial \wt{\CD}_0}) = 1$).  We let $\CD_{00}$ be the quantum surface described by $h_{00} = \wt{h}_0 + \tfrac{2}{\gamma} \log L_{00}^{-1}$ where $L_{00} = \qbmeasure{\wt{h}_0}(\ccwBoundary{w_0^0}{w_1^0}{\partial \wt{\CD}_0})$ and we write $x_0^{00}$, $x_1^{00}$ for the points which correspond to $w_0^0$, $w_1^0$.  We refer to $(\CD_{00}, x_{0}^{00}, x_1^{00})$ as the child of $(\CD_{0},x_0^0,x_1^0)$ and then continue the exploration in $\CD_{00}$ from $x_0^{00}$ to $x_1^{00}$.  We note in this case that
\begin{equation}
\label{eqn:too_many_chunk_dev_bound}
2 \leq L_{00} \leq 6 + c_F 2^{(1-4/\kappa) J}.
\end{equation}

We recall that the probability that the exploration fails at given stage is $O(2^{-c J})$ for a constant $c > 0$.  This proves that the number of times that the exploration is restarted is dominated by the size of a subcritical Galton-Watson tree where each node can have at most two children and the probability of having either one or two children is $O(2^{-c J})$.  Recalling that we have taken $J = \xi_0 k$, it therefore follows that there exists $n \in \N$ (which does not grow with $k$) so that the probability that the exploration is restarted more than $n$ times is $O(2^{-(1+\beta) k})$.  The proof is completed by taking a union bound over $1 \leq j \leq N_k$ and then over $k$.  

Suppose that we have some $m \in \N$ and $o \in \{0,1\}^m$ so that $(\CD_o,x_0^o,x_1^o)$ is part of the exploration.  Then the maximal chunk size which makes up chunks in the exploration of $\CD_o$ from $x_0^o$ to $x_1^o$ (relative to the ambient quantum disk $\CD$) is given by $2^{-\xi_0 J} \prod_{o'} L_{o'}^{4/\kappa}$ where the product is over $o'$ which are along the ancestral line of $o$ using the tree structure defined above.  Combining the upper and lower bounds from \eqref{eqn:upward_jump_small_dev_bound}--\eqref{eqn:too_many_chunk_dev_bound} we see that there are constants $c_0, c_1 > 0$ so that
\[ 2^{-c_0 m J} \leq 2^{-\xi_0 J} \prod_{o'} L_{o'}^{4/\kappa} \leq 2^{-c_1 m J}.\]
Arguing as in~\eqref{eqn:b_area_lbd}, \eqref{eqn:b_diam_lbd} it therefore follows that the corresponding path $\omega$ has a neighborhood of size at least a constant times $2^{-c_0 n \xi_2 k/\alpha_\UBD}$ around it which is contained in the aforementioned chunks.  Since we will always take $m \leq n$ where $n$ is chosen as above depending on $\beta$, the result follows by possibly increasing the value of $\zeta > 0$.  
\end{proof}

In what follows, for $L > 0$ we let $\qdiskCarpet{\gamma}{L}$ be the law on quantum surfaces decorated by a loop ensemble and marked point $(D,h,\Gamma,z)$ which can be sampled from as follows.  First suppose that $(D,\wt{h},\wt{x},\wt{y}) \sim \qdiskL{\gamma}{L}$ and $\wt{\Gamma}$ is an independent $\CLE_\kappa$ on $D$ with carpet $\wt{\Upsilon}$.  Then the law of $(h,\Gamma)$ is the one which has Radon-Nikodym derivative with respect to the law of $(\wt{h},\wt{\Gamma})$ given by $\CZ^{-1} \qcarpet{\wt{h}}{\wt{\Upsilon}}$ where $\CZ$ is a normalization constant.  Finally, conditionally on everything else, $z$ has law $\qcarpet{h}{\Upsilon}$.

\begin{lemma}
\label{lem:interior_point_distance_to_boundary_finite}
Suppose that $D$ is a bounded, simply connected domain and $(D,h,\Gamma,z) \sim \qdiskCarpet{\gamma}{1}$.  Fix $0 < \alpha_\UBD < 2 < \alpha_\LBD$ and $\epsilon_0 > 0$.  Let $E$ be the event that for every $w \in D$ and $\epsilon \in (0,\epsilon_0)$ with $B(w,\epsilon) \subseteq D$ we have that $\epsilon^{\alpha_\LBD} \leq \qmeasure{h}(B(w,\epsilon)) \leq \epsilon^{\alpha_\UBD}$.  Then the law of $(\medianHP{\epsilon})^{-1} \metapprox{\epsilon}{z}{\partial D}{\Gamma}$ is tight as $\epsilon \to 0$.
\end{lemma}
\begin{proof}
Let~$\eta$ be an $\SLE_\kappa^0(\kappa-6)$ coupled with $\Gamma$ as a CPI and which is targeted at $z$ and parameterized by the quantum natural time of its trunk.  Let $\tau_0 = 0$, $\CD_0 = \CD$, and for each $j \geq 1$ we inductively define $\tau_j$ and $\CD_j$ as follows.  We let $\tau_j$ be the first time $t \geq \tau_{j-1}$ that $\eta([\tau_{j-1},\tau_j])$ disconnects~$z$ from $\partial \CD_{j-1}$ and we let $\CD_j$ be the quantum surface parameterized by the component of $\CD \setminus \eta([0,\tau_j])$ containing~$z$.  We also let $\wt{\CD}_j$ be the quantum surface parameterized by the other component of $\CD_{j-1} \setminus \eta([0,\tau_j])$ with $\eta(\tau_j)$ on its boundary.

Let $L_j$ (resp.\ $\wt{L}_j$) be the quantum length of $\partial \CD_j$ (resp.\ $\partial \wt{\CD}_j$).  Fix $p \in (0,4/\kappa-1/2)$; note that this interval is non-empty as $\kappa \in (8/3,4)$.  We claim that there exists a constant $c > 0$ so that $\E[L_j^p] \leq 2^{-c j}$ and $\E[\wt{L}_j^p] \leq c^{-1} 2^{-c j}$ for each $j \in \N$.  To see this, for each $t \geq 0$ we let $D_t$ denote the component of $D \setminus \eta([0,t])$ which contains $z$ and let $X_t$ be the quantum length of $\partial D_t$.  By the construction of $\qcarpet{h}{\Upsilon}$ given in \cite{msw2020simplecle}, for any stopping time $\tau$ for the filtration $\CF_t = \sigma(X_s : s \leq t)$ we have that $\E[ \qcarpet{h}{\Upsilon}(\Upsilon \cap D_\tau) \giv \CF_\tau] = c_0 X_\tau^{4/\kappa-1/2}$ for a constant $c_0 > 0$.  In particular, if we fix $r > 0$ and let $\tau_r = \inf\{ t \geq 0 : X_t \geq r\}$ then we have that
\begin{align*}
\p[ \sup_{t \geq 0} X_t \geq r]
&\leq \p[ \tau_r < \infty]
 \leq r^{-4/\kappa+1/2} \E[ X_{\tau_r}^{4/\kappa-1/2} \one_{\{\tau_r < \infty\}}	 ]\\
&\leq c_0^{-1} r^{-4/\kappa+1/2} \E[\qcarpet{h}{\Upsilon}(\Upsilon)]
 = O(r^{-4/\kappa+1/2}).
\end{align*}
Then the above implies that $\E[ (\sup_{t \geq 0} X_t)^p ] < \infty$ and therefore $\E[L_1^p] < \infty$.  Since $X_t \to 0$ as $t \to \infty$ a.s.\ we have that $L_j \to 0$ as $j \to \infty$ a.s.  Consequently, it follows from the dominated convergence theorem (with dominating function given by $(\sup_{t \geq 0} X_t)^p$) that $\E[L_j^p] \to 0$ as $j \to \infty$.  Thus as $\E[L_j^p] = (\E[L_1^p])^j$, the only way that $\E[ L_j^p ] \to 0$ as $j \to \infty$ can hold is if $\E[L_1^p] < 1$.  This proves that there exists a constant $c > 0$ so that
\begin{equation}
\label{eqn:l_j_moment}
\E[L_j^p] \leq 2^{-c j} \quad\text{for all}\quad j \in \N.
\end{equation}
We note that the law of $\wt{L}_j$ is stochastically dominated by $L_{j-1} \sup_{t \geq 0} \wt{X}_t$ where $\wt{X}$ is an independent copy of $X$.  It therefore follows that by possibly decreasing the value of $c > 0$ we have that
\begin{equation}
\label{eqn:wt_l_j_moment}
\E[ \wt{L}_j^p] \leq c^{-1} \E[ L_{j-1}^p] \leq c^{-1} 2^{-cj} \quad\text{for all}\quad j \in \N.
\end{equation}

Let $h_j$ (resp.\ $\wt{h}_j$) be the field which describes $\CD_j^z = \CD_j-z$ (resp.\ $\wt{\CD}_j^z = \wt{\CD}_j-z$).  Let $\Gamma_j$ be the loops of $\Gamma$ which are contained $\CD_j$ and then translated by $-z$.  Let $\Upsilon_j$ be the carpet of $\Gamma_j$.  Then $h_j$ together with $\Upsilon_j$ and marked by $0$ describes a sample from the law $\qdiskCarpet{\gamma}{L_j}$.  Also, $\wt{h}_j$ describes a sample from the law $\qdiskL{\gamma}{\wt{L}_j}$.  Let $R_j = L_j^{2/\alpha_\LBD}$ and $\wt{R}_j = \wt{L}_j^{2/\alpha_\LBD}$.  Set
\[ H_j(\cdot) = h_j(R_j \cdot) + Q \log R_j - \frac{2}{\gamma} \log L_j \quad\text{and}\quad \wt{H}_j(\cdot) = \wt{h}_j(\wt{R}_j \cdot) + Q \log \wt{R}_j - \frac{2}{\gamma} \log \wt{L}_j.\]
Then $H_j$ together with $R_j^{-1} \Upsilon_j$ and marked by $0$ describes a sample from the law $\qdiskCarpet{\gamma}{1}$.  Also, $\wt{H}_j$ describes a sample from the law~$\qdiskL{\gamma}{1}$.  On $E$, we have that $\epsilon^{\alpha_\LBD} \leq \qmeasure{h_j}(B(w,\epsilon)) \leq \epsilon^{\alpha_\UBD}$ for all $\epsilon \in (0,\epsilon_0)$ and $w \in \CD_j^z$ so that $B(w,\epsilon) \subseteq \CD_j^z$.  We thus have on $E$ that for all $\epsilon \in (0,\epsilon_0 / R_j)$ and $w \in R_j^{-1} \CD_j^z$ with $B(w,\epsilon) \subseteq R_j^{-1} \CD_j^z$ that
\[ \qmeasure{H_j}(B(w,\epsilon)) = L_j^{-2} \qmeasure{h_j}(B(R_j w,R_j \epsilon)) \geq L_j^{-2} R_j^{\alpha_\LBD} \epsilon^{\alpha_\LBD} = \epsilon^{\alpha_\LBD}.\]

Suppose that $\beta \in (0,2/\alpha_\LBD)$.  It therefore follows from Proposition~\ref{prop:boundary_distance_tail_bounds} that on $E$ and $\epsilon \in (0,\epsilon_0)$ the expected $\beta$-H\"older norm of $(\medianHP{R_j^{-1}\epsilon})^{-1} \metapprox{R_j^{-1}\epsilon}{\cdot}{\cdot}{R_j^{-1} \Gamma_j}$ on $\partial R_j^{-1} \CD_j^z$ (with respect to quantum length) is finite.  Likewise, on $E$ and $\epsilon \in (0,\epsilon_0)$ the expected $\beta$-H\"older norm of $(\medianHP{\wt{R}_j^{-1} \epsilon})^{-1} \metapprox{\wt{R}_j^{-1} \epsilon}{\cdot}{\cdot}{\wt{R}_j^{-1} \wt{\Gamma}_j}$ on $\partial \wt{R}_j^{-1}\wt{\CD}_j^z$ (with respect to quantum length) is also finite.  Note that
\begin{align*}
    (\medianHP{R_j^{-1} \epsilon})^{-1} \metapprox{R_j^{-1} \epsilon}{\cdot}{\cdot}{R_j^{-1}\Gamma_j}
&=  \frac{(\medianHP{\epsilon})^{-1}}{(\medianHP{\epsilon})^{-1}} (\medianHP{R_j^{-1} \epsilon})^{-1} \big( R_j^{-2} \metapprox{\epsilon}{\cdot}{\cdot}{\Gamma_j} \big)
 = R_j^{-2} \frac{\medianHP{\epsilon}}{\medianHP{R_j^{-1} \epsilon}}   \big( (\medianHP{\epsilon})^{-1} \metapprox{\epsilon}{\cdot}{\cdot}{\Gamma_j} \big).
\end{align*}
Lemma~\ref{lem:covering_lemma} implies that $\medianHP{\epsilon} \geq R_j \medianHP{R_j^{-1} \epsilon}$.  Inserting this into the above implies that
\begin{equation}
\label{eqn:metrjbound}
(\medianHP{\epsilon})^{-1} \metapprox{\epsilon}{\cdot}{\cdot}{\Gamma_j} \leq R_j (\medianHP{R_j^{-1} \epsilon})^{-1} \metapprox{R_j^{-1} \epsilon}{\cdot}{\cdot}{R_j^{-1}\Gamma_j}.
\end{equation}
We likewise have that
\begin{equation}
\label{eqn:mettrjbound}
 (\medianHP{\epsilon})^{-1} \metapprox{\epsilon}{\cdot}{\cdot}{\wt{\Gamma}_j} \leq \wt{R}_j (\medianHP{\wt{R}_j^{-1} \epsilon})^{-1} \metapprox{\wt{R}_j^{-1} \epsilon}{\cdot}{\cdot}{\wt{R}_j^{-1}\wt{\Gamma}_j}.
\end{equation}
By the Borel-Cantelli lemma and~\eqref{eqn:l_j_moment}, \eqref{eqn:wt_l_j_moment}, we have for a constant $a > 0$ that $R_j \leq 2^{- aj}$ and $\wt{R}_j \leq 2^{- aj}$ for all $j \in \N$ large enough.  On these events and $E$, we thus have that the expectations of $\sup_{u,v \in \partial \CD_j^z} (\medianHP{\epsilon})^{-1} \metapprox{\epsilon}{u}{v}{\Gamma_j}$ and $\sup_{u,v \in \partial \wt{\CD}_j^z} (\medianHP{\epsilon})^{-1} \metapprox{\epsilon}{u}{v}{\Gamma_j}$ are $O(2^{-a j})$.  Combining everything and summing over $j$ completes the proof of the result.
\end{proof}

\begin{lemma}
\label{lem:bounds_holder}
Suppose that $D \subseteq \C$ is a simply connected domain, $(D,h,x,y) \sim \qdiskL{\gamma}{1}$, $0 < \alpha_\UBD < 2 < \alpha_\LBD$, and $\epsilon_0 > 0$.  Let $E$ be the event that for all $\epsilon \in (0,\epsilon_0)$ and $z \in D$ with $B(z,\epsilon) \subseteq D$ we have that $\epsilon^{\alpha_\LBD} \leq \qmeasure{h}(B(z,\epsilon)) \leq \epsilon^{\alpha_\UBD}$.  For each $\beta > 0$ there exists $\beta_\LBD > 0$ so that for each $\epsilon \in (0,\epsilon_0)$ the following is true.  Let $F$ be the event that for all (prime ends) $a, b \in \partial D$ with $\qbmeasure{h}(\ccwBoundary{a}{b}{\partial D}) = \epsilon$ (resp.\ $\qbmeasure{h}(\cwBoundary{a}{b}{\partial D}) = \epsilon$) we have that $\diam(\ccwBoundary{a}{b}{\partial D}) \geq \epsilon^{\beta_\LBD}$ (resp.\ $\diam(\cwBoundary{a}{b}{\partial D}) \geq \epsilon^{\beta_\UBD}$).  Then $\p[F^c \cap E] = O(\epsilon^\beta)$.
\end{lemma}
\begin{proof}
Suppose that $z \in \ccwBoundary{x}{y}{\partial D}$ is chosen from $\qbmeasure{h}$.  Let $\varphi \colon \D \to D$ be a conformal transformation which takes $-i$ to $x$, $i$ to $y$, and $1$ to $z$.  Let $\wt{h} = h \circ \varphi + Q \log|\varphi'|$.  Then $(\D,h,-i,i) \sim \qdiskL{\gamma}{1}$.  Fix $\epsilon_0 > 0$, $\sigma_\LBD > 1$, and let $\wt{E}$ be the event that
\begin{enumerate}[(i)]
\item for all $\epsilon \in (0,\epsilon_0)$ and $z \in \D$ so that $B(z,\epsilon) \subseteq \D$ we have that $\qmeasure{\wt{h}}(B(z,\epsilon)) \geq \epsilon^{\sigma_\LBD}$ and
\item for all $\epsilon \in (0,\epsilon_0)$ and intervals $I \subseteq \partial \D$ with $|I| = \epsilon^{\sigma_\LBD}$ we have that $\nu_{\wt{h}}(I) \leq \epsilon$.
\end{enumerate}
Then for each $p > 0$ there exists $\sigma_\LBD > 1$ so that $\p[\wt{E}^c] = O(\epsilon_0^p)$.

Suppose that $\epsilon \in (0,\epsilon_0)$ and $a,b \in \partial \D$ are such that $\qbmeasure{\wt{h}}(\ccwBoundary{a}{b}{\partial \D}) = \epsilon$.  Let $w \in \D$ have the same argument as the center of $\ccwBoundary{a}{b}{\partial \D}$ with $\dist(w,\partial \D) = \diam(\ccwBoundary{a}{b}{\partial \D})$ and let $r = \dist(w,\partial \D)/2$.  On $\wt{E} \cap E$ we have that $r \asymp |\ccwBoundary{a}{b}{\partial \D}| \geq \epsilon^{\sigma_\LBD}$.  Thus if $\diam(\varphi(B(w,r)) \leq \epsilon_0$ we have that
\begin{equation}
\label{eqn:diam_area_bounds}
c_0^{-1} \epsilon^{\sigma_\LBD^2} \leq \qmeasure{\wt{h}}(B(w,r)) = \qmeasure{h}(\varphi(B(w,r))) \leq c_0 (\diam(\varphi(B(w,r))))^{\alpha_\UBD}.
\end{equation}
That is, we either have that $\diam(\varphi(B(w,r))) \gtrsim \epsilon^{\sigma_\LBD^2/\alpha_\UBD}$ or $\diam(\varphi(B(w,r))) \geq \epsilon_0$.  In either case, the assertion of the lemma thus follows from elementary harmonic measure considerations.
\end{proof}

\begin{lemma}
\label{lem:quantum_typical_point_diamter_lbd}
Suppose that $D$ is a bounded, simply connected domain and $(D,h,\Gamma,z) \sim \qdiskCarpet{\gamma}{1}$.  Fix $0 < \alpha_\UBD < 2 < \alpha_\LBD$ and $\epsilon_0 > 0$.  Let $E$ be the event that for every $w \in D$ and $\epsilon \in (0,\epsilon_0)$ so that $B(w,\epsilon) \subseteq D$ we have that $\epsilon^{\alpha_\LBD} \leq \qmeasure{h}(B(w,\epsilon)) \leq \epsilon^{\alpha_\UBD}$.  For each $r > 0$ let
\[ \approxball{\epsilon}{z}{r} = \{ w \in \Upsilon : (\medianHP{\epsilon})^{-1} \metapprox{\epsilon}{z}{w}{\Gamma} \leq r\}.\]
There exists $\alpha_\ball > 0$ depending only on $\alpha_\LBD, \alpha_\UBD$ such that uniformly in $\epsilon \in (0,\delta)$, 
\[ \p[E,\ \diam(\approxball{\epsilon}{z}{\delta^{\alpha_\ball}} ) \leq \delta] \to 0 \quad\text{as}\quad \delta \to 0\]
faster than any power of~$\delta$.
\end{lemma}
\begin{proof}
The proof of this lemma is similar to that of Lemma~\ref{lem:interior_point_distance_to_boundary_finite}.  Indeed, we let $\eta$ be an $\SLE_\kappa^0(\kappa-6)$ coupled with $\Gamma$ as a CPI and which is targeted at $z$.  Let $\tau_0 = 0$, $\CD_0 = \CD$, and for each $j \geq 1$ we inductively define $\tau_j$ and $\CD_j$ as follows.  We let $\tau_j$ be the first time $t \geq \tau_{j-1}$ that $\eta([\tau_{j-1},\tau_j])$ disconnects $z$ from $\partial \CD_{j-1}$ and we let $\CD_j$ be the quantum surface parameterized by the component of $\CD \setminus \eta([0,\tau_j])$ containing $z$.  We let $L_j$ be the quantum length of $\partial \CD_j$.  For each $k$, we let $N_k$ be the number of $j$ so that $L_j \in (2^{-k-1},2^{-k}]$.  Then there exists $p \in (0,1)$ so that $N_k$ is stochastically dominated by a geometric random variable with parameter $p$.  Let $k_0 \in \N$ be such that $k \geq k_0$ holds if and only if $2^{-k} \leq \delta$.  Fix $a_1 \in (0,1)$.  Then a union bound implies that $\p[ \exists k \geq k_0 : N_k \geq 2^{a_1 k}] \to 0$ as $\delta \to 0$ faster than any power of $\delta$.

For each $n,k$ we let $\tau_{n,k}$ be the $n$th value of $j \in \N$ so that $L_j \in (2^{-k-1},2^{-k}]$.  Then it follows from Proposition~\ref{prop:boundary_distance_tail_bounds} as in the proof of Lemma~\ref{lem:interior_point_distance_to_boundary_finite} (recall in particular~\eqref{eqn:metrjbound}) that
\[ \p[E,\ (\medianHP{\epsilon})^{-1} \sup_{u,v \in \partial \CD_{\tau_{n_k}}} \metapprox{\epsilon}{u}{v}{\Upsilon} \geq c 2^{-(2/\alpha_\LBD) k}] \to 0 \quad\text{as}\quad c \to \infty\]
faster than any negative power of $c$.  In particular, there exists $a_2 \in (0,1)$ so that
\[ \p[E,\ (\medianHP{\epsilon})^{-1} \sup_{u,v \in \partial \CD_{\tau_{n_k}}} \metapprox{\epsilon}{u}{v}{\Upsilon} \geq 2^{-a_2 k}] \to 0 \quad\text{as}\quad k \to \infty\]
faster than any power of $2^{-k}$.  Combining this with a union bound, we thus see that the probability of the event that $E$ occurs and either $N_k \geq 2^{a_1 k}$ or $(\medianHP{\epsilon})^{-1} \sup_{u,v \in \partial \CD_{\tau_{n_k}}} \metapprox{\epsilon}{u}{v}{\Upsilon} \geq 2^{-a_2 k}$ for some $k \geq  k_0$ tends to $0$ as $\delta \to 0$ faster than any power of $\delta$.   The same also applies to $(\medianHP{\epsilon})^{-1} \sup_{u,v \in \partial \wt{\CD}_{\tau_{n_k}}} \metapprox{\epsilon}{u}{v}{\Upsilon}$.

Let $N$ be the largest $k \in \N$ so that $L_k \geq \delta$.  Then it follows from the above that there exists $a_3 \in (0,1)$ so that the probability of $E$ and the event that $\approxball{\epsilon}{z}{\delta^{a_3}}$ does not contain $\cup_{j=N+1}^\infty \partial \CD_j$ tends to $0$ as $\delta \to 0$ faster than any power of $\delta$ or does not contain an interval of quantum length at least $\delta$ on $\partial \CD_N$.  To complete the proof, it suffices to show that there exists $a_4 > 0$ so that the diameter of such an interval is very likely to be at least $\delta^{a_4}$.  This, in turn, follows from Lemma~\ref{lem:bounds_holder} and as the Radon-Nikodym derivative between $\qdiskCarpet{\gamma}{L_k}$ and $\qdiskL{\gamma}{L_k}$ has a finite $p$th moment for some $p > 1$.
\end{proof}

\begin{figure}[ht!]
\begin{center}
\includegraphics[scale=1]{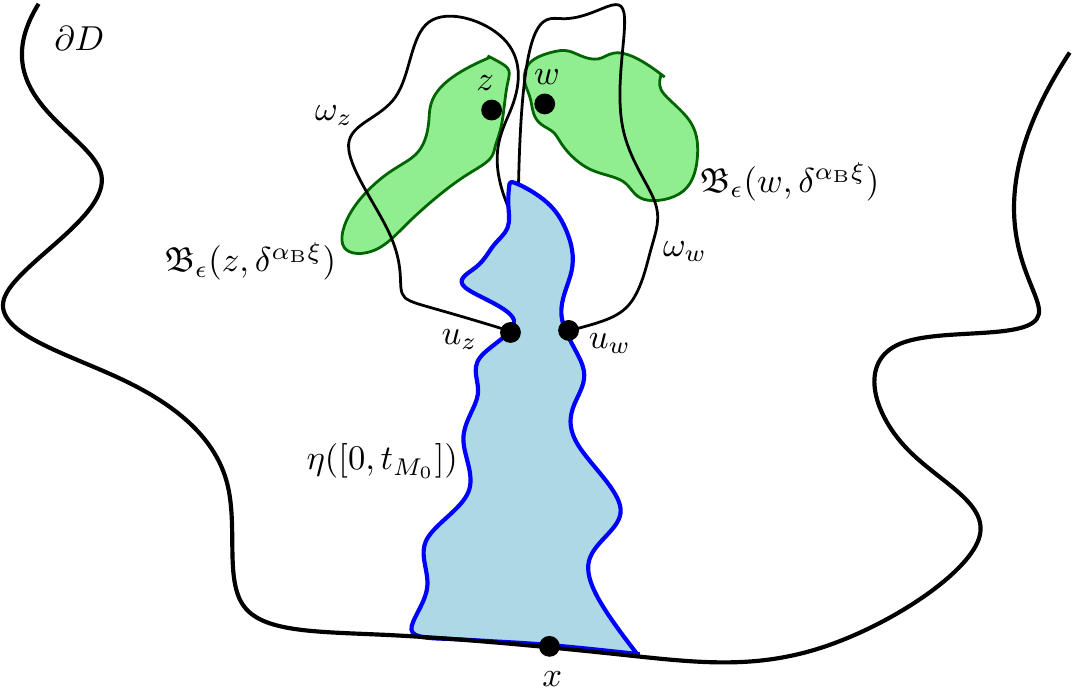}	
\end{center}
\caption{\label{fig:carpet_points_close}  Illustration of the setup to prove Lemma~\ref{lem:distance_two_quantum_typical}.  Shown in blue is $\eta([0,t_{m_0}])$ and $\approxball{\epsilon}{z}{\delta^{\alpha_\ball \xi}}$, $\approxball{\epsilon}{w}{\delta^{\alpha_\ball \xi}}$ in green.}
\end{figure}

\begin{lemma}
\label{lem:distance_two_quantum_typical}
Suppose that $D$ is a bounded, simply connected domain and $(D,h,\Gamma,z) \sim \qdiskCarpet{\gamma}{1}$.  Fix $0 < \alpha_\UBD < 2 < \alpha_\LBD$, $C, \epsilon_0 > 0$, let $\alpha_\PP$ be as in Lemma~\ref{lem:disk_in_disk}, and $\alpha_\HO$ be as in Proposition~\ref{prop:space_filling_on_quantum_disk}.
\begin{itemize}
\item Let $E_1$ be the event that for every $w \in D$ and $\epsilon \in (0,\epsilon_0)$ with $B(w,\epsilon) \subseteq D$ we have that $\epsilon^{\alpha_\LBD} \leq \qmeasure{h}(B(w,\epsilon)) \leq \epsilon^{\alpha_\UBD}$ and also $\qmeasure{h}(D) \leq C$.
\item Assume that $\Gamma$ is generated from a GFF $h^\IG$ on $D$ where the associated exploration tree is rooted at $x \in \partial D$ picked from $\qbmeasure{h}$, let $\eta'$ be the associated space-filling $\SLE_{\kappa'}$, and let $E_2$ be the event that the $\alpha_\HO$-H\"older norm of $\eta'$ (when parameterized by $\qmeasure{h}$) is at most $C$.
\item Let $E_3$ be the event that no loop of $\Gamma$ has an $(\epsilon,\alpha_\PP)$-pinch point for $\epsilon \in (0,\epsilon_0)$.
\end{itemize}
Let $E = E_1 \cap E_2 \cap E_3$.  There exists $\alpha_\KC \in (0,1)$ depending only on $\alpha_\LBD$, $\alpha_\UBD$, $\alpha_\PP$, and $\alpha_\HO$ so that the following is true.  Suppose that $w$ is sampled conditionally independently of $z$ from $\qcarpet{h}{\Upsilon}$ given $h$, $\Gamma$ and let $F = \{|z-w| \leq \delta\}$.  Uniformly in $\epsilon \in (0,\delta)$ we have that
\[ \p[E,\ F,\ (\medianHP{\epsilon})^{-1} \metapprox{\epsilon}{z}{w}{\Upsilon} \geq \delta^{\alpha_\KC}] \to 0 \quad\text{as}\quad \delta \to 0
\]
faster than any power of $\delta$.
\end{lemma}
\begin{proof}
See Figure~\ref{fig:carpet_points_close} for an illustration of the setup for the proof.  Suppose that $x \in \partial D$ is picked from $\qbmeasure{h}$ and let $\eta'$ be a space-filling $\SLE_{\kappa'}$ in $D$ starting from $x$.  We assume that $\Gamma$ and $\eta'$ are coupled together in that they are both generated from the same instance $h^\IG$ of the GFF on $D$ which, in turn, is taken to be conditionally independent of $h$ given $x$.  We assume that $\eta'$ is subsequently parameterized by $\qmeasure{h}$.  Suppose that $w \in \Upsilon$ is picked from $\qcarpet{h}{\Upsilon}$ conditionally independently of $z$ given everything else.  We let $\eta_z$ be the $\SLE_\kappa(\kappa-6)$ process generated by $h^\IG$ and targeted at $z$ so that its trunk $\eta_z'$ is equal to $\eta'$ targeted at $z$.  In particular, we assume that $\eta$ is parameterized according to the amount of quantum area its trunk $\eta_z'$ has disconnected from $z$.  Note that the total amount of time that it takes $\eta'$ to fill $D$ is equal to $\qmeasure{h}(D)$ which, on $E_1$, is at most $C$.  This implies that on $E_1$ the time interval on which $\eta_z'$ is defined is contained in $[0,C]$.  We note that on $E_2$ we have that $\eta'$ has an $\alpha_\HO$-H\"older norm which is at most $C$ and therefore the same is true for $\eta_z'$.

For each $m \in \N$ we let $t_m = C^{-1} m \delta^{1/(\alpha_\HO \alpha_\PP)}$.  Fix $\beta > (\alpha_\HO \alpha_\PP)^{-1}$ and let $\xi > 0$ be such that the conclusion of Lemma~\ref{lem:disk_neighborhood} holds for this set of parameters.  For each $m \in \N$ we let $F_m$ be the event that the event that the conclusion of Lemma~\ref{lem:disk_neighborhood} holds for the quantum surface $\CD_m$ which is parameterized by the component of $D \setminus \eta([0,t_m])$ which contains $z$.  Lemma~\ref{lem:disk_neighborhood} implies that $\p[E_1 \cap F_m^c] = O(\delta^\beta)$.

Let $\tau = \inf\{t \geq 0 : |\eta_z'(t) - z| \leq \delta^{1/\alpha_\PP}\}$.  Let $m_0 \in \N$ be such that $\tau \in (t_{m_0},t_{m_0+1}]$.  On $F = \{|z-w| \leq \delta\}$ we have that $\eta_{[0,\tau]}$ hence also $\eta|_{[0,t_{m_0}]}$ does not disconnect $z$, $w$ (for otherwise $\Gamma$ would have a loop with a $(\delta^{1/\alpha_\PP},\alpha_\PP)$-pinch point).  On $F \cap F_{m_0}$, there exists a path $\omega_z$ (resp.\ $\omega_w$) in $\Upsilon \cap \CD_{m_0}$ which satisfies the assertions of Lemma~\ref{lem:disk_neighborhood} for $z$ (resp.\ $w$).  In particular, if $U_z$ (resp.\ $U_w$) denotes the component of $\CD_{m_0} \setminus \omega_z$ (resp.\ $\CD_{m_0} \setminus \omega_w$) which contains $z$ (resp.\ $w$) then
\[  \sup_{u \in \partial U_z} \metapprox{\epsilon}{u}{\eta([0,t_{m_0}]) \cup \omega_z}{\Upsilon} \leq \delta^\xi \quad\text{and}\quad \sup_{u \in \partial U_w} \metapprox{\epsilon}{u}{\eta([0,t_{m_0}]) \cup \omega_w}{\Upsilon} \leq \delta^\xi \quad\text{on}\quad F \cap F_{m_0}.\] 

Let $\alpha_\ball > 0$ be as in the statement of Lemma~\ref{lem:quantum_typical_point_diamter_lbd}.  Let $G$ be the event that $\diam(\approxball{\epsilon}{z}{\delta^{\alpha_\ball \xi}}) \geq \delta^{\xi}$ and $\diam(\approxball{\epsilon}{w}{\delta^{\alpha_\ball \xi}}) \geq \delta^{\xi}$.  Lemma~\ref{lem:quantum_typical_point_diamter_lbd} implies $\p[E_1 \cap G^c] \to 0$ as $\delta \to 0$ faster than any power of $\delta$.  On $F \cap F_{m_0} \cap G$, we have that $\approxball{\epsilon}{z}{\delta^{\alpha_\ball \xi}} \cap \partial U_z \neq \emptyset$ and $\approxball{\epsilon}{w}{\delta^{\alpha_\ball \xi}} \cap \partial U_w \neq \emptyset$.  By combining everything we see that
\[  \metapprox{\epsilon}{z}{\eta([0,t_{m_0}])}{\Upsilon} \leq 2\delta^{\xi} \quad\text{and}\quad \metapprox{\epsilon}{w}{\eta([0,t_{m_0}])}{\Upsilon} \leq 2\delta^{\xi} \quad\text{on}\quad F \cap F_{m_0} \cap G.\] 

Moreover, let $u_z$ (resp.\ $u_w$) be a point in $\partial U_z$ (resp.\ $\partial U_w$) where $\omega_z$ (resp.\ $\omega_w$) hits $\eta([0,t_{m_0}])$.  On $F \cap F_{m_0}$, we know that $\diam(U_z) \leq \delta^{\xi}$ and $\diam(U_w) \leq \delta^{\xi}$.  Lemma~\ref{lem:bounds_holder} implies that by possibly decreasing the value of $\xi > 0$ we have that the quantum length of $\partial U_z \cap \eta([0,t_{m_0}])$, $\partial U_w \cap \eta([0,t_{m_0}])$, and the shorter arc of $\partial \CD_{m_0}$ from $u_z$ to $u_w$ are all at most $\delta^{\xi}$.  For each $m \in \N$, we let $H_m$ be the event that the event of Proposition~\ref{prop:boundary_distance_tail_bounds} holds for $\CD_m$.  Using the aforementioned bounds, by possibly decreasing $\xi > 0$ if necessary, we then have that
\[  \metapprox{\epsilon}{z}{u_z}{\Upsilon} \leq 2\delta^{\xi} \quad\text{and}\quad \metapprox{\epsilon}{w}{u_w}{\Upsilon} \leq 2\delta^{\xi} \quad\text{on}\quad F \cap F_{m_0} \cap G \cap H_{m_0}\]
and likewise
\[ \metapprox{\epsilon}{u_z}{u_w}{\Upsilon} \leq \delta^{\xi} \quad\text{on}\quad F \cap F_{m_0} \cap G \cap H_{m_0}.\]
Altogether, this implies that
\[ \metapprox{\epsilon}{z}{w}{\Upsilon} \leq 5\delta^{\xi} \quad\text{on}\quad F \cap F_{m_0} \cap G \cap H_{m_0}.\]
Let $M_0 = C^2 \delta^{-(\alpha_\PP \alpha_\HO)^{-1}}$.  On $E$, we have that $[0,t_{M_0}]$ contains the interval on which $\eta_z$ is defined.  Moreover,
\begin{align*}
  &\p\!\left[ E \cap F \cap \bigcup_{m=1}^{M_0} F_m^c \right] \leq \sum_{m=1}^{M_0} \p[ E \cap F \cap F_m^c] \to 0 \quad\text{as}\quad \delta \to 0,\\
  &\p\!\left[ E \cap F \cap \bigcup_{m=1}^{M_0} H_m^c \right] \leq \sum_{m=1}^{M_0} \p[ E \cap F \cap H_m^c] \to 0 \quad\text{as}\quad \delta \to 0, \quad\text{and}\\
  &\p[ E \cap F \cap G^c] \to 0 \quad\text{as}\quad \delta \to 0
\end{align*}
faster than any power of $\delta$.  Combining everything completes the proof.
\end{proof}

We can now complete the proof of Proposition~\ref{prop:interior_tightness}.

\begin{proof}[Proof of Proposition~\ref{prop:interior_tightness}]
Let $(z_n)$ be an i.i.d.\ sequence chosen from $\qcarpet{h}{\Upsilon}$.  We let $\alpha_\net > 0 $ be as in the statement of Lemma~\ref{lem:quantum_measure_points_dense}.  For each $k \in \N$, we let $N_j = 2^{\alpha_\net j}$.  Then Lemma~\ref{lem:quantum_measure_points_dense} implies that there a.s.\ exists $j_0 \in \N$ so that $j \geq j_0$ implies that $z_1,\ldots,z_{N_j}$ is a $2^{-j}$-net of $\Upsilon$.  That is, $\Upsilon \subseteq \cup_{k=1}^{N_j} B(z_k,2^{-j})$.  Let $N_\epsilon = \lfloor \log_2 \epsilon^{-2} \rfloor$.  Lemma~\ref{lem:distance_two_quantum_typical} and the Borel-Cantelli lemma imply, by possibly increasing the value of $j_0$, we have $j_0 \leq j \leq N_\epsilon$ implies that if $|z_\ell - z_k| \leq 2^{-j}$ for $1 \leq \ell,k \leq N_j$ then 
\begin{align}
\label{eqn:normalized_distance_good}
(\medianHP{\epsilon})^{-1}\metapprox{\epsilon}{z_\ell}{z_k}{\Gamma} \leq 2^{-\alpha_\KC j}.
\end{align}

Suppose that $z,w \in \Upsilon$ and fix $J \in \N$ so that $2^{-J-1} \leq |z-w| < 2^{-J}$.  We assume that $j_0 \leq J \leq N_\epsilon$.  Then we can find sequences $m_j,n_j$ with $1 \leq m_j, n_j \leq N_j$ for every $j \leq J$ so that $z \in B(z_{m_j},2^{-j})$ and $w \in B(z_{n_j},2^{-j})$.  It therefore follows from~\eqref{eqn:normalized_distance_good} that
\begin{align}
\label{eqn:z_bound}
(\medianHP{\epsilon})^{-1}\metapprox{\epsilon}{z}{z_{m_{J}}}{\Gamma}
\leq \sum_{j={J}}^{N_\epsilon} (\medianHP{\epsilon})^{-1} \metapprox{\epsilon}{z_{m_{j+1}}}{z_{m_j}}{\Gamma}
\leq 2^{-\alpha_\KC J}.
\end{align}
We similarly have that
\begin{align}
\label{eqn:w_bound}
(\medianHP{\epsilon})^{-1}\metapprox{\epsilon}{w}{z_{n_{J}}}{\Gamma}
\leq \sum_{j={J}}^{N_\epsilon}(\medianHP{\epsilon})^{-1} \metapprox{\epsilon}{z_{n_{j+1}}}{z_{n_j}}{\Gamma}
\leq 2^{-\alpha_\KC J}.
\end{align}
Note also that~\eqref{eqn:normalized_distance_good} implies
\begin{align}
\label{eqn:typ_bound}
(\medianHP{\epsilon})^{-1}\metapprox{\epsilon}{z_{m_{J}}}{z_{n_J}}{\Gamma} \leq 2^{-\alpha_\KC J}.
\end{align}
Combining~\eqref{eqn:z_bound}, \eqref{eqn:w_bound}, and~\eqref{eqn:typ_bound} implies that
\[ (\medianHP{\epsilon})^{-1}\metapprox{\epsilon}{z}{w}{\Gamma} \leq 2^{-\alpha_\KC J}.\]
This proves what we want in the case that $J \geq j_0$.  The result follows for $J < j_0$ by considering a finite chain of points $x_0 = z,\ldots,x_m = w$ in $\Upsilon$ with $|x_i - x_{i+1}| \leq 2^{-j_0}$ and applying the triangle inequality.
\end{proof}

\section{Positive definiteness and geodesics}
\label{sec:pos_def}

The purpose of this section is to prove the positive definiteness of a subsequential limit $\met{\cdot}{\cdot}{\Gamma}$ constructed in Section~\ref{sec:interior_tightness}.  We will first show in Section~\ref{subsec:simultaneous_limits} that for every sequence $(\epsilon_j)$ of positive real numbers decreasing to $0$ one can find a single subsequence $(\epsilon_{j_k})$ which can be used to define a subsequentially limiting pseudometric for every non-trivial simply connected domain simultaneously, provided one restricts to paths which stay away from the domain boundary.  In Section~\ref{subsec:locality}, we will develop a notion of a ``local set'' for a $\CLE_\kappa$ instance $\Gamma$ equipped with a subsequential limit $\met{\cdot}{\cdot}{\Gamma}$.  We will then prove the positive definiteness of the subsequential limit in Section~\ref{subsec:proof_of_pos_def} (see also the beginning of Section~\ref{subsec:proof_of_pos_def} for an outline of the steps used to prove this result).  Finally, we will deduce the geodesic property of a subsequential limit in Section~\ref{subsec:geodesics} and prove the comparability of quantiles defined in the finite volume and half-planar setting in Section~\ref{subsec:quantiles_comparable}.

\subsection{Simultaneous limits}
\label{subsec:simultaneous_limits}

Throughout, we will need to restrict our attention to a countable family of open sets such that for every open set $D \subseteq \C$ there exists $D_1 \subseteq D_2 \subseteq \cdots$ in the family so that $D = \cup_{n=1}^\infty D_n$.  There are various choices that one could make and the particular one is not important for what follows, but for the sake of concreteness we will consider the set $\dyad$ of domains $U \subseteq \C$ for which there exists $k \in \N$ and closed squares $S_1,\ldots,S_n$ with side length $2^{-k}$ and corners in $(2^{-k} \Z)^2$ so that $U = \interior{\cup_{j=1}^n S_j}$.  For a domain $D \subseteq \C$, we let $\dyad(D)$ be the set of $U \in \dyad$ with $\closure{U} \subseteq D$.  Finally, we let $\dyad_2(D)$ be the collection of pairs $(U_1,U_2)$ with $U_1,U_2 \in \dyad(D)$ and $\closure{U_1} \subseteq U_2$.

\begin{lemma}
\label{lem:joint_subsequential}
For every sequence $(\epsilon_k)$ of positive numbers decreasing to $0$ there exists a subsequence $(\epsilon_{j_k})$ so that the following is true.  Suppose that $D \subseteq \C$ is a simply connected domain and $(U_1,U_2) \in \dyad_2(D)$.  For $z,w \in U_2$, let $\metapproxres{\epsilon}{U_2}{z}{w}{\Gamma} = \inf_\omega \lebneb{\epsilon}(\omega)$ where the infimum is over paths $\omega \colon [0,1] \to \Upsilon \cap U_2$ with $\omega(0) = z$, $\omega(1) = w$.  Let $E$ be the event that there exists a path in $(U_2 \setminus U_1) \cap \Upsilon$ with positive distance from $\partial U_1 \cup \partial U_2$ and which disconnects $\partial U_1$ from $\partial U_2$.  Then $(\medianHP{\epsilon_{j_k}})^{-1} \metapproxres{\epsilon_{j_k}}{U_2}{\cdot}{\cdot}{\Gamma}|_{U_1^2} \one_E$ converges as $k \to \infty$ weakly with respect to $(\funcset_4,\funcmet_4)$.
\end{lemma}

We emphasize that in the statement of Lemma~\ref{lem:joint_subsequential}, the subsequence $(\epsilon_{j_k})$ of $(\epsilon_k)$ does not depend on the choice of $D$ and $(U_1,U_2) \in \dyad_2(D)$.

\begin{proof}[Proof of Lemma~\ref{lem:joint_subsequential}]
Suppose that $D \subseteq \C$ is a simply connected domain and $\Gamma$ is a $\CLE_\kappa$ on $D$.  Let $(D,h)$ have law $\qdiskWeighted{\gamma}{\ell}$ and let $w \in D$, $x \in \partial D$, respectively, be independently picked from $\qmeasure{h}, \qbmeasure{h}$.  Let $(U_1,U_2) \in \dyad_2(D)$ and let $E$ be as in the statement of the lemma.  Let $\eta$ be an $\SLE_\kappa(\kappa-6)$ in $D$ which is coupled with $\Gamma$ as a CPI starting from $x$ and targeted at $w$.  Let $\tau$ be the first time that $w$ is disconnected from $\partial D$ and let $U_w$ be the component of $D \setminus \eta([0,\tau])$ which contains $w$.  Let $\ell$ be the quantum length of $\partial U_w$.  Given $\ell$, the law of the quantum surface described by $(U_w,h)$ is $\qdiskWeighted{\gamma}{\ell}$.  Since the conditional law given $\eta|_{[0,\tau]}$ of the loops $\Gamma_w$ of $\Gamma$ contained in $U_w$ is that of a $\CLE_\kappa$ in $U_w$ (on the event that $U_w$ is not a loop of $\Gamma$), it follows from Proposition~\ref{prop:interior_tightness} that $(\medianHP{\epsilon})^{-1} \metapprox{\epsilon}{\cdot}{\cdot}{\Gamma_w}$ is tight with respect to $(\funcset_4,\funcmet_4)$.

Let $(w_j)$ be a sequence chosen from $\qmeasure{h}$ and, for each $j$, we let $\eta_j$ be an $\SLE_\kappa(\kappa-6)$ in $D$ coupled with $\Gamma$ as a CPI starting from $x$ and targeted at $w_j$.  We assume that the $w_j,\eta_j$ are taken to be conditionally independent given everything else.  Let $U_{w_j}$ and $\Gamma_{w_j}$ be as above with $w_j,\eta_j$ in place of $w,\eta$.  By Proposition~\ref{prop:cpi_path_close}, the conditional probability given $\Upsilon$ that a CPI is close to any fixed simple path in $\Upsilon$ is positive, hence it follows that on $E$ if we let $J$ be the smallest $j$ so that $U_1 \subseteq U_{w_j} \subseteq U_2$ then $J < \infty$ a.s.  We note on $E$ that $\metapproxres{\epsilon}{U_2}{\cdot}{\cdot}{\Gamma}|_{U_1^2} \leq \metapprox{\epsilon}{\cdot}{\cdot}{\Gamma_{w_J}}$.  From this, it follows that $\metapproxres{\epsilon}{U_2}{\cdot}{\cdot}{\Gamma}|_{U_1^2} \one_E$ is tight with respect to $(\funcset_4,\funcmet_4)$.

What we have explained above implies that for each simply connected domain $D \subseteq \C$, $(U_1,U_2) \in \dyad_2(D)$, and every sequence $(\epsilon_k)$ of positive numbers decreasing to $0$ there exists a subsequence $(\epsilon_{j_k})$ so that $(\medianHP{\epsilon_{j_k}})^{-1}\metapproxres{\epsilon_{j_k}}{U_2}{\cdot}{\cdot}{\Gamma}|_{U_1^2} \one_E$ converges weakly with respect to $(\funcset_4,\funcmet_4)$.  Indeed, this follows since the collection of such $(U_1,U_2)$, $D$ is countable.  Our aim now is to upgrade this statement to show that it is possible to make a choice of subsequence of $(\epsilon_k)$ which does not depend on $D$, $U_1$, or~$U_2$.

Let $(\epsilon_k)$ be any sequence of positive numbers which decrease to $0$.  First of all, by passing to a diagonal subsequence, we can apply the above argument to get a subsequence $(\epsilon_{j_k})$ of $(\epsilon_k)$ so that the following is true.  For every simply connected domain $D \subseteq \C$ which is in $\dyad$ and $(U_1,U_2) \in \dyad_2(D)$ we have that $(\medianHP{\epsilon_{j_k}})^{-1} \metapproxres{\epsilon_{j_k}}{U_2}{\cdot}{\cdot}{\Gamma}|_{U_1^2} \one_{E}$ converges weakly as $k \to \infty$ with respect to $(\funcset_4,\funcmet_4)$.

We now suppose that $D \subseteq \C$ is a simply connected domain and $(U_1,U_2) \in \dyad_2(D)$.  To finish the proof, we will show that $(\medianHP{\epsilon_{j_k}})^{-1} \metapproxres{\epsilon_{j_k}}{U_2}{\cdot}{\cdot}{\Gamma}|_{U_1^2} \one_{E}$ converges weakly as $k \to \infty$ with respect to $(\funcset_4,\funcmet_4)$.

Let $(D_m)$ be a sequence of simply connected domains in $\dyad$ so that $\cup_m D_m = D$.  We assume that $\closure{U_2} \subseteq D_1$.  For each $m$, we let $\Gamma_m$ be a $\CLE_\kappa$ on $D_m$.  In order to prove the result, it suffices to show that the total variation distance between the law of the loops $\Gamma^{U_2}$ of $\Gamma$ which intersect $U_2$ and that of the loops $\Gamma_m^{U_2}$ of $\Gamma_m$ which intersect $U_2$ tends to $0$ as $m \to \infty$.  We assume that $\Gamma$ and $\Gamma_m$ are coupled together so that they all come from the same instance of the Brownian loop-soup on $\C$.  Then we note that each loop $\CL_m \in \Gamma_m^{U_2}$ is contained in a loop $\CL \in \Gamma^{U_2}$.  Moreover, $\CL_m = \CL$ if $\CL$ does not intersect $D \setminus D_m$.  It therefore suffices to show that there a.s.\ exists $m_0 \in \N$ so that $m \geq m_0$ implies that none of the loops of $\Gamma^{U_2}$ intersect $D \setminus D_m$.  Let $\varphi$ be a conformal transformation from~$D$ to~$\D$.  If the above did not hold, then $\varphi(\Gamma^{U_2})$ (the set of loops of the $\CLE_\kappa$ given by $\varphi(\Gamma)$ which intersect $\varphi(U_2)$) would with positive probability contain a loop which intersects $\partial \D$.  This, in turn, cannot happen as $\varphi(U_2) \cap \partial \D = \emptyset$.
\end{proof}

Suppose that we have the setup as in Lemma~\ref{lem:joint_subsequential}.  By taking as limit as $k \to \infty$, we have that the family $\metres{U_2}{\cdot}{\cdot}{\Gamma}$ is defined simultaneously for all $U_2 \in \dyad(D)$.  We note that $\metres{U_2}{\cdot}{\cdot}{\Gamma}$ is decreasing as $U_2$ increases.  Suppose that $(U_n)$ is an increasing sequence of domains in $D$, each of which is in $\dyad$, so that $D = \cup_n U_n$.  We let $\metplus{\cdot}{\cdot}{\Gamma}$ be the limit as $n \to \infty$ of $\metres{U_n}{\cdot}{\cdot}{\Gamma}$.  We note that $\metplus{\cdot}{\cdot}{\Gamma}$ does not depend on $(U_n)$.  Indeed, suppose that $(\wt{U}_n)$ is another sequence of such domains.  Then for every $n$ there exists $m$ so that $U_n \subseteq \wt{U}_m$ hence $\metres{U_n}{\cdot}{\cdot}{\Gamma} \geq \metres{\wt{U}_m}{\cdot}{\cdot}{\Gamma}$ and likewise for every $n$ there exists $m$ so that $\wt{U}_n \subseteq U_m$ hence $\metres{\wt{U}_n}{\cdot}{\cdot}{\Gamma} \geq \metres{U_m}{\cdot}{\cdot}{\Gamma}$.  Moreover, we note that each for each $U_1 \in \dyad(D)$ fixed there a.s.\ exists $n_0 \in \N$ so that $n \geq n_0$ implies that the event $E$ from Lemma~\ref{lem:joint_subsequential} holds for $(U_1,U_{2,n})$ as the loops of a $\CLE_\kappa$ with $\kappa \in (8/3,4)$ a.s.\ do not hit the domain boundary.

\newcommand{\zerop}[1]{\mathrm{Zero}^+(#1)}
\newcommand{\zero}[1]{\mathrm{Zero}(#1)}

For a simply connected domain $D \subseteq \C$ and a $\CLE_\kappa$ process $\Gamma$ on $D$ equipped with $\metplus{\cdot}{\cdot}{\Gamma}$, we let $\zerop{\Gamma}$ be the closure of the set of pairs of points $z,w \in D$ such that there exists $U_2 \in \dyad(D)$ so that $\metres{U_2}{z}{w}{\Gamma} = 0$.  We also set $\zero{\Gamma}$ be the closure of the set of pairs of points in $z,w \in D$ such that $\metplus{z}{w}{\Gamma} = 0$.

\begin{lemma}
\label{lem:zero_length_conf}
Suppose that $D$, $\wt{D}$ are simply connected domains in $\C$ and $\varphi \colon D \to \wt{D}$ is a conformal map.  Suppose that $\Gamma$ (resp.\ $\wt{\Gamma}$) is a $\CLE_\kappa$ on $D$.  Then $\varphi(\zerop{\Gamma}) \stackrel{d}{=} \zerop{\wt{\Gamma}}$.  The same is true with $\zero{\Gamma}$, $\zero{\wt{\Gamma}}$ in place of $\zerop{\Gamma}$, $\zerop{\wt{\Gamma}}$ in the case that $\sup_{z \in D} |\varphi'(z)| < \infty$.  
\end{lemma}

We emphasize that in the statement of Lemma~\ref{lem:zero_length_conf} when we consider $\zerop{\Gamma}$, $\zerop{\wt{\Gamma}}$ we have restricted to points that can be connected by a ``zero length path'' which stays away from the domain boundary.  The particular application that we have in mind for the case $\sup_{z \in D} |\varphi'(z)| < \infty$ is when $\varphi$ is a translation or scaling map.

\begin{proof}[Proof of Lemma~\ref{lem:zero_length_conf}]
Let us first handle the case of $\zerop{\Gamma}$, $\zerop{\wt{\Gamma}}$.  We assume that $\wt{\Gamma} = \varphi(\Gamma)$.  Using the Skorokhod representation theorem for weak convergence and for each $k$ writing $\Gamma_k$ (resp.\ $\wt{\Gamma}_k$) for a copy of $\Gamma$ (resp.\ $\wt{\Gamma}$) with $\wt{\Gamma}_k = \varphi(\Gamma_k)$, we can couple everything together onto a common probability space so that $(\medianHP{\epsilon_{j_k}})^{-1} \metapproxres{\epsilon_{j_k}}{U_2}{\cdot}{\cdot}{\Gamma_k}$ and $(\medianHP{\epsilon_{j_k}})^{-1} \metapproxres{\epsilon_{j_k}}{\wt{U}_2}{\cdot}{\cdot}{\wt{\Gamma}_k}$ a.s.\ converge as $k \to \infty$ for each $U_2 \in \dyad(D)$, $\wt{U}_2 \in \dyad(\wt{D})$ where the subsequence is as in Lemma~\ref{lem:joint_subsequential}.  We may further assume that the $\Gamma_k$, $\wt{\Gamma}_k$ converge and call their limit $\Gamma$, $\wt{\Gamma}$ with carpets $\Upsilon$, $\wt{\Upsilon}$, respectively.

Suppose that $U_2 \in \dyad(D)$ and $z,w \in \Upsilon$ satisfy $\metres{U_2}{z}{w}{\Gamma} = 0$.  Fix $\xi > 0$.  Then there a.s.\ exists $k_0 \in \N$ so that for every $k \geq k_0$ there is a path $\omega_k \colon [0,1] \to \Upsilon \cap U_2$ with $\omega_k(0) = z$ and $\omega_k(1) = w$ so that $(\medianHP{\epsilon_{j_k}})^{-1} \lebneb{\epsilon_{j_k}}(\omega_k) \leq \xi$.  Now suppose that $\wt{U}_2 \in \dyad(\wt{D})$ is such that $\varphi(U_2) \subseteq \wt{U}_2$.  As $|\varphi'|$ is bounded on $U_2$ it follows that there exists $C > 0$ depending only on $U_2$ and $\varphi$ so that the $(\medianHP{\epsilon_{j_k}})^{-1} \lebneb{\epsilon_{j_k}}(\varphi(\omega_k)) \leq C \xi$.  As $\varphi(\eta_k)$ connects $\varphi(z)$ to $\varphi(w)$ and $\xi > 0$ was arbitrary, it follows that $\metres{\wt{U}_2}{\varphi(z)}{\varphi(w)}{\wt{\Gamma}} = 0$.  This implies that $\varphi(\zerop{\Gamma}) \subseteq \zerop{\wt{\Gamma}}$.  The same argument applies with $\varphi^{-1}$ in place of $\varphi$ to see that $\zerop{\wt{\Gamma}} \subseteq \varphi(\zerop{\Gamma})$, which proves the result.

In the case that $\sup_{z \in D} |\varphi'(z)| < \infty$, the same argument applies with $\zero{\Gamma}$, $\zero{\wt{\Gamma}}$ in place of $\zerop{\Gamma}$, $\zerop{\wt{\Gamma}}$. 
\end{proof}

\subsection{Locality}
\label{subsec:locality}

We will assume that we have taken a subsequence as in Lemma~\ref{lem:joint_subsequential} so that we have $\metplus{\cdot}{\cdot}{\cdot}$ as described just before Lemma~\ref{lem:zero_length_conf} is defined for every non-trivial simply connected domain simultaneously.  Although we not yet shown that $\metplus{\cdot}{\cdot}{\cdot}$ is a metric as we have not finished proving that it is positive definite, we can still define the $\metplus{\cdot}{\cdot}{\cdot}$-length of a path in the same way that the length is defined for a metric.  For $U \subseteq D$ open, we can also define the interior-internal pseudometric on $\closure{U}$ to be the infimum of the $\metplus{\cdot}{\cdot}{\cdot}$-lengths of paths which are contained in $U$ except possibly at their endpoints.  Throughout, if $K \subseteq \closure{D}$ is closed we let $K^*$ be the closure of the union of $K$ and the points surrounded by loops of $\Gamma$ which intersect $K$.  If $U \subseteq D$ is open, we let $U^* = D \setminus (D \setminus U)^*$.

\begin{definition}
\label{def:local}
Suppose that $D \subseteq \C$ is a simply connected domain and $\Gamma$ is a $\CLE_\kappa$ on $D$.  Let $K \subseteq \closure{D}$ be a random closed set which is coupled with $\Gamma$ so that $K \cup \partial D$ is a.s.\ connected.  We say that $K$ is \emph{local} for $(\Gamma,\metplus{\cdot}{\cdot}{\Gamma})$ if the following is true.  For each $\epsilon > 0$, we let $\CF_K^\epsilon$ be the $\sigma$-algebra which is generated by $K$, the loops of $\Gamma$ which intersect~$K$, and the interior-internal pseudometric associated with $\metplus{\cdot}{\cdot}{\Gamma}$ in the $\epsilon$-neighborhood of $K$.  Let $\CF_K = \cap_{\epsilon > 0} \CF_K^\epsilon$.  Then the conditional law of $\Gamma$ and the interior-internal pseudometric associated with $\metplus{\cdot}{\cdot}{\Gamma}$ in each component $U$ of $D \setminus K^*$ given $\CF_K$ is independently that of $(\Gamma_U, \metplus{\cdot}{\cdot}{\Gamma_U})$ where $\Gamma_U$ is a $\CLE_\kappa$ in $U$.
\end{definition}

We now give a useful criterion for checking locality.

\begin{lemma}
\label{lem:loc_characterization}
Suppose that $D \subseteq \C$ is a simply connected domain and let $\Gamma$ be a $\CLE_\kappa$ on $D$.  Let $K \subseteq \closure{D}$ be a random closed set which is coupled with $\Gamma$ so that $K \cup \partial D$ is a.s.\ connected.  Then $K$ is local for $(\Gamma, \metplus{\cdot}{\cdot}{\Gamma})$ if and only if the following is true.  Suppose that $U \subseteq D$ is open. For each $\epsilon > 0$, let $(D \setminus U)^\epsilon$ be the $\epsilon$-neighborhood of $D \setminus U$ and let $\Gamma_\epsilon$ be the loops of $\Gamma$ contained in $(D \setminus U)^\epsilon$.  Let $\CF_U^\epsilon$ be the $\sigma$-algebra generated by the loops of $\Gamma$ which intersect $D \setminus U$ and $\metplus{\cdot}{\cdot}{\Gamma_\epsilon}$ and let $\CF_U = \cap_{\epsilon > 0} \CF_U^\epsilon$.  Given $\CF_U$, the event that $K^* \cap U = \emptyset$ is independent of the loops $\Gamma_U$ of $\Gamma$ contained in $U^*$ and $\metplus{\cdot}{\cdot}{\Gamma_U}$.
\end{lemma}
\begin{proof}
If $K$ is local for $\Gamma$, then it is obvious that the property in the lemma statement holds.

Suppose instead that $K$ satisfies the property in the lemma statement.  For each $n \in \N$, let $K_n$ be the closure of the union of $K$ and the dyadic squares with side length $2^{-n}$ and corners in $(2^{-n} \Z)^2$ which intersect $K$.  Let $\CF_n$ be the $\sigma$-algebra generated by $K_n$, the loops of $\Gamma$ which intersect $K_n$, and the $\metplus{\cdot}{\cdot}{\Gamma}$-length of every path which is contained in $K_n$.  The assumption from the lemma statement implies that the conditional law of $\Gamma$ and the interior-internal pseudometric in each component $U$ of $D \setminus K_n^*$ is given by a $\CLE_\kappa$ in $U$ coupled with $\metplus{\cdot}{\cdot}{\Gamma_U}$.  The backward martingale convergence theorem implies that the conditional law given $\cap_n \CF_n$ of the loops of $\Gamma$ and the interior-internal pseudometric in each component $U$ of $D \setminus K^*$ is given by a $\CLE_\kappa$ $\Gamma_U$ in $U$ and $\metplus{\cdot}{\cdot}{\Gamma_U}$.  This implies the result as $\CF_K \subseteq \cap_n \CF_n$.
\end{proof}

\subsection{Statement and proof of positive definiteness}
\label{subsec:proof_of_pos_def}

\begin{proposition}
\label{prop:overall_no_zero_length}
Suppose that $\Gamma_\D$ is a $\CLE_\kappa$ in $\D$ and let $D$ be the set of points surrounded by the loop $\CL \in \Gamma_D$ which surrounds $0$.  Given $\CL$, let $\Gamma$ be a $\CLE_\kappa$ in $D$.  Almost surely, for all $z,w \in \Upsilon$ distinct we have that $\met{z}{w}{\Gamma} > 0$.  That is, $\met{\cdot}{\cdot}{\Gamma}$ is a.s.\ a metric on $\Upsilon$.
\end{proposition}

As we will explain below, using a CPI exploration and Proposition~\ref{prop:cpi_path_close} we will see that in order to prove Proposition~\ref{prop:overall_no_zero_length} it suffices to establish the following proposition.

\begin{proposition}
\label{prop:no_zero_crossings}
Suppose that $\eta$ is an $\SLE_\kappa^1(\kappa-6)$ in $\D$ from $-i$ to $i$ coupled with a $\CLE_\kappa$ process $\Gamma_\D$ on $\D$ as a CPI.  On the event that $\eta$ disconnects $0$ from $\partial \D$ and the component $D$ of $\D \setminus \eta$ is not a loop of $\Gamma_\D$ and is surrounded counterclockwise by $\eta$, we let $\Gamma$ be a $\CLE_\kappa$ in $D$.  Otherwise, we let $D = \emptyset$.  On $D \neq \emptyset$, we almost surely have for all $z,w \in \partial D$ distinct that $\metplus{z}{w}{\Gamma} > 0$.  That is, $\metplus{\cdot}{\cdot}{\Gamma}|_{(\partial D)^2}$ is a.s.\ a metric on $\partial D$.
\end{proposition}

Let us now describe the main steps to prove Proposition~\ref{prop:no_zero_crossings}.

\begin{enumerate}
\item[Step 1.] We first collect some preliminary lemmas in Section~\ref{subsubsec:prelim_lemmas}.  In particular, in Lemma~\ref{lem:zero_metric_ball_does_not_hit} we will work in the following half-planar setup.  We let $\CC = (\C,h,0,\infty) \sim \qconeW{\gamma}{4}$ which is decorated by independent $\SLE_\kappa$ curves $\eta_\pm$ so that the surface $\CH_+$ parameterized by the component of $\C \setminus (\eta_- \cup \eta_+)$ which is to the left of $\eta_+$ is a quantum half-plane.  We assume that $\eta_\pm$ are parameterized according to quantum length.  We let $\Gamma_+$ be a $\CLE_\kappa$ in $\CH_+$, $\Upsilon_+$ its carpet, and we let $\metplus{\cdot}{\cdot}{\Gamma_+}$ be as above.  Then we will show that for any $u < v < r < s$ fixed we have $\metplus{\eta_+([u,v])}{\eta_+([r,s])}{\Gamma_+} > 0$ with positive probability and $\metplus{\eta_+([u,v])}{\eta_+([r,s])}{\Gamma_+} > 0$ has probability tending to $1$ as $s \downarrow r$.  We next show in Lemma~\ref{lem:zero_metric_ball_is_local} that if $\Gamma$ is a $\CLE_\kappa$ on a simply connected domain $D \subseteq \C$, $x,y \in \partial D$ are distinct, then the closure $K$ of the set of points $z$ in the carpet $\Upsilon$ of $\Gamma$ with $\metplus{z}{\ccwBoundary{x}{y}{\partial D}}{\Gamma} = 0$ is local and in Lemma~\ref{lem:in_between_stuff_is_small} that the non-loop points on $\partial K$ have zero harmonic measure in any component of $D \setminus K^*$.  We also show in Lemma~\ref{lem:no_zero_length_on_boundary} that if $D$ is as in Proposition~\ref{prop:no_zero_crossings} and $(D,h) \sim \qdiskL{\gamma}{1}$ then the quantum length of the set of $z \in \partial D$ for which there exists $w \in \closure{D} \setminus \{z\}$ with $\metplus{z}{w}{\Gamma_+} = 0$ is a.s.\ zero, the analogous statement holds for each loop $\CL \in \Gamma$, and the same statements hold with harmonic measure in place of quantum length.
\item[Step 2.] We next define and establish some properties of the ``shields'' in Section~\ref{subsubsec:shield_def}.  We suppose we are working on a Jordan domain $D$ as in the statement of Proposition~\ref{prop:no_zero_crossings} and we have a $\CLE_\kappa$ process $\Gamma$ in $D$ and let $\Upsilon$ be its carpet.  Roughly speaking, a shield is a closed, connected set $A$ so that the following is true.  If $O \subseteq D$ is open, $a_1,a_2 \in O \cap \Upsilon$, and there exists $\delta > 0$ so that $a_1,a_2$ are in different components of $O^\delta \setminus A$ where $O^\delta$ is the $\delta$-neighborhood of $O$, then $\metplus{a_1}{a_2}{\Gamma_O} > 0$ where $\Gamma_O$ are the loops of $\Gamma$ contained in~$O^*$.  See Figure~\ref{fig:shield_def} for an illustration.  This condition should be thought of as giving that a ``zero-length path'' cannot pass through~$A$.  Such shields exist because~$\partial K$ from above can be used to build a set which has this property.
\item[Step 3.] We fix a countable dense set $(r_k)$ of $D$ and, for each $k$, let $\CL_k$ be the loop which surrounds~$r_k$ (we note that there a.s.\ exists such a loop for all $k$ simultaneously).  For $i,j \in \N$ we say that $i \sim_0 j$ if $\CL_i$, $\CL_j$ can be connected by a shield and let $\sim$ be the finest equivalence relation which contains $\sim_0$. We also let $A_k$ be the closure of the union of the points surrounded by the $\CL_i$ so that $i \sim k$.  We suppose that $z_1,\ldots,z_4 \in \partial D$ are distinct points given in counterclockwise order.  We will show in Lemma~\ref{lem:ak_cross_finite} that for each $\delta > 0$ there can be at most finitely many such sets $A_k$ which are not contained in the $\delta$-neighborhood of $\ccwBoundary{z_2}{z_3}{\partial D}$.  In Section~\ref{subsubsec:disconnecting_by_shields}, we aim to use the sets $A_k$ to disconnect $\ccwBoundary{z_1}{z_2}{\partial D}$ and $\ccwBoundary{z_3}{z_4}{\partial D}$.  If the closure $K$ of the union of the $A_k$ which intersect $\ccwBoundary{z_2}{z_3}{\partial D}$ also intersects $\ccwBoundary{z_4}{z_1}{\partial D}$ then by the aforementioned fact there must exist an $A_k$ which intersects both $\ccwBoundary{z_2}{z_3}{\partial D}$ and $\ccwBoundary{z_4}{z_1}{\partial D}$.  If there is no such $A_k$, then as $K$ is local we can keep repeating the same experiment in the component of $D \setminus K$ with $\ccwBoundary{z_4}{z_1}{\partial D}$ on its boundary.  As such a shield has a positive chance of connecting the two such arcs, this process will eventually discover a finite number of such sets $K$ which disconnect $\ccwBoundary{z_1}{z_2}{\partial D}$ and $\ccwBoundary{z_3}{z_4}{\partial D}$.  Since the part of $\partial K$ with distance at least $\delta > 0$ from $\ccwBoundary{z_2}{z_3}{\partial D}$ is contained in a finite union of the $A_k$'s (and likewise when we iterate the exploration), we can find a finite collection of $A_k$'s which disconnect $\ccwBoundary{z_1}{z_2}{\partial D}$ from $\ccwBoundary{z_3}{z_4}{\partial D}$.
\item[Step 4.] We complete the proof by arguing that $\metplus{\ccwBoundary{z_1}{z_2}{\partial D}}{x}{\Gamma} > 0$ where $x$ is an intersection point between two adjacent shields.  This will complete the proof as a ``zero length path'' cannot pass through any given shield, so can only connect $\ccwBoundary{z_1}{z_2}{\partial D}$ and $\ccwBoundary{z_3}{z_4}{\partial D}$ by passing through such an intersection point.  We will prove that $\metplus{\ccwBoundary{z_1}{z_2}{\partial D}}{x}{\Gamma} > 0$ by arguing that such an intersection point is disconnected from $\ccwBoundary{z_1}{z_2}{\partial D}$ by a further shield connecting two ``typical points'' on two $\CLE_\kappa$ loops (one from the previous shield and one from the next shield) and a typical loop point $w$ has the property that $\metplus{z}{w}{\Gamma} > 0$ for all $z \in \Upsilon \setminus \{w\}$ using the results described in Step 1.
\end{enumerate}

\subsubsection{Preliminary lemmas}
\label{subsubsec:prelim_lemmas}

\begin{lemma}
\label{lem:zero_metric_ball_does_not_hit}
Let $\CC = (\C,h,0,\infty) \sim \qconeW{\gamma}{4}$ which is decorated by independent $\SLE_\kappa$ curves $\eta_\pm$ so that the surface $\CH_+$ parameterized by the component of $\C \setminus (\eta_- \cup \eta_+)$ which is to the left of $\eta_+$ is a quantum half-plane.  We assume that $\eta_\pm$ are parameterized according to quantum length.  We let $\Gamma_+$ be a $\CLE_\kappa$ in $\CH_+$, $\Upsilon_+$ its carpet, and we let $\metplus{\cdot}{\cdot}{\Gamma_+}$ be as above.  For any $u < v < r < s$ we have that $\p[\metplus{\eta_+([u,v])}{\eta_+([r,s])}{\Gamma_+} > 0] > 0$.  Moreover, for $u < v < r$ fixed, $\p[\metplus{\eta_+([u,v])}{\eta_+([r,s])}{\Gamma_+} > 0] \to 1$ as $s \downarrow r$.
\end{lemma}
\begin{proof}
We will first prove the first part of the statement.  For contradiction, let us suppose that there exist $u < v < r < s$ so that $\metplus{\eta_+([u,v])}{\eta_+([r,s])}{\Gamma_+} = 0$ a.s.\  By applying Lemma~\ref{lem:zero_length_conf} with a translation and scaling map, we may assume that $u=0$ and $v=1$.  We assume that~$\CC$ is embedded into~$\C$ so that $|\eta_+(1)| = 1$.  Let $\eta_+^\cp$ be $\eta_+$ parameterized by capacity and let $\tau$ be the first time~$t$ that $\eta_+^\cp([0,t])$ has quantum length equal to $1$ so that $\eta_+^\cp([0,\tau]) = \eta_+([0,1])$.  Then for any $0 < a < b$ fixed, there is a positive chance that the quantum length of $\eta_+^\cp([0,\tau+a])$ is at most $r$ and the quantum length of $\eta_+^\cp([0,\tau+b])$ is at least $s$, that is, $\eta_+([r,s]) \subseteq \eta_+^\cp([\tau+a,\tau+b])$.  This holds, moreover, even if we condition on $\eta_+^\cp$.  It therefore follows that $\metplus{\eta_+([0,1])}{\eta_+^\cp([\tau+a,\tau+b])}{\Gamma_+} = 0$ a.s.\  Since $0 < a < b$ were arbitrary, it follows that $\metplus{\eta_+([0,1])}{\eta_+^\cp([\tau+a,\tau+b])}{\Gamma_+} = 0$ a.s.\ for all $0 < a < b$ simultaneously.  Therefore $\metplus{\eta_+([0,1])}{\eta_+([r,s])}{\Gamma_+} = 0$ a.s.\ for all $1 < r < s$.  By the continuity of~$\metplus{\cdot}{\cdot}{\Gamma_+}$, it follows that $\metplus{\eta_+([0,1])}{\eta_+(r)}{\Gamma_+} = 0$ a.s.\ for all $r > 1$.  By applying scaling again (Lemma~\ref{lem:zero_length_conf}), we see that $\metplus{\eta_+(0)}{\eta_+(r)}{\Gamma_+} = 0$ for all $r > 0$.  Finally, applying a translation and scaling map (Lemma~\ref{lem:zero_length_conf}) and the continuity of $\metplus{\cdot}{\cdot}{\Gamma_+}$, we have that $\metplus{\eta_+(r)}{\eta_+(s)}{\Gamma_+} = 0$ a.s.\ for all $r, s > 0$.  This is a contradiction to the comparability of the quantiles $\quantHP{p}{\epsilon}$, hence proves the first statement.

We now prove the second statement.  Suppose that there exist $0 < u < v < r$ such that  $\p[ \metplus{\eta_+([u,v])}{\eta_+([r,s])}{\Gamma_+} > 0]$ does not tend to $1$ as $s \downarrow r$.  Then $\p[ \metplus{\eta_+([u,v])}{\eta_+(r)}{\Gamma_+} = 0] > 0$. By Lemma~\ref{lem:zero_length_conf} and translation and scaling, this implies that there exists $p_0 > 0$ so that $\p[\metplus{\eta_+(0)}{\eta_+([R,2R])}{\Gamma_+} = 0] = p_0$ for all $R > 0$.  Suppose that we have $0 < r_1 < r_2$.  Let $\eta_1$, $\eta_2$ be $\SLE_\kappa^1(\kappa-6)$ processes coupled as conditionally independent CPIs in $\Upsilon_+$ respectively starting from $\eta_+(r_1)$, $\eta_+(r_2)$.  Let~$\tau_i$ be the first time that $\eta_i$ disconnects~$0$ from~$\infty$.  Let also $T_i$ be the part of the boundary of the unbounded component of $\CH_+ \setminus \eta_i([0,\tau_i])$ which is contained in $\eta_i([0,\tau_i])$.  Let $B_i$ be the part of the boundary of the component of $\CH_+ \setminus \eta_i([0,\tau_i])$ with $\eta_+(0)$ on its boundary which is contained in $\eta_i([0,\tau_i])$.

We note that $\p[\metplus{T_1}{B_2}{\Gamma_+} > 0] > 0$ for all $r_2 > r_1$ sufficiently large.  Indeed, if not, then $\p[\metplus{T_1}{B_2}{\Gamma_+} > 0] = 0$ for all $r_2 > r_1$.  By applying scaling (Lemma~\ref{lem:zero_length_conf}), this implies that $\p[ \metplus{\eta_+(0)}{B_1}{\Gamma_+} = 0] = 1$ for all $r_1 > 0$.  This is in fact true by translation for all $\eta_+(r)$ for $r$ in some interval.  With positive probability we also have that $\metplus{\eta_-([a,b])}{\eta_+([a,b])}{\Gamma_+} = 0$.  We then obtain that with positive probability we have that $\metplus{\eta_+(r)}{\eta_+(0)}{\Gamma_+} = 0$ for all $r$ in some interval.  By the scale invariance of the law of $\zero{\Gamma_+}$ (Lemma~\ref{lem:zero_length_conf}), we get that $\Fd^+$ vanishes with positive probability, which we know it cannot by the comparability of the quantiles $\quantHP{p}{\epsilon}$.  This proves our assertion.

We now define sequences of numbers as follows.  Suppose we take $r_1 = 1$ and $r_2$ as above so that $\p[\metplus{T_1}{B_2}{\Gamma_+} > 0] > 0$.  We then let $r_j = r_2^{j-1}$ for all $j \geq 2$.  By scaling, there exists $p_1 > 0$ so that $\p[\metplus{T_j}{B_{j+1}}{\Gamma_+} > 0] = p_1$ for all $j$.  For each $j$ let $U_j$ be the domain bounded by $T_j$, $B_{j+1}$, $\eta_-$, and~$\eta_+$.  For each $k$, we let $N_k$ be the number of $1 \leq j \leq k$ so that $\p[\metplus{T_j}{B_{j+1}}{\Gamma_+} > 0 \giv U_j] \geq p_1/2$.  By Birkhoff's ergodic theorem, we have that $N_k /k$ a.s.\ tends to a positive limit as $k \to \infty$.  It therefore follows that $\p[\metplus{T_1}{B_k}{\Gamma_+} >0] \to 1$ as $k \to \infty$.  This contradicts the statement that there exists $p_0 > 0$ so that $\p[\metplus{\eta_+(0)}{\eta_+([R,2R])}{\Gamma_+} = 0] = p_0$ for all $R > 0$, which completes the proof of the second assertion.
\end{proof}

Suppose that $D \subseteq C$ is a simply connected domain and $\Gamma$ is a $\CLE_\kappa$ on $D$.  We let $Z_{x,y}$ be the closure of the union of $\{z \in \Upsilon : \metplus{z}{\ccwBoundary{x}{y}{\partial D}}{\Gamma} = 0\}$ and the loops of $\Gamma$ which intersect it.  Let also $Z_{x,y}^+$ be the closure of the union of $\{z \in \Upsilon : \exists (w,z) \in \zerop{\Gamma},\ w \in \ccwBoundary{x}{y}{\partial D}\}$ and the loops of $\Gamma$ which intersect it.  If $\metplus{\cdot}{\cdot}{\Gamma}$ is continuous up to $\partial D$, then we note that $Z_{x,y}^+ \subseteq Z_{x,y}$ since if $z \in \Upsilon$ and there exists $w \in \ccwBoundary{x}{y}{\partial D}$ so that $(w,z) \in \zerop{\Gamma}$ then $\metplus{z}{w}{\Gamma} = 0$ hence $\metplus{z}{\ccwBoundary{x}{y}{\partial D}}{\Gamma} = 0$.  We consider both cases since $Z_{x,y}^+$ is useful due to its conformal invariance properties (Lemma~\ref{lem:zero_length_conf}) while $Z_{x,y}$ is useful in settings in which we will not apply a conformal map.

\begin{lemma}
\label{lem:zero_metric_ball_is_local}
Suppose that $D \subseteq \C$ is a simply connected domain, $\Gamma$ is a $\CLE_\kappa$ in $D$, and $x,y \in \partial D$ are distinct.  Then both $Z_{x,y}$ and $Z_{x,y}^+$ are local for $(\Gamma,\metplus{\cdot}{\cdot}{\Gamma})$.
\end{lemma}
\begin{proof}
Suppose that $U \subseteq D$ is open.  Using the notation of the statement of Lemma~\ref{lem:loc_characterization}, whether $Z_{x,y} \cap U = \emptyset$ is determined by $\CF_U^\epsilon$ for each $\epsilon > 0$ and hence by $\CF_U$.  Consequently, the assertion of the lemma for $Z_{x,y}$ follows from Lemma~\ref{lem:loc_characterization}.  The same argument applies for $Z_{x,y}^+$.
\end{proof}

\begin{lemma}
\label{lem:in_between_stuff_is_small}
Suppose that we have the setup described in Lemma~\ref{lem:zero_metric_ball_is_local}.  The following is a.s.\ true.  Suppose that $z$ is in the interior of a component $U$ of $D \setminus Z_{x,y}$.  Then the harmonic measure of the set of points in $\partial U \setminus \partial D$ as seen from $z$ which are not contained in a loop of $\Gamma$ is zero.  The same is true with $Z_{x,y}^+$ in place of $Z_{x,y}$.
\end{lemma}
\begin{proof}
We will prove the result in the case of $Z_{x,y}$.  The argument in the case of $Z_{x,y}^+$ is analogous.  In order for the assertion of the lemma to be non-trivial, we need to assume
\begin{equation}
\label{eqn:k_non_empty}
\p[Z_{x,y} \setminus \partial D \neq \emptyset] > 0.
\end{equation}
We note that~\eqref{eqn:k_non_empty} implies that with positive probability there exists $w \in \Upsilon$ so that if $(U_{2,n})$ is an increasing sequence in $\dyad(D)$ with $D = \cup_n U_{2,n}$ then there exists $n_0$ so that for all $n \geq n_0$ we have that $\metres{U_{2,n}}{w}{\partial U_{2,n}}{\Gamma} = 0$.  By Lemma~\ref{lem:zero_length_conf}, we note that this statement must hold for every simply connected domain $D \subseteq \C$.  

Suppose that~\eqref{eqn:k_non_empty} holds and the assertion of the lemma does not hold.  Then there exists a point $z \in D$ with rational coordinates so that with positive probability $z \notin Z_{x,y}$ and the harmonic measure of the points in $\partial U \setminus \partial D$, $U$ the component of $D \setminus Z_{x,y}$ which contains $z$, which are not contained in a loop of $\Gamma$ is positive.  By Lemma~\ref{lem:zero_metric_ball_is_local}, we know that the conditional law of the loops of $\Gamma$ contained in $U$ is that of a $\CLE_\kappa$.  That is, we can write these loops as $\varphi(\Gamma_\D)$ where $\varphi \colon \D \to U$ is the unique conformal transformation with $\varphi(0) = z$ and $\varphi'(0) > 0$ and $\Gamma_\D$ is a $\CLE_\kappa$ in $\D$ with carpet $\Upsilon_\D$.  By~\eqref{eqn:k_non_empty}, it is a positive probability event for $\Gamma_\D$ that there exists $u \in \Upsilon_\D$ so that if $(U_{2,n})$ is an increasing sequence in $\dyad(\D)$ with $\D = \cup_n U_{2,n}$ then $\metres{U_{2,n}}{u}{\partial U_{2,n}}{\Gamma_\D} = 0$ for all $n$ large enough.  Consequently, there exists $v_n \in \partial U_{2,n}$ so that $\metres{U_{2,n}}{u}{v_n}{\Gamma_\D} = 0$ for all $n$ large enough.  By passing to a subsequence if necessary, we may assume that $(v_n)$ converges to a limit $v \in \partial \D$.  Let $\Theta$ be uniform in $[0,2\pi]$ independently of everything else.  Then $e^{i \Theta} \Gamma_\D$ has the same law as $\Gamma_\D$.  Let $X$ be such that $\varphi(X)$ is equal to those points in $\partial U \setminus \partial D$ which are not contained in a loop of $\Gamma$.  Then we know that $X$ has positive harmonic measure as seen from $z$.  Therefore there is a positive chance that $e^{i \Theta} v \in X$ and it is a positive probability event that $v \in X$.  This is a contradiction because on the event that $v \in X$ we have by Lemma~\ref{lem:zero_length_conf} that $\metplus{\varphi(u)}{\varphi(X)}{\varphi(\Gamma_\D)} = 0$, which contradicts the definition of~$Z_{x,y}$.
\end{proof}

\begin{figure}[ht!]
\begin{center}
\includegraphics[scale=1]{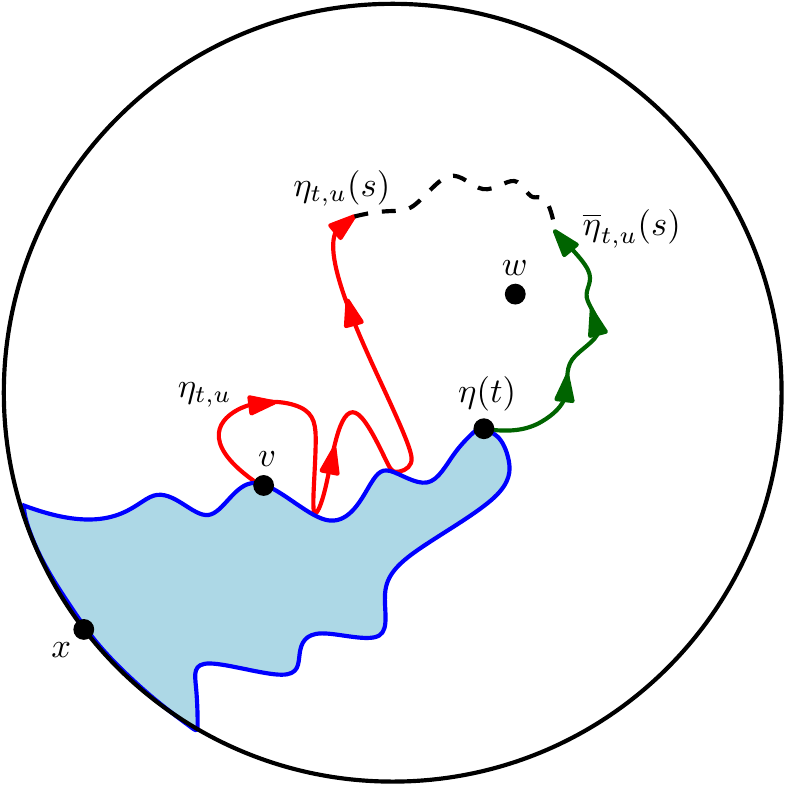}	
\end{center}
\caption{\label{fig:zero_lero_length_on_boundary} Illustration of the proof of Lemma~\ref{lem:no_zero_length_on_boundary}.  The set of points disconnected from $w$ by $\eta|_{[0,t]}$ is shown in blue and recall that $D_t$ is the component of $\D \setminus \eta([0,t])$ containing $w$.  The point $v$ is such that the quantum length of $\ccwBoundary{v}{\eta(t)}{\partial D_t}$ is equal to $u$.  The path $\eta_{t,u}$ (red) is the right boundary of $\eta|_{[t,\infty)}$ stopped when it first hits $v$ as viewed from $v$ and $\ol{\eta}_{t,u}$ its time-reversal.  The dashed curve is the remainder of $\eta_{t,u}$ given $\eta_{t,u}|_{[0,s]}$ and $\ol{\eta}_{t,u}|_{[0,s]}$ shown on the event that it does not hit $\partial D_t$.  On this event, its law is absolutely continuous with respect to an $\SLE_\kappa$ from $\eta_{t,u}(s)$ to $\ol{\eta}_{t,u}(s)$.}
\end{figure}

\begin{lemma}
\label{lem:no_zero_length_on_boundary}
Suppose that we have the setup of Proposition~\ref{prop:no_zero_crossings} and are on the event that $D \neq \emptyset$.  Suppose $h$ is such that $(D,h) \sim \qdiskL{\gamma}{1}$.  Let $\Gamma$ be an independent $\CLE_\kappa$ on $D$.  Let $X$ be the set of points $z \in \partial D$ so that there exists $w \in \closure{D} \setminus \{z\}$ with $\metplus{z}{w}{\Gamma} = 0$.
\begin{enumerate}[(i)]
\item\label{it:boundary_meas_zero} Then $\p[\qbmeasure{h}(X) = 0]=1$ and $X$ is independent of~$\metplus{\cdot}{\cdot}{\Gamma}$.
\item\label{it:loop_boundary_meas_zero} Similarly, it a.s.\ holds that if $\CL \in \Gamma$, $D_\CL$ is the set of points surrounded by $\CL$, and $X_\CL$ is the set of $z \in \CL$ so that there exists $w \in \closure{D} \setminus (D_\CL \cup \{z\})$ with $\metplus{z}{w}{\Gamma} = 0$ then $\qbmeasure{h}(X_\CL) = 0$.
\end{enumerate}
Moreover, the first assertion of the lemma holds with harmonic measure in $D$ in place of $\qbmeasure{h}$ and the second assertion holds with harmonic measure in the unbounded component of $\C \setminus \CL$ in place of $\qbmeasure{h}$.
\end{lemma}
\begin{proof}
See Figure~\ref{fig:zero_lero_length_on_boundary} for an illustration of the setup of the proof.  Suppose that $h$ is such that $(\D,h) \sim \qdiskWeighted{\gamma}{1}$.  Let $w$ be sampled from $\qmeasure{h}$ and $x$ from $\qbmeasure{h}$ conditionally independently of everything else.  Given $x$, $w$, let $\eta$ be an $\SLE_\kappa^1(\kappa-6)$ in $\D$ targeted at $w$ sampled conditionally independently of everything else and then parameterized by the quantum natural time of its trunk~$\eta'$.  The $\tau$ be the first time that $\eta$ disconnects $w$ from $\partial \D$ and let $E$ be the event that the component $D$ of $\D \setminus \eta([0,\tau])$ is not a loop of $\Gamma$ and is surrounded counterclockwise by $\eta$.  On $E$ and given the quantum length $\ell$ of $\partial D$ we note that $(D,h) \sim \qdiskWeighted{\gamma}{\ell}$.  It suffices to prove~\eqref{it:boundary_meas_zero} in this setting.

Let $t > 0$ be rational.  Assume we are working on the event that $t < \tau$ and let $D_t$ be the component of $\D \setminus \eta([0,t])$ which contains $w$.  Fix $u > 0$ rational and let $v \in \partial D_t$ be such that $\qbmeasure{h}(\cwBoundary{\eta(t)}{v}{\partial D_t}) = u$.  Then we note that $\eta|_{[t,\infty)}$ targeted at $v$ is an $\SLE_\kappa^1(\kappa-6)$ process in $D_t$ from $\eta(t)$ to $v$.  In particular, its trunk is an $\SLE_{\kappa'}(\kappa'-6)$ process in $D_t$ from $\eta(t)$ to $v$.  It thus follows from \cite[Theorem~1.4]{ms2016ig1} (see also \cite[Figure~2.5]{mw2017intersections}) that its right boundary (as viewed from $v$) $\eta_{t,u}$ is an $\SLE_\kappa(2-\kappa;\kappa-4)$ process.  Assume that $\eta_{t,u}$ is parameterized according to quantum length and let $\ol{\eta}_{t,u}$ be the time-reversal of $\eta_{t,u}$.  Fix $s > 0$ rational and suppose that $\eta_{t,u}([0,s]) \cap \ol{\eta}_{t,u}([0,s]) = \emptyset$.  On the event that the remainder $\wt{\eta}_{t,u}$ of $\eta_{t,u}$ does not hit $\partial D_t$, it follows from Lemma~\ref{lem:middle_part_abs_cont} that it is absolutely continuous with respect to an $\SLE_\kappa$ process in the component $D_{s,t,u}$ of $D_t \setminus (\eta_{t,u}([0,s]) \cup \ol{\eta}_{t,u}([0,s]))$ with $\eta_{t,u}(s)$, $\ol{\eta}_{t,u}(s)$ on its boundary from $\eta_{t,u}(s)$ to $\ol{\eta}_{t,u}(s)$.  By considering the setup from Lemma~\ref{lem:zero_metric_ball_does_not_hit}, it follows that by applying a conformal transformation $\C \setminus \eta_- \to D_{s,t,u}$ which takes $0$ (resp.\ $\infty$) to $\eta_{t,v}(s)$ (resp.\ $\ol{\eta}_{t,v}(s)$) we have that with $\Gamma_{s,t,u}$ a $\CLE_\kappa$ in $D_{s,t,u}$ and $X_{s,t,u}$ the set of $z \in \wt{\eta}_{s,t,u}$ so that there exists $w \in \closure{D_{s,t,u}} \setminus \{z\}$ with $\metplus{z}{w}{\Gamma_{s,t,u}} = 0$ we have that $\p[\qbmeasure{h}(X_{s,t,u}) = 0]=1$.  The first part of~\eqref{it:boundary_meas_zero} follows since this statement a.s.\ holds for all $s,t,u > 0$ rational simultaneously.

 To see that $X$ is independent of~$\metplus{\cdot}{\cdot}{\Gamma}$, we simply note that $X$ is determined by the restriction of $\metplus{\cdot}{\cdot}{\Gamma}$ to the $\epsilon$-neighborhood of $\partial D$ for each $\epsilon > 0$.

Suppose that we are on the event that the harmonic measure of $X$ is positive in $D$.  Then as $X$ is independent of $h$, it is easy to see that $\qbmeasure{h}(X) > 0$.  Thus part~\eqref{it:boundary_meas_zero}  implies that $X$ a.s.\ has zero harmonic measure in $D$.

We now prove~\eqref{it:loop_boundary_meas_zero}.  Let $y \in \partial D$ be picked from $\qbmeasure{h}$ independently of $x$ and let $\eta$ be an $\SLE_\kappa(\kappa-6)$ in $D$ from $x$ to $y$ coupled with $\Gamma$ as a CPI.  Fix $\delta > 0$ and for each $j \in \N$ let $\tau_j$ be the $j$th time that $\eta$ hits a loop of $\Gamma$ with quantum length at least $\delta$.  Let $\CL_j$ be this loop of $\Gamma$.  Assume we have explored $\eta([0,\tau_j^-])$ and the first $\delta$ units of quantum length of $\CL_j$ starting from where $\eta$ first hits $\CL_j$ and let $\eta_j$ be the rest of $\CL_j$.  Then the conditional law of $\eta_j$ given what we have explored is that of an $\SLE_\kappa$ in the remaining domain.  In particular, we can sample from the joint law of $\eta_j$ and the unexplored loops of $\Gamma$ by first sampling from the law in the half-planar setup and then applying a conformal transformation from $\C \setminus \eta_-$ as in the proof of the first assertion of the lemma.  This proves that the quantum length of the set of points $z$ in $\eta_j$ for which there exists $w \in \Upsilon \setminus \{z\}$ with $\metplus{z}{w}{\Gamma_j} = 0$ is a.s.\ $0$.  Since $\delta > 0$ was arbitrary, we see that the quantum length of the set of points $z$ in $\CL_j$ for which there exists $w \in \Upsilon \setminus \{z\}$ with $\metplus{z}{w}{\Gamma_j} = 0$ is a.s.\ $0$.  This does not quite complete the proof as $\eta$ does not hit every loop of $\Gamma$.  However, Proposition~\ref{prop:cpi_path_close} implies that the conditional probability given~$\Upsilon$ that~$\eta$ hits any loop $\CL$ of~$\Gamma$ in a fixed interval of $\CL$ is positive.  Combining, this implies~\eqref{it:loop_boundary_meas_zero}.  The assertion regarding the harmonic measure in the unbounded component of $\C \setminus \CL$ in place of $\qbmeasure{h}$ follows similarly as in the case of $X$ described above.
\end{proof}

\subsubsection{Definition of the shields}
\label{subsubsec:shield_def}

\begin{figure}[ht!]
\begin{center}
\includegraphics[scale=1]{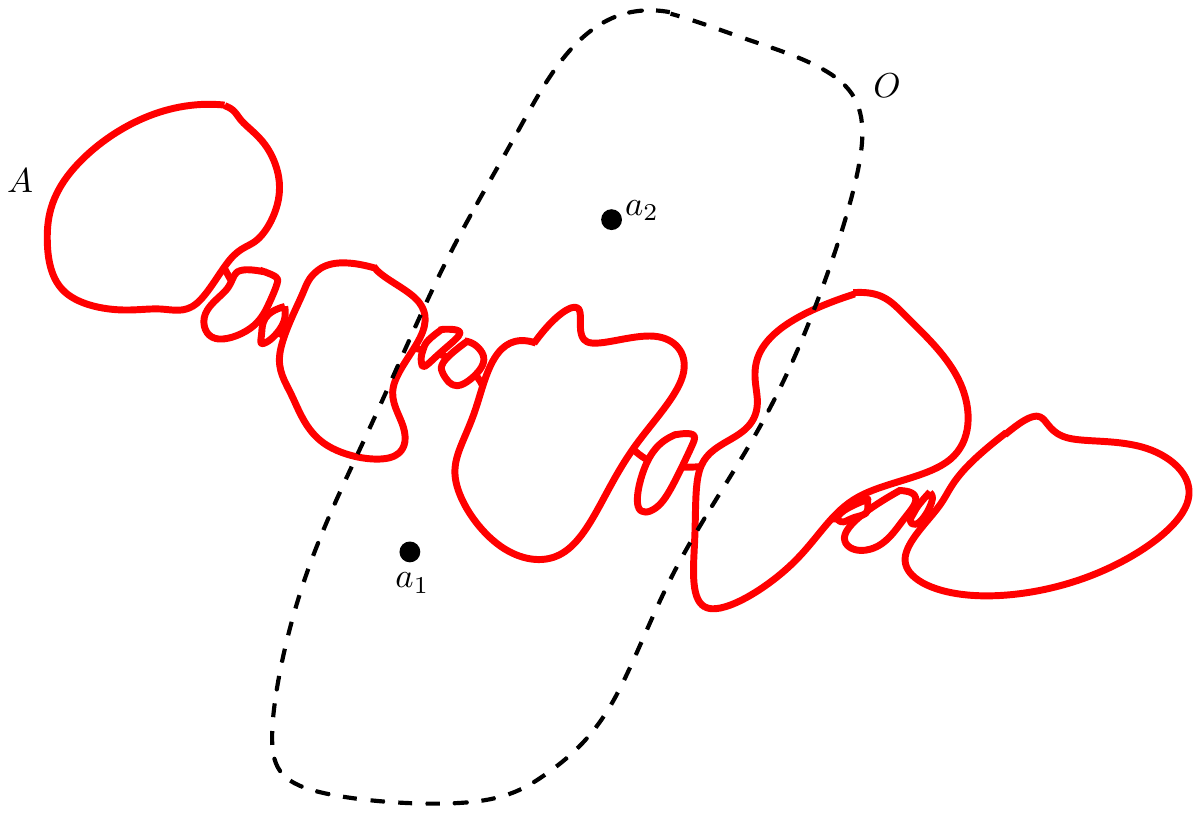}	
\end{center}
\caption{\label{fig:shield_def} Illustration of the definition of a shield $A$, whose boundary is shown in red.  We note that the set of points surrounded by loops contained in $A$ are dense in $A$ but are not equal to all of $A$ since $\CLE_\kappa$ loops with $\kappa \in (8/3,4)$ do not intersect each other.  The region bounded by the dashed line is $O$.}
\end{figure}

We now suppose that we are in the setting of the statement of Proposition~\ref{prop:no_zero_crossings} and we have four distinct marked points $z_1,\ldots,z_4 \in \partial D$ given in counterclockwise order.  Let $(r_k)$ be an enumeration of the points with rational coordinates in $D$.  For each~$k$, we let~$\CL_k$ be the loop of $\Gamma$ which surrounds $r_k$.  (We note that there a.s.\ exists a unique such loop for all $k \in \N$ simultaneously.)  We say that $j \sim_0 k$ if $j = k$ or if there exists a closed and connected set~$A \subseteq D$ (so that $\dist(A,\partial D) > 0$) such that $\CL_j, \CL_k \subseteq A$ and the harmonic measure of both $\CL_j$ and $\CL_k$ in $\C \setminus A$ are positive and so that the following properties hold (see Figure~\ref{fig:shield_def} for an illustration).
\begin{enumerate}[(I)]
\item\label{it:ic1} Let $U$ be the unbounded component of $\C \setminus A$.  Then $A = \C \setminus U$.
\item\label{it:ic2} The set of $z$ points in $\partial A$ so that there exists $\CL \in \Gamma$ with $z \in \CL$ and so that $\CL$ has positive harmonic measure in $\C \setminus A$ is dense in $\partial A$.
\item\label{it:ic3}   If $O \subseteq D$ is open, $a_1,a_2 \in O^*$, $\Gamma_O$ is equal to the loops of $\Gamma$ which are contained in $O^*$, and there exists $\delta > 0$ so that $a_1,a_2$ are in different components of $O^\delta \setminus A$ where $O^\delta$ is the $\delta$-neighborhood of $O$, then $\metplus{a_1}{a_2}{\Gamma_O} > 0$.
\end{enumerate}
We then let $\sim$ be the finest equivalence relation on $\N$ so that $j \sim_0 k$ implies $j \sim k$.  For each $k \in \N$, we let $A_k$ be closure of the union of the points surrounded by $\CL_j$ such that $j \sim k$.  We let
\[ E_k = \{ A_k \cap \ccwBoundary{z_2}{z_3}{\partial D} \neq \emptyset \},\quad \CK = \{k \in \N : E_k \text{ occurs}\},\quad \text{and}\quad K = \closure{\cup_{k \in \CK} A_k}.\]

We are now going to collect some properties of the $A_k$'s.  In particular, we will show in Lemma~\ref{lem:e_k_pos} that $\p[E_k] > 0$ for all $k \in \N$ and in Lemma~\ref{lem:boundary_good} that with $Y_k = \{z \in A_k : \metplus{z}{\partial A_k}{\Gamma} = 0\}$ we have that $U_k = A_k \setminus Y_k$ is simply connected and $\partial U_k = Y_k$.  Lemma~\ref{lem:boundary_shield} is an intermediate step in both of these lemmas as it gives a construction of a set~$A$ satisfying properties~\eqref{it:ic1}--\eqref{it:ic3} from above.  We will then define a clockwise exploration of the~$A_k$'s which hit $\ccwBoundary{z_2}{z_3}{\partial D}$ which we will prove in Lemma~\ref{lem:component_local} is local for $(\Gamma,\metplus{\cdot}{\cdot}{\Gamma})$ and use in Lemma~\ref{lem:ak_cross_finite} to show that the set of $A_k$'s which hit $\ccwBoundary{z_2}{z_3}{\partial D}$ and leave the $\delta$-neighborhood of $\ccwBoundary{z_2}{z_3}{\partial D}$ is finite for each $\delta > 0$.

\begin{lemma}
\label{lem:boundary_shield}
Suppose that we have the setup described above and fix $x,y \in \partial D$ distinct.  Let $W$ be a component of $D \setminus Z_{x,y}$ which is not surrounded by a loop of $\Gamma$, let $I$ be an arc of $\partial W \cap \partial Z_{x,y}$ (in the sense of prime ends in $\partial W$) which has positive distance from $\partial D$, and let $\wt{A}$ be the closure of the union of $I$ and the loops of $\Gamma$ whose intersection with $I$ has positive harmonic measure in $W$.  Let $A$ be the complement of the unbounded component of $\C \setminus \wt{A}$.  On $A \neq \emptyset$, we have that $A$ satisfies properties~\eqref{it:ic1}--\eqref{it:ic3} above.  The same is also true with $Z_{x,y}^+$ in place of $Z_{x,y}$.
\end{lemma}
\begin{proof}
We will give the proof in the case of $Z_{x,y}$; the argument in the case of $Z_{x,y}^+$ is analogous.

Property~\eqref{it:ic1} holds by definition.  Lemma~\ref{lem:in_between_stuff_is_small} implies that if $w \in W$ then the harmonic measure as seen from $w$ of the set of points in $I$ which are not in a loop of $\Gamma$ is zero.  Therefore~\eqref{it:ic2} holds.  Suppose that $O \subseteq D$ is an open set, $a_1,a_2 \in O^*$, and $\delta > 0$ is such that $a_1$, $a_2$ are in different components of $O^\delta \setminus A$ where $O^\delta$ is the $\delta$-neighborhood of~$O$.  Let~$\Gamma_O$ be the loops of~$\Gamma$ which are contained in~$O^*$.  If $a_1 \in W$ and $a_2 \in O \setminus \closure{W}$ then by the definition of $A$ we must have that $\metplus{a_1}{a_2}{\Gamma_O} > 0$ (for otherwise $a_1 \in Z_{x,y}$).  Now suppose that $a_1,a_2 \in W$.  If $\metplus{a_1}{a_2}{\Gamma_O} = 0$ then there must exist $u \in A \cap O^*$ which is not in a loop of $\Gamma$ so that $\metplus{a_1}{u}{\Gamma_O} = 0$.  This also contradicts the definition of $Z_{x,y}$ so we must have that $\metplus{a_1}{a_2}{\Gamma_O} > 0$ (for otherwise $a_1 \in Z_{x,y}$).  If $a_1$, $a_2$ are in different components of $O^*$ then $\metplus{a_1}{a_2}{\Gamma_O} = \infty$.  Lastly suppose that $a_1,a_2 \in O \setminus \closure{W}$.  Suppose that $a_1$, $a_2$ are in the same component of $O^*$ (for otherwise $\metplus{a_1}{a_2}{\Gamma_O} = \infty$).  Since the loops of $\Gamma$ do not intersect each other, it follows that if $a_1$, $a_2$ are in different components of $O^\delta \setminus A$ then any path in $\Upsilon \cap O^\delta$ which  connects $a_1$, $a_2$ must enter $\ol{W}$.  Suppose that $\metplus{a_1}{a_2}{\Gamma_O} = 0$ and let $Z_{a_1}$ be the set of $u \in O^*$ so that $\metplus{a_1}{u}{\Gamma_O} = 0$.  If $Z_{a_1} \cap W \neq \emptyset$ then we get a contradiction to the definition of $Z_{x,y}$ because then there exists $u \in W$ with $\metplus{u}{\ccwBoundary{x}{y}{\partial D}}{\Gamma} = 0$.  If $Z_{a_1} \cap W = \emptyset$, then as $Z_{a_1}$ is closed and connected and contains $a_2$ it follows that its intersection with $\partial W$ has positive harmonic measure in $W$.  Indeed, the reason for this is that since the loops of $\Gamma$ do not intersect each other there must exist $\CL \in \Gamma$ which is contained in $A$ so that $\CL \cup \partial W$ disconnects $a_1,a_2$ in~$O^{\delta/2}$.  Since $\CL$ is in $A$, the points disconnected from $\infty$ by $\CL \cup I$ are in $A$ and its intersection with $I$ must have positive harmonic measure in $W$.  Thus as $Z_{a_1}$ must contain $\CL \cap I$ we get a contradiction to Lemma~\ref{lem:in_between_stuff_is_small}.  Altogether, this proves that~\eqref{it:ic3} holds. 
\end{proof}

\begin{lemma}
\label{lem:e_k_pos}
Suppose that we have the setup described just above.  For each $k \in \N$ we have that $\p[E_k] > 0$.
\end{lemma}
\begin{proof}
We first note that if we have a smooth and simple curve $\omega$ which connects $r_k$ to a point in $\ccwBoundaryOpen{z_2}{z_3}{\partial D}$ then with $\wt{A}$ the closure of the union of the points surrounded by the loops of $\Gamma$ which intersect $\omega$ and $A$ the complement of the unbounded component of $\C \setminus \wt{A}$ we have that $A$ satisfies properties~\eqref{it:ic1}, \eqref{it:ic2} as above.  If property~\eqref{it:ic3} holds with positive probability, then we have shown that $\p[E_k] > 0$.

Suppose now that property~\eqref{it:ic3} does not hold for this set.  Then we in particular have that with positive probability there exist $z,w \in \Upsilon \setminus \partial D$ distinct with $\metplus{z}{w}{\Gamma} = 0$.  Note that for any $z \in \Upsilon$ the set $\{w \in \Upsilon : \metplus{z}{w}{\Gamma} = 0\}$ is connected.  Consequently, by using a CPI and applying Proposition~\ref{prop:cpi_path_close} and Lemma~\ref{lem:zero_length_conf}, we see that with positive probability the set $\{(z,w) \in \zerop{\Gamma} : z \in \Upsilon,\ w \in \partial D\}$ is non-empty.  Applying Lemma~\ref{lem:zero_length_conf} a second time implies that there exists $p > 0$ so that for all $x,y \in \partial D$ distinct we have that $\p[Z_{x,y}^+ \neq \emptyset] \geq p$.  Moreover, Lemma~\ref{lem:zero_metric_ball_does_not_hit} implies that $\diam(Z_{x,y}^+) \to 0$ in probability as $x \to y$.  Thus by choosing $x,y \in \ccwBoundaryOpen{z_2}{z_3}{\partial D}$ sufficiently close we have both $Z_{x,y}^+ \cap \ccwBoundary{z_2}{z_3}{\partial D} \neq \emptyset$ and $Z_{x,y}^+ \cap \partial D \subseteq \ccwBoundary{z_2}{z_3}{\partial D}$ with positive probability.  On this event, let $W_{x,y}$ be the component of $D \setminus Z_{x,y}^+$ with $z_4$ on its boundary.   Lemma~\ref{lem:boundary_shield} implies that if $\wt{A}$ is the closure of the union of an arc of $\partial W_{x,y} \cap Z_{x,y}^+$ (in the sense of prime ends in $W_{x,y}$) with positive distance to $\partial D$ and the loops of $\Gamma$ whose intersection with it has positive harmonic measure in $W_{x,y}$ then the complement $A$ of the unbounded component of $\C \setminus \wt{A}$ satisfies properties~\eqref{it:ic1}--\eqref{it:ic3}.  If the intersection of the arc with $\CL_i$, some $i \in \N$, has positive harmonic measure then all of the points surrounded by the loops of $\Gamma$ which intersect the arc and have positive harmonic measure are contained in $A_i$ and hence $\partial W_{x,y} \cap Z_{x,y}^+ \subseteq A_i$.  In particular, $A_i \cap \ccwBoundary{x}{y}{\partial D} \neq \emptyset$.

To finish the proof, it is left to explain why $A_k \cap \ccwBoundary{x}{y}{\partial D} \neq \emptyset$ with positive probability.  Let $\varphi$ the unique conformal transformation which takes $r_i$ to $r_k$ and fixes $x$.  Lemma~\ref{lem:zero_length_conf} implies that $\varphi(A_i) \stackrel{d}{=} A_k$.  Consequently $\p[A_k \cap \partial D \neq \emptyset] > 0$.   Let $x_0$ be a point in $\partial D$ chosen from harmonic measure as viewed from $r_k$ independently of everything else and let $y_0$ be such that the harmonic measure of $\ccwBoundary{x_0}{y_0}{\partial D}$ as seen from $r_k$ is equal to that of $\ccwBoundary{x}{y}{\partial D}$.  If $\p[A_k \cap \ccwBoundary{x_0}{y_0}{\partial D} \neq \emptyset] =0$ then $\p[ A_k \cap \partial D \neq \emptyset] = 0$ so we conclude that $\p[A_k \cap \ccwBoundary{x_0}{y_0}{\partial D} \neq \emptyset] > 0$.  Let $\psi \colon D \to D$ be the unique conformal transformation which fixes $r_k$ and sends $x_0$ to $x$.  Then $\psi(A_k) \stackrel{d}{=} A_k$ by Lemma~\ref{lem:zero_length_conf}.  Altogether, this implies that $\p[A_k \cap \ccwBoundary{x}{y}{\partial D} \neq \emptyset] > 0$, which completes the proof.
\end{proof}

\begin{lemma}
\label{lem:boundary_good}
Suppose that we have the setup described above.  Then the following a.s.\ hold.
\begin{enumerate}[(i)]
 \item\label{it:boundary_zero} If $z,w$ are in the same component of $\partial A_k \setminus \partial D$ then $\metplus{z}{w}{\Gamma} = 0$.
 \item\label{it:simply_connected} Let $Y_k$ be the set of $z \in A_k$ such that $\metplus{z}{\partial A_k}{\Gamma} = 0$.  Then $U_k = A_k \setminus Y_k$ is simply connected.
 \item\label{it:boundary_equal} $\partial U_k = Y_k$.
 \end{enumerate}
\end{lemma}
\begin{proof}
We will first prove part~\eqref{it:boundary_zero}.  As a first step, we will argue that $\partial A_k = \partial \interior{A_k}$.  We clearly have that $\partial \interior{A_k} \subseteq \partial A_k$.  Suppose that $z \in \partial A_k$ and fix $\epsilon > 0$.  By the definition of $A_k$, there exists $i \in \N$ with $i \sim k$ and a closed set $A \subseteq D$ with $\dist(A, \partial D) > 0$ satisfying \eqref{it:ic1}--\eqref{it:ic3} which contains $\CL_i$ and $\partial A$ contains some point $w \in B(z,\epsilon) \cap A_k$.  By~\eqref{it:ic2}, there exists $\CL \in \Gamma$ which is contained in $A$, has positive harmonic measure in $\C \setminus A$, and  intersects $B(w,\epsilon)$ hence $B(z,2\epsilon)$.  Since the points surrounded by $\CL$ are in $\interior{A_k}$, it follows that $\interior{A_k} \cap B(z,2\epsilon) \neq \emptyset$.  This proves that $z \in \partial \interior{A_k}$ as $\epsilon > 0$ was arbitrary.

Let $I$ be a component  of $\partial A_k \setminus \partial D$.  We assume that there exists $z,w \in I$ distinct so that $\metplus{z}{w}{\Gamma} > 0$ and we will obtain a contradiction.

Let $X = \{u \in I : \metplus{u}{\partial D}{\Gamma} > 0\}$.  We first claim that $X$ is dense in $I$.  Suppose for contradiction that it is not dense in $I$.  Then there exists $v \in I$ and $\epsilon > 0$ so that with $J = I \cap B(v,\epsilon)$ we have that $a \in J$ implies that $\metplus{a}{\partial D}{\Gamma} = 0$.  We claim that this implies that there exists $x,y \in \partial D$ distinct, $u \in I$, and $\delta > 0$ so that $B(u,\delta) \cap I \subseteq Z_{x,y}$ and $I$ is not contained in $Z_{x,y}$.  To see this, suppose that $x_1,\ldots,x_n$ are distinct points given in counterclockwise order in $\partial D$ and write $x_{n+1} = x_1$.  Let $1 \leq i_1 < \cdots < i_m \leq n$ be a minimal collection so that $B(v,\epsilon) \cap I \subseteq \cup_{\ell=1}^m Z_{x_{i_\ell},x_{i_\ell+1}}$.  Then there exists $u \in B(v,\epsilon) \cap I$ which is not in $\cup_{\ell=1}^{m-1} Z_{x_{i_\ell},x_{i_\ell+1}}$.  Let $\delta$ be the distance of $u$ to $\cup_{\ell=1}^{m-1} Z_{x_{i_\ell},x_{i_\ell+1}}$.  Then $\delta > 0$ as $\cup_{\ell=1}^{m-1} Z_{x_{i_\ell},x_{i_\ell+1}}$ is closed and we must have that $B(u,\delta) \cap I \subseteq Z_{x_{i_m},x_{i_m+1}}$.  If it were true that $I \subseteq Z_{x_{i_m},x_{i_m+1}}$ no matter how close we choose the spacing between the $x_i$'s, then by taking a limit as the spacing tends to $0$ we would obtain that there exists $x \in \partial D$ so that $\metplus{a}{x}{\Gamma} = 0$ for all $a \in I$.  This, in turn, contradicts our assumption that there exists exists $z,w \in I$ distinct so that $\metplus{z}{w}{\Gamma} > 0$.  We can thus find $x,y \in \partial D$ distinct, $u \in I$, and $\delta > 0$ so that $B(u,\delta) \cap I \subseteq Z_{x,y}$ and $I$ is not contained in $Z_{x,y}$.  Then by considering a boundary arc of a component $W$ of $D \setminus Z_{x,y}$ which is not a loop of $\Gamma$, we get a contradiction to the definition of $A_k$ because then Lemma~\ref{lem:boundary_shield} gives that we can use $Z_{x,y}$ to construct a set $A$ which satisfies properties~\eqref{it:ic1}--\eqref{it:ic3} and connects a loop of $\Gamma$ contained in $A_k$ to a loop of $\Gamma$ which is not contained in $A_k$.  This contradicts the definition of $A_k$ and therefore $X$ is dense in $I$.

Since $X$ is dense in $I$ and there exists $z,w \in I$ distinct so that $\metplus{z}{w}{\Gamma} > 0$ we may assume that $z,w \in I$ are such that $\metplus{z}{\partial D}{\Gamma} > 0$ and $\metplus{z}{w}{\Gamma} > 0$.  This implies that there exists $\epsilon > 0$ so that for all $v \in B(z,\epsilon)$ we have that $\metplus{v}{\partial D}{\Gamma} > 0$ and $\metplus{v}{w}{\Gamma} > 0$.  Let $z_1 \in \interior{A_k} \cap B(z,\epsilon)$ and $z_2 \in B(z,\epsilon) \setminus A_k$.  Let $A^0$ be the set of points $u$ so that $\metplus{u}{[z_1,z_2] \cup (\partial A_k \cap B(z,\epsilon))}{\Gamma} = 0$.  Then $A^0$ has positive distance to $\partial D$.  Let $\wt{A}$ be the closure of the union of $A^0$ and the loops of $\Gamma$ which intersect it and let $A$ be the complement of the unbounded component of $\C \setminus \wt{A}$.  Then $A$ satisfies properties~\eqref{it:ic1}--\eqref{it:ic3} and contains a loop inside of $A_k$ and a loop outside of $A_k$.  This also contradicts the definition of $A_k$.  Altogether, this completes the proof of part~\eqref{it:boundary_zero}.

We now turn to prove part~\eqref{it:boundary_zero}. As $Y_k$ is connected, to show that $U_k$ is simply connected it suffices to show that $U_k$ is connected.  Suppose that $U_k$ is not connected.  Let $V_k$ be the component of $U_k$ whose closure contains $\CL_k$ and the set of points surrounded by $\CL_k$.  Let $V$ be any other component of $U_k$ and assume that $V$ contains the points surrounded by some loop $\CL_i \in \Gamma$.  Then there cannot exist a set $A$ which satisfies \eqref{it:ic1}--\eqref{it:ic3} and has a loop $\CL_V$ (resp.\ $\wt{\CL}$) in $V$ (resp.\ $V_k$) whose harmonic measure in $\C \setminus A$ is positive because~\eqref{it:ic3} will in particular be violated.  To explain this point in further detail, suppose that there are such loops $\CL_V$, $\wt{\CL}$ and such a set~$A$.  Assume that $\partial V \cap \partial D = \emptyset$ so that $u,v \in \partial V$ implies $\metplus{u}{v}{\Gamma} = 0$ for simplicity (the case that $\partial V \cap \partial D = \emptyset$ is similar).  As the harmonic measure of $\CL_V$ in $\C \setminus A$ is positive, it follows from Lemma~\ref{lem:no_zero_length_on_boundary} that there is a point in $\CL_V \cap \partial A$ which is in $V$ hence $\partial V \setminus A \neq \emptyset$.  Since $A \cap \partial V$ is closed, this implies that $\partial V \setminus A$ has at least two points.  Let $u,v$ be distinct prime ends in $\partial V \setminus A$.  Then either $\cwBoundary{u}{v}{\partial V}$ or $\ccwBoundary{u}{v}{\partial V}$ is disconnected by $A$.  Assume that we are in the former situation.  Then we can find $O$ open and $\delta > 0$ small so that $u,v \in O$ and $A$ disconnects $u,v$ in the $\delta$-neighborhood of $O$, and $O$ contains $\cwBoundary{u}{v}{\partial V}$.  By the definition of $A$, this would imply that $\metplus{u}{v}{\Gamma_{O}} > 0$, which is a contradiction.  The same argument applies in the case that $A$ disconnects $\ccwBoundary{u}{v}{\partial U}$.  Therefore if $i \sim_0 k$ then the points surrounded by $\CL_i$ are contained in $V_k$, which implies that $U_k$ has only one component.

That~\eqref{it:boundary_equal} holds follows from the same argument used to show that $\partial A_k = \partial \interior{A_k}$ given in the beginning of the proof.
\end{proof}

\begin{figure}[ht!]
\begin{center}
\includegraphics[scale=1]{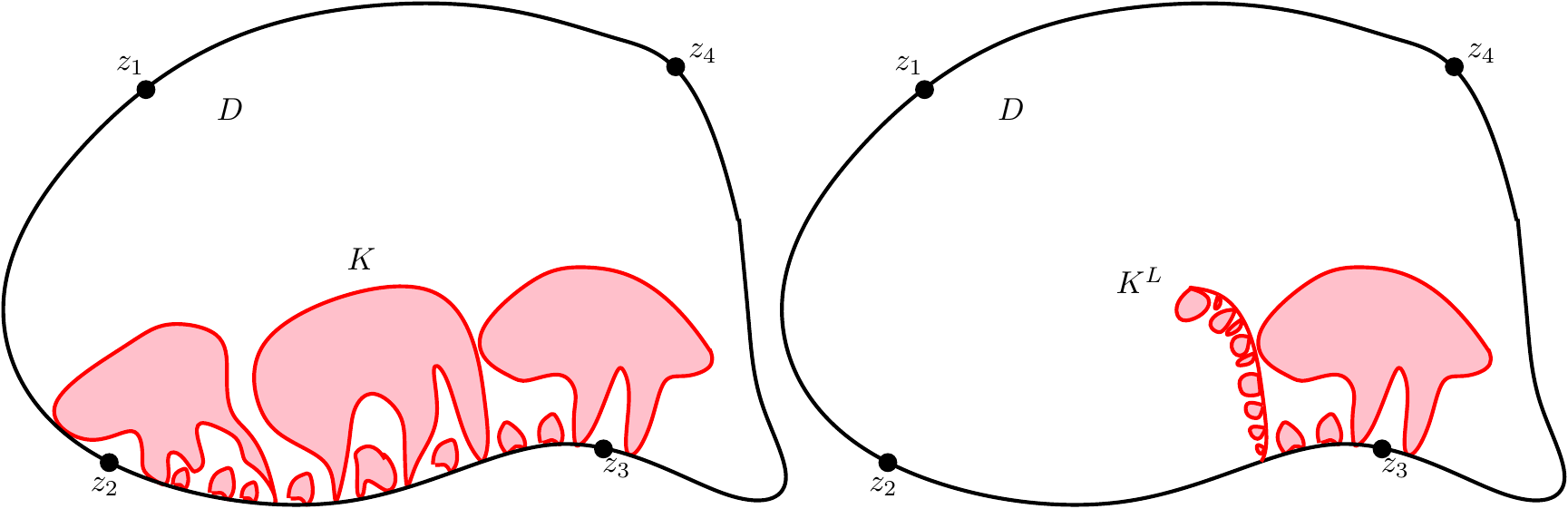}	
\end{center}
\caption{\label{fig:exploration_illustration} Illustration of the exploration.  {\bf Left:} Shown are the sets $A_k$ which intersect $\ccwBoundary{z_2}{z_3}{\partial D}$, i.e., $E_k$ occurs (equivalently $k \in \CK$).  There is a natural clockwise ordering of these sets based on their rightmost intersection with $\ccwBoundary{z_2}{z_3}{\partial D}$ (though one has to be careful in the case that multiple such sets have the same rightmost intersection point).  {\bf Right:} We can explore the $A_k$ for which $k \in \CK$ occurs from right to left using the aforementioned ordering and we can also partially explore an $A_k$ by conditioning on part of the set $Y_k$ of $z$ with $\metplus{z}{\partial A_k}{\Gamma} = 0$ starting from the rightmost intersection with $\ccwBoundary{z_2}{z_3}{\partial D}$ as well as the loops of $\Gamma$ which intersect this part of $Y_k$.  One can think of this exploration as being continuously parameterized by using the extremal length in the remaining domain between $\ccwBoundary{z_4}{z_1}{\partial D}$ and $\ccwBoundary{z_2}{z_3}{\partial D}$ together with the part of the boundary which is in the exploration.}
\end{figure}

We are now going to describe an exploration of the $A_k$'s for $k \in \CK$ which goes in the clockwise direction.  See Figure~\ref{fig:exploration_illustration} for an illustration.  Since we do not have any extra information on the $A_k$'s for $k \in \CK$ (e.g., regularity of $\partial A_k$) beyond the fact that they are closed sets and satisfy the assertions of Lemma~\ref{lem:boundary_good}, some care will be needed in describing this exploration.

We note that there is a natural clockwise ordering of the $A_k$'s for $k \in \CK$.  Indeed, if $i,j \in \CK$ then we say that $i \leq_\CK j$ if the rightmost point of $A_i \cap \ccwBoundary{z_2}{z_3}{\partial D}$ viewed as a prime end in the boundary of the component $W_{i,j}$ of $D \setminus A_j$ containing $A_i$ is to the right of the rightmost prime end in $\partial W_{i,j}$ corresponding to a point in $A_j \cap \ccwBoundary{z_2}{z_3}{\partial D}$.  If $i,j \in \CK$ and $i \sim j$ then we will write $i =_\CK j$.  Finally, if $i,j \in \CK$ and $i \leq_\CK j$ and $i \neq_\CK j$ then we will write $i <_\CK j$.

\newcommand{\EL}[3]{\mathrm{EL}(#1,#2;#3)}

For a simply connected domain $D \subseteq \C$ and arcs $I,J \subseteq \partial D$ we let $\EL{I}{J}{D}$ denote the extremal length between $I$ and $J$ in $D$.  Suppose we are on the event that $K \cap \ccwBoundary{z_4}{z_1}{\partial D} = \emptyset$.  Let $W$ be the component of $D \setminus K$ with $\ccwBoundary{z_4}{z_1}{\partial D}$ on its boundary and let $I = \partial W \cap (K \cup \ccwBoundary{z_2}{z_3}{\partial D})$.  We note that $\EL{I}{\ccwBoundary{z_4}{z_1}{\partial D}}{W} \leq \EL{\ccwBoundary{z_2}{z_3}{\partial D}}{\ccwBoundary{z_4}{z_1}{\partial D}}{D}$.  For each $j \in \CK$ we let $K_j$ be the closure of the union of the $A_k$'s for $k \in \CK$ and with $k \leq_\CK j$.  We also let $W_j$ be the component of $D \setminus K_j$ with $\ccwBoundary{z_4}{z_1}{\partial D}$ on its boundary and let $I_j = \partial W_j \cap (K_j \cup \ccwBoundary{z_2}{z_3}{\partial D})$.  We let $\CK_L$ be the set of $j \in \CK$ with $\EL{I_j}{\ccwBoundary{z_4}{z_1}{\partial D}}{W_j} \leq L$.  Finally, we let $K_0^L = \cap_{j \in \CK_L} K_j$, $W^L$ be the component of $D \setminus K^L$ with $\ccwBoundary{z_4}{z_1}{\partial D}$ on its boundary, and $I^L = \partial W^L \cap (K^L \cup \ccwBoundary{z_2}{z_3}{\partial D})$.  Then we have that $\EL{I^L}{\ccwBoundary{z_4}{z_1}{\partial D}}{W^L} \leq L$.

If $\EL{I^L}{\ccwBoundary{z_4}{z_1}{\partial D}}{W^L} < L$, then there exists $j \in \CK$ and $K_j = K_0^L$.  In this case, we let $Y_j$ and $U_j$ be as in~\eqref{it:simply_connected} of Lemma~\ref{lem:boundary_good} and let $\varphi \colon \D \to U_j$ be a conformal map which sends $1$ to the rightmost point of $Y_j \cap \ccwBoundary{z_2}{z_3}{\partial D}$.  For each $\theta \in [0,2\pi)$, we let $I_j^\theta = \varphi(\ccwBoundary{1}{e^{i\theta}}{\partial \D})$.  We then let $A_j^\theta$ be equal to the closure of the union of $I_j^\theta$ and the loops of $\Gamma$ which intersect $I_j^\theta$.  We then let $K_j^\theta$ be equal to the closure of the union of the $A_k$ with $k \in \CK$ and $k <_\CK j$ together with $A_j^\theta$.  We let $W_j^\theta$ be the component of $D \setminus K_j^\theta$ with $\ccwBoundary{z_4}{z_1}{\partial D}$ on its boundary and let $J_j^\theta = \partial W_j^\theta \cap (K_j^\theta \cup \ccwBoundary{z_2}{z_3}{\partial D})$.  Finally, we let $\theta_L$ be the smallest $\theta \in [0,2\pi)$ so that $\EL{J_j^\theta}{\ccwBoundary{z_4}{z_1}{\partial W_j^\theta}}{W_j^\theta} \leq L$.  We then let $K^L = K_j^{\theta_L}$.

\begin{lemma}
\label{lem:component_local}
For each $L$, we have that $K^L$ is local for $(\Gamma,\metplus{\cdot}{\cdot}{\Gamma})$.  We also have that $K$ is local for $(\Gamma,\metplus{\cdot}{\cdot}{\Gamma})$.
\end{lemma}
\begin{proof}
We will deduce the assertion from Lemma~\ref{lem:loc_characterization}.  Suppose that $U \subseteq D$ is open, $\epsilon > 0$, and we use the notation from the statement of Lemma~\ref{lem:loc_characterization}.  Then the event $K^L \cap U^* = \emptyset$ is determined by~$\CF_U^\epsilon$.  Since $\epsilon > 0$ was arbitrary, we have that $K^L \cap U^* = \emptyset$ is determined by~$\CF_U$, which proves the result for $K^L$.  The same argument gives the locality of $K$ for $(\Gamma,\metplus{\cdot}{\cdot}{\Gamma})$.
\end{proof}

\begin{lemma}
\label{lem:ak_cross_finite}
Assume that there exist $x,y \in \partial D$ distinct so that $\p[ Z_{x,y} \neq \emptyset] > 0$.  For each $\delta > 0$, the number of distinct sets $A_k$ with $k \in \CK$ so that
\begin{enumerate}[(i)]
\item $A_k$ is not disconnected from $\ccwBoundary{z_4}{z_1}{\partial D}$ by $A_j$ for $j \in \CK$ with $j \leq_\CK k$ and
\item $A_k$ is not contained in the $\delta$-neighborhood of $\ccwBoundary{z_2}{z_3}{\partial D}$
\end{enumerate}
is a.s.\ finite.  In particular, $\partial K \cap \partial W \subseteq \cup_{k \in \CK} A_k$ so that for each $z \in \partial K \cap \partial W$ there exists $k \in \CK$ with $z \in A_k$.
\end{lemma}
\begin{proof}
Fix $\xi, \delta > 0$.  We note that by the local finiteness of $\CLE_\kappa$ there a.s.\ exist at most finitely many distinct sets $A_k$ so that $A_k$ contains a loop with diameter at least $\xi$.  Recall that $K^L$ is increasing as $L$ decreases.  We let $L_1$ be the largest value of $L$ so that $K^L = K_j$ for some $j \in \CK$ and~$A_j$ contains a loop with diameter at least $\xi$.  Given that we have defined $L_1,\ldots,L_m$, we let~$L_{m+1}$ be the largest $L < L_{m+1}$ so that $K^L = K_j$ for some $j \in \CK$ and $A_j$ contains a loop with diameter at least $\xi$.

Fix a value of $m \in \N$ and let $w_m$ be the leftmost point on $K^{L_m} \cap \partial D$ relative to $z_3$.  Let $x_m$ (resp.\ $y_m$) be the first point in the clockwise (resp.\ counterclockwise) direction along $\partial W^{L_m}$ from $w_m$ which has distance at least $\delta^2$ from $w_m$.  It follows from Lemmas~\ref{lem:e_k_pos}, \ref{lem:component_local} that conditionally on $K^{L_m}$, there is a positive chance that there exists $k \in \N$ so that $A_k$ disconnects $\ccwBoundary{x_m}{y_m}{\partial W^{L_m}}$ from $W^{L_m} \cap \partial B(w_m,\delta)$ in $W^{L_m}$.  Let $N_m$ be the smallest value of $n \geq m$ so that $K^{L_n}$ disconnects $\ccwBoundary{x_m}{w_m}{\partial W^{L_m}}$ from $\ccwBoundary{z_4}{z_1}{\partial D}$.  Then the number of sets $A_j$ discovered by the exploration in $[L_{N_m},L_m]$ which leave the $\delta$-neighborhood of $\ccwBoundary{z_2}{z_3}{\partial D}$ is stochastically dominated by a geometric random variable (whose parameter does not depend on $\xi$).  Since $\xi > 0$ was arbitrary and the parameter of the aforementioned geometric does not depend on $\xi$, the assertion of the lemma follows.
\end{proof}

\begin{lemma}
\label{lem:kl_harmonic_zero}
For each $L$ in each component $W$ of $D \setminus K^L$ the harmonic measure as seen from $w \in W$ of the set of those points of $\partial K^L \setminus \partial K$ which are not contained in a loop of $\Gamma$ is a.s.\ zero.
\end{lemma}
\begin{proof}
Lemma~\ref{lem:component_local} implies that $K^L$ is local for $(\Gamma,\metplus{\cdot}{\cdot}{\Gamma})$.  Consequently, the result follows from the same argument used to prove Lemma~\ref{lem:in_between_stuff_is_small}.  In particular, if $W$ is a component of $D \setminus K^L$ and the set of those points $z \in \partial K^L \setminus \partial K$ which are not in a loop of $\Gamma$ had positive harmonic measure then with positive probability  there would exist some $A_j$ and $z \in \interior{A_j}$ so that with $Y_j$ as in~\eqref{it:simply_connected} of Lemma~\ref{lem:boundary_good} we have that $\metplus{z}{Y_j}{\Gamma} = 0$ and $z \notin Y_j$.  This, in turn, is a contradiction.
\end{proof}

\subsubsection{Disconnecting the opposing arcs by shields}
\label{subsubsec:disconnecting_by_shields}

We consider the following exploration.  Let $D_1 = D$, $\Gamma_1 = \Gamma$, and let $K_1$ be defined in the same way as $K$ in terms of $D_1$, $\Gamma_1$ as described at the beginning of Section~\ref{subsubsec:shield_def}.  We also let $z_2^1 = z_2$ and $z_3^1 = z_3$.  For each $j \geq 1$, given that we have defined the domain $D_j$, marked boundary points $z_1,z_2^j,z_3^j,z_4$, loops $\Gamma_j$ in $D_j$, on $K_j \cap \ccwBoundary{z_4}{z_1}{\partial D} = \emptyset$ we define $K_{j+1}$, $D_{j+1}$, $z_1,z_2^{j+1},z_3^{j+1},z_4$, and $\Gamma_{j+1}$ as follows.  We let $z_2^{j+1}$ (resp.\ $z_3^{j+1}$) be the leftmost (resp.\ rightmost) point on $\ccwBoundary{z_1}{z_2}{\partial D_j}$ (resp.\ $\ccwBoundary{z_3}{z_4}{\partial D_j}$) relative to $z_2$ (resp.\ $z_3$) which is contained in $K_j$.  We let $D_{j+1}$ be the component of $D_j \setminus K_j$ with $\ccwBoundary{z_4}{z_1}{\partial D_j}$ on its boundary and let $\Gamma_{j+1}$ be the loops of $\Gamma_j$ which are contained in $D_{j+1}$.  We note that by Lemma~\ref{lem:component_local} we have that $\cup_{i=1}^j K_i$ is local for $(\Gamma,\metplus{\cdot}{\cdot}{\Gamma})$.  In particular, the conditional law of $\Gamma_{j+1}$ and the interior-internal pseudometric in $D_{j+1}$ is given by that of a $\CLE_\kappa$ in $D_{j+1}$ equipped with $\metplus{\cdot}{\cdot}{\Gamma_{j+1}}$.  Let $K_{j+1}$ be defined in the same way as $K$ as described at the beginning of Section~\ref{subsubsec:shield_def} but in terms of $D_{j+1}$, $\Gamma_{j+1}$, $z_2^{j+1}$, and $z_3^{j+1}$. As $\EL{\ccwBoundary{z_2^j}{z_3^j}{\partial D_j}}{\ccwBoundary{z_4}{z_1}{\partial D_j}}{D_j}$ is decreasing in $j \in \N$, we see from Lemmas~\ref{lem:zero_length_conf}, \ref{lem:e_k_pos} that there exists $p_0 > 0$ so that
\begin{equation}
\label{eqn:stoch_dom_geom}
\p[ K_{j+1} \cap \ccwBoundary{z_4}{z_1}{\partial D} \neq \emptyset \giv K_j \cap \ccwBoundary{z_4}{z_1}{\partial D} = \emptyset ] \geq p_0 \quad\text{for each}\quad j \in \N.
\end{equation}
Let $N$ be the first $j$ so that $K_j \cap \ccwBoundary{z_4}{z_1}{\partial D} \neq \emptyset$.  Then~\eqref{eqn:stoch_dom_geom} implies that $N$ is stochastically dominated by a geometric random variable with parameter $p_0$.  In particular, $N$ is a.s.\ finite.

\newcommand{\bk}{\mathbf k}

For each $1 \leq j \leq N$ and $k \in \N$ we let $A_k^j$, $E_k^j$ be defined as above using $\metplus{\cdot}{\cdot}{\Gamma_j}$ and $\ccwBoundary{z_2^j}{z_3^j}{\partial D_j}$ in place of $\metplus{\cdot}{\cdot}{\Gamma}$ and $\ccwBoundary{z_2}{z_3}{\partial D}$.  Lemma~\ref{lem:ak_cross_finite} implies that there a.s.\ exist $k_1,\ldots,k_N$ so that the $A_{k_j}^j$ disconnect $\ccwBoundary{z_1}{z_2}{\partial D}$ and $\ccwBoundary{z_3}{z_4}{\partial D}$.  For $\bk = (k_1,\ldots,k_n) \in \N^n$, we let $E_\bk$ be the event that $A_{k_1}^1,\ldots,A_{k_n}^n$ are distinct, disconnect $\ccwBoundary{z_1}{z_2}{\partial D}$ and $\ccwBoundary{z_3}{z_4}{\partial D}$, and that no proper subset of $\{A_{k_1}^1,\ldots,A_{k_n}^n\}$ disconnects $\ccwBoundary{z_1}{z_2}{\partial D}$ and $\ccwBoundary{z_3}{z_4}{\partial D}$.  Then what we have shown so far implies that $\cup_{\bk} E_\bk$ a.s.\ occurs where the union is over $\bk \in \cup_{n \in \N} \N^n$.

Suppose that we are working on $E_\bk$.  We let $x_0$ be the rightmost point in $A_1^1 \cap \partial D$ relative to $z_2$.  For each $1 \leq j \leq n-1$ we let $x_j$ be the rightmost intersection point of $A_{k_{j+1}}^{j+1}$ with $A_{k_j}^j$.  We let $x_n$ be the rightmost intersection point of $A_{k_n}^n$ with $\partial D$.  For each $1 \leq j \leq n$, we let $Y_{k_j}^j$ and $U_{k_j}^j$ be as in~\eqref{it:simply_connected} of Lemma~\ref{lem:boundary_good} and let $C_j = \ccwBoundary{x_{j-1}}{x_j}{\partial U_{k_j}^j}$.  Then we note that $\cup_{j=1}^n C_j$ disconnects $\ccwBoundary{z_1}{z_2}{\partial D}$ from $\ccwBoundary{z_3}{z_4}{\partial D}$.  For each $1 \leq j \leq n$, we let $Q_j$ be the closure of the union of $C_j$ and the loops of $\Gamma$ which intersect it.  We let $Q_\bk = \cup_{j=1}^n Q_j$.  We note that $Q_\bk$ cannot intersect $\ccwBoundary{z_1}{z_2}{\partial D}$.  Indeed, if $Q_j \cap \ccwBoundary{z_1}{z_2}{\partial D} \neq \emptyset$ then as the loops of $\Gamma$ cannot hit $\ccwBoundary{z_1}{z_2}{\partial D}$ we would have that $C_j \cap \ccwBoundary{z_1}{z_2}{\partial D} \neq \emptyset$.  This implies that $C_j$ intersects $\ccwBoundary{z_1}{z_2}{\partial D}$ at one of its endpoints (i.e., $x_{j-1}$ or $x_j$ is in $\ccwBoundary{z_1}{z_2}{\partial D}$) or there exists $y \in \ccwBoundaryOpen{x_{j-1}}{x_j}{\partial U_{k_j}^j}$ which is also in $\ccwBoundary{z_1}{z_2}{\partial D}$.  The latter cannot happen as it would contradict Lemma~\ref{lem:boundary_good} which implies that each $U_{k_j}^j$ is connected.  The former a.s.\ cannot happen because the probability that any fixed boundary point is the rightmost in $A_{k_{j+1}}^{j+1}$ is a.s.\ zero by Lemma~\ref{lem:zero_length_conf}.  Let $W_\bk$ be the component of $D \setminus Q_\bk$ with $\ccwBoundary{z_1}{z_2}{\partial D}$ on its boundary.  We also let $\Gamma_\bk$ be the loops of $\Gamma$ contained in $W_\bk$.

We note that it can be that $Q_1$ does not contain $\CL_{k_1}$.  Since $K^L$ as defined earlier visits $Y_{k_1}^1$ in counterclockwise order, it follows from Lemma~\ref{lem:kl_harmonic_zero} that there exists $i \in \N$ so that~$Q_1$ contains~$\CL_i$.  Applying the same principle to $2 \leq j \leq n$ implies that each~$Q_j$ also contains a loop of $\Gamma$.  In particular, if we let $F_\bk$ be the event that $E_\bk$ occurs and each $Q_j$ contains $\CL_j$ we have that $\cup_{\bk} F_\bk$ a.s.\ occurs where the union is over $\bk \in \cup_{n \in \N} \N^n$.

\begin{lemma}
\label{lem:disconnecting_local}
On $F_\bk$, given $Q_\bk$ the conditional law of $\Gamma_\bk$ and the associated interior-internal pseudometric in $W_\bk$ is given by that of a $\CLE_\kappa$ in $W_\bk$ equipped with $\metplus{\cdot}{\cdot}{\Gamma_\bk}$.	
\end{lemma}
\begin{proof}
This follows from the same argument used to prove Lemma~\ref{lem:loc_characterization}.  In particular, suppose that $U \subseteq D$ is open, $\epsilon > 0$, and we let $\CF_U^\epsilon$, $\CF_U$ be as in the statement of Lemma~\ref{lem:loc_characterization}.  Then it suffices to show that the event $F_\bk \cap \{ Q_\bk \cap U = \emptyset\}$ is determined by $\CF_U^\epsilon$.  This, however, is easy to see from the definition of $F_\bk$ and $Q_\bk$. 
\end{proof}

\begin{lemma}
\label{lem:disconnecting_harmonic_loops}
Suppose we are working on $F_\bk$ and that $z \in W_\bk$.  Then the harmonic measure as seen from $z$ of the points in $D \cap \partial Q_\bk$ which are not in a loop of $\Gamma$ is a.s.\ equal to~$0$.
\end{lemma}
\begin{proof}
This is a consequence of Lemma~\ref{lem:kl_harmonic_zero}.
\end{proof}

\begin{figure}[ht!]
\begin{center}
\includegraphics[scale=1]{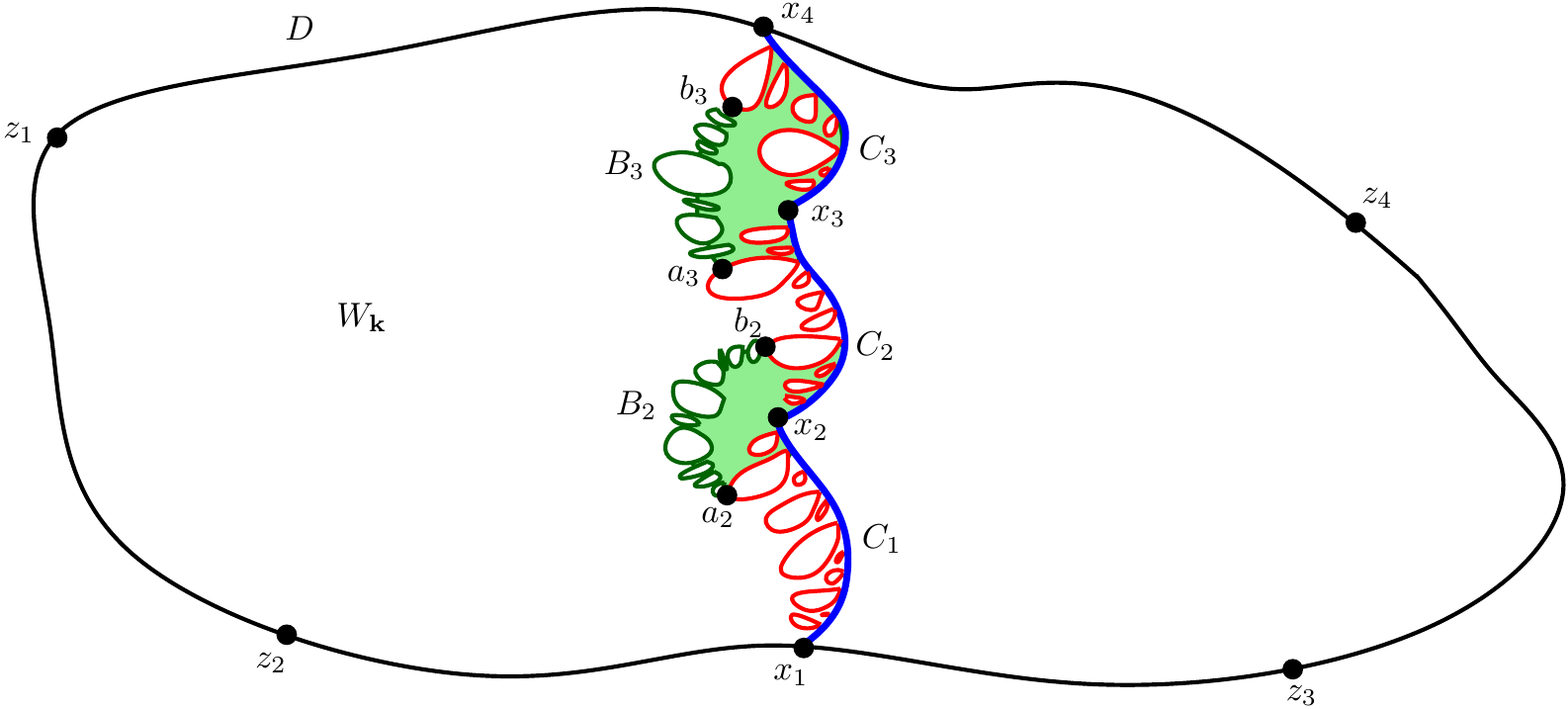}	
\end{center}
\caption{\label{fig:pos_def_proof} Illustration of the proof of Proposition~\ref{prop:no_zero_crossings}.  The sets $C_1$, $C_2$, $C_3$ are shown in blue and the loops of $\Gamma$ which intersect them are shown in red.  The green sets are $B_2$ and $B_3$ and the loops of $\Gamma$ which hit $B_2$, $B_3$ are shown in dark green.  We note that the illustration might suggest that the argument used to prove Proposition~\ref{prop:no_zero_crossings} gives that $\metplus{\ccwBoundary{z_2}{z_3}{\partial D}}{\ccwBoundary{z_4}{z_1}{\partial D}}{\Gamma} = 0$ since if $z,w$ are in the same component of $C_i \setminus \partial D$ for some $i$ then $\metplus{z}{w}{\Gamma} = 0$.  However, there is the possibility that one of the $C_i$'s hits $\ccwBoundary{z_3}{z_4}{\partial D}$ in which case there may be $z,w$ in different components of $C_i \setminus \partial D$ with $\metplus{z}{w}{\Gamma} > 0$.}
\end{figure}

\begin{proof}[Proof of Proposition~\ref{prop:no_zero_crossings}]
See Figure~\ref{fig:pos_def_proof} for an illustration of the proof.  We assume that we have the setup described just before the statement of Lemma~\ref{lem:disconnecting_local}.  Suppose that $\metplus{\ccwBoundary{z_1}{z_2}{\partial D}}{\ccwBoundary{z_3}{z_4}{\partial D}}{\Gamma} = 0$.  Fix $n \in \N$, $\bk \in \N^n$, and suppose that we are working on $F_\bk$.  By the definition of $Q_\bk$, we must have that $\metplus{\ccwBoundary{z_1}{z_2}{\partial D}}{x_j}{\Gamma_\bk} = 0$ for some $j$.  Thus to prove the result it suffices to show that $\metplus{\ccwBoundary{z_1}{z_2}{\partial D}}{x_j}{\Gamma_\bk} > 0$ for all $0 \leq j \leq n$.

Let $(U_{2,m})$ be an increasing sequence in $\dyad(D)$ with $\cup_m U_{2,m} = W_\bk$.  Fix $1 \leq j \leq n$ and let $a_j$, $b_j$ be distinct in $\partial W_\bk$ and distinct from $x_j$ so that $x_j \in I_j = \ccwBoundaryOpen{a_j}{b_j}{\partial W_\bk}$.  For each $\epsilon > 0$ we let $I_j^\epsilon$ be the $\epsilon$-neighborhood of $I_j$.  We let $B_j$ be the set of $u \in W_\bk$ such that for every $\epsilon > 0$ there exists $m_0 \in \N$ so that $m \geq m_0$ implies that $\metres{U_{2,m}}{u}{I_j^\epsilon}{\Gamma_\bk} = 0$.  We note that $\diam(B_j) \to 0$ as $a_j \to x_j$ from the left and $b_j \to x_j$ from the right by Lemmas~\ref{lem:zero_length_conf} and~\ref{lem:zero_metric_ball_does_not_hit}.  Fix $\delta > 0$ and assume we have chosen $a_j,b_j$ sufficiently close to $x_j$ so that the conditional probability that $\diam(B_j) \leq \delta$ given everything else is at least $1-\delta$.

We note that if $u \in W_\bk \setminus B_j$ then $\metres{U_{2,m}}{u}{B_j}{\Gamma_\bk} > 0$ for all $m \in \N$.  In particular, the only way that $\metplus{u}{x_j}{\Gamma_\bk} = 0$ can hold is if $\metplus{u}{a_j}{\Gamma_\bk} = 0$ or $\metplus{u}{b_j}{\Gamma_\bk} = 0$.  By Lemma~\ref{lem:disconnecting_harmonic_loops}, we know that there are loops of $\Gamma$ in $\partial W_\bk$ which are arbitrarily close to $x_j$.  Combining this with Lemma~\ref{lem:no_zero_length_on_boundary} we also know that there exist points $a_j, b_j \in \partial W_\bk$ which are arbitrarily close to $x_j$ with the property that $\metplus{u}{a_j}{\Gamma_\bk} > 0$ for all $u \in \closure{W_\bk} \setminus \{a_j\}$ and $\metplus{u}{b_j}{\Gamma_\bk} > 0$ for all $u \in \closure{W_\bk} \setminus \{b_j\}$.  Moreover, whether $a_j$, $b_j$ have this property is independent of $\metplus{\cdot}{\cdot}{\Gamma_\bk}$.  Therefore we may assume without loss of generality that we have made this choice.  Altogether, we see that the conditional probability that $\metplus{\ccwBoundary{z_1}{z_2}{\partial D}}{x_j}{\Gamma_\bk} > 0$ given everything else is at least $1-\delta$.  Since $\delta > 0$ was arbitrary, we have that $\metplus{\ccwBoundary{z_1}{z_2}{\partial D}}{x_j}{\Gamma_\bk} > 0$ a.s.  Combining completes the proof.
\end{proof}

\begin{proof}[Proof of Proposition~\ref{prop:overall_no_zero_length}]
Suppose that we have the setup of Proposition~\ref{prop:overall_no_zero_length}.  Suppose that with positive probability there exist $z,w \in \Upsilon$ distinct with $\met{z}{w}{\Gamma} = 0$.  We claim that we may assume without loss of generality that $z,w \in \Upsilon \cap D$.  To see this, let $Z = \{ u \in \Upsilon : \met{z}{u}{\Gamma} = 0\}$.  Then~$Z$ is a connected set.  If $Z \subseteq \partial D$, then we can write $Z = \cwBoundary{a}{b}{\partial D}$ for $a,b \in \partial D$ distinct.  This, in turn, contradicts the statement of Lemma~\ref{lem:no_zero_length_on_boundary} which completes the proof of the claim in this case.  For the remainder of the proof, we shall therefore assume that we have $z,w \in \Upsilon \setminus \partial D$ with $\met{z}{w}{\Gamma} = 0$.

Fix $x,y \in \partial D$ distinct and let $\eta$ be an $\SLE_\kappa^1(\kappa-6)$ process in $D$ from $x$ to $y$ coupled with $\Gamma$ as a CPI.  It follows from Proposition~\ref{prop:cpi_path_close} that with positive conditional probability given everything else, $\eta$ surrounds a domain~$U$ counterclockwise (so that there are no loops of $\Gamma$ on its boundary) so that there are points $u,v \in \partial U$ distinct with $\metplus{u}{v}{\Gamma_U} = 0$ where $\Gamma_U$ denotes the loops of $\Gamma$ contained in $U$.  By applying a conformal transformation $D \to \D$ and using Lemma~\ref{lem:zero_length_conf}, we get a contradiction to Proposition~\ref{prop:no_zero_crossings}, which completes the proof.
\end{proof}

\subsection{Subsequential limits are geodesic}
\label{subsec:geodesics}

\begin{proposition}
\label{prop:geodesic_limit}
Suppose that we have the setup of Theorem~\ref{thm:cle_loop} and that $(\epsilon_j)$ is a sequence of positive numbers decreasing to $0$ so that $(\medianHP{\epsilon_j})^{-1} \metapprox{\epsilon_j}{\cdot}{\cdot}{\Gamma}$ converges weakly as $j \to \infty$.  Let  $\met{\cdot}{\cdot}{\Gamma}$ have the law of the subsequentially limiting metric on $\Upsilon$.  Then $\met{\cdot}{\cdot}{\Gamma}$ is a.s.\ geodesic.
\end{proposition}
\begin{proof}
To prove that a metric space $(X,d)$ equipped with a good probability measure $\mu$ (i.e., $\mu(U) > 0$ for every non-empty open set $U \subseteq X$) is geodesic, it suffices to show that if $z,w$ are independent samples from $\mu$ then there a.s.\ exists $u$ such that $d(z,u) = d(u,w) = d(z,w) / 2$.  That is, $z,w$ a.s.\ have a midpoint.  In our case, we will take the metric to be $\met{\cdot}{\cdot}{\Gamma}$ and the good measure to be given by $\qcarpet{h}{\Upsilon}$ where $h$ is a field which describes an independent sample from $\qdiskL{\gamma}{1}$ parameterized by $D$.

Let $(\epsilon_j)$ be a sequence of positive numbers tending to $0$ as $j \to \infty$ along which $(\medianHP{\epsilon_j})^{-1}\metapprox{\epsilon_j}{\cdot}{\cdot}{\Gamma}$ converges weakly.  Suppose that we pick $z,w$ conditionally independently from $\qcarpet{h}{\Upsilon}$ given everything else and let $\omega_j \colon [0,1] \to \Upsilon$ be a continuous path such that $\omega_j(0) = z$, $\omega_j(1) = w$ and $\lebneb{\epsilon_j}(\omega_j) \leq \metapprox{\epsilon_j}{z}{w}{\Gamma} + \epsilon_j \medianHP{\epsilon_j}$.  Let $s_j \in [0,1]$ be such that
\begin{align}
\label{eqn:s_j_def}
|\lebneb{\epsilon_j}(\omega_j([0,s_j])) - \tfrac{1}{2}\metapprox{\epsilon_j}{z}{w}{\Gamma}| \leq 2\epsilon_j \medianHP{\epsilon_j}.
\end{align}
Fix $A, \delta > 0$ and let $t_j$ be the last time after $s_j$ that $\dist(\omega_j(t_j),\omega_j([0,s_j))) = \delta$.  Let $\beta > 0$ be as in the statement of Proposition~\ref{prop:interior_tightness} and let $X_j$ be the $\beta$-H\"older norm of $(\medianHP{\epsilon_j})^{-1} \metapprox{\epsilon_j}{\cdot}{\cdot}{\Gamma}$.  On $\{X_j \leq A\}$, we have that $\metapprox{\epsilon_j}{\omega_j(t_j)}{\omega_j([0,s_j])}{\Gamma} \leq A \delta^\beta$ hence there exists $r_j \in [0,s_j]$ and a path $\wt{\omega}_j \colon [0,1] \to \Upsilon$ with $\wt{\omega}_j(0) = \omega_j(r_j)$ and $\wt{\omega}_j(1) = \omega_j(t_j)$ so that $\lebneb{\epsilon_j}(\wt{\omega}_j) \leq (A \delta^\beta + \epsilon_j)\medianHP{\epsilon_j}$.  Let $\wh{\omega}_j$ be given by starting with $\omega_j$ and then replacing $\omega_j|_{[r_j,t_j]}$ with $\wt{\omega}$.  Then we have that
\begin{equation}
\label{eqn:middle_part_bound}
\lebneb{\epsilon_j}(\wh{\omega}_j([r_j,t_j])) \leq (A \delta^\beta + \epsilon_j) \medianHP{\epsilon_j} \quad\text{on}\quad \{X_j \leq A\}.
\end{equation}
On $\{X_j \leq A\}$ we also have that
\begin{align}
\metapprox{\epsilon_j}{z}{w}{\Gamma}
&\leq 
\lebneb{\epsilon_j}(\wh{\omega}_j([0,1])) \notag \\
&\leq
\lebneb{\epsilon_j}(\omega_j([0,r_j])) + 
\lebneb{\epsilon_j}(\wt{\omega}_j([r_j,t_j])) + 
\lebneb{\epsilon_j}(\omega_j([t_j,1])) \notag\\
&\leq \tfrac{1}{2}\metapprox{\epsilon_j}{z}{w}{\Gamma} + (3\epsilon_j + A \delta^\beta) \medianHP{\epsilon_j} + \lebneb{\epsilon_j}(\omega_j([t_j,1])) \quad\text{(by ~\eqref{eqn:s_j_def}, \eqref{eqn:middle_part_bound})}. \label{eqn:right_part_lbd}
\end{align}
Rearranging~\eqref{eqn:right_part_lbd} implies that
\[ \lebneb{\epsilon_j}(\omega_j([t_j,1])) \geq \frac{1}{2} \metapprox{\epsilon_j}{z}{w}{\Gamma} - (3\epsilon_j + A \delta^\beta) \medianHP{\epsilon_j} \quad\text{on}\quad \{X_j \leq A\}.\]
Assuming that $j$ is sufficiently large so that $2\epsilon_j \leq \delta$, we have that $\dist(\omega_j([0,r_j]),\omega_j([t_j,1])) \geq \delta \geq 2\epsilon_j$ so that
\[ \lebneb{\epsilon_j}(\omega_j([0,r_j])) + \lebneb{\epsilon_j}(\omega_j([t_j,1])) \leq \lebneb{\epsilon_j}(\omega_j([0,1])).\]
Applying~\eqref{eqn:s_j_def} again and rearranging gives
\begin{equation}
\label{eqn:t_j_to_1_bound}
\lebneb{\epsilon_j}(\omega_j([t_j,1])) \leq \frac{1}{2} \metapprox{\epsilon_j}{z}{w}{\Gamma} + 2\epsilon_j \medianHP{\epsilon_j}.
\end{equation}
By passing to a further subsequence if necessary, we may assume that the joint law of $z$, $w$, $\omega_j(r_j)$, $\omega_j(t_j)$, $(\medianHP{\epsilon_j})^{-1} \metapprox{\epsilon_j}{\cdot}{\cdot}{\Gamma}$, $h$, and $\Gamma$ converges weakly.  Write $z$, $w$, $u$, $v$, $\met{\cdot}{\cdot}{\Gamma}$, $h$, $\Gamma$ for the limit.  Let $X$ be the $\beta$-H\"older norm of $\met{\cdot}{\cdot}{\Gamma}$.  On $\{X \leq A\}$ we then have that
\begin{align*}
	\met{z}{u}{\Gamma} &\leq \frac{1}{2} \met{z}{w}{\Gamma} \quad\text{(by~\eqref{eqn:s_j_def})},\\
	\met{u}{v}{\Gamma} &\leq A \delta^\beta \quad\text{(by~\eqref{eqn:middle_part_bound})}\quad\text{and}\\ 
	\met{v}{w}{\Gamma} &\leq \frac{1}{2} \met{z}{w}{\Gamma} + A \delta^\beta \quad\text{(by~\eqref{eqn:t_j_to_1_bound})}.
\end{align*}
This implies that
\[ \met{u}{w}{\Gamma} \leq \frac{1}{2} \met{z}{w}{\Gamma} + 2A \delta^\beta \quad\text{hence}\quad \met{z}{u}{\Gamma} \geq \met{z}{w}{\Gamma} - \met{u}{w}{\Gamma} \geq \frac{1}{2}\met{z}{w}{\Gamma} - 2 A \delta^\beta .\]
Since $\delta > 0$ was arbitrary and the metric $\met{\cdot}{\cdot}{\Gamma}$ is compact, we see that on $\{X \leq A\}$ there is a.s.\ a midpoint for $z$, $w$.  This completes the proof as $X$ is a.s.\ finite and $A > 0$ was arbitrary.
\end{proof}

\subsection{Comparability of normalizations}
\label{subsec:quantiles_comparable}

\begin{proposition}
\label{prop:quantiles_comparable}
There exist constants $c_1, c_2 > 0$ so that for all $\epsilon \in (0,1)$ we have that
\[ c_1 \medianHP{\epsilon} \leq \median{\epsilon} \leq c_2 \medianHP{\epsilon}.\]
\end{proposition}
\begin{proof}
Let $D_\epsilon = \sup_{z,w \in \partial D} \metapprox{\epsilon}{z}{w}{\Gamma}$.  By the definition of $\median{\epsilon}$, we have that
\[ \p[ \median{\epsilon}^{-1} D_\epsilon \geq 1] \geq 1/2 \quad\text{and}\quad \p[ \median{\epsilon}^{-1} D_\epsilon \leq 1] \geq 1/2.\]
Suppose that there exists a sequence $(\epsilon_j)$ of positive numbers with $\epsilon_j \to 0$ such that $\medianHP{\epsilon_j}/\median{\epsilon_j} \to 0$ as $j \to \infty$.  By passing to a subsequence if necessary, we may assume without loss of generality that $(\medianHP{\epsilon_j})^{-1} \metapprox{\epsilon_j}{\cdot}{\cdot}{\Gamma}$ converges weakly as $j \to \infty$.  Then $(\medianHP{\epsilon_j})^{-1} D_{\epsilon_j}$ converges weakly to a positive and finite random variable.  Consequently,
\begin{align*}
   \p[ \median{\epsilon_j}^{-1} D_{\epsilon_j} \geq 1]
&= \p\!\left[ \frac{\medianHP{\epsilon_j}}{\median{\epsilon_j}} (\medianHP{\epsilon_j})^{-1} D_{\epsilon_j} \geq 1 \right] \to 0 \quad\text{as} \quad j \to \infty. 
\end{align*}
This contradicts the definition of $\median{\epsilon_j}$.  Similarly, if there exists a sequence $(\epsilon_j)$ of positive numbers with $\epsilon_j \to 0$ such that $\medianHP{\epsilon_j}/\median{\epsilon_j} \to \infty$ as $j \to \infty$ then (passing to a further subsequence if necessary) we have that
\begin{align*}
   \p[ \median{\epsilon_j}^{-1} D_{\epsilon_j} \leq 1]
&= \p\!\left[ \frac{\medianHP{\epsilon_j}}{\median{\epsilon_j}} (\medianHP{\epsilon_j})^{-1} D_{\epsilon_j} \leq 1 \right] \to 0 \quad\text{as} \quad j \to \infty. 
\end{align*}
This also contradicts the definition of $\median{\epsilon_j}$.  Combining proves the result.
\end{proof}

\appendix

\section{Modulus of continuity of space-filling $\SLE$ with quantum parameterization}
\label{app:mod_of_cont}

\begin{proposition}
\label{prop:space_filling_on_quantum_disk}
There exists a constant $\alpha_\HO \in (0,1)$ so that the following is true.  Suppose that $(\D,h,-i,i) \sim \qdiskL{\gamma}{1}$.  Let $\eta'$ be a space-filling $\SLE_{\kappa'}$ loop on $\D$ from $-i$ to $-i$ which is sampled independently of $h$ and then subsequently parameterized by quantum area.  Then $\eta'$ is a.s.\ $\alpha_\HO$-H\"older continuous.
\end{proposition}

We will prove Proposition~\ref{prop:space_filling_on_quantum_disk} by using that for each $\xi > 1$ and all $\epsilon > 0$ small enough, whenever a space-filling $\SLE_{\kappa'}$ travels distance $\epsilon$ it fills in a ball of radius $\epsilon^\xi$ (Proposition~\ref{prop:space_filling_fills_in_balls_on_disk}) together with a lower bound for the quantum measure on a quantum disk (Lemma~\ref{lem:quantum_disk_multifractal}).

\begin{proposition}
\label{prop:space_filling_fills_in_balls_on_disk}
Suppose that $\eta'$ is a space-filling $\SLE_{\kappa'}$ on $\D$ from $-i$ to $-i$.  For each $\xi > 1$ there a.s.\ exists $\epsilon_0 > 0$ so that for all $\epsilon \in (0,\epsilon_0)$ the following is true.  For any $0 \leq a < b$, on the event that $\diam(\eta'([a,b])) \geq \epsilon$ we have that $\eta'([a,b])$ contains a disk of radius at least $\epsilon^\xi$.
\end{proposition}

The following is a restatement of \cite[Proposition~3.4]{ghm2020kpz} as well as the discussion in \cite[Remark~3.9]{ghm2020kpz}.

\begin{lemma}
\label{lem:space_filling_fills_in_balls_in_plane}
Suppose that $\eta'$ is a space-filling $\SLE_{\kappa'}$ from $\infty$ to $\infty$ in $\C$.  For $\xi > 1$, $R > 0$, and $\epsilon > 0$, we let $E_\epsilon$ be the event that the following is true.  For each $\delta \in (0,\epsilon]$ and each $a < b \in \R$ with $\eta'([a,b]) \subseteq B(0,R)$ and $\diam(\eta'([a,b])) \geq \delta$ the set $\eta'([a,b])$ contains a ball of radius at least~$\delta^\xi$.  Then $\p[E_\epsilon] \to 1$ as $\epsilon \to 0$ faster than any power of $\epsilon$.
\end{lemma}

\begin{lemma}
\label{lem:gff_rn_bound}
Suppose that $M > 0$ and that $h$ is a GFF on $\D$ with Dirichlet boundary conditions which are at most $M$ in absolute value.  Fix $z \in \D$ and $\epsilon > 0$ so that $B(z,2\epsilon) \subseteq \D$.  Let $\Fh$ be the harmonic extension of the values of $h$ from $\partial B(z,2\epsilon)$ to $B(z,2\epsilon)$.  Then for every $p > 0$ there exist a constant $C_{p,M} < \infty$ so that
\[ \E\left[ \exp\left( p \sup_{w \in B(z,\epsilon)} | \Fh(w) - \Fh(z)| \right) \right] \leq C_{p,M}.\]
The same likewise holds if $h$ is a whole-plane GFF normalized so that its average on $\partial \D$ is any fixed value in $[-M,M]$.
\end{lemma}
\begin{proof}
This is a standard sort of calculation for the GFF; see the proof of \cite[Lemma~4.4]{mq2020geodesics} as well as the argument used to prove Lemma~\ref{lem:dirichlet_energy}, in particular \eqref{eqn:harmonic_bound1}, \eqref{eqn:harmonic_bound2}, and \eqref{eqn:harmonic_bound3}.
\end{proof}

\begin{proof}[Proof of Proposition~\ref{prop:space_filling_fills_in_balls_on_disk}]
Suppose that $\wt{h}$ is a GFF on $\h$ with boundary conditions given by $-\lambda'+2\pi \chi$ on $\R_-$ and $-\lambda'$ on $\R_+$.  Let $h = \wt{h} \circ \psi - \chi \arg \psi'$ where $\psi \colon \D \to \h$ is a conformal map which takes $-i$ to $0$.  Let $\eta'$ be the space-filling $\SLE_{\kappa'}$ loop associated with $h$.  Fix $z \in \D$ and $r > 0$ so that $B(z,2r) \subseteq \D$.  Let $\wh{h}$ be a whole-plane GFF with its additive constant taken so that its average on $\partial \D$ is uniform in $[0, 2\pi \chi]$ so that $\wh{h}$ has the law of a representative of a whole-plane GFF with values modulo $2\pi \chi$.  Let $\phi$ be a $C_0^\infty$ function which is $1$ on $B(z,r)$ and $0$ outside of $B(z,2r)$.  Let also $\Fh$ (resp.\ $\wh{\Fh}$) be the function which is harmonic in $B(z,2r)$ with values given by those of $h$ (resp.\ $\wh{h}$) on $\partial B(z,2r)$.  Finally, let $g = \phi(\wh{\Fh}-\Fh)$ so that $h + g = \wh{h}$ in $B(z,r)$.  We assume that~$h$, $\wh{h}$ are coupled together onto a common probability space so that the projections of $h$, $\wh{h}$ onto the space of functions which are supported in $B(z,r)$ are the same and the projections of $h$, $\wh{h}$ onto the orthogonal complement are independent.  Then the Radon-Nikodym derivative between the law of $\wh{h}|_{B(z,r)}$ and $h|_{B(z,r)}$ is given by $\CZ = \E[ \exp( (h,g)_\nabla - \| g \|_\nabla^2/2) \giv h|_{B(z,r)}]$.  Lemma~\ref{lem:gff_rn_bound} and Jensen's inequality then imply that $\CZ$ has finite moments of all orders (see, e.g., the proof of \cite[Lemma~4.2]{mq2020geodesics}).  Combining this with Lemma~\ref{lem:space_filling_fills_in_balls_in_plane} implies that the following is true.  Fix $\xi > 1$.  For $\epsilon \in (0,r)$ we let $E_{z,r,\epsilon}$ be the event that for every $0 \leq a < b < \infty$ such that $\eta'([a,b]) \subseteq B(z,r)$ with $\diam (\eta'([a,b])) \geq \epsilon$ we have that $\eta'([a,b])$ contains a ball of diameter at least $\epsilon^\xi$.  Then $\p[ E_{z,r,\epsilon}] \to 1$ as $\epsilon \to 0$ faster than any power of $\epsilon$.  By combining this with the Borel-Cantelli lemma it thus follows that the following is true.  There a.s.\ exists $\epsilon_0 > 0$ so that for all $\epsilon \in (0,\epsilon_0)$ and $0 \leq a < b$ such that $\dist( \eta'([a,b]), \partial \D) \geq \diam(\eta'([a,b])) \geq \epsilon$ we have that $\eta'([a,b])$ contains a ball of diameter $\epsilon^\xi$.

In order to finish proving the result, we need to consider the case that $\dist(\eta'([a,b]), \partial \D) \leq \diam(\eta'([a,b]))$.  For each $k \in \N$ and $1 \leq j \leq N_k := 2^k$ we let $z_{j,k} = \exp(2\pi i j / N_k)$ so that $z_{1,k},\ldots,z_{N_k,k}$ are equally spaced points on $\partial \D$.  We also let $\eta_{j,k}$ be the flow line of $h$ starting from $z_{j,k}$ with angle $\pi/2$ stopped at the first time it exits $B_{j,k} = B(z_{j,k}, 2^{-\xi k})$.  Then $\eta_{j,k}$ gives part of the left boundary of $\eta'$ stopped upon hitting $z_{j,k}$.  We let $F_{j,k}$ be the event that $\eta_{j,k}$ hits $\partial B_{j,k}$ at distance at least $2^{-\xi k}/100$ from $\partial \D$.  It follows from \cite[Lemma~2.3]{mw2017intersections} that there exists $p > 0$ (which does not depend on $k$) so that $\p[F_{j,k}] \geq p$ for each $j$.  For each $j$, let $\CG_{j,k}$ be the $\sigma$-algebra generated by $\eta_i$ for $i \neq j$.  We claim further that there exists $p > 0$ (which does not depend on $k$) so that $\p[F_{j,k} \giv \CG_{j,k}] \geq p$.  To see that this is the case, we let $\varphi_j$ be the unique conformal transformation from the component of $\D \setminus \cup_{i \neq j} \eta_i$ which contains $0$ to $\D$ which fixes $0$ and $z_j$.  Consider the GFF $h \circ \varphi_j^{-1} - \chi \arg (\varphi_j^{-1})'$ on $\D$.  Its boundary data in a neighborhood of $z_j = \varphi_j(z_j)$ is the same as that of $h$.  Moreover, distortion estimates for conformal maps imply that $\varphi_j$ looks like the identity near $z_j$.  It thus follows from \cite[Lemma~2.3]{mw2017intersections} that there exists $p > 0$ so that $\varphi_j(\eta_j)$ has chance at least $p$ of exiting $B(\varphi_j(z_j),2^{-\xi k})$ at distance at least $2^{-\xi k}/10$ from $\partial \D$.  Altogether, we see that the collection of events $F_{j,k}$ is stochastically dominated from below by a collection of $N_k$ i.i.d.\ Bernoulli random variables with success probability $p$.  In particular, for each $\alpha > 0$ there exists $C > 0$ so that the probability that there exists a run of $C k$ of the $F_{j,k}$ which do not occur is at most $2^{-\alpha k}$.  Combining, this implies that the following is true.  If we have any segment of $\eta'$ which travels distance $C k 2^{-k}$ along $\partial \D$, it must with overwhelming probability get to distance at least $2^{-\xi k}/100$ from $\partial \D$.  So the claim proved in the first paragraph implies that it also fills a ball of size at least $c_0 2^{- \xi^2 k}$ for a constant $c_0 > 0$.  This proves the result as $\xi > 1$ was arbitrary.
\end{proof}

\begin{lemma}
\label{lem:free_boundary_half_plane_mass_bound}
For each $\alpha \in \R$ there exists a constant $\alpha_\LBD > 0$ so that the following is true.  Suppose that $h$ is equal to the sum of a free boundary GFF on $\h$ and $-\alpha\log|\cdot|$ with the additive constant fixed so that its average on $\h \cap \partial \D$ is equal to $0$.  There a.s.\ exists $\epsilon_0 > 0$ so that for all $\epsilon \in (0,\epsilon_0)$ and $z \in \h \cap \D$ with $B(z,\epsilon) \subseteq \h \cap \D$ we have that $\qmeasure{h}(B(z,\epsilon)) \geq \epsilon^{\alpha_\LBD}$.
\end{lemma}
\begin{proof}
We will give the proof in the case that $\alpha = 0$.  The result for general values of $\alpha$ follows because adding $-\alpha\log|\cdot|$ to the field decreases the mass in a ball of radius $\epsilon > 0$ in $\h \cap \D$ by at most a constant times the factor $\epsilon^{\alpha \gamma}$.

We note that we can write $h$ as the even part of a whole-plane GFF $\wt{h}$ on $\C$ with the additive constant fixed so that its average on $\partial \D$ is equal to $0$ (see \cite[Section~3.2]{s2016zipper}).  Fix $\delta > 0$.  Since there a.s.\ exists $\epsilon_0  > 0$ so that $|\wt{h}_\epsilon(z)|\leq (2+\delta) \log \epsilon^{-1}$ for all $\epsilon_0 \in (0,\epsilon_0)$ and $z \in \D$ \cite{hms2010thick}, it follows that $|h_\epsilon(z)| \leq (2+\delta) \sqrt{2} \log \epsilon^{-1}$ for all $\epsilon \in (0,\epsilon_0)$ and $z \in \D \cap \h$.  The result thus follows by combining this with \cite[Proposition~4.6]{ds2011kpz} and the Borel-Cantelli lemma.
\end{proof}

\begin{lemma}
\label{lem:quantum_disk_multifractal}
There exists $\alpha_\LBD > 0$ so that the following is true.  Suppose that $(\D,h,-i,i) \sim \qdiskL{\gamma}{1}$.  There a.s.\ exists $\epsilon_0 > 0$ so that for all $\epsilon \in (0,\epsilon_0)$ and $z \in \D$ we have that $\qmeasure{h}(B(z,\epsilon)) \geq \epsilon^{\alpha_\LBD}$.
\end{lemma}
\begin{proof}
We suppose that $(\strip,h,-\infty,+\infty)$ is sampled from the infinite measure on quantum disks conditioned on having its projection onto $\CH_1(\strip)$ being at least $0$.  Let $X_u$ be the average of $h$ on $u+(0,i\pi)$.  We take the embedding into $\strip$ so that $\inf\{u \in \R : X_u = 0\} = 0$.  Then we know that $X_u = B_{2u} + (\gamma-Q)u$ for $u \geq 0$ where $B$ is a standard Brownian motion with $B_0 = 0$.  Let $\varphi \colon \strip \to \h$ be the map $z \mapsto \exp(-z)$.  Then the restriction to $\h \cap \D$ of $\wt{h} = h \circ \varphi^{-1} + Q \log|(\varphi^{-1})'|$ has the law of the corresponding restriction of the sum of a free boundary GFF on $\h$ and $-\gamma\log|\cdot|$ with the additive constant fixed so its average on $\h \cap \partial \D$ is equal to $0$.  Consequently, Lemma~\ref{lem:free_boundary_half_plane_mass_bound} implies that there exists $\alpha_\LBD > 0$ and a.s.\ exists $\epsilon_0 > 0$ so that for all $\epsilon \in (0,\epsilon_0)$ we have that $\qmeasure{\wt{h}}(B(z,\epsilon)) \geq \epsilon^{\alpha_\LBD}$ for all $z \in \h \cap \D$ with $B(z,\epsilon) \subseteq \h \cap \D$.  It therefore follows that the same is true with $\qmeasure{h}$ in place of $\qmeasure{\wt{h}}$ and $z \in [0,\log \tfrac{1}{\epsilon}] \times (0,\pi)$.

Let $\tau = \sup\{u \in \R : X_u = 0\}$.  Then as $t \mapsto X_{-t+\tau}$ has the same law as $X$, it follows that the above statement holds for all $z \in [\tau-\log \tfrac{1}{\epsilon}, \tau] \times (0,\pi)$.

Altogether, we have proved that the following is true.  There exists a constant $\alpha_\LBD > 0$ so that for a.e.\ instance of $(\strip,h,-\infty,+\infty)$ there exists $\epsilon_0 > 0$ so that for all $\epsilon \in (0,\epsilon_0)$ and $z \in [-\log \epsilon^{-1}, \log \epsilon^{-1}] \times (0,\pi)$ with $B(z,\epsilon) \subseteq \strip$ we have that $\qmeasure{h}(B(z,\epsilon)) \geq \epsilon^{\alpha_\LBD}$.  In particular, the result holds for $\qdiskL{\gamma}{\ell}$ for a.e.\ value of $\ell \in [1,2]$.  Since a sample from $\qdiskL{\gamma}{1}$ can be produced by first sampling from $\qdiskL{\gamma}{\ell}$ and then adding $-\tfrac{2}{\gamma} \log \ell$ to the field, it follows that the result also holds for a sample from $\qdiskL{\gamma}{1}$ which is parameterized by $\strip$.  This proves the result because then we can conformally map $\D$ to $\strip$ with $\pm i$ sent to $\pm \infty$ and note that a ball of radius $\epsilon$ centered at $\pm i$ will contain $[\log \epsilon^{-1},\infty) \times (0,\pi)$ and $(-\infty, \log \epsilon] \times (0,\pi)$, respectively.
\end{proof}

\begin{proof}[Proof of Proposition~\ref{prop:space_filling_on_quantum_disk}]
Suppose that $(\D,h,-i,i) \sim \qdiskL{\gamma}{1}$ and $\eta'$ is a space-filling $\SLE_{\kappa'}$ from $-i$ to $-i$ which is sampled independently of $h$ and then parameterized by quantum area.  Let $\alpha_\LBD > 0$ be as in the statement of Lemma~\ref{lem:quantum_disk_multifractal}.  Then Lemma~\ref{lem:quantum_disk_multifractal} implies that there a.s.\ exists $\epsilon_0 > 0$ so that for all $\epsilon \in (0,\epsilon_0)$ and $z \in \D$ with $B(z,\epsilon) \subseteq \D$ we have that $\qmeasure{h}(B(z,\epsilon)) \geq \epsilon^{\alpha_\LBD}$.  Fix $\xi > 1$.  Proposition~\ref{prop:space_filling_fills_in_balls_on_disk} implies that there a.s.\ exists $\epsilon_0 > 0$ so that for every $\epsilon \in (0,\epsilon_0)$ and $0 \leq a < b$ with $\diam (\eta'([a,b])) \geq \epsilon$ we have that $\eta'([a,b])$ fills in a ball of diameter at least $\epsilon^\xi$.  This implies that $\qmeasure{h}(\eta'([a,b])) \geq \epsilon^{\alpha_\LBD \xi}$, which gives that $b-a \geq \epsilon^{\alpha_\LBD \xi}$ as $\eta'$ is parameterized by quantum area.  This proves that $\eta'$ is $(\alpha_\LBD \xi)^{-1}$-H\"older continuous, which gives the result.
\end{proof}

We are now going to use Proposition~\ref{prop:space_filling_fills_in_balls_on_disk} to establish a non-self-tracing property for $\CLE_\kappa$ loops.  We state the result when the domain is taken to be $\strip$ rather than $\D$ due to how we will use it in Appendix~\ref{app:carpet_measure}.

\begin{lemma}
\label{lem:loop_balls_nearby}
Fix $\xi > 1$ and suppose that $\Gamma$ is a $\CLE_\kappa$ on $\strip$.  There a.s.\ exists $\epsilon_0 \in (0,1)$ so that the following is true for every $\epsilon \in (0,\epsilon_0)$.  
\begin{enumerate}[(i)]
\item For every loop $\CL \in \Gamma$ contained in $[\log \epsilon,\log \tfrac{1}{\epsilon}] \times (0,\pi)$, and $x \in \CL$ there exists $z \in \strip$ not surrounded by $\CL$ such that $B(z,\epsilon^\xi) \cap \CL = \emptyset$ and $|x-z| \leq \epsilon$.
\item Suppose that $k \in \N$ and $U$ is a complementary component after performing the iterated $\cwBCLE_{\kappa'}(0)$/$\ccwBCLE_\kappa(-\kappa/2)$ construction of $\Gamma$ $k$ times which is not surrounded by a loop of $\Gamma$.  If $\diam(U) \geq \epsilon$ and $U \subseteq [\log \epsilon, \log \tfrac{1}{\epsilon}]$, then there exists $z \in U$ such that $B(z,\epsilon^\xi) \subseteq U$. 
\end{enumerate}
\end{lemma}

In order to prove Lemma~\ref{lem:loop_balls_nearby}, let us recall how the iterated $\BCLE$ construction of~$\Gamma$ works in the context of imaginary geometry.  Suppose that~$h$ is a GFF on $\h$ with boundary conditions given by~$\lambda'$ on~$\R_-$ and $\lambda' - 2\pi \chi$ on~$\R_+$.  For each $x \in \R$, we let~$\eta_x'$ be the counterflow line of~$h$ from~$0$ to~$x$.  Then the collection of counterflow lines~$\eta_x'$ together form the tree which generates a $\cwBCLE_{\kappa'}(0)$, say~$\Gamma_0$.  Let~$U$ be a component of $\h \setminus \Gamma_0$.  Then~$U$ is either surrounded by a loop (clockwise) or by a false loop (counterclockwise) of~$\Gamma_0$.

Suppose that $U$ is surrounded clockwise.  Then there exists $x \in \R$ so that $U$ is a component of $\h \setminus \eta_x'$.  Let $z$ be the first (equivalently, last) point on $\partial U$ which is visited by $\eta_x'$.  Let $\varphi \colon U \to \h$ be a conformal transformation which takes $z$ to $0$ and consider the field $\wt{h} = h \circ \varphi^{-1} - \chi \arg (\varphi^{-1})'$.  Then $\wt{h}$ is a GFF on $\h$ with boundary conditions given by $-\lambda'$ on $\R_-$ and $-\lambda' - 2\pi \chi$ on $\R_+$.  Let
\[ c_1 =\lambda(1-\kappa/2) + \lambda' + 2\pi \chi = \frac{5\pi}{2} \chi.\]
Then $\wt{h}+c_1$ has boundary conditions given by $\lambda(1-\kappa/2) + 2\pi \chi$ on $\R_-$ and $\lambda(1-\kappa/2)$ on $\R_+$.  For each $x \in \R$, we let $\eta_x$ be the flow line of $\wt{h}+c_1$ starting from $0$ and targeted at $x$.  Then the collection of flow lines $\eta_x$ together form the tree which generates a $\ccwBCLE_\kappa(-\kappa/2)$, say~$\Gamma_1$.  Let~$V$ be a component of $\h \setminus \Gamma_1$.  Then~$V$ is surrounded by a loop (counterclockwise) or a false loop (clockwise) of~$\Gamma_1$.

If $V$ is surrounded counterclockwise, then $V$ corresponds to a loop of the $\CLE_\kappa$ associated with the iterated $\BCLE$ exploration so the process terminates.  Suppose that $V$ is surrounded clockwise.  Then there exists $w \in \partial V$ with the property that if $x \in \R$ is such that $\eta_x$ traces part of $\partial V$ then it hits $w$ first.  Let $\psi \colon V \to \h$ be a conformal transformation which takes $w$ to $0$.  Then $\wh{h} = \wt{h} \circ \psi^{-1} - \chi \arg (\psi^{-1})' + c_1$ is a GFF on $\h$ with boundary conditions given by $\lambda$ on $\R_-$ and $\lambda-2\pi \chi$ on $\R_+$.  Let
\[ c_2 = -\frac{\pi}{2} \chi.\]
Then $\wh{h}+c_2$ has boundary conditions $\lambda'$ on $\R_-$ and $\lambda'- 2\pi \chi$ on $\R_+$.  These are the same as the boundary conditions that we started with and so we can start the exploration afresh.

Suppose that $U$ is surrounded counterclockwise and let $\varphi$, $\wt{h}$ be as defined above.  Then $\wt{h}$ is a GFF on $\h$ with boundary conditions given by $\lambda'+2\pi \chi$ on $\R_-$ and $\lambda'$ on $\R_+$.  Let
\[ c_3 = -2\pi \chi.\]
Then $\wt{h}+c_3$ has boundary conditions $\lambda'$ on $\R_-$ and $\lambda' - 2\pi \chi$ on $\R_+$.  These are the same boundary conditions as in the beginning of the exploration, so we start the exploration afresh.

Altogether, we see that in order to explore to a given loop $\CL$ of $\Gamma$, we have a certain number of counterclockwise steps by counterflow lines (each leading to an angle change of $-2\pi$) and then every clockwise step by a counterflow line (each leading to an angle change of $5\pi/2$) is followed by a successful counterclockwise step by a flow line (exploration terminates) or a clockwise step by a flow line (leading to an angle change of $-\pi/2$).  Modulo $2\pi$, we thus have that there is just one angle for the flow lines which make up the associated $\CLE_\kappa$ loops, namely $\pi/2$.  Similarly, the counterflow lines which are involved all have the same ``angle'' modulo $2\pi$.  This means that they can all be viewed as arising from the same space-filling $\SLE_{\kappa'}$.

\begin{proof}[Proof of Lemma~\ref{lem:loop_balls_nearby}]
We will first prove the first assertion of the lemma.  Fix any compact set $K \subseteq \closure{\h}$ and $\xi > 1$.  We assume that we have a $\CLE_\kappa$ process $\Gamma$ which has been generated from an iterated $\BCLE$ procedure using the GFF $h$ on $\h$ as described just above.  Let $\eta'$ be the space-filling $\SLE_{\kappa'}$ process associated with $h$.  Recall that its left boundary stopped upon hitting any point $z \in \h$ is the flow line starting from $z$ with angle $\pi/2$.  By \cite[Proposition~3.4]{ghm2020kpz} we have that there a.s.\ exists $\epsilon_0 > 0$ so that for all $\epsilon \in (0,\epsilon_0)$ and $x \in K$ the following is true.  Let $\tau_x = \inf\{t \geq 0 : \eta'(t) = x\}$ and let $\sigma_x = \inf\{t \geq \tau_x : |\eta'(t) - x| \geq \epsilon\}$.  Then $\eta'|_{[\tau_x,\sigma_x]}$ fills all of the points of a ball of radius $\epsilon^\xi$.  Now suppose that $\CL$ is a loop of $\Gamma$ which is contained in $K$.  As explained above, $\CL$ corresponds to a flow line of $h$ with angle $\pi /2$.  In particular, $\eta'$ visits the points of~$\CL$ in reverse chronological order (relative to the order in which the points are drawn by the corresponding flow line) and visits all of the points of~$\CL$ before entering its interior.  Suppose that $x \in \CL$.  As explained above, we have that $\eta'|_{[\tau_x,\sigma_x]}$ fills all of the points of a ball of radius $\epsilon^\xi$ provided $\epsilon \in (0,\epsilon_0)$.  During this interval of time, $\eta'$ cannot enter into the region surrounded by~$\CL$, so this ball is disjoint from and is not surrounded by~$\CL$.  The result follows by applying a conformal transformation $\h \to \strip$.

The second assertion of the lemma follows from the same argument as the first assertion since once the space-filling $\SLE_{\kappa'}$ enters $U$ it must fill it entirely before leaving.
\end{proof}

\section{Density of quantum carpet typical points}
\label{app:carpet_measure}

\begin{lemma}
\label{lem:quantum_measure_points_dense}
Suppose that $\CD = (\strip,h,-\infty,+\infty) \sim \qdiskL{\gamma}{1}$ and $\Gamma$ is an independent $\CLE_\kappa$ on $\CD$.  There exists $\alpha_\net > 0$ so that there a.s.\ exists $\epsilon_0 > 0$ so that for all $\epsilon \in (0,\epsilon_0)$ the following is true.  Let $(x_j)$ be an i.i.d.\ sequence picked from $\qcarpet{h}{\Upsilon}$ and let $N_\epsilon = \epsilon^{-\alpha_\net}$.  For all $\epsilon \in (0,\epsilon_0)$ we have that $\Upsilon \cap ( [\log \epsilon, \log \tfrac{1}{\epsilon}] \times (0,\pi)) \subseteq \cup_{j=1}^{N_\epsilon} B(x_j,\epsilon)$.
\end{lemma}

Since the proof of Lemma~\ref{lem:quantum_measure_points_dense} will require several steps, let us first outline the argument before proceeding with the details.

\begin{enumerate}
\item[Step 1.] Use that each loop has a ball nearby which is not surrounded by a loop (Lemma~\ref{lem:loop_balls_nearby}).  This is a version of (and is deduced from) the non-self-tracing property of space-filling $\SLE_{\kappa'}$.
\item[Step 2.] We will consider a sample from the law $\qdiskL{\gamma}{1}$ parameterized by $\strip$ and then prove an upper bound on the modulus and Dirichlet energy of the harmonic extension of the field values to dyadic squares (Lemma~\ref{lem:dirichlet_energy}).  This will be used later on to control various Radon-Nikodym derivatives.
\item[Step 3.] We will establish an upper bound on the quantum measure of small Euclidean balls for a sample from the law $\qdiskL{\gamma}{1}$ when parameterized by $\strip$ (Lemma~\ref{lem:disk_mass_upper_bound}).
\item[Step 4.] We will establish a lower bound on $\qcarpet{h}{\Upsilon}(B(x,\epsilon))$ where $x \in \CL$ and $\CL \in \Gamma$ is a loop with quantum length at least $\epsilon$ (Lemma~\ref{lem:mass_near_by}).
\item[Step 5.] We will prove in Lemma~\ref{lem:quantum_disk_cover} that if we have a sample from the law $\qdiskL{\gamma}{1}$ parameterized by $\D$ then we need at most a polynomial number of independent samples from $\qmeasure{h}$ in order to obtain an $\epsilon$-net of $\D$.
\item[Step 6.] We combine the above estimates in order to complete the proof of Lemma~\ref{lem:quantum_measure_points_dense}.  We first let $N_\epsilon'$ be a large negative power of $\epsilon$ and then choose i.i.d.\ points $z_1,\ldots,z_{N_\epsilon'}$ from $\qmeasure{h}$.  Step 5 implies that this collection of points is likely to be an $\epsilon$-net of $\D$.  We aim to show that if $B(z_j,\epsilon) \cap \Upsilon \neq \emptyset$ then $\qcarpet{h}{\Upsilon}(B(z_j,2\epsilon))$ is with overwhelming probability at least a power of~$\epsilon$.  Since $\E[ \qcarpet{h}{\Upsilon}(\Upsilon)] < \infty$, we can also obtain an upper bound on $\qcarpet{h}{\Upsilon}(\Upsilon)$, so assuming the previous claim we have that if we pick $x$ according to $\qcarpet{h}{\Upsilon}$ then it has probability at least a power of $\epsilon$ of being in $B(z_j,2\epsilon)$ in which case $B(z_j,\epsilon) \subseteq B(x,4\epsilon)$.  As the collection $z_1,\ldots,z_{N_\epsilon'}$ is likely to form an $\epsilon$-net of $\D$, it follows that if we choose $N_\epsilon$ to be a sufficiently negative power of $\epsilon$ and then choose $x_1,\ldots, x_{N_\epsilon}$ i.i.d.\ from $\qcarpet{h}{\Upsilon}$ then the collection $x_1,\ldots, x_{N_\epsilon}$ is likely to be a $4\epsilon$-net of $\Upsilon$.  By Step 4, if  $B(x,\epsilon) \cap \CL \neq \emptyset$ where $\CL \in \Gamma$ has quantum length at least $\epsilon$ then $\qcarpet{h}{\Upsilon}(B(x,2\epsilon))$ is at least a power of $\epsilon$.  Since $\qmeasure{h}(\Upsilon) = 0$ a.s., we have that $x$ is a.s.\ in a component of $\D \setminus \Upsilon$ which is surrounded by a loop of $\Gamma$.  The remainder of the proof is focused on bounding the total quantum area surrounded by loops of~$\Gamma$ with small quantum length to conclude that the loop which surrounds~$x$ is likely to be at least a power of~$\epsilon$.
\end{enumerate}

At this point in the article, it will be important for us to recall the precise definition of $\qdiskL{\gamma}{1}$ given in \cite{dms2014mating}.  We recall that the Dirichlet inner product of functions $f,g \in C_0^\infty$ is defined by
\[ (f,g)_\nabla = \frac{1}{2\pi} \int \nabla f(z) \cdot \nabla g(z) dz\]
where $dz$ denotes Lebesgue measure.  We also let $\| \cdot \|_\nabla$ denote the associated norm.  Then the Dirichlet inner product is defined more generally for $f,g$ with $\| f \|_\nabla < \infty$, $\|g \|_\nabla < \infty$.  It is easiest to give the definition of $\qdiskL{\gamma}{1}$ parameterized by $\strip$.  We let $\CH(\strip)$ be the closure with respect to $(\cdot,\cdot)_\nabla$ of those functions which are $C^\infty$ in $\strip$ and have zero mean.  We recall that $\CH(\strip)$ admits the orthogonal decomposition $\CH_1(\strip) \oplus \CH_2(\strip)$ where $\CH_1(\strip)$ consists of those functions which have mean zero on vertical lines of the form $u + [0,i \pi]$ for $u \in \R$ and $\CH_2(\strip)$ consists of those functions which are constant on such vertical lines.  We first consider the infinite measure $\qdisk{\gamma}$ on distributions $h$ on $\strip$ which can be ``sampled'' from as follows (see \cite[Section~4.5]{dms2014mating} for more details).
\begin{itemize}
\item ``Sample'' an excursion $Z$ of a Bessel process of dimension $3-\tfrac{4}{\gamma^2}$.  Take the projection of $h$ onto $\CH_1(\strip)$ to be given by $\tfrac{2}{\gamma} \log Z$ reparameterized to have quadratic variation $2dt$.
\item Take its projection onto $\CH_2(\strip)$ to be the corresponding projection of an independent GFF on $\strip$ with free boundary conditions.
\end{itemize}
As in \cite[Section~4.5]{dms2014mating}, for each $\ell > 0$ we then define $\qdiskL{\gamma}{\ell}$ to be the law on quantum surfaces given by taking the above distribution and conditioning the boundary length of $\partial \strip$ to be exactly $\ell$.  We note that there is one free parameter to fix the embedding of a sample from $\qdiskL{\gamma}{\ell}$ into $\strip$ as defined above, namely the horizontal translation.  There are a variety of different choices which can be convenient based on the particular situation.  The points at $\pm \infty$ are special in the sense that they are ``quantum typical'' boundary points.  More precisely, this means that the following is true.  Suppose that $(\strip,h,-\infty,+\infty) \sim \qdiskL{\gamma}{1}$ and that $x,y \in \partial \strip$ are picked independently from $\qbmeasure{h}$ and $\varphi \colon \strip \to \strip$ is a conformal transformation which takes $-\infty$ to $x$ and $+\infty$ to $y$.  Then, modulo horizontal translation, $h$ and $h \circ \varphi + Q\log|\varphi'|$ have the same law \cite[Proposition~A.8]{dms2014mating}.

\newcommand{\harm}{\mathrm{HARM}}

\begin{lemma}
\label{lem:dirichlet_energy}
Suppose that $\CD = (\strip,h,-\infty,+\infty) \sim \qdiskL{\gamma}{1}$.  We choose the horizontal translation so that $(-\infty,0]$ and $[0,\infty)$ have the same quantum length.  Let $\CQ_n$ be the set of dyadic squares $S$ with side length $2^{-n}$ such that the square $\wt{S}$ with twice the side length and the same center as $S$ is contained in $[-n,n] \times (0,\pi)$.  For each $S \in \CQ_n$, we let $\Fh_S$ be the function which is harmonic in $\wt{S}$ with boundary values given by those of $h$.  There exists $\alpha_\harm > 0$ so that there a.s.\ exists $n_0 \in \N$ so that the following is true.  For every $n \geq n_0$ and $S \in \CQ_n$ we have that
\begin{align}
\label{eqn:harmonic_disk_bound}
\begin{split}
\sup_{z \in S} |\Fh_S(z)| \leq \alpha_\harm n, &\quad \sup_{z,w \in S} |\Fh_S(z) - \Fh_S(w)| \leq \alpha_\harm n^{1/2}, \quad\text{and}\\
 &\int_S | \nabla \Fh_S(z)|^2 dz \leq \alpha_\harm n.
\end{split}
\end{align}
\end{lemma}
\begin{proof}
We first note that since $\Fh_S$ is harmonic there exists a constant $c > 0$ so that if $\sup_{z,w \in S} |\Fh_S(z) - \Fh_S(w)| \leq \alpha_\harm n^{1/2}$ then we have that $\sup_{z \in S} |\nabla \Fh_S(z)| \leq c \alpha_\harm n^{1/2} / 2^{-n}$.  Thus it suffices to establish the first two inequalities in~\eqref{eqn:harmonic_disk_bound} (and then possibly increase the value of $\alpha_\harm$).

We will prove the result in a slightly different setting and then deduce it as a consequence for $\qdiskL{\gamma}{1}$.  In particular, we suppose that $\CD = (\strip,h,-\infty,+\infty)$ is sampled from $\qdisk{\gamma}$ conditioned so that the supremum of its projection onto $\CH_1(\strip)$ is equal to $0$.  We recall that its projection onto $\CH_2(\strip)$ is equal to that of the corresponding projection of an independent GFF on $\strip$ with free boundary conditions.

Suppose that $\wt{h}$ is a GFF on $\strip$ with free boundary conditions with its additive constant fixed so that its average on $(0, i \pi)$ is equal to $0$.  Let $S \in \CQ_n$ be such that~$\wt{S}$ is adjacent to the $y$-axis.  Let~$\wh{S}$ be the square with the same center as~$S$ but with $3/2$ times the side length.  Thus, $S \subseteq \wh{S} \subseteq \wt{S}$.  Let $p$ be the Poisson kernel in $\wh{S}$ and let $z_0$ be the center of $S$.  Then there exists a constant $c_0 > 0$ (which does not depend on $n$) so that for all $z \in S$ and $w \in \partial \wh{S}$ we have that $p(z,w) \leq c_0 p(z_0,w)$.  Let $\wt{\Fh}_S$ be the function which is harmonic in $\wt{S}$ with the same boundary conditions as $\wt{h}$.  We thus have for all $z \in S$ that
\begin{align}
\label{eqn:harmonic_bound1}
|\wt{\Fh}_S(z)|
&= \left| \int_{\partial \wh{S}} p(z,w) \wt{\Fh}_S(w) dw \right|
 \leq \int_{\partial \wh{S}} p(z,w) |\wt{\Fh}_S(w)| dw
 \leq c_0 \int_{\partial \wh{S}} p(z_0,w) |\wt{\Fh}_S(w)| dw
\end{align}
where $dw$ denotes Lebesgue measure on $\partial \wh{S}$.  By Jensen's inequality, we have that
\begin{align}
\label{eqn:harmonic_bound2}
 \E\!\left[ \exp\left(c_0 \int p(z_0,w) |\wt{\Fh}_S(w)| dw \right) \right]
&\leq \int p(z_0,w) \E[ \exp(c_0 |\wt{\Fh}_S(w)| )] dw.
\end{align}
It follows from the explicit form of the Green's function with Neumann boundary conditions on $\strip$ that there exists a constant $c_1 > 0$ so that for $w \in \partial \wh{S}$ we have that $\wt{\Fh}_S(w)$ is a Gaussian random variable of mean zero and variance at most $c_1 n$.  Therefore there exists a constant $c_2 > 0$ so that $\E[ \exp(c_0 \wt{\Fh}_S(w) )] \leq e^{c_2 n}$ and $\E[ \exp(-c_0 \wt{\Fh}_S(w) )] \leq e^{c_2 n}$ for all $w \in \partial \wh{\CS}$.  This implies that there exists a constant $c_3 > 0$ so that
\begin{equation}
\label{eqn:harmonic_bound3}
\E[ \exp(c_0 |\wt{\Fh}_S(w)| )] \leq \exp( c_3 n).
\end{equation}
Altogether, this implies that
\[ \E\!\left[\exp\left( \sup_{z \in S} |\wt{\Fh}_S(z)|\right)\right] \leq \exp(c_3 n).\]
Therefore there exists a constant $c_4 > 0$ so that
\[ \p\!\left[ \sup_{z \in S} |\wt{\Fh}_S(z)| \geq c_4 n \right] \leq 2^{-3n}.\]

We now suppose more generally that we have $S \in \CQ_n$ (with $\wt{S}$ not necessarily adjacent to the $y$-axis) and let $\wt{X}$ be the average of $\wt{h}$ on the vertical line $L$ which is adjacent to the right side of $\wt{S}$.  As $\wt{h} - \wt{X}$ has the law of a free boundary GFF on $\strip$ with the additive constant normalized so that its average on $L$ is equal to $0$, what we have argued above implies that
\[ \p[ \sup_{z \in S} |\wt{\Fh}_S(z) - \wt{X}| \geq c_4 n] \leq 2^{-3n}.\]
We note that $\wt{X}$ is a Gaussian random variable with mean zero and variance at most $2n$.  Thus by possibly increasing the value of $c_4$, the above in turn implies
\[ \p[ \sup_{z \in S} |\wt{\Fh}_S(z)| \geq c_4 n] \leq 2^{-3n}.\]
We assume that $h$ and $\wt{h}$ are coupled together so that their projection onto $\CH_2(\strip)$ is the same.  Let~$\wt{X}_*$ be the infimum of the average of $\wt{h}$ on vertical lines with real part in $[-n,n]$.   As the projection of $h$ onto $\CH_1(\strip)$ is in $\R_-$ and the projection of $\wt{h}$ onto $\CH_1(\strip)$ restricted to $[-n,n] \times (0,\pi)$ is at least $\wt{X}_*$ we have for every $S \in \CQ_n$ and $z \in S$ that $\Fh_S(z) \leq \wt{\Fh}_S(z) - \wt{X}_*$.  Since there exists a constant $c_5 > 0$ so that $\p[ |\wt{X}_*| \geq c_5 n] \leq 2^{-3n}$, it follows from the above and a union bound that for a constant $c_6 > 0$ we have that 
\[ \p[ \max_{S \in \CQ_n} (\sup_{z \in S} \Fh_S(z)) \geq c_6 n] \leq 2^{-n}.\]
By possibly increasing the value of $c_6$, we similarly have that
\[ \p[ \min_{S \in \CQ_n} (\inf_{z \in S} \Fh_S(z)) \leq -c_6 n] \leq 2^{-n}.\]
Altogether, this proves the first inequality in~\eqref{eqn:harmonic_disk_bound} for the field $(\CS,h,-\infty,+\infty)$ by applying the Borel-Cantelli lemma.  The same argument with $\Fh_S(z) - \Fh_S(z_S)$, $z_S$ the center of $S$, in place of $\Fh_S$ gives the second inequality in~\eqref{eqn:harmonic_disk_bound}.  The reason that we obtain an $O(n^{1/2})$ rather than $O(n)$ upper bound is that $\wt{\Fh}_S(z) - \wt{\Fh}_S(z_S)$ has bounded variance.  The same argument more generally implies that the same is true if $(\CS,h,-\infty,+\infty)$ is sampled from $\qdisk{\gamma}$ conditioned on the supremum of the projection onto $\CH_1(\strip)$ being equal to any fixed constant.

Let us now explain how to deduce the result where $(\CS,h,-\infty,+\infty) \sim \qdiskL{\gamma}{1}$ with the horizontal translation taken as in the statement of the lemma.  Note that in this case $h$ is dominated from above by the corresponding field when we sample it from $\qdisk{\gamma}$ conditioned so that the boundary length is in $[1,2]$ and from below when the boundary length is conditioned on being in $[1/2,1]$ (with the horizontal translation taken in the same manner).  Indeed, this follows because for each $\ell > 0$ we can produce a sample from the law $\qdiskL{\gamma}{\ell}$ by first sampling from the law $\qdiskL{\gamma}{1}$ and then adding $\tfrac{2}{\gamma} \log \ell$ to the field.  Therefore it suffices to prove the statement when the boundary length is conditioned on being in any fixed non-trivial interval $[a,b]$.  This, in turn, follows as we have shown above that it holds a.s.\ when the maximum of the projection of $h$ onto $\CH_1(\strip)$ is equal to any fixed value.  In particular, the result holds for a.e.\ ``sample'' from $\qdisk{\gamma}$.
\end{proof}

\begin{lemma}
\label{lem:disk_mass_upper_bound}
Suppose that $\gamma \in (\sqrt{8/3},2)$ and that $\CD = (\strip,h,-\infty,+\infty) \sim \qdiskL{\gamma}{1}$.  There exists $\alpha_\UBD > 0$ so that there a.s.\ exists $\epsilon_0 > 0$ so that the following is true.  For every $\epsilon \in (0,\epsilon_0)$ and square $S \subseteq \strip$ with side length at most $\epsilon$ we have that $\qmeasure{h}(S) \leq \epsilon^{\alpha_\UBD}$.  The same is also true if $\CD$ is parameterized by $\D$ instead of by $\strip$.
\end{lemma}

The statement of Lemma~\ref{lem:disk_mass_upper_bound} in fact holds for all $\gamma \in (0,2)$ but we have restricted to the case that $\gamma \in (\sqrt{8/3},2)$ in order to give a shorter proof.

\begin{proof}[Proof of Lemma~\ref{lem:disk_mass_upper_bound}]
Suppose that $(\h,\wt{h},0,\infty)$ is a quantum half-plane.  We note that there exists $\alpha_\UBD > 0$ so that for any fixed compact set $K \subseteq \h$ there a.s.\ exists $\epsilon_0 > 0$ so that for all $z \in K$ and $\epsilon \in (0,\epsilon_0)$ we have that $\qmeasure{h}(B(z,\epsilon)) \leq \epsilon^{\alpha_\UBD}$.  Indeed, this follows from \cite[Proposition~3.5]{rv2010revisited} and the Borel-Cantelli lemma.

Let $\Gamma$ be a $\CLE_\kappa$ in $\h$ which is independent of $\wt{h}$ and let $\eta$ be an $\SLE_\kappa^0(\kappa-6)$ process in~$\h$ from~$0$ to~$\infty$ which is coupled with~$\Gamma$ to be a CPI.  Let~$\tau$ be the first time that~$\eta$ hits a loop $\CL$ of $\Gamma$ with quantum length at least $1$ and let $\ell$ be the quantum length of $\CL$.  Let~$D_\CL$ be the domain which is surrounded by~$\CL$.  Then we have that $(D_\CL,h) \sim \qdiskL{\gamma}{\ell}$ and therefore $(D_\CL,h-\tfrac{2}{\gamma} \log \ell) \sim \qdiskL{\gamma}{1}$.  We can thus couple the quantum surface $(D_\CL,h-\tfrac{2}{\gamma}\log \ell)$ and the quantum surface described by $(\strip,h,-\infty,+\infty)$ to be the same.  Let $\varphi \colon \strip \to D_\CL$ be the corresponding embedding map.  Then we know that $\varphi$ is a.s.\ H\"older continuous \cite{rs2005basic}.
 
Fix $p \in (0,1)$ and a compact set $K$ so that $\p[ D_\CL \subseteq K] \geq 1-p$.  Then by what is explained in the first paragraph, we see that on this event there exists $\epsilon_0 \in (0,1)$ so that for all $z \in D_\CL$ we have that $\qmeasure{h}(B(z,\epsilon)) \leq \epsilon^{\alpha_\UBD}$.  Since $p \in (0,1)$ was arbitrary we see that there a.s.\ exists $\epsilon_0 \in (0,1)$ so that for all $z \in D_\CL$ with $B(z,\epsilon) \subseteq D_\CL$ we have that $\qmeasure{h}(B(z,\epsilon)) \leq \epsilon^{\alpha_\UBD}$.  The result thus follows by the a.s.\ H\"older continuity of $\varphi$.
\end{proof}

\begin{lemma}
\label{lem:number_of_loops}
Suppose that $\CD = (D,h,x,y) \sim \qdiskL{\gamma}{1}$.  Let $\Gamma$ be a $\CLE_\kappa$ on $D$ which is independent of $h$.  For each $\epsilon > 0$, let $N_\epsilon$ be the number of loops of $\Gamma$ with quantum length at least $\epsilon$.  Then for each $a > 0$ there exists $b > 0$ so that $\p[ N_\epsilon \geq \epsilon^{-4/\kappa-1/2-a}] = O(\epsilon^b)$.	
\end{lemma}
\begin{proof}
Fix $a > 0$.  Then we want to show that
\[ \limsup_{\epsilon \to 0} \frac{\log \p[ N_\epsilon \geq \epsilon^{-4/\kappa-1/2-a}]}{\log \epsilon^{-1}} < 0.\]
Suppose that this is not the case.  Then there exists a sequence $(\epsilon_j)$ of positive numbers with $\epsilon_j \to 0$ as $j \to \infty$ so that
\[ \lim_{j \to \infty} \frac{\log \p[ N_{\epsilon_j} \geq \epsilon_j^{-4/\kappa-1/2-a}]}{\log \epsilon_j^{-1}} = 0.\]
Let $\eta$ be an $\SLE_\kappa(\kappa-6)$ in $D$ starting from $x$ with the property that whenever it intersects itself it continues into the component which has the largest boundary length (and is not a $\CLE_\kappa$ loop).  Fix $b_1 > 0$ and let $M_j$ be the number of downward jumps of size at least $\epsilon_j^{b_1}$ made by the boundary length process for the target component.  By Poisson concentration, we have for any $b_2 > 0$ that $\p[M_j \leq \epsilon_j^{b_2 -4 b_1/\kappa}] \to 0$ as $j \to \infty$ faster than any power of $\epsilon_j$.  Suppose we are working on the event that $M_j \geq \epsilon_j^{b_2 - 4b_1/\kappa}$.  Let $(\CD_k)$ be the sequence of quantum surfaces described by the components of $D \setminus \eta$ which correspond to these downward jumps and let $\ell_k$ be the quantum length of $\partial \CD_k$.  Then the $(\CD_k)$ are conditionally independent given the $(\ell_k)$ and the conditional law of $\CD_k$ given $\ell_k$ is $\qdiskL{\gamma}{\ell_k}$.  Moreover, we have that the loops of~$\Gamma$ contained in~$\CD_k$ are a~$\CLE_\kappa$ in~$\CD_k$.  Consequently, it follows from \cite[Theorem~1.3]{msw2020simplecle} that there exists $p_0 > 0$ so that the probability that the number of loops of $\Gamma$ in $\CD_k$ with quantum length at least $\epsilon_j$ is at least $\epsilon_j^{b_1-4/\kappa-1/2}$ is at least $p_0$.  By the conditional independence and binomial concentration, we have on $M_j \geq \epsilon_j^{b_2 - 4b_1/\kappa}$ that the number loops in $\CD$ with quantum length at least $\epsilon_j$ is at least a constant times $M_j \epsilon_j^{b_1 - 4/\kappa-1/2}$ off an event whose probability tends to $0$ as $j \to \infty$ faster than any power of $\epsilon_j$.  Note that the overall exponent of $\epsilon_j$ in this expression is:
\[ \left(b_2 - \frac{4b_1}{\kappa} \right) + b_1 - \frac{4}{\kappa} - \frac{1}{2}.\]
In particular, the coefficient of $b_1$ is negative as $\kappa \in (8/3,4)$.  Thus by choosing $b_2 > 0$ sufficiently small, we obtain a contradiction to \cite[Theorem~1.3]{msw2020simplecle}.
\end{proof}

\begin{lemma}
\label{lem:mass_near_by}
Suppose that $\CD = (\strip,h,-\infty,+\infty) \sim \qdiskL{\gamma}{1}$ and~$\Gamma$ is an independent $\CLE_\kappa$ on~$\CD$.  There exists $\alpha_\LBD > 0$ so that there a.s.\ exists $\epsilon_0 > 0$ so that the following is true.  For every loop $\CL$ of $\Gamma$ with quantum length $\epsilon \in (0,\epsilon_0)$ and $x \in \CL$ with $x \in [\log \epsilon,\log \tfrac{1}{\epsilon}] \times (0,\pi)$ we have that $\qcarpet{h}{\Upsilon}(\Upsilon \cap B(x,\epsilon)) \geq \epsilon^{\alpha_\LBD}$.
\end{lemma}
\begin{proof}

\noindent{\it Step 1. Upper bound on number of loops.}  Fix $a > 4/\kappa-1/2$.  Let $E$ be the event that there are at least $\epsilon^{-a}$ loops of $\Gamma$ with boundary length at least $\epsilon$.  Lemma~\ref{lem:number_of_loops} implies that there exists $b > 0$ so that $\p[E] = O(\epsilon^b)$.

\newcommand{\DIAM}{\mathrm{DIAM}}

\noindent{\it Step 2.  Lower bound on loop diameter.}  Let $F$ be the event that $E$ holds or there is a loop of $\Gamma$ with quantum length at least $\epsilon$ which surrounds a region with quantum area at most $\epsilon^{2+\zeta}$.  Note that \cite[Theorem~1.2]{ag2019disk} implies that the probability that sample from $\qdiskL{\gamma}{\epsilon}$ has quantum area smaller than $\epsilon^{2+\zeta}$ decays to $0$ as $\epsilon \to 0$ faster than any power of $\epsilon$.  It therefore follows that $\p[F] = O(\epsilon^b)$.  The Borel-Cantelli lemma thus implies that there a.s.\ exists $\epsilon_0 > 0$ so that for all $\epsilon \in (0,\epsilon_0)$ if $\CL$ is a loop of $\Gamma$ with boundary length $\epsilon$ then it surrounds a domain with quantum area at least $\epsilon^{2+\zeta}$.  Lemma~\ref{lem:disk_mass_upper_bound} implies that there exists $\alpha_\UBD > 0$ so that, by possibly decreasing the value of $\epsilon_0 > 0$, if $S \subseteq \strip$ is a square with side length $\epsilon \in (0,\epsilon_0)$ then $\qmeasure{h}(S) \leq \epsilon^{\alpha_\UBD}$.  By combining it therefore follows that there exists $\alpha_\DIAM > 0$ so that (possibly decreasing the value of $\epsilon_0$) the diameter of any loop $\CL$ with quantum length at least $\epsilon \in (0,\epsilon_0)$ is at least $\epsilon^{\alpha_\DIAM}$.

\noindent{\it Step 3.  Exploration.} We now consider a \emph{center exploration} of $\CD$.  More precisely, we let $\eta$ be an $\SLE_\kappa(\kappa-6)$ process in $\CD$ coupled with $\Gamma$ to be a CPI with the property that whenever the trunk hits itself or a loop of $\Gamma$ the process always continues into the complementary component (which is not a loop of $\Gamma$) with the largest boundary length.  Fix $\epsilon > 0$ and for each $j$, we let $\tau_j$ be the $j$th time that the center exploration has an upward jump of size at least $\epsilon$ (hence discovers a loop of $\Gamma$ with quantum length at least $\epsilon$).  Let $\CL_j$ be the loop discovered by the center exploration at time $\tau_j$.  Let $x_1,\ldots,x_n$ be an $\epsilon^{\alpha_\DIAM}/2$-net of $[\log \epsilon, \log \tfrac{1}{\epsilon}] \times (0,\pi)$.  Then note that $n = O( \epsilon^{-2 \alpha_\DIAM} \log \tfrac{1}{\epsilon})$.  In particular, if $x \in \CL_j$ with $x \in [\log \epsilon,\log \tfrac{1}{\epsilon}] \times (0,\pi)$ then there exists $1 \leq i_j \leq n$ so that $|x - x_{i_j}| \leq \epsilon^{\alpha_\DIAM}/2$.  Fix $\alpha > 2 \alpha_\DIAM$, let $m = \epsilon^{\alpha_\DIAM-\alpha/2}$, and let $y_{1,j},\ldots,y_{m,j}$ be points in $B(x_{i_j},\epsilon^{\alpha_\DIAM}) \cap \CL_j$ with $|y_{i,j} - y_{k,j}| \geq \epsilon^{\alpha/2}$ for $i \neq k$.  Fix $\alpha' > \alpha$.  Lemma~\ref{lem:loop_balls_nearby} implies for each $k$ that there exists~$z_{k,j}$ and~$r_{k,j}$ with $|y_{k,j}-z_{k,j}| \leq \epsilon^\alpha$, $r_{k,j} \geq \epsilon^{\alpha'}$, $|y_{k,j}-z_{k,j}| \leq 2 r_{k,j}$ so that $B(z_{k,j},r_{k,j})$ is disjoint from $\CL_j$ and not surrounded by it.  Let~$D_j$ be the target component of the center exploration at the time~$\tau_j$.  By Lemma~\ref{lem:loop_balls_nearby}, decreasing the value of $\epsilon_0 > 0$ if necessary we may assume that each $B(z_{k,j},r_{k,j})$ is disjoint from $\eta([0,\tau_j])$ and is not surrounded by a loop of $\Gamma$ visited by $\eta|_{[0,\tau_j]}$.  By possibly increasing the value of $\alpha'$ we may also assume that $B(z_{k,j}, r_{k,j})$ is contained in a dyadic square $S_{i,j}$ so that if $\wt{S}_{i,j}$ is the square with the same center and twice the side length then $\wt{S}_{i,j}$ does not intersect $\eta([0,\tau_j])$ and the distance of $\wt{S}_{i,j}$ to $\eta([0,\tau_j])$ is proportional to its diameter.  We note by the Beurling estimate the probability that a Brownian motion starting in $B(z_{k,j},\epsilon^\alpha)$ exits $B(z_{k,j},\epsilon^{\alpha/2})$ before hitting $\eta([0,\tau_j])$ is $O(\epsilon^{\alpha/4})$.  In particular, the probability that a Brownian motion starting in $B(z_{k,j},\epsilon^\alpha)$ hits $B(z_{\ell,j},\epsilon^\alpha)$ for $\ell \neq k$ is $O(\epsilon^{\alpha/4})$.

Let $w_0,w_1$ be the two prime ends on $\partial D_j$ which correspond to the point where the loop $\CL_j$ is attached to the target component boundary of the center exploration at the time $\tau_j^-$ so that $\ccwBoundary{w_0}{w_1}{\partial D_j} = \CL_j$.  Let $w_2$ be some other point on $(\partial D_j) \setminus \CL_j$.  Let $\varphi_j \colon D_j \to \h$ be the unique conformal map which takes $w_0$ to $-1$, $w_1$ to $1$, and $w_2$ to $\infty$.  Let $h^\IG$ be a GFF on $\h$ with the boundary values given by $-\lambda' - \pi \chi$ so that the loops of $\Gamma$ in $D_j$ can be viewed as the image under $\varphi_j^{-1}$ of the $\CLE_\kappa$, say $\Gamma_j$, generated by $h_j^\IG$.  (The boundary data here is different from as described after the statement of Lemma~\ref{lem:loop_balls_nearby} because we have taken the root of the $\cwBCLE_{\kappa'}(0)$ exploration tree to be at $\infty$ rather than at $0$.)  By the conformal invariance of Brownian motion, we have for $i \neq k$ the probability that a Brownian motion starting from any point in $\varphi_j(B(z_{i,j},\epsilon^{\alpha}))$ hits $\varphi_j(B(z_{k,j},\epsilon^{\alpha}))$ before hitting $\partial \h$ is $O(\epsilon^{\alpha/4})$.

We will now restrict our attention to those $i,j$ such that $\wt{S}_{i,j}$ is contained in $D_j$.  Let $w_{i,j} = \re(\varphi_j(z_{i,j}))$, $s_{i,j} = 2|w_{i,j}-\varphi_j(z_{i,j})|$, and $U_{i,j} = B(w_{i,j},s_{i,j}) \cap \h$.  Then it follows that the probability that a Brownian motion starting in any point in $U_{i,j}$ hits $U_{k,j}$ for $k \neq i$ before exiting $\h$ is $O(\epsilon^{\alpha/4})$.  We let $\eta_{i,j}$ be the flow line of $h_j^\IG$ starting from $w_{i,j} + s_{i,j}/4$ with angle $\pi/2$ stopped upon first exiting $U_{i,j}$.  Then it follows from \cite[Lemma~2.5]{mw2017intersections} that there exists $p \in (0,1)$ so that the probability that $\eta_{i,j}$ disconnects $\varphi_j(B(z_{i,j},r_{i,j}/2))$ from $\infty$ is at least $p$.  The same remains true if we condition on $\eta_{k,j}$ for $k \neq i$.  Indeed, let $\psi_{i,j}$ be the unique conformal transformation from $\h \setminus (\cup_{k \neq i} \eta_{k,j})$ to $\h$ with $\psi_{i,j}(\varphi_j(z_{i,j})) = i$ and which fixes $\infty$.  Then the boundary data for the field $h_{i,j}^\IG = h_j^\IG \circ \psi_{i,j}^{-1} - \chi \arg (\psi_{i,j}^{-1})'$ is at most $\lambda - \pi \chi/2$ and is at least $-\lambda - \pi \chi$.  In particular, the boundary data for $h_{i,j}^\IG$ is bounded above and below by deterministic constants.  Moreover, for all $\epsilon > 0$ sufficiently small it is equal to $-\lambda - \pi \chi$ in a $B(0,100)$.  This implies that the law of $\psi_{i,j}(\eta_{i,j})$ stopped upon exiting $B(0,50)$ is absolutely continuous with respect to the law of the corresponding flow line if the boundary data were equal to $-\lambda - \pi\chi$ everywhere.

We note that $\eta_{i,j}$ is equal to the left boundary of the counterflow line of $h_j^\IG$ from $\infty$ targeted at $w_{i,j}$ (stopped upon exiting $U_i$).  Let $W_{i,j}$ be the component of $\h \setminus \eta_{i,j}$ which contains $\varphi_j(z_i)$.  On the event $A_{i,j}$ that $W_{i,j}$ is bounded, it follows that the conditional law of the loops of $\Gamma_j$ which are contained in $W_{i,j}$ is given by that of a $\CLE_\kappa$ in $W_{i,j}$.  Let $V_{i,j} = \varphi_j^{-1}(W_{i,j})$ and let $\psi_{i,j} \colon V_{i,j} \to \D$ be the unique conformal map with $\psi_{i,j}(z_{i,j}) = 0$ and $\psi_{i,j}'(z_{i,j}) > 0$.  Then the image $\Gamma_{i,j}$ of the loops of $\Gamma$ in $V_{i,j}$ under $\psi_{i,j}$ is a $\CLE_\kappa$ in $\D$. By distortion estimates for conformal maps, we have that $\psi_{i,j}(B(z_{i,j},s_{i,j}))$ contains $B(0,1/100)$.  Suppose that $h_{i,j}$ is a zero-boundary GFF on $\D$.  For each $\delta \in (0,1)$, we note that $h_{i,j}$ restricted to $B(0,\delta)$ is absolutely continuous with respect to the corresponding restriction of the field which defines a sample from $\qdiskL{\gamma}{1}$ parameterized by $\D$.  It therefore follows that the quantum measure on the carpet of $\Gamma_{i,j}$ is defined.  Moreover, for each $p_1 \in (0,1)$ there exists $c > 0$ so that the probability that the measure of the carpet in $B(0,1/100)$ is at least $c$ is at least $p_1$.

Let $h^{i,j} = h_{i,j} \circ \psi_{i,j} + Q \log |\psi_{i,j}'|$.  Then $h^{i,j}$ is the sum of the zero boundary GFF $h_{i,j} \circ \psi_{i,j}$ on $V_{i,j}$ and the harmonic function $-Q \log|\psi_{i,j}'|$.  By distortion estimates for conformal maps and using that $s_{i,j} \geq \epsilon^{\alpha'}$, we have that $|\psi_{i,j}'| = \Theta(s_{i,j}) = O(\epsilon^{-\alpha'})$.  By writing $\wt{h}_{i,j} = h_{i,j} \circ \psi_{i,j} = h^{i,j} - Q \log|\psi_{i,j}'|$, we thus see that the following is true.  There exists $p > 0$ so that the probability that the quantum measure on the carpet of $\Gamma$ in $B(z_{i,j},s_{i,j})$ associated with the zero boundary GFF $\wt{h}_{i,j}$ on $V_{i,j}$ is at least $s_{i,j}^{\alpha Q} \geq \epsilon^{\alpha \alpha' Q}$ is at least $p$.  We are now going to use an absolute continuity argument together with Lemma~\ref{lem:dirichlet_energy} in order to transfer this lower bound to obtains a lower bound for the amount of quantum measure in the carpet of $\Gamma$ in $B(z_{i,j},s_{i,j})$ associated with $h$ instead of $\wt{h}_{i,j}$.

Let $S_{i,j}$ be the dyadic square which contains $B(z_{i,j},s_{i,j})$ so that if $\wt{S}_{i,j}$ is the square with the same center and twice the side length then we have that $\wt{S}_{i,j} \cap \eta([0,\tau_j]) = \emptyset$.  Let $\Fh_{S_{i,j}}$ be as in Lemma~\ref{lem:dirichlet_energy}, let $\wh{S}_{i,j}$ be the square with the same center as $S_{i,j}$ and $3/2$ times the side length, and let $\phi$ be a $C_0^\infty$ function which is $1$ on $S_{i,j}$ and $0$ outside of $\wh{S}_{i,j}$.  Then we have that $h - \phi \Fh_{S_{i,j}}$ restricted to $S_{i,j}$ is equal in distribution to a zero boundary GFF on $\wt{S}_{i,j}$ restricted to $S_{i,j}$.  The Radon-Nikodym derivative between the law of the former with respect to the latter is given by
\[ \CZ_{i,j} = \E\left[ \exp\left( (h,-\phi \Fh_{S_{i,j}})_\nabla - \frac{1}{2} \| \phi \Fh_{S_{i,j}} \|_\nabla^2 \right) \giv h|_{S_{i,j}}, h|_{\wt{S}_{i,j}^c} \right].\]
By Lemma~\ref{lem:dirichlet_energy}, we have for all $\epsilon > 0$ sufficiently small that $\CZ_{i,j}^p = O(\epsilon^{-q})$.  By applying H\"older's inequality, we thus see that there exists $\alpha_\LBD > 0$ so that the probability that the amount of mass in $V_{i,j}$ is at least $\epsilon^{\alpha_\LBD}$ is at least $\epsilon^{\alpha_\LBD}$.  By choosing $\alpha_\LBD > 0$ sufficiently large, we thus see that, with overwhelming probability we have that $\qcarpet{h}{\Upsilon}(B(x,\epsilon^{\alpha_\DIAM})) \geq \epsilon^{\alpha_\LBD}$ for all $\epsilon \in (0,\epsilon_0)$.

We iterate the above until the exploration terminates.  Upon doing so, we continue perform the center exploration in each of the components which have been discovered (and correspond to downward jumps).  By Step 1, all of these explorations will discover $O(\epsilon^{-a})$ number of loops with boundary length at least $\epsilon$.  Therefore the result follows by performing a union bound.
\end{proof}

\begin{lemma}
\label{lem:quantum_disk_cover}
Suppose that $\CD = (\D,h,x,y) \sim \qdiskL{\gamma}{1}$.  There exists $\alpha_\net > 0$ so that there a.s.\ exists $\epsilon_0 > 0$ so that for all $\epsilon \in (0,\epsilon_0)$ the following is true.  Let $(z_j)$ be an i.i.d.\ sequence chosen from $\qmeasure{h}$ and let $N_\epsilon = \epsilon^{-\alpha_\net}$.  Then $z_1,\ldots,z_{N_\epsilon}$ forms an $\epsilon$-net of $\D$.
\end{lemma}
\begin{proof}
We first choose $c > 0$ sufficiently large so that with $E_c = \{ \qmeasure{h}(\CD) \geq \epsilon^{-c}\}$ we have that $\p[E_c] = O(\epsilon^\alpha)$; this is possible by \cite[Theorem~1.2]{ag2019disk}.  Lemma~\ref{lem:disk_mass_upper_bound} implies that there exists $\alpha_\LBD > 0$ so that there a.s.\ exists $\epsilon_0 > 0$ so that for all $\epsilon \in (0,\epsilon_0)$ and $z \in \D$ we have that $\qmeasure{h}(B(z,\epsilon)) \geq \epsilon^{\alpha_\LBD}$.  Suppose that $(z_j)$ is an i.i.d.\ sequence chosen from $\qmeasure{h}$.  Suppose that $x \in \D$ is fixed.  Then we have that
\begin{align*}
\p[ z_1,\ldots,z_{N_\epsilon} \notin B(x,\epsilon),\ E_c^c,\ \epsilon < \epsilon_0] \leq (1-\epsilon^{\alpha_\LBD -c})^{N_\epsilon} \leq \exp(-\epsilon^{\alpha_\LBD - c - \alpha_\net}).
\end{align*}
If we choose $\alpha_\net > c-\alpha_\LBD$ we see that the above probability decays to $0$ as $\epsilon \to 0$ faster than any power of $\epsilon$.  The result thus follows by applying the above to the collection of points $(\epsilon \Z)^2 \cap \D$ and using a union bound.
\end{proof}

We can now give the proof of Lemma~\ref{lem:quantum_measure_points_dense}.

\begin{proof}[Proof of Lemma~\ref{lem:quantum_measure_points_dense}]
Suppose that we have the setup as described in the statement of the lemma and suppose that $z$ is picked from $\qmeasure{h}$.  We will first explain why it suffices to show that for every $a > 0$ there exists $b > 0$ (deterministic) and $\epsilon_0 > 0$ (random) so that
\begin{equation}
\label{eqn:measure_points_dense_main_step}
\p[ B(z,\epsilon) \cap \Upsilon \neq \emptyset,\ \qcarpet{h}{\Upsilon}(B(z,\epsilon)) \leq \epsilon^b,\ \epsilon < \epsilon_0] = O(\epsilon^a).
\end{equation}
Assuming~\eqref{eqn:measure_points_dense_main_step}, we claim that we can pick $\alpha_\net > 0$ large enough so that
\begin{equation}
\label{eqn:measure_points_dense_main_step2}
\p[ B(z,\epsilon) \cap \Upsilon \neq \emptyset,\ x_j \notin B(z,\epsilon)\ \forall 1 \leq j \leq N_\epsilon,\ \epsilon < \epsilon_0] = O(\epsilon^a).
\end{equation}
Indeed, this follows since the construction of $\qcarpet{h}{\Upsilon}$ in \cite{msw2020simplecle} implies that $\E[ \qcarpet{h}{\Upsilon}(\Upsilon)] < \infty$ so by Markov's inequality we have that $\p[ \qcarpet{h}{\Upsilon}(\Upsilon) \geq \epsilon^{-a}] = O(\epsilon^a)$.  Fix $\alpha > 0$ and fix $\alpha_\net' > 0$ as in Lemma~\ref{lem:quantum_disk_cover}.  Let $N_\epsilon' = \epsilon^{-\alpha_\net'}$ and let $(z_j)$ be an i.i.d.\ sequence chosen from $\qmeasure{h}$.  Then Lemma~\ref{lem:quantum_disk_cover} implies that by possibly decreasing the value of $\epsilon_0 > 0$ we a.s.\ have for all $\epsilon \in (0,\epsilon_0)$ that $\Upsilon \subseteq \cup_{j=1}^{N_\epsilon'} B(z_j,\epsilon)$.  We therefore have that
\begin{align*}
  \p[ \Upsilon \subseteq \cup_{j=1}^{N_\epsilon} B(x_j,4\epsilon),\ \epsilon < \epsilon_0]
&\geq \p[\cap_{j=1}^{N_\epsilon'} \{ B(z_j,\epsilon) \cap \Upsilon \neq \emptyset \Rightarrow \exists i : x_i \in B(z_j,\epsilon)\},\ \epsilon < \epsilon_0]\\
&\geq 1 - N_\epsilon' \p[ B(z,\epsilon) \cap \Upsilon \neq \emptyset,\ x_i \notin B(z,\epsilon)\ \forall 1 \leq i \leq N_\epsilon,\ \epsilon < \epsilon_0]\\
&= 1  - N_\epsilon' O(\epsilon^a),
\end{align*}
where in the last step we used~\eqref{eqn:measure_points_dense_main_step2}.  By choosing $a = \alpha + \alpha_\net'$, this gives us the statement we want to prove, up to a redefinition of $\epsilon$.

Fix $c > 0$.  By Lemma~\ref{lem:number_of_loops} there exists $d > 0$ so that the number of loops of $\Gamma$ with quantum length in $[2^{-k-1},2^{-k}]$ is $O(2^{(4/\kappa+1/2+c) k})$ of an event with probability $O(2^{-d k})$.  By \cite[Theorem~1.2]{ag2019disk}, the expected quantum area surrounded by such a loop is $O(2^{-2k})$.  Therefore the quantum area surrounded by all such loops of $\Gamma$ is $O(2^{(c+d+4/\kappa-3/2) k})$ off an event with probability $O(2^{-dk})$.  Note that $4/\kappa-3/2 < 0$ since $\kappa \in (8/3,4)$ so we can choose $c,d > 0$ sufficiently small so that $c+d+4/\kappa-3/2 < 0$.  It therefore follows that the quantum area surrounded by loops of $\Gamma$ with quantum length at most $2^{-k}$ is also $O(2^{(c+d+4/\kappa-3/2) k})$ off an event with probability $O(2^{-d k})$.  By taking $k = \lfloor \log \epsilon^{-1} \rfloor$, we see that the quantum area surrounded of those loops of $\Gamma$ with quantum length at most $\epsilon$ is $O(\epsilon^{c+d+3/2-4/\kappa})$ off an event with probability $O(\epsilon^d)$.  As $\qmeasure{h}(\Upsilon) = 0$ a.s., each $z_j$ is a.s.\ surrounded by a loop of $\Gamma$.  The above discussion implies that there exists $b > 0$ so that the probability that any one of the $z_j$ with $1 \leq j \leq N_\epsilon'$ is surrounded by a loop of $\Gamma$ with quantum length which is at most $\epsilon^b$ is $O(\epsilon^a)$.  Lemma~\ref{lem:mass_near_by} implies we have (possibly decreasing $\epsilon_0 > 0$) that if $B(z_j,\epsilon) \cap \Upsilon \neq \emptyset$ and $z_j$ is not surrounded by a loop of $\Gamma$ with quantum length at most $\epsilon^b$ then $\qcarpet{h}{\Upsilon}(B(z_j,2\epsilon)) \geq \epsilon^{\alpha_\LBD}$.  Combining proves~\eqref{eqn:measure_points_dense_main_step}, which completes the proof.
\end{proof}

\section{L\'evy process estimates}
\label{app:levy_process}

Suppose that $\CH = (\h,h,0,\infty)$ is a quantum half-plane and $\eta$ is an independent $\SLE_\kappa^0(\kappa-6)$ on $\h$ from $0$ to $\infty$ which is subsequently parameterized by the quantum natural time of its trunk.  For each $t \geq 0$, we let $L_t$ (resp.\ $R_t$) denote the change in the quantum length of the left (resp.\ right) side of the outer boundary of $\eta([0,t])$.  In \cite[Section~4]{msw2020simplecle}, it is shown that $L_t, R_t$ are independent $4/\kappa$-stable L\'evy processes with the same law.  Moreover, if we write the L\'evy measure of $L_t$ (equivalently, $R_t$) as $a_+ |s|^{-4/\kappa-1} \one_{s \geq 0} ds + a_- |s|^{-4/\kappa-1} \one_{s < 0} ds$ then we have that the ratio $u = a_+ / a_-$ of the intensities of the upward to downward jumps is given by
\[ u = -\cos(4\pi/\kappa).\]
Recall that $\beta = (a_+ - a_-)/(a_+ + a_-)$.  Assuming that we have normalized the process so that $a_+ + a_- = 1$, the relationship between $u$ and $\beta$ is thus given by
\[ u = \frac{1+\beta}{1-\beta} \quad\text{hence}\quad \beta = \frac{-1+u}{1+u}.\]
Therefore
\begin{equation}
\label{eqn:beta_formula}
\beta = -(\cot(2\pi/\kappa))^2 = -(\cot(\pi \kappa'/8))^2.
\end{equation}
Recall from \cite[Chapter VIII]{bertoin1996levy} that the positivity parameter $P$ of an $\alpha$-stable L\'evy process as a function of $\beta$ is given by
\[ P = \frac{1}{2} + \frac{1}{\pi \alpha} \arctan(\beta \tan(\pi \alpha/2)).\]
In our case, $\alpha = 4/\kappa = \kappa'/4$.  Plugging in the value of $\beta$ from~\eqref{eqn:beta_formula} and this value of $\alpha$ gives
\[ P = 1-\frac{\kappa}{8}.\]

Fix $\kappa \in (8/3,4)$.  Let $X^1, X^2$ be i.i.d.\ $4/\kappa$-stable L\'evy processes $X^1,X^2$ with positivity parameter $P$ and starting from $0$ (so that $(X^1,X^2) \stackrel{d}{=} (L,R)$ from above).  Let $I_t^j = \inf_{0 \leq s \leq t} X_s^j$ and $S_t^j = \sup_{0 \leq s \leq t} X_s^j$ for $j=1,2$, respectively, be the running infimum and supremum of $X^j$.  We let $\tau^j = \inf\{t \geq 1 : X_t^j = I_t^j\}$ for $j=1,2$ and $\tau = \tau^1 \wedge \tau^2$.

\begin{lemma}
\label{lem:tauj_tail}
We have that $\p[\tau^1 \geq x] \asymp x^{-\kappa/8}$ and $\p[\tau^1 \geq x, I_1^1 \geq -1] \asymp x^{-\kappa/8}$ as $x \to \infty$.
\end{lemma}
\begin{proof}
Let $\wt{X}$ be an independent copy of $X^1$ and let $\wt{I}$ be its running infimum.  Then we have that
\begin{align*}
   \p[ \tau^1 \geq x+1 \giv I_1^1, X_1^1]
&= \p[ \wt{I}_x \geq I_1^1 - X_1^1]
 = \p[ \wt{I}_1 \geq x^{-\kappa/4}(I_1^1 - X_1^1)]\\
&\asymp x^{-\kappa/8} (X_1^1 - I_1^1)^{1/2},
\end{align*}
where in the last line we applied \cite[Chapter VIII, Proposition 2]{bertoin1996levy}.  Taking expectations of both sides and using that $|X_1^1 - I_1^1|$ has a finite mean hence a finite $1/2$-moment, we see that $\p[\tau^1 \geq x] \asymp x^{-\kappa/8}$.  This proves the first assertion of the lemma.  The same argument gives the second assertion except we take a conditional expectation of $X_1^1$ given $I_1^1 \geq -1$.
\end{proof}

As a consequence of Lemma~\ref{lem:tauj_tail}, we have that
\begin{equation}
\label{eqn:tautail}
\p[ \tau \geq x] = \p[\tau^1 \geq x] \p[\tau^2 \geq x] \asymp x^{-\kappa/4} \quad\text{and}\quad \p[ \tau \geq x, I_1^1 \geq -1] \asymp x^{-\kappa/4} \quad\text{as}\quad x \to \infty.
\end{equation}

Fix $M \geq 1$.

\begin{lemma}
\label{lem:bottom_length_moment_bound}
There exists a constant $c > 0$ so that $-\E[ I_{\tau \wedge M}] \geq c (\log M)$.
\end{lemma}

Before we give the proof of Lemma~\ref{lem:bottom_length_moment_bound}, we first need to collect the following lemma.

\begin{lemma}
\label{lem:levy_moment_estimate}
There exists a constant $c > 0$ so that the following is true.  For each $k \in \N$ let $\sigma_k^1 = \inf\{t \geq 0 : X_t^1 \geq k^{\kappa/4}\}$.  Then
\[ \E[ X_{\sigma_k^1}^1 \giv  \tau^1 \geq \sigma_k^1] \leq c k^{\kappa/4}.\]
\end{lemma}
\begin{proof}
Fix $\delta > 0$.  For each $j \geq 0$, we let $S_j^{\delta,1} = \sup_{t \in [j \delta,(j+1) \delta]} X_t^1$.  Fix $u \geq k^{\kappa/4}$.  Then we have that
\begin{align}
     &\p[ X_{\sigma_k^1}^1 \geq u \giv \tau^1 \geq \sigma_k^1 ]
\leq \sum_{j=0}^\infty \p[ S_j^{\delta,1} \geq u, \sigma_k^1 \geq j \delta \giv \tau^1 \geq \sigma_k^1] \notag\\
=& \sum_{j=0}^\infty \frac{\p[ S_j^{\delta,1} \geq u, \sigma_k^1 \geq j \delta, \tau^1 \geq \sigma_k^1]}{\p[\tau^1 \geq \sigma_k^1]} \notag\\
\leq& \sum_{j=0}^\infty \frac{\p[ S_j^{\delta,1} \geq u, \sigma_k^1 \in [j\delta,(j+1)\delta], \tau^1 \geq j \delta]}{\p[\tau^1 \geq \sigma_k^1]} \notag\\
=& \sum_{j=0}^\infty \frac{\p[ S_j^{\delta,1} \geq u \giv \sigma_k^1 \in [j\delta,(j+1)\delta], \tau^1 \geq j \delta]\p[\sigma_k^1 \in [j\delta,(j+1)\delta], \tau^1 \geq j \delta]}{\p[\tau^1 \geq \sigma_k^1]}. \label{eqn:first_bound}
\end{align}
We note for $v < k^{\kappa/4}$ that
\begin{align*}
&\p[ S_j^{\delta,1} \geq u \giv \sigma_k^1 \in [j\delta,(j+1)\delta], \tau^1 \geq j \delta, X_{j\delta}^1 = v]
=\p[ S_j^{\delta,1} \geq u \giv S_j^{\delta,1} \geq k^{\kappa/4}, \sigma_k^1 \geq j\delta, \tau^1 \geq j \delta, X_{j\delta}^1 = v]\\
=&\p[ S_j^{\delta,1} \geq u \giv S_j^{\delta,1} \geq k^{\kappa/4}, X_{j\delta}^1 = v]
 = \frac{\p[ S_j^{\delta,1} \geq u \giv X_{j \delta}^1 =v]}{\p[ S_j^{\delta,1} \geq k^{\kappa/4} \giv X_{j \delta}^1 =v ]}.
\end{align*}
By \cite[Chapter VIII, Proposition~4]{bertoin1996levy}, the ratio above converges to $(u-v)^{-4/\kappa}/(k^{\kappa/4}-v)^{-4/\kappa}$ as $\delta \to 0$.  Assume that $u \geq 2 k^{4/\kappa}$.  Then we have for a constant $c > 0$ that
\[ \left(\frac{u-v}{k^{\kappa/4}-v} \right)^{-4/\kappa} = u^{-4/\kappa} \left(\frac{k^{\kappa/4}-v}{1-v/u} \right)^{4/\kappa} \leq c k u^{-4/\kappa}.\]
Altogether, taking a limit as $\delta \to 0$ and inserting these bounds into~\eqref{eqn:first_bound} implies for a constant $c > 0$ that
\[ \p[ X_{\sigma_k^1}^1 \geq u \giv \tau^1 \geq \sigma_k^1 ] \leq c k u^{-4/\kappa} \quad\text{for}\quad u \geq 2k^{\kappa/4}.\]
Therefore for constants $c_1,c_2 > 0$ we have that
\begin{align*}
\E[ X_{\sigma_k^1}^1 \giv  \tau^1 \geq \sigma_k^1]
&\leq 2k^{\kappa/4} + \E[ X_{\sigma_k^1}^1 \one_{\{X_{\sigma_k^1}^1 \geq 2k^{\kappa/4}\}}  \giv  \tau^1 \geq \sigma_k^1]\\
&\leq 2 k^{\kappa/4} + c_1 k \int_{2k^{\kappa/4}}^\infty u^{-4/\kappa} du \leq c_2 k^{\kappa/4}.
\end{align*}
\end{proof}

\begin{proof}[Proof of Lemma~\ref{lem:bottom_length_moment_bound}]
We first claim that
\begin{equation}
\label{eqn:x_giv_tau_tail}
\p[ X_k^1 \geq R k^{\kappa/4} \giv \tau \geq k, I_1^1 \geq -1] \to 0 \quad\text{as}\quad R \to \infty \quad\text{uniformly in}\quad k.
\end{equation}
To see this, fix $R \geq 2$. For each $k$, let $\sigma_k^1 = \inf\{t \geq 0 : X_t^1 \geq k^{\kappa/4}\}$ be as in Lemma~\ref{lem:levy_moment_estimate}.  Then we have that
\begin{align}
   \p[ X_k^1 \geq R k^{\kappa/4} \giv \tau \geq k,I_1^1 \geq -1]
&= \p[ X_k^1 \geq R k^{\kappa/4} \giv \tau^1 \geq k,I_1^1 \geq -1] \quad\text{($X^1, X^2$ independent)} \notag\\
&= \frac{\p[ X_k^1 \geq R k^{\kappa/4}, \tau^1 \geq k,I_1^1 \geq -1]}{\p[\tau^1 \geq k,I_1^1 \geq -1]} \notag\\
&\lesssim \frac{\p[ X_k^1 \geq R k^{\kappa/4}, \tau^1 \geq k]}{\p[\tau^1 \geq k]} \quad\text{(by Lemma~\ref{lem:tauj_tail})} \notag\\
&\leq \frac{\p[ X_k^1 \geq R k^{\kappa/4}, \tau^1 \geq \sigma_k^1]}{\p[\tau^1 \geq k \geq \sigma_k^1]} \notag\\
&\leq \frac{\p[ X_k^1 \geq R k^{\kappa/4} \giv \tau^1 \geq \sigma_k^1]\p[\tau^1 \geq \sigma_k^1]}{\p[\tau^1 \geq k \geq \sigma_k^1]}. \label{eqn:x_k_given_bound}
\end{align}
We note that
\begin{align*}
\p[ \tau^1 \geq k \geq \sigma_k^1]
&\leq \p[\tau^1 \geq k, \tau^1 \geq \sigma_k^1]
 = \p[ \tau^1 \geq k \giv \tau^1 \geq \sigma_k^1] \p[ \tau^1 \geq \sigma_k^1].
\end{align*}
It is easy to see that $\p[ \tau^1 \geq k \giv \tau^1 \geq \sigma_k^1]$ is bounded from below by a positive constant for $k \in \N$.  Consequently, \eqref{eqn:x_k_given_bound} is bounded from above by a constant times $\p[ X_k^1 \geq R k^{\kappa/4} \giv \tau^1 \geq \sigma_k^1]$ which tends to $0$ as $R \to \infty$ uniformly in $k$ by Markov's inequality and Lemma~\ref{lem:levy_moment_estimate}.  This proves~\eqref{eqn:x_giv_tau_tail}.

Assume that we have chosen $R \geq 2$ sufficiently large so that
\begin{equation}
\label{eqn:x1alphakbound}
\p[ X_k^1 \leq R k^{\kappa/4} \giv \tau \geq k, I_1^1 \geq -1] \geq 1/2 \quad\text{for all}\quad k.
\end{equation}
Let $E_k^1$ be the event that $X^1|_{[k,k+1]}$ makes a downward jump of size at least $(R+1)k^{\kappa/4}$.  Since the number of such jumps that $X^1|_{[k,k+1]}$ makes is distributed as a Poisson random variable with mean proportional to $k^{-1}$, we have that
\begin{equation}
\label{eqn:ek1givlbd}	
\p[E_k^1 \giv X_k^1 \leq R k^{\kappa/4}, \tau \geq k, I_1^1 \geq -1] = \p[ E_k^1] \gtrsim k^{-1}.
\end{equation}
It therefore follows that there exist constants $c_1,c_2,c_3 > 0$ so that
\begin{align*}
   -\E[ I_{\tau \wedge M}]
&\geq \sum_{k=1}^M k^{\kappa/4} \p[E_k^1,\ \tau \geq k, X_k^1 \leq R k^{\kappa/4}, I_1^1 \geq -1]\\
&\geq \sum_{k=1}^M k^{\kappa/4} \times c_1 k^{-1} \p[\tau \geq k, I_1^1 \geq -1] \quad\text{(by~\eqref{eqn:x1alphakbound}, \eqref{eqn:ek1givlbd})}\\
&\geq \sum_{k=1}^M c_2 k^{-1} \quad\text{(by~\eqref{eqn:tautail})}\\
 &\geq c_3 (\log M).
\end{align*}
This proves the result.
\end{proof}

\begin{lemma}
\label{lem:top_length_moment_bound}
We have that 
$\E[ (X_M^1 - I_M^1) \one_{\tau \geq M}] = O(1)$.
\end{lemma}
\begin{proof}
On the event $\tau \geq M$, we have that $I_M^1 = I_1^1$.  Fix $p \in (1,4/\kappa)$ and let $q > 1$ be such that $p^{-1} + q^{-1} = 1$.  Since $|I_1^1|$ has a finite $p$th moment \cite[Chapter VIII, Proposition~4]{bertoin1996levy}, we have for a constant $c > 0$ that
\begin{align}
     \E[ -I_M^1 \one_{\tau \geq M}]
&= \E[ -I_1^1 \one_{\tau \geq M}]
 \leq \E[ |I_1^1|^p]^{1/p} \p[ \tau \geq M]^{1/q}
 \leq  c M^{-q \kappa/4} = O(1). \label{eqn:infbound}
\end{align}
Note that we used~\eqref{eqn:tautail} in the second inequality.

Let $\sigma_M^1 = \inf\{t \geq 0 : X_t^1 \geq M^{\kappa/4}\}$.  By the independence of $X^1,X^2$, we have that
\begin{align}
\label{eqn:xm_bound1}
\E[ X_M^1 \one_{\tau \geq M} (\one_{\sigma_M^1 \leq M} + \one_{\sigma_M^1 > M})]
&\asymp M^{-\kappa/8} \E[ X_M^1 \one_{\tau^1 \geq M} (\one_{\sigma_M^1 \leq M} + \one_{\sigma_M^1 > M})]
\end{align}
For each $t \geq 0$ let $\CF_t^1 = \sigma(X_s^1 : s \leq t)$.  We moreover have that
\begin{align}
\E[ X_M^1 \one_{\tau^1 \geq M} (\one_{\sigma_M^1 \leq M} + \one_{\sigma_M^1 > M})]
&\leq M^{\kappa/4}\p[ \tau^1 \geq M] + \E[ X_M^1 \one_{\sigma_M^1 \leq M \leq \tau^1}] \notag\\
&\lesssim M^{\kappa/8} + \E[ X_M^1 \one_{\sigma_M^1 \leq M \leq \tau^1}] \quad\text{(by Lemma~\ref{lem:tauj_tail})} \notag\\
&\leq M^{\kappa/8} + \E[ |X_M^1 - X_{\sigma_M^1}^1| \one_{\sigma_M^1 \leq M \leq \tau^1}] + \E[ X_{\sigma_M^1}^1 \one_{\sigma_M^1 \leq M \leq \tau^1}] \notag\\
&\lesssim M^{\kappa/8} + \E[ |X_M^1 - X_{\sigma_M^1}^1| \one_{\sigma_M^1 \leq M \leq \tau^1}] \quad\text{(by Lemmas~\ref{lem:tauj_tail},  \ref{lem:levy_moment_estimate})}. \label{eqn:xm_bound_int}
\end{align}
We emphasize that in the last line we used that $ \E[ X_{\sigma_M^1}^1 \one_{\sigma_M^1 \leq M \leq \tau^1}] \leq \E[ X_{\sigma_M^1}^1 \giv \sigma_M^1 \leq \tau^1] \p[ \sigma_M^1 \leq \tau^1]$ and $\p[ \tau^1 \geq M \giv \tau^1 \geq \sigma_M^1]$ is positive uniformly in $M$ so that $\p[ \tau^1 \geq \sigma_M^1] \lesssim \p[ \tau^1 \geq M] \asymp M^{-\kappa/8}$.  We further have that~\eqref{eqn:xm_bound_int} is bounded from above by
\begin{align} 
& M^{\kappa/8} + \E[ \sup_{0 \leq t \leq M} |X_{\sigma_M^1+t}^1 - X_{\sigma_M^1}^1| \one_{\sigma_M^1 \leq M \leq \tau^1}] \notag\\
\leq& M^{\kappa/8} + \E[ \sup_{0 \leq t \leq M} |X_{\sigma_M^1+t}^1 - X_{\sigma_M^1}^1| \one_{\sigma_M^1 \leq \tau^1}] \notag\\
=& M^{\kappa/8} + \E[ \E[ \sup_{0 \leq t \leq M} |X_{\sigma_M^1+t}^1 - X_{\sigma_M^1}^1| \giv \CF_{\sigma_M^1}^1] \one_{\sigma_M^1 \leq \tau^1}] \notag\\
\leq& M^{\kappa/8} + \E[S_M^1] \p[ \sigma_M^1 \leq \tau^1]. \label{eqn:xm_bound2}
\end{align}
In the final inequality, we used that $\E[ \sup_{0 \leq t \leq M} |X_{\sigma_M^1+t}^1 - X_{\sigma_M^1}^1| \giv \CF_{\sigma_M^1}^1] = \E[ S_M^1]$.  As we explained just above, $\p[ \sigma_M^1 \leq \tau^1] \lesssim M^{-\kappa/8}$ and it follows from \cite[Chapter VII, Corollary~2]{bertoin1996levy} that $\E[S_M^1] = M^{\kappa/4} \E[ S_1^1] \asymp M^{\kappa/4}$.  Combining this with~\eqref{eqn:xm_bound1} and~\eqref{eqn:xm_bound2} implies that $\E[ X_M^1 \one_{\tau \geq M}] = O(1)$ which, when combined with~\eqref{eqn:infbound}, proves the result.
\end{proof}

\begin{lemma}
\label{lem:stablejumpsum}
Fix $\alpha \in (1,2)$ and suppose that $X$ is an $\alpha$-stable L\'evy process with upward (and possibly downward) jumps.  There exists a constant $c > 0$ so that the following is true.  For each $n \in \N$, let $S_n$ be the sum of the largest $n$ upward jumps made by $X|_{[0,1]}$.  Then
\[ \E[S_n] \leq c n^{1-1/\alpha}.\]
\end{lemma}
\begin{proof}
\noindent{\it Step 1: Setup.}  Fix $k \in \Z$.  Recall that the number $N_k$ of upward jumps made by $X|_{[0,1]}$ of size in $[e^k,e^{k+1}]$ is distributed as a Poisson random variable with mean $m_k$ proportional to $e^{-\alpha k}$.

Let $k_0 \in \Z$ be such that $k \geq k_0$ if and only if $m_k \leq 1/2$.  Let $k_n$ be the largest value of $k \in \Z$ so that $m_k \geq 2 n$.  For each $k \in \N$, let $E_k = \{ N_k < n\}$.  Then if we set
\begin{align*}
S^1 = \sum_{k=k_0}^\infty N_k,\quad
S^2 = \sum_{k=k_n}^{k_0-1} N_k, \quad\text{and}\quad
S^3 = \sum_{j=-\infty}^{k_n-1} N_j \one_{E_{j+1}},
\end{align*}
we have that $S_n \leq S^1 + S^2 + S^3$.

\noindent{\it Step 2: Bound for $\E[S^1]$.} As $\alpha > 1$,
\begin{equation}
\label{eqn:s1bound}
\E[S^1] \leq \sum_{k=k_0}^\infty e^{k+1} m_k \asymp \sum_{k=k_0}^\infty e^{(1-\alpha)k} < \infty.
\end{equation}

\noindent{\it Step 3: Bound for $\E[S^2]$.}  Note that $k_n = -\alpha^{-1} \log n + O(1)$.  We thus have that
\begin{equation}
\label{eqn:s2bound}
\E[S^2] \leq \sum_{k = k_n}^{k_0-1} e^{k+1} m_k \asymp \sum_{k= k_n}^{k_0-1} e^{(1-\alpha) k} \asymp n^{1-1/\alpha}.
\end{equation}

\noindent{\it Step 4: Bound for $\E[S^3]$.} Since $k \leq k_n$ implies that $m_k \geq 2n$, by Poisson concentration there exists a constant $c > 0$ so that
\begin{equation}
\label{eqn:nksmall}	
\p[ E_k ] \leq \exp(- c m_k) \quad\text{for all}\quad k \leq k_n.
\end{equation}
We thus have that
\begin{align}
\E[S^3]
&\leq \sum_{k=-\infty}^{k_n-1} e^k \E[ N_k] \p[E_{k+1}] \quad\text{(independence of $N_k$ and $E_{k+1}$)} \notag\\
&\lesssim \sum_{k=-\infty}^{k_n-1} e^k m_k e^{-c m_k} \quad\text{(by~\eqref{eqn:nksmall})} \notag\\
&< \infty. \label{eqn:s3bound}
\end{align}

The result thus follows by combining~\eqref{eqn:s1bound}, \eqref{eqn:s2bound}, and~\eqref{eqn:s3bound}.
\end{proof}

\bibliographystyle{abbrv}
\bibliography{references}

\end{document}